\documentclass[a4paper,10pt]{amsart}
\usepackage[utf8]{inputenc}
\usepackage{lmodern}
\usepackage[T1]{fontenc}
\usepackage{microtype} 	
\usepackage{enumitem}
\usepackage{amsmath,amssymb,amsfonts,amsthm,latexsym,mathrsfs}
\usepackage{hyperref}
\hypersetup{colorlinks=true,linkcolor=red,citecolor=blue}
\usepackage{booktabs}
\usepackage{longtable}
\usepackage{lipsum}
\theoremstyle{plain}

\newtheorem*{theorem*}{{\bf Theorem}}

\newtheorem*{corollary*}{{\bf Corollary}}
\newtheorem{proposition}[subsection]{{\bf Proposition}}

\theoremstyle{definition}

\theoremstyle{remark}

\numberwithin{equation}{section}
\DeclareMathOperator{\B}{B}
\DeclareMathOperator{\M}{M}
\DeclareMathOperator{\K}{K}
\DeclareMathOperator{\Hom}{Hom}
\DeclareMathOperator{\cwedge}{\curlywedge}
\begin{document}
\baselineskip=14pt
\title{Bogomolov multipliers of all groups of order $128$}
\author[U. Jezernik]{Urban Jezernik}
\address[Urban Jezernik]{
Institute of Mathematics, Physics, and Mechanics \\
Jadranska 19 \\
1000 Ljubljana \\
Slovenia}
\thanks{}
\email{urban.jezernik@imfm.si}
\author[P. Moravec]{Primo\v{z} Moravec}
\address[Primo\v{z} Moravec]{
Department of Mathematics \\
University of Ljubljana \\
Jadranska 21 \\
1000 Ljubljana \\
Slovenia}
\thanks{}
\email{primoz.moravec@fmf.uni-lj.si}
\thanks{}
\date{\today}
\begin{abstract}
\noindent This note is an implementation of the algorithm for computing Bogomolov multipliers as given in \cite{Mor11} in combination with \cite{Eic08} to effectively determine the multipliers of groups of order $128$.
The two serving purposes are a continuation of the results \cite{Chu08,Chu09}, and an application \cite{Jez13}.
\end{abstract}
\maketitle
\section{Introduction}
\label{sec:introduction}

\noindent 
Let $G$ be a group and let $G\cwedge G$ be the group generated by the symbols $x\cwedge y$ for all pairs $x,y\in G$, subject to the following
relations:
\[	
\begin{aligned}
\label{eq:ext}
xy\cwedge z = (x^y\cwedge z^y)(y\cwedge z), \quad 
x\cwedge yz = (x\cwedge z)(x^z\cwedge y^z), \quad
a\cwedge b = 1,
\end{aligned}
\]
for all $x,y,z\in G$ and all $a,b\in G$ with $[a,b]=1$. The group $G\cwedge G$ was first studied in \cite{Mor11} and is called the {\it curly exterior square} of $G$. There is a canonical epimorphism $G\cwedge G\to [G,G]$ whose kernel is denoted by $\B_0(G)$. Its significance was pointed
out in \cite{Mor11} where it was shown that $\Hom(\B_0(G),\mathbb{Q}/\mathbb{Z})$ is naturally isomorphic to the unramifed Brauer group of a field extension $\mathbb{C}(V)^G/\mathbb{C}$ over $\mathbb{Q}/\mathbb{Z}$. The unramified Brauer group is a well known obstruction to Noether's problem \cite{Noe16} asking whether or not $\mathbb{C}(V)^G$ is purely transcendental over $\mathbb{C}$. Following Kunyavski\u\i \ \cite{Kun08}, we say that $\B_0(G)$ is the {\it Bogomolov multiplier} of $G$. Bogomolov multipliers can also be interpreted as measures of how the commutator relations in groups fail to follow from the so-called universal ones, see \cite{Jez13} for further details.

Based on the above description of Bogomolov multipliers, an algorithm for computing $\B_0(G)$ and $G\cwedge G$ when $G$ is a polycyclic group was developed in \cite{Mor11}. The purpose of this paper is to describe a new algorithm for computing Bogomolov multipliers and curly exterior squares of polycyclic groups. It is based on an algorithm for computing Schur multipliers that was developed by Eick and Nickel \cite{Eic08}, and a Hopf-type formula for $\B_0(G)$ that was found in \cite{Mor11}. An advantage of the new algorithm is that it enables a systematic trace of which elements of $\B_0(G)$ are in fact non-trivial, thus providing an efficient tool of double-checking non-triviality of Bogomolov multipliers by hand. The algorithm has been implemented in {\sf GAP} \cite{GAP}.

Hand calculations of Bogomolov multipliers were done for groups of order 32 by 
Chu, Hu, Kang, and Prokhorov \cite{Chu08}, and groups of order 64 by
Chu, Hu, Kang, and Kunyavski\u\i \ \cite{Chu09}. In a similar way, Bogomolov multipliers of groups of order $p^5$ were determined in \cite{Hos11,Hos12}, and for groups of order $p^6$ this was done recently by Chen and Ma \cite{Che13}.
We apply the above mentioned algorithm to determine Bogomolov multipliers of all groups of order 128. Our contribution is an explicit description of generators of Bogomolov multipliers of these groups. There are 2328 groups of order 128, and they were classified by James, Newman, and O'Brien \cite{Jam90}. Instead of considering all of them, we use the fact \cite{Jam90} that these groups belong to 115 isoclinism families according to Hall \cite{Hal40}, together with the fact that isoclinic groups have isomorphic Bogomolov multipliers \cite{Mor14}.
It turns out that there are precisely eleven isoclinism families whose Bogomolov multipliers are non-trivial. For each of these families we explicitly determine $\B_0(G)$ for a chosen representative $G$.
We mention that the results of this paper form a basis for proving the main result of \cite{Jez13}.

The outline of the paper is as follows. In Section \ref{sec:algorithm} we describe the new algorithm for computing Bogomolov multipliers and curly exterior squares of polycyclic groups.
We then proceed to determine the multipliers of groups of order $128$.
A short summary of the results is provided in Section \ref{sec:summary},
and the full details of all the calculations are given in Section \ref{sec:calculations}.

This manuscript is an extended version of \cite{arx} where only the groups of order 128 that yield non-trivial Bogomolov multipliers are considered.

\section{The algorithm}
\label{sec:algorithm}

\noindent 
Let $G$ be a finite polycyclic group, presented by a power-commutator presentation with a polycyclic generating sequence $g_i$ with $1 \leq i \leq n$ for some $n$ subject to the relations
\[
	\begin{aligned}
	g_i^{e_i} & = \textstyle \prod_{k=i+1}^n g_k^{x_{i,k}} & & \text{ for } 1 \leq i \leq n, \\
	[g_i, g_j] & = \textstyle \prod_{k=i+1}^n g_k^{y_{i,j,k}} & & \text{ for } 1 \leq j < i \leq n.
	\end{aligned}
\]
Note that when printing such a presentation, we hold to standard practice and omit the trivial commutator relations.
For every relation except the trivial commutator relations (the reason being these get factored out in the next step), introduce a new abstract generator, a so-called {\em tail}, append the tail to the relation, and make it central.
In this way, we obtain a group generated by $g_i$ with $1 \leq i \leq n$ and $t_{\ell}$ with $1 \leq \ell \leq m$ for some $m$, subject to the relations
\[
	\begin{aligned}
	g_i^{e_i} & = \textstyle \prod_{k=i+1}^n g_k^{x_{i,k}} \cdot t_{\ell(i)} & & \text{ for } 1 \leq i \leq n, \\
	[g_i, g_j] & = \textstyle \prod_{k=i+1}^n g_k^{y_{i,j,k}} \cdot t_{\ell(i,j)} & & \text{ for } 1 \leq j < i \leq n,
	\end{aligned}
\]
with the tails $t_{\ell}$ being central.
This presentation gives a central extension $G_{\emptyset}^{*}$ of $\langle t_{\ell} \mid 1 \leq \ell \leq m \rangle$ by $G$, but the given relations may not determine a consistent power-commutator presentation.
Evaluating the consistency relations 
\[
	\begin{aligned}
	g_k (g_j g_i) &= (g_k g_j) g_i & & \text{ for } k > j > i, \\
	(g_j^{e_j}) g_i &= g_j^{e_j - 1}(g_i g_i) & & \text{ for } j > i, \\
	g_j (g_i^{e_i}) &= (g_jg_i) g_i^{e_i - 1} & & \text{ for } j > i, \\
	(g_i^{e_i}) g_i &= g_i (g_i^{e_i}) & & \text{ for all } i
	\end{aligned}
\]
in the extension gives a system of relations between the tails.
Having these in mind, the above presentation of $G_{\emptyset}^{*}$ amounts to a pc-presented quotient of the universal central extension $G^{*}$ of the quotient system, backed by the theory of the {\em tails routine} and {\em consistency checks}, see \cite{Nic93,Sim94,Eic08}.
Beside the consistency enforced relations, we evaluate the commutators $[g,h]$ in the extension with the elements $g,h$ commuting in $G$, which potentially impose some new tail relations.
In the language of exterior squares, this step amounts to determining the subgroup $\M_0(G)$ of the Schur multiplier, see \cite{Mor11}.
The procedure may be simplified by noticing that the conjugacy class of a single commutator induces the same relation throughout.
Let $G_0^*$ be the group obtained by factoring $G_{\emptyset}^*$ by these additional relations.
Computationally, we do this by applying Gaussian elimination over the integers to produce a generating set for all of the relations between the tails at once, and collect them in a matrix $T$.
Applying a transition matrix $Q^{-1}$ to obtain the Smith normal form of $T = PSQ$ gives a new basis for the tails, say $t_{\ell}^*$.
The abelian invariants of the group generated by the tails are recognised as the elementary divisors of $T$.
Finally, the Bogomolov multiplier of $G$ is identified as the torsion subgroup of $\langle t_{\ell}^* \mid 1 \leq \ell \leq m \rangle$ inside $G_0^{*}$, the theoretical background of this being the following proposition.

\begin{proposition}
\label{p:torsion}
Let $G$ be a finite group, presented by $G = F/R$ with $F$ free of rank $n$.
Denote by $\K(F)$ the set of commutators in $F$.
Then $\B_0(G)$ is isomorphic to the torsion subgroup of $R/\langle \K(F) \cap R \rangle$, and the torsion-free factor $R/([F,F]\cap R)$ is free abelian of rank $n$.
Moreover, every complement $C$ to $\B_0(G)$ in $R/\langle \K(F) \cap R \rangle$ yields a commutativity preserving central extension of $\B_0(G)$ by $G$.
\end{proposition}
\begin{proof}
Using the Hopf-type formula for the Bogomolov multiplier $\B_0(G)\cong ([F,F]\cap R)/\langle \K(F)\cap R\rangle$ from \cite{Mor11}, the proposition follows from the  arguments given in \cite[Corollary 2.4.7]{Kar87}.
By construction and \cite{Eic08}, we have $G_0^{*} \cong F/\langle \K(F) \cap R \rangle$,
and the complement $C$ gives the extension $G_0^{*}/C$.
\end{proof}

Taking the derived subgroup of the extension $G_0^{*}$ and factoring it by a complement of the torsion part of the subgroup generated by the tails thus gives a consistent power-commutator presentation of the curly exterior square $G \curlywedge G$, see \cite{Eic08,Mor11}.
With each of the groups below, we also output the presentation of $G_0^{*}$ factored by a complement of $\B_0(G)$ and expressed in the new tail basis $t_i^*$ as to explicitly point to the nonuniversal commutator relations with respect to the commutator presentation of the original group.
\section{A summary of results}
\label{sec:summary}

\noindent There are precisely $11$ isoclinism families of groups of order $128$ whose Bogomolov multipliers are nontrivial, see Table \ref{tab:report}.
These are the families $\Phi_i$ with $i \in \{ 16, 30, 31, 37, 39, 43, 58, 60, 80, 106, 114 \}$ of \cite{Jam90}.
Their multipliers are all isomorphic to $C_2$, except those of the family $\Phi_{30}$ with which we get $C_2 \times C_2$.
The exceptional groups belonging to the latter family have been, together with their odd prime counterpart, further investigated in \cite{Jez13}.
For each of the families with nontrivial multipliers, we also give the identification number as implemented in {\sf GAP} \cite{GAP} of a selected representative that was used for determining the family's multiplier.

\begin{table}[htb]
\caption{Isoclinism families of groups of order $128$ with nontrivial Bogomolov multipliers.}
\label{tab:report}
\begin{longtable}{ccc}
\toprule
Family & {\sf GAP} ID & $\B_0$ \\
\midrule \endhead 
$\ref{number:16}$ & $227$ & $C_{2}$  \\
$\ref{number:30}$ & $1544$ & $C_{2} \times C_{2}$  \\
$\ref{number:31}$ & $1345$ & $C_{2}$  \\
$\ref{number:37}$ & $242$ & $C_{2}$  \\
$\ref{number:39}$ & $36$ & $C_{2}$  \\
$\ref{number:43}$ & $1924$ & $C_{2}$  \\
$\ref{number:58}$ & $417$ & $C_{2}$  \\
$\ref{number:60}$ & $446$ & $C_{2}$  \\
$\ref{number:80}$ & $950$ & $C_{2}$  \\
$\ref{number:106}$ & $144$ & $C_{2}$  \\
$\ref{number:114}$ & $138$ & $C_{2}$  \\
\bottomrule
\end{longtable}
\end{table}

All-in-all, there are $230$ groups of order $128$ with nontrivial Bogomolov multipliers out of a total of $2328$ groups of this order.
For all these groups, Noether's rationality problem \cite{Noe16} therefore has a negative solution.
\section{The calculations}
\label{sec:calculations}

\begin{enumerate}[leftmargin=0cm]


\item \label{number:1} 
Let the group $G$ be the representative of this family given by the presentation
\[
\begin{aligned}
\langle g_{1}, \,g_{2}, \,g_{3}, \,g_{4}, \,g_{5}, \,g_{6}, \,g_{7} & \mid & g_{1}^{2} &= g_{2}, \\ 
 & & g_{2}^{2} &= g_{3}, \\ 
 & & g_{3}^{2} &= g_{4}, \\ 
 & & g_{4}^{2} &= g_{5}, \\ 
 & & g_{5}^{2} &= g_{6}, \\ 
 & & g_{6}^{2} &= g_{7}, \\ 
 & & g_{7}^{2} &= 1\rangle. \\ 
\end{aligned}
\]
We add 7 tails to the presentation as to form a quotient of the universal central extension of the system: 
$g_{1}^{2} = g_{2} t_{1}$,
$g_{2}^{2} = g_{3} t_{2}$,
$g_{3}^{2} = g_{4} t_{3}$,
$g_{4}^{2} = g_{5} t_{4}$,
$g_{5}^{2} = g_{6} t_{5}$,
$g_{6}^{2} = g_{7} t_{6}$,
$g_{7}^{2} =  t_{7}$.
Consistency checks enforce no relations between the tails.
Scanning through the conjugacy class representatives of $G$ and the generators of their centralizers, we see that no new relations are imposed.
We thus have that the torison subgroup of the group generated by the tails is trivial, implying that the Bogomolov multiplier is also trivial.


\item \label{number:2} 
Let the group $G$ be the representative of this family given by the presentation
\[
\begin{aligned}
\langle g_{1}, \,g_{2}, \,g_{3}, \,g_{4}, \,g_{5}, \,g_{6}, \,g_{7} & \mid & g_{1}^{2} &= g_{4}, \\ 
 & & g_{2}^{2} &= g_{5}, & [g_{2}, g_{1}]  &= g_{3}, \\ 
 & & g_{3}^{2} &= 1, \\ 
 & & g_{4}^{2} &= g_{6}, \\ 
 & & g_{5}^{2} &= g_{7}, \\ 
 & & g_{6}^{2} &= 1, \\ 
 & & g_{7}^{2} &= 1\rangle. \\ 
\end{aligned}
\]
We add 8 tails to the presentation as to form a quotient of the universal central extension of the system: 
$g_{1}^{2} = g_{4} t_{1}$,
$g_{2}^{2} = g_{5} t_{2}$,
$[g_{2}, g_{1}] = g_{3} t_{3}$,
$g_{3}^{2} =  t_{4}$,
$g_{4}^{2} = g_{6} t_{5}$,
$g_{5}^{2} = g_{7} t_{6}$,
$g_{6}^{2} =  t_{7}$,
$g_{7}^{2} =  t_{8}$.
Carrying out consistency checks gives the following relations between the tails:
\[
\begin{aligned}
g_{2}^2 g_{1} & = g_{2} (g_{2} g_{1})& \Longrightarrow & & t_{3}^{2}t_{4} & = 1 \\
\end{aligned}
\]
Scanning through the conjugacy class representatives of $G$ and the generators of their centralizers, we see that no new relations are imposed.
Collecting the coefficients of these relations into a matrix yields
\[
T = \bordermatrix{
{} & t_{1} & t_{2} & t_{3} & t_{4} & t_{5} & t_{6} & t_{7} & t_{8} \cr
{} &  &  & 2 & 1 &  &  &  &  \cr
}.
\]
It follows readily that the nontrivial elementary divisor of the Smith normal form of $T$ is equal to $1$. The torsion subgroup of the group generated by the tails is thus trivial, thereby showing $\B_0(G) = 1$.


\item \label{number:3} 
Let the group $G$ be the representative of this family given by the presentation
\[
\begin{aligned}
\langle g_{1}, \,g_{2}, \,g_{3}, \,g_{4}, \,g_{5}, \,g_{6}, \,g_{7} & \mid & g_{1}^{2} &= g_{4}, \\ 
 & & g_{2}^{2} &= g_{5}, & [g_{2}, g_{1}]  &= g_{3}, \\ 
 & & g_{3}^{2} &= g_{6}, & [g_{3}, g_{1}]  &= g_{6}, \\ 
 & & g_{4}^{2} &= g_{7}, \\ 
 & & g_{5}^{2} &= 1, & [g_{5}, g_{1}]  &= g_{6}, \\ 
 & & g_{6}^{2} &= 1, \\ 
 & & g_{7}^{2} &= 1\rangle. \\ 
\end{aligned}
\]
We add 10 tails to the presentation as to form a quotient of the universal central extension of the system: 
$g_{1}^{2} = g_{4} t_{1}$,
$g_{2}^{2} = g_{5} t_{2}$,
$[g_{2}, g_{1}] = g_{3} t_{3}$,
$g_{3}^{2} = g_{6} t_{4}$,
$[g_{3}, g_{1}] = g_{6} t_{5}$,
$g_{4}^{2} = g_{7} t_{6}$,
$g_{5}^{2} =  t_{7}$,
$[g_{5}, g_{1}] = g_{6} t_{8}$,
$g_{6}^{2} =  t_{9}$,
$g_{7}^{2} =  t_{10}$.
Carrying out consistency checks gives the following relations between the tails:
\[
\begin{aligned}
g_{5}^2 g_{1} & = g_{5} (g_{5} g_{1})& \Longrightarrow & & t_{8}^{2}t_{9} & = 1 \\
g_{3}^2 g_{1} & = g_{3} (g_{3} g_{1})& \Longrightarrow & & t_{5}^{2}t_{9} & = 1 \\
g_{2}^2 g_{1} & = g_{2} (g_{2} g_{1})& \Longrightarrow & & t_{3}^{2}t_{4}t_{8}^{-1} & = 1 \\
g_{2} g_{1}^{2} & = (g_{2} g_{1}) g_{1}& \Longrightarrow & & t_{3}^{2}t_{4}t_{5}t_{9} & = 1 \\
\end{aligned}
\]
Scanning through the conjugacy class representatives of $G$ and the generators of their centralizers, we see that no new relations are imposed.
Collecting the coefficients of these relations into a matrix yields
\[
T = \bordermatrix{
{} & t_{1} & t_{2} & t_{3} & t_{4} & t_{5} & t_{6} & t_{7} & t_{8} & t_{9} & t_{10} \cr
{} &  &  & 2 & 1 &  &  &  & 1 & 1 &  \cr
{} &  &  &  &  & 1 &  &  & 1 & 1 &  \cr
{} &  &  &  &  &  &  &  & 2 & 1 &  \cr
}.
\]
It follows readily that the nontrivial elementary divisors of the Smith normal form of $T$ are all equal to $1$. The torsion subgroup of the group generated by the tails is thus trivial, thereby showing $\B_0(G) = 1$.


\item \label{number:4} 
Let the group $G$ be the representative of this family given by the presentation
\[
\begin{aligned}
\langle g_{1}, \,g_{2}, \,g_{3}, \,g_{4}, \,g_{5}, \,g_{6}, \,g_{7} & \mid & g_{1}^{2} &= g_{6}, \\ 
 & & g_{2}^{2} &= g_{7}, & [g_{2}, g_{1}]  &= g_{4}, \\ 
 & & g_{3}^{2} &= 1, & [g_{3}, g_{1}]  &= g_{5}, \\ 
 & & g_{4}^{2} &= 1, \\ 
 & & g_{5}^{2} &= 1, \\ 
 & & g_{6}^{2} &= 1, \\ 
 & & g_{7}^{2} &= 1\rangle. \\ 
\end{aligned}
\]
We add 9 tails to the presentation as to form a quotient of the universal central extension of the system: 
$g_{1}^{2} = g_{6} t_{1}$,
$g_{2}^{2} = g_{7} t_{2}$,
$[g_{2}, g_{1}] = g_{4} t_{3}$,
$g_{3}^{2} =  t_{4}$,
$[g_{3}, g_{1}] = g_{5} t_{5}$,
$g_{4}^{2} =  t_{6}$,
$g_{5}^{2} =  t_{7}$,
$g_{6}^{2} =  t_{8}$,
$g_{7}^{2} =  t_{9}$.
Carrying out consistency checks gives the following relations between the tails:
\[
\begin{aligned}
g_{3}^2 g_{1} & = g_{3} (g_{3} g_{1})& \Longrightarrow & & t_{5}^{2}t_{7} & = 1 \\
g_{2}^2 g_{1} & = g_{2} (g_{2} g_{1})& \Longrightarrow & & t_{3}^{2}t_{6} & = 1 \\
\end{aligned}
\]
Scanning through the conjugacy class representatives of $G$ and the generators of their centralizers, we see that no new relations are imposed.
Collecting the coefficients of these relations into a matrix yields
\[
T = \bordermatrix{
{} & t_{1} & t_{2} & t_{3} & t_{4} & t_{5} & t_{6} & t_{7} & t_{8} & t_{9} \cr
{} &  &  & 2 &  &  & 1 &  &  &  \cr
{} &  &  &  &  & 2 &  & 1 &  &  \cr
}.
\]
It follows readily that the nontrivial elementary divisors of the Smith normal form of $T$ are all equal to $1$. The torsion subgroup of the group generated by the tails is thus trivial, thereby showing $\B_0(G) = 1$.


\item \label{number:5} 
Let the group $G$ be the representative of this family given by the presentation
\[
\begin{aligned}
\langle g_{1}, \,g_{2}, \,g_{3}, \,g_{4}, \,g_{5}, \,g_{6}, \,g_{7} & \mid & g_{1}^{2} &= g_{6}, \\ 
 & & g_{2}^{2} &= g_{7}, & [g_{2}, g_{1}]  &= g_{5}, \\ 
 & & g_{3}^{2} &= 1, & [g_{3}, g_{2}]  &= g_{5}, \\ 
 & & g_{4}^{2} &= 1, & [g_{4}, g_{1}]  &= g_{5}, \\ 
 & & g_{5}^{2} &= 1, \\ 
 & & g_{6}^{2} &= 1, \\ 
 & & g_{7}^{2} &= 1\rangle. \\ 
\end{aligned}
\]
We add 10 tails to the presentation as to form a quotient of the universal central extension of the system: 
$g_{1}^{2} = g_{6} t_{1}$,
$g_{2}^{2} = g_{7} t_{2}$,
$[g_{2}, g_{1}] = g_{5} t_{3}$,
$g_{3}^{2} =  t_{4}$,
$[g_{3}, g_{2}] = g_{5} t_{5}$,
$g_{4}^{2} =  t_{6}$,
$[g_{4}, g_{1}] = g_{5} t_{7}$,
$g_{5}^{2} =  t_{8}$,
$g_{6}^{2} =  t_{9}$,
$g_{7}^{2} =  t_{10}$.
Carrying out consistency checks gives the following relations between the tails:
\[
\begin{aligned}
g_{4}^2 g_{1} & = g_{4} (g_{4} g_{1})& \Longrightarrow & & t_{7}^{2}t_{8} & = 1 \\
g_{3}^2 g_{2} & = g_{3} (g_{3} g_{2})& \Longrightarrow & & t_{5}^{2}t_{8} & = 1 \\
g_{2}^2 g_{1} & = g_{2} (g_{2} g_{1})& \Longrightarrow & & t_{3}^{2}t_{8} & = 1 \\
\end{aligned}
\]
Scanning through the conjugacy class representatives of $G$ and the generators of their centralizers, we obtain the following relations induced on the tails:
\[
\begin{aligned}
{[g_{3} g_{4} g_{5} , \, g_{1} g_{2} g_{4} ]}_G & = 1 & \Longrightarrow & & t_{5}t_{7}t_{8} & = 1 \\
{[g_{2} g_{5} , \, g_{1} g_{3} g_{4} ]}_G & = 1 & \Longrightarrow & & t_{3}t_{5}^{-1} & = 1 \\
\end{aligned}
\]
Collecting the coefficients of these relations into a matrix yields
\[
T = \bordermatrix{
{} & t_{1} & t_{2} & t_{3} & t_{4} & t_{5} & t_{6} & t_{7} & t_{8} & t_{9} & t_{10} \cr
{} &  &  & 1 &  &  &  & 1 & 1 &  &  \cr
{} &  &  &  &  & 1 &  & 1 & 1 &  &  \cr
{} &  &  &  &  &  &  & 2 & 1 &  &  \cr
}.
\]
It follows readily that the nontrivial elementary divisors of the Smith normal form of $T$ are all equal to $1$. The torsion subgroup of the group generated by the tails is thus trivial, thereby showing $\B_0(G) = 1$.


\item \label{number:6} 
Let the group $G$ be the representative of this family given by the presentation
\[
\begin{aligned}
\langle g_{1}, \,g_{2}, \,g_{3}, \,g_{4}, \,g_{5}, \,g_{6}, \,g_{7} & \mid & g_{1}^{2} &= g_{5}, \\ 
 & & g_{2}^{2} &= 1, & [g_{2}, g_{1}]  &= g_{4}, \\ 
 & & g_{3}^{2} &= 1, & [g_{3}, g_{2}]  &= g_{6}, \\ 
 & & g_{4}^{2} &= g_{6}, & [g_{4}, g_{1}]  &= g_{6}, & [g_{4}, g_{2}]  &= g_{6}, \\ 
 & & g_{5}^{2} &= g_{7}, \\ 
 & & g_{6}^{2} &= 1, \\ 
 & & g_{7}^{2} &= 1\rangle. \\ 
\end{aligned}
\]
We add 11 tails to the presentation as to form a quotient of the universal central extension of the system: 
$g_{1}^{2} = g_{5} t_{1}$,
$g_{2}^{2} =  t_{2}$,
$[g_{2}, g_{1}] = g_{4} t_{3}$,
$g_{3}^{2} =  t_{4}$,
$[g_{3}, g_{2}] = g_{6} t_{5}$,
$g_{4}^{2} = g_{6} t_{6}$,
$[g_{4}, g_{1}] = g_{6} t_{7}$,
$[g_{4}, g_{2}] = g_{6} t_{8}$,
$g_{5}^{2} = g_{7} t_{9}$,
$g_{6}^{2} =  t_{10}$,
$g_{7}^{2} =  t_{11}$.
Carrying out consistency checks gives the following relations between the tails:
\[
\begin{aligned}
g_{4}^2 g_{2} & = g_{4} (g_{4} g_{2})& \Longrightarrow & & t_{8}^{2}t_{10} & = 1 \\
g_{4}^2 g_{1} & = g_{4} (g_{4} g_{1})& \Longrightarrow & & t_{7}^{2}t_{10} & = 1 \\
g_{3}^2 g_{2} & = g_{3} (g_{3} g_{2})& \Longrightarrow & & t_{5}^{2}t_{10} & = 1 \\
g_{2}^2 g_{1} & = g_{2} (g_{2} g_{1})& \Longrightarrow & & t_{3}^{2}t_{6}t_{8}t_{10} & = 1 \\
g_{2} g_{1}^{2} & = (g_{2} g_{1}) g_{1}& \Longrightarrow & & t_{3}^{2}t_{6}t_{7}t_{10} & = 1 \\
\end{aligned}
\]
Scanning through the conjugacy class representatives of $G$ and the generators of their centralizers, we obtain the following relations induced on the tails:
\[
\begin{aligned}
{[g_{3} g_{4} g_{6} , \, g_{2} g_{4} g_{6} ]}_G & = 1 & \Longrightarrow & & t_{5}t_{8}t_{10} & = 1 \\
\end{aligned}
\]
Collecting the coefficients of these relations into a matrix yields
\[
T = \bordermatrix{
{} & t_{1} & t_{2} & t_{3} & t_{4} & t_{5} & t_{6} & t_{7} & t_{8} & t_{9} & t_{10} & t_{11} \cr
{} &  &  & 2 &  &  & 1 &  & 1 &  & 1 &  \cr
{} &  &  &  &  & 1 &  &  & 1 &  & 1 &  \cr
{} &  &  &  &  &  &  & 1 & 1 &  & 1 &  \cr
{} &  &  &  &  &  &  &  & 2 &  & 1 &  \cr
}.
\]
It follows readily that the nontrivial elementary divisors of the Smith normal form of $T$ are all equal to $1$. The torsion subgroup of the group generated by the tails is thus trivial, thereby showing $\B_0(G) = 1$.


\item \label{number:7} 
Let the group $G$ be the representative of this family given by the presentation
\[
\begin{aligned}
\langle g_{1}, \,g_{2}, \,g_{3}, \,g_{4}, \,g_{5}, \,g_{6}, \,g_{7} & \mid & g_{1}^{2} &= g_{4}, \\ 
 & & g_{2}^{2} &= g_{5}, & [g_{2}, g_{1}]  &= g_{3}, \\ 
 & & g_{3}^{2} &= 1, & [g_{3}, g_{1}]  &= g_{6}, \\ 
 & & g_{4}^{2} &= g_{7}, & [g_{4}, g_{2}]  &= g_{6}, \\ 
 & & g_{5}^{2} &= 1, \\ 
 & & g_{6}^{2} &= 1, \\ 
 & & g_{7}^{2} &= 1\rangle. \\ 
\end{aligned}
\]
We add 10 tails to the presentation as to form a quotient of the universal central extension of the system: 
$g_{1}^{2} = g_{4} t_{1}$,
$g_{2}^{2} = g_{5} t_{2}$,
$[g_{2}, g_{1}] = g_{3} t_{3}$,
$g_{3}^{2} =  t_{4}$,
$[g_{3}, g_{1}] = g_{6} t_{5}$,
$g_{4}^{2} = g_{7} t_{6}$,
$[g_{4}, g_{2}] = g_{6} t_{7}$,
$g_{5}^{2} =  t_{8}$,
$g_{6}^{2} =  t_{9}$,
$g_{7}^{2} =  t_{10}$.
Carrying out consistency checks gives the following relations between the tails:
\[
\begin{aligned}
g_{4}^2 g_{2} & = g_{4} (g_{4} g_{2})& \Longrightarrow & & t_{7}^{2}t_{9} & = 1 \\
g_{3}^2 g_{1} & = g_{3} (g_{3} g_{1})& \Longrightarrow & & t_{5}^{2}t_{9} & = 1 \\
g_{2}^2 g_{1} & = g_{2} (g_{2} g_{1})& \Longrightarrow & & t_{3}^{2}t_{4} & = 1 \\
g_{2} g_{1}^{2} & = (g_{2} g_{1}) g_{1}& \Longrightarrow & & t_{3}^{2}t_{4}t_{5}t_{7}t_{9} & = 1 \\
\end{aligned}
\]
Scanning through the conjugacy class representatives of $G$ and the generators of their centralizers, we see that no new relations are imposed.
Collecting the coefficients of these relations into a matrix yields
\[
T = \bordermatrix{
{} & t_{1} & t_{2} & t_{3} & t_{4} & t_{5} & t_{6} & t_{7} & t_{8} & t_{9} & t_{10} \cr
{} &  &  & 2 & 1 &  &  &  &  &  &  \cr
{} &  &  &  &  & 1 &  & 1 &  & 1 &  \cr
{} &  &  &  &  &  &  & 2 &  & 1 &  \cr
}.
\]
It follows readily that the nontrivial elementary divisors of the Smith normal form of $T$ are all equal to $1$. The torsion subgroup of the group generated by the tails is thus trivial, thereby showing $\B_0(G) = 1$.


\item \label{number:8} 
Let the group $G$ be the representative of this family given by the presentation
\[
\begin{aligned}
\langle g_{1}, \,g_{2}, \,g_{3}, \,g_{4}, \,g_{5}, \,g_{6}, \,g_{7} & \mid & g_{1}^{2} &= g_{4}, \\ 
 & & g_{2}^{2} &= 1, & [g_{2}, g_{1}]  &= g_{3}, \\ 
 & & g_{3}^{2} &= g_{5}g_{7}, & [g_{3}, g_{1}]  &= g_{5}, & [g_{3}, g_{2}]  &= g_{5}, \\ 
 & & g_{4}^{2} &= g_{6}, \\ 
 & & g_{5}^{2} &= g_{7}, & [g_{5}, g_{1}]  &= g_{7}, & [g_{5}, g_{2}]  &= g_{7}, \\ 
 & & g_{6}^{2} &= 1, \\ 
 & & g_{7}^{2} &= 1\rangle. \\ 
\end{aligned}
\]
We add 12 tails to the presentation as to form a quotient of the universal central extension of the system: 
$g_{1}^{2} = g_{4} t_{1}$,
$g_{2}^{2} =  t_{2}$,
$[g_{2}, g_{1}] = g_{3} t_{3}$,
$g_{3}^{2} = g_{5}g_{7} t_{4}$,
$[g_{3}, g_{1}] = g_{5} t_{5}$,
$[g_{3}, g_{2}] = g_{5} t_{6}$,
$g_{4}^{2} = g_{6} t_{7}$,
$g_{5}^{2} = g_{7} t_{8}$,
$[g_{5}, g_{1}] = g_{7} t_{9}$,
$[g_{5}, g_{2}] = g_{7} t_{10}$,
$g_{6}^{2} =  t_{11}$,
$g_{7}^{2} =  t_{12}$.
Carrying out consistency checks gives the following relations between the tails:
\[
\begin{aligned}
g_{3}(g_{2} g_{1}) & = (g_{3} g_{2}) g_{1}  & \Longrightarrow & & t_{9}t_{10}^{-1} & = 1 \\
g_{5}^2 g_{2} & = g_{5} (g_{5} g_{2})& \Longrightarrow & & t_{10}^{2}t_{12} & = 1 \\
g_{3}^2 g_{2} & = g_{3} (g_{3} g_{2})& \Longrightarrow & & t_{6}^{2}t_{8}t_{10}^{-1} & = 1 \\
g_{3}^2 g_{1} & = g_{3} (g_{3} g_{1})& \Longrightarrow & & t_{5}^{2}t_{8}t_{9}^{-1} & = 1 \\
g_{2}^2 g_{1} & = g_{2} (g_{2} g_{1})& \Longrightarrow & & t_{3}^{2}t_{4}t_{6}t_{8}t_{12} & = 1 \\
g_{2} g_{1}^{2} & = (g_{2} g_{1}) g_{1}& \Longrightarrow & & t_{3}^{2}t_{4}t_{5}t_{8}t_{12} & = 1 \\
\end{aligned}
\]
Scanning through the conjugacy class representatives of $G$ and the generators of their centralizers, we see that no new relations are imposed.
Collecting the coefficients of these relations into a matrix yields
\[
T = \bordermatrix{
{} & t_{1} & t_{2} & t_{3} & t_{4} & t_{5} & t_{6} & t_{7} & t_{8} & t_{9} & t_{10} & t_{11} & t_{12} \cr
{} &  &  & 2 & 1 &  & 1 &  & 1 &  &  &  & 1 \cr
{} &  &  &  &  & 1 & 1 &  & 1 &  & 1 &  & 1 \cr
{} &  &  &  &  &  & 2 &  & 1 &  & 1 &  & 1 \cr
{} &  &  &  &  &  &  &  &  & 1 & 1 &  & 1 \cr
{} &  &  &  &  &  &  &  &  &  & 2 &  & 1 \cr
}.
\]
It follows readily that the nontrivial elementary divisors of the Smith normal form of $T$ are all equal to $1$. The torsion subgroup of the group generated by the tails is thus trivial, thereby showing $\B_0(G) = 1$.


\item \label{number:9} 
Let the group $G$ be the representative of this family given by the presentation
\[
\begin{aligned}
\langle g_{1}, \,g_{2}, \,g_{3}, \,g_{4}, \,g_{5}, \,g_{6}, \,g_{7} & \mid & g_{1}^{2} &= g_{7}, \\ 
 & & g_{2}^{2} &= 1, & [g_{2}, g_{1}]  &= g_{4}, \\ 
 & & g_{3}^{2} &= 1, & [g_{3}, g_{1}]  &= g_{5}, & [g_{3}, g_{2}]  &= g_{6}, \\ 
 & & g_{4}^{2} &= 1, \\ 
 & & g_{5}^{2} &= 1, \\ 
 & & g_{6}^{2} &= 1, \\ 
 & & g_{7}^{2} &= 1\rangle. \\ 
\end{aligned}
\]
We add 10 tails to the presentation as to form a quotient of the universal central extension of the system: 
$g_{1}^{2} = g_{7} t_{1}$,
$g_{2}^{2} =  t_{2}$,
$[g_{2}, g_{1}] = g_{4} t_{3}$,
$g_{3}^{2} =  t_{4}$,
$[g_{3}, g_{1}] = g_{5} t_{5}$,
$[g_{3}, g_{2}] = g_{6} t_{6}$,
$g_{4}^{2} =  t_{7}$,
$g_{5}^{2} =  t_{8}$,
$g_{6}^{2} =  t_{9}$,
$g_{7}^{2} =  t_{10}$.
Carrying out consistency checks gives the following relations between the tails:
\[
\begin{aligned}
g_{3}^2 g_{2} & = g_{3} (g_{3} g_{2})& \Longrightarrow & & t_{6}^{2}t_{9} & = 1 \\
g_{3}^2 g_{1} & = g_{3} (g_{3} g_{1})& \Longrightarrow & & t_{5}^{2}t_{8} & = 1 \\
g_{2}^2 g_{1} & = g_{2} (g_{2} g_{1})& \Longrightarrow & & t_{3}^{2}t_{7} & = 1 \\
\end{aligned}
\]
Scanning through the conjugacy class representatives of $G$ and the generators of their centralizers, we see that no new relations are imposed.
Collecting the coefficients of these relations into a matrix yields
\[
T = \bordermatrix{
{} & t_{1} & t_{2} & t_{3} & t_{4} & t_{5} & t_{6} & t_{7} & t_{8} & t_{9} & t_{10} \cr
{} &  &  & 2 &  &  &  & 1 &  &  &  \cr
{} &  &  &  &  & 2 &  &  & 1 &  &  \cr
{} &  &  &  &  &  & 2 &  &  & 1 &  \cr
}.
\]
It follows readily that the nontrivial elementary divisors of the Smith normal form of $T$ are all equal to $1$. The torsion subgroup of the group generated by the tails is thus trivial, thereby showing $\B_0(G) = 1$.


\item \label{number:10} 
Let the group $G$ be the representative of this family given by the presentation
\[
\begin{aligned}
\langle g_{1}, \,g_{2}, \,g_{3}, \,g_{4}, \,g_{5}, \,g_{6}, \,g_{7} & \mid & g_{1}^{2} &= g_{7}, \\ 
 & & g_{2}^{2} &= 1, & [g_{2}, g_{1}]  &= g_{5}, \\ 
 & & g_{3}^{2} &= 1, & [g_{3}, g_{1}]  &= g_{6}, \\ 
 & & g_{4}^{2} &= 1, & [g_{4}, g_{2}]  &= g_{5}, \\ 
 & & g_{5}^{2} &= 1, \\ 
 & & g_{6}^{2} &= 1, \\ 
 & & g_{7}^{2} &= 1\rangle. \\ 
\end{aligned}
\]
We add 10 tails to the presentation as to form a quotient of the universal central extension of the system: 
$g_{1}^{2} = g_{7} t_{1}$,
$g_{2}^{2} =  t_{2}$,
$[g_{2}, g_{1}] = g_{5} t_{3}$,
$g_{3}^{2} =  t_{4}$,
$[g_{3}, g_{1}] = g_{6} t_{5}$,
$g_{4}^{2} =  t_{6}$,
$[g_{4}, g_{2}] = g_{5} t_{7}$,
$g_{5}^{2} =  t_{8}$,
$g_{6}^{2} =  t_{9}$,
$g_{7}^{2} =  t_{10}$.
Carrying out consistency checks gives the following relations between the tails:
\[
\begin{aligned}
g_{4}^2 g_{2} & = g_{4} (g_{4} g_{2})& \Longrightarrow & & t_{7}^{2}t_{8} & = 1 \\
g_{3}^2 g_{1} & = g_{3} (g_{3} g_{1})& \Longrightarrow & & t_{5}^{2}t_{9} & = 1 \\
g_{2}^2 g_{1} & = g_{2} (g_{2} g_{1})& \Longrightarrow & & t_{3}^{2}t_{8} & = 1 \\
\end{aligned}
\]
Scanning through the conjugacy class representatives of $G$ and the generators of their centralizers, we obtain the following relations induced on the tails:
\[
\begin{aligned}
{[g_{2} g_{5} , \, g_{1} g_{3} g_{4} ]}_G & = 1 & \Longrightarrow & & t_{3}t_{7}^{-1} & = 1 \\
\end{aligned}
\]
Collecting the coefficients of these relations into a matrix yields
\[
T = \bordermatrix{
{} & t_{1} & t_{2} & t_{3} & t_{4} & t_{5} & t_{6} & t_{7} & t_{8} & t_{9} & t_{10} \cr
{} &  &  & 1 &  &  &  & 1 & 1 &  &  \cr
{} &  &  &  &  & 2 &  &  &  & 1 &  \cr
{} &  &  &  &  &  &  & 2 & 1 &  &  \cr
}.
\]
It follows readily that the nontrivial elementary divisors of the Smith normal form of $T$ are all equal to $1$. The torsion subgroup of the group generated by the tails is thus trivial, thereby showing $\B_0(G) = 1$.


\item \label{number:11} 
Let the group $G$ be the representative of this family given by the presentation
\[
\begin{aligned}
\langle g_{1}, \,g_{2}, \,g_{3}, \,g_{4}, \,g_{5}, \,g_{6}, \,g_{7} & \mid & g_{1}^{2} &= g_{7}, \\ 
 & & g_{2}^{2} &= 1, & [g_{2}, g_{1}]  &= g_{5}, \\ 
 & & g_{3}^{2} &= 1, & [g_{3}, g_{1}]  &= g_{6}, & [g_{3}, g_{2}]  &= g_{5}, \\ 
 & & g_{4}^{2} &= 1, & [g_{4}, g_{1}]  &= g_{5}, \\ 
 & & g_{5}^{2} &= 1, \\ 
 & & g_{6}^{2} &= 1, \\ 
 & & g_{7}^{2} &= 1\rangle. \\ 
\end{aligned}
\]
We add 11 tails to the presentation as to form a quotient of the universal central extension of the system: 
$g_{1}^{2} = g_{7} t_{1}$,
$g_{2}^{2} =  t_{2}$,
$[g_{2}, g_{1}] = g_{5} t_{3}$,
$g_{3}^{2} =  t_{4}$,
$[g_{3}, g_{1}] = g_{6} t_{5}$,
$[g_{3}, g_{2}] = g_{5} t_{6}$,
$g_{4}^{2} =  t_{7}$,
$[g_{4}, g_{1}] = g_{5} t_{8}$,
$g_{5}^{2} =  t_{9}$,
$g_{6}^{2} =  t_{10}$,
$g_{7}^{2} =  t_{11}$.
Carrying out consistency checks gives the following relations between the tails:
\[
\begin{aligned}
g_{4}^2 g_{1} & = g_{4} (g_{4} g_{1})& \Longrightarrow & & t_{8}^{2}t_{9} & = 1 \\
g_{3}^2 g_{2} & = g_{3} (g_{3} g_{2})& \Longrightarrow & & t_{6}^{2}t_{9} & = 1 \\
g_{3}^2 g_{1} & = g_{3} (g_{3} g_{1})& \Longrightarrow & & t_{5}^{2}t_{10} & = 1 \\
g_{2}^2 g_{1} & = g_{2} (g_{2} g_{1})& \Longrightarrow & & t_{3}^{2}t_{9} & = 1 \\
\end{aligned}
\]
Scanning through the conjugacy class representatives of $G$ and the generators of their centralizers, we obtain the following relations induced on the tails:
\[
\begin{aligned}
{[g_{2} g_{5} , \, g_{1} g_{3} g_{4} ]}_G & = 1 & \Longrightarrow & & t_{3}t_{6}^{-1} & = 1 \\
{[g_{2} g_{4} g_{5} , \, g_{1} g_{4} ]}_G & = 1 & \Longrightarrow & & t_{3}t_{8}t_{9} & = 1 \\
\end{aligned}
\]
Collecting the coefficients of these relations into a matrix yields
\[
T = \bordermatrix{
{} & t_{1} & t_{2} & t_{3} & t_{4} & t_{5} & t_{6} & t_{7} & t_{8} & t_{9} & t_{10} & t_{11} \cr
{} &  &  & 1 &  &  &  &  & 1 & 1 &  &  \cr
{} &  &  &  &  & 2 &  &  &  &  & 1 &  \cr
{} &  &  &  &  &  & 1 &  & 1 & 1 &  &  \cr
{} &  &  &  &  &  &  &  & 2 & 1 &  &  \cr
}.
\]
It follows readily that the nontrivial elementary divisors of the Smith normal form of $T$ are all equal to $1$. The torsion subgroup of the group generated by the tails is thus trivial, thereby showing $\B_0(G) = 1$.


\item \label{number:12} 
Let the group $G$ be the representative of this family given by the presentation
\[
\begin{aligned}
\langle g_{1}, \,g_{2}, \,g_{3}, \,g_{4}, \,g_{5}, \,g_{6}, \,g_{7} & \mid & g_{1}^{2} &= g_{4}, \\ 
 & & g_{2}^{2} &= g_{5}, & [g_{2}, g_{1}]  &= g_{3}, \\ 
 & & g_{3}^{2} &= g_{6}, \\ 
 & & g_{4}^{2} &= g_{7}, & [g_{4}, g_{2}]  &= g_{6}, \\ 
 & & g_{5}^{2} &= 1, & [g_{5}, g_{1}]  &= g_{6}, \\ 
 & & g_{6}^{2} &= 1, \\ 
 & & g_{7}^{2} &= 1\rangle. \\ 
\end{aligned}
\]
We add 10 tails to the presentation as to form a quotient of the universal central extension of the system: 
$g_{1}^{2} = g_{4} t_{1}$,
$g_{2}^{2} = g_{5} t_{2}$,
$[g_{2}, g_{1}] = g_{3} t_{3}$,
$g_{3}^{2} = g_{6} t_{4}$,
$g_{4}^{2} = g_{7} t_{5}$,
$[g_{4}, g_{2}] = g_{6} t_{6}$,
$g_{5}^{2} =  t_{7}$,
$[g_{5}, g_{1}] = g_{6} t_{8}$,
$g_{6}^{2} =  t_{9}$,
$g_{7}^{2} =  t_{10}$.
Carrying out consistency checks gives the following relations between the tails:
\[
\begin{aligned}
g_{5}^2 g_{1} & = g_{5} (g_{5} g_{1})& \Longrightarrow & & t_{8}^{2}t_{9} & = 1 \\
g_{4}^2 g_{2} & = g_{4} (g_{4} g_{2})& \Longrightarrow & & t_{6}^{2}t_{9} & = 1 \\
g_{2}^2 g_{1} & = g_{2} (g_{2} g_{1})& \Longrightarrow & & t_{3}^{2}t_{4}t_{8}^{-1} & = 1 \\
g_{2} g_{1}^{2} & = (g_{2} g_{1}) g_{1}& \Longrightarrow & & t_{3}^{2}t_{4}t_{6}t_{9} & = 1 \\
\end{aligned}
\]
Scanning through the conjugacy class representatives of $G$ and the generators of their centralizers, we see that no new relations are imposed.
Collecting the coefficients of these relations into a matrix yields
\[
T = \bordermatrix{
{} & t_{1} & t_{2} & t_{3} & t_{4} & t_{5} & t_{6} & t_{7} & t_{8} & t_{9} & t_{10} \cr
{} &  &  & 2 & 1 &  &  &  & 1 & 1 &  \cr
{} &  &  &  &  &  & 1 &  & 1 & 1 &  \cr
{} &  &  &  &  &  &  &  & 2 & 1 &  \cr
}.
\]
It follows readily that the nontrivial elementary divisors of the Smith normal form of $T$ are all equal to $1$. The torsion subgroup of the group generated by the tails is thus trivial, thereby showing $\B_0(G) = 1$.


\item \label{number:13} 
Let the group $G$ be the representative of this family given by the presentation
\[
\begin{aligned}
\langle g_{1}, \,g_{2}, \,g_{3}, \,g_{4}, \,g_{5}, \,g_{6}, \,g_{7} & \mid & g_{1}^{2} &= g_{7}, \\ 
 & & g_{2}^{2} &= 1, & [g_{2}, g_{1}]  &= g_{5}, \\ 
 & & g_{3}^{2} &= 1, & [g_{3}, g_{1}]  &= g_{6}, & [g_{3}, g_{2}]  &= g_{5}g_{6}, \\ 
 & & g_{4}^{2} &= 1, & [g_{4}, g_{1}]  &= g_{5}g_{6}, & [g_{4}, g_{2}]  &= g_{5}, \\ 
 & & g_{5}^{2} &= 1, \\ 
 & & g_{6}^{2} &= 1, \\ 
 & & g_{7}^{2} &= 1\rangle. \\ 
\end{aligned}
\]
We add 12 tails to the presentation as to form a quotient of the universal central extension of the system: 
$g_{1}^{2} = g_{7} t_{1}$,
$g_{2}^{2} =  t_{2}$,
$[g_{2}, g_{1}] = g_{5} t_{3}$,
$g_{3}^{2} =  t_{4}$,
$[g_{3}, g_{1}] = g_{6} t_{5}$,
$[g_{3}, g_{2}] = g_{5}g_{6} t_{6}$,
$g_{4}^{2} =  t_{7}$,
$[g_{4}, g_{1}] = g_{5}g_{6} t_{8}$,
$[g_{4}, g_{2}] = g_{5} t_{9}$,
$g_{5}^{2} =  t_{10}$,
$g_{6}^{2} =  t_{11}$,
$g_{7}^{2} =  t_{12}$.
Carrying out consistency checks gives the following relations between the tails:
\[
\begin{aligned}
g_{4}^2 g_{2} & = g_{4} (g_{4} g_{2})& \Longrightarrow & & t_{9}^{2}t_{10} & = 1 \\
g_{4}^2 g_{1} & = g_{4} (g_{4} g_{1})& \Longrightarrow & & t_{8}^{2}t_{10}t_{11} & = 1 \\
g_{3}^2 g_{2} & = g_{3} (g_{3} g_{2})& \Longrightarrow & & t_{6}^{2}t_{10}t_{11} & = 1 \\
g_{3}^2 g_{1} & = g_{3} (g_{3} g_{1})& \Longrightarrow & & t_{5}^{2}t_{11} & = 1 \\
g_{2}^2 g_{1} & = g_{2} (g_{2} g_{1})& \Longrightarrow & & t_{3}^{2}t_{10} & = 1 \\
\end{aligned}
\]
Scanning through the conjugacy class representatives of $G$ and the generators of their centralizers, we obtain the following relations induced on the tails:
\[
\begin{aligned}
{[g_{2} g_{5} g_{6} , \, g_{1} g_{4} g_{5} ]}_G & = 1 & \Longrightarrow & & t_{3}t_{9}^{-1} & = 1 \\
{[g_{2} g_{4} g_{5} g_{6} , \, g_{1} g_{3} g_{4} g_{5} ]}_G & = 1 & \Longrightarrow & & t_{3}t_{6}^{-1}t_{8}t_{9}^{-1} & = 1 \\
{[g_{2} g_{3} g_{5} g_{6} , \, g_{1} g_{3} g_{5} ]}_G & = 1 & \Longrightarrow & & t_{3}t_{5}t_{6}^{-1} & = 1 \\
\end{aligned}
\]
Collecting the coefficients of these relations into a matrix yields
\[
T = \bordermatrix{
{} & t_{1} & t_{2} & t_{3} & t_{4} & t_{5} & t_{6} & t_{7} & t_{8} & t_{9} & t_{10} & t_{11} & t_{12} \cr
{} &  &  & 1 &  &  &  &  &  & 1 & 1 &  &  \cr
{} &  &  &  &  & 1 &  &  & 1 & 1 & 1 & 1 &  \cr
{} &  &  &  &  &  & 1 &  & 1 &  & 1 & 1 &  \cr
{} &  &  &  &  &  &  &  & 2 &  & 1 & 1 &  \cr
{} &  &  &  &  &  &  &  &  & 2 & 1 &  &  \cr
}.
\]
It follows readily that the nontrivial elementary divisors of the Smith normal form of $T$ are all equal to $1$. The torsion subgroup of the group generated by the tails is thus trivial, thereby showing $\B_0(G) = 1$.


\item \label{number:14} 
Let the group $G$ be the representative of this family given by the presentation
\[
\begin{aligned}
\langle g_{1}, \,g_{2}, \,g_{3}, \,g_{4}, \,g_{5}, \,g_{6}, \,g_{7} & \mid & g_{1}^{2} &= g_{5}, \\ 
 & & g_{2}^{2} &= 1, & [g_{2}, g_{1}]  &= g_{4}, \\ 
 & & g_{3}^{2} &= 1, & [g_{3}, g_{1}]  &= g_{7}, & [g_{3}, g_{2}]  &= g_{7}, \\ 
 & & g_{4}^{2} &= g_{6}, & [g_{4}, g_{1}]  &= g_{6}, & [g_{4}, g_{2}]  &= g_{6}, \\ 
 & & g_{5}^{2} &= g_{7}, \\ 
 & & g_{6}^{2} &= 1, \\ 
 & & g_{7}^{2} &= 1\rangle. \\ 
\end{aligned}
\]
We add 12 tails to the presentation as to form a quotient of the universal central extension of the system: 
$g_{1}^{2} = g_{5} t_{1}$,
$g_{2}^{2} =  t_{2}$,
$[g_{2}, g_{1}] = g_{4} t_{3}$,
$g_{3}^{2} =  t_{4}$,
$[g_{3}, g_{1}] = g_{7} t_{5}$,
$[g_{3}, g_{2}] = g_{7} t_{6}$,
$g_{4}^{2} = g_{6} t_{7}$,
$[g_{4}, g_{1}] = g_{6} t_{8}$,
$[g_{4}, g_{2}] = g_{6} t_{9}$,
$g_{5}^{2} = g_{7} t_{10}$,
$g_{6}^{2} =  t_{11}$,
$g_{7}^{2} =  t_{12}$.
Carrying out consistency checks gives the following relations between the tails:
\[
\begin{aligned}
g_{4}^2 g_{2} & = g_{4} (g_{4} g_{2})& \Longrightarrow & & t_{9}^{2}t_{11} & = 1 \\
g_{4}^2 g_{1} & = g_{4} (g_{4} g_{1})& \Longrightarrow & & t_{8}^{2}t_{11} & = 1 \\
g_{3}^2 g_{2} & = g_{3} (g_{3} g_{2})& \Longrightarrow & & t_{6}^{2}t_{12} & = 1 \\
g_{3}^2 g_{1} & = g_{3} (g_{3} g_{1})& \Longrightarrow & & t_{5}^{2}t_{12} & = 1 \\
g_{2}^2 g_{1} & = g_{2} (g_{2} g_{1})& \Longrightarrow & & t_{3}^{2}t_{7}t_{9}t_{11} & = 1 \\
g_{2} g_{1}^{2} & = (g_{2} g_{1}) g_{1}& \Longrightarrow & & t_{3}^{2}t_{7}t_{8}t_{11} & = 1 \\
\end{aligned}
\]
Scanning through the conjugacy class representatives of $G$ and the generators of their centralizers, we obtain the following relations induced on the tails:
\[
\begin{aligned}
{[g_{3} g_{7} , \, g_{1} g_{2} g_{3} ]}_G & = 1 & \Longrightarrow & & t_{5}t_{6}t_{12} & = 1 \\
\end{aligned}
\]
Collecting the coefficients of these relations into a matrix yields
\[
T = \bordermatrix{
{} & t_{1} & t_{2} & t_{3} & t_{4} & t_{5} & t_{6} & t_{7} & t_{8} & t_{9} & t_{10} & t_{11} & t_{12} \cr
{} &  &  & 2 &  &  &  & 1 &  & 1 &  & 1 &  \cr
{} &  &  &  &  & 1 & 1 &  &  &  &  &  & 1 \cr
{} &  &  &  &  &  & 2 &  &  &  &  &  & 1 \cr
{} &  &  &  &  &  &  &  & 1 & 1 &  & 1 &  \cr
{} &  &  &  &  &  &  &  &  & 2 &  & 1 &  \cr
}.
\]
It follows readily that the nontrivial elementary divisors of the Smith normal form of $T$ are all equal to $1$. The torsion subgroup of the group generated by the tails is thus trivial, thereby showing $\B_0(G) = 1$.


\item \label{number:15} 
Let the group $G$ be the representative of this family given by the presentation
\[
\begin{aligned}
\langle g_{1}, \,g_{2}, \,g_{3}, \,g_{4}, \,g_{5}, \,g_{6}, \,g_{7} & \mid & g_{1}^{2} &= g_{5}, \\ 
 & & g_{2}^{2} &= 1, & [g_{2}, g_{1}]  &= g_{4}, \\ 
 & & g_{3}^{2} &= 1, & [g_{3}, g_{2}]  &= g_{7}, \\ 
 & & g_{4}^{2} &= g_{6}, & [g_{4}, g_{1}]  &= g_{6}, & [g_{4}, g_{2}]  &= g_{6}, \\ 
 & & g_{5}^{2} &= g_{7}, \\ 
 & & g_{6}^{2} &= 1, \\ 
 & & g_{7}^{2} &= 1\rangle. \\ 
\end{aligned}
\]
We add 11 tails to the presentation as to form a quotient of the universal central extension of the system: 
$g_{1}^{2} = g_{5} t_{1}$,
$g_{2}^{2} =  t_{2}$,
$[g_{2}, g_{1}] = g_{4} t_{3}$,
$g_{3}^{2} =  t_{4}$,
$[g_{3}, g_{2}] = g_{7} t_{5}$,
$g_{4}^{2} = g_{6} t_{6}$,
$[g_{4}, g_{1}] = g_{6} t_{7}$,
$[g_{4}, g_{2}] = g_{6} t_{8}$,
$g_{5}^{2} = g_{7} t_{9}$,
$g_{6}^{2} =  t_{10}$,
$g_{7}^{2} =  t_{11}$.
Carrying out consistency checks gives the following relations between the tails:
\[
\begin{aligned}
g_{4}^2 g_{2} & = g_{4} (g_{4} g_{2})& \Longrightarrow & & t_{8}^{2}t_{10} & = 1 \\
g_{4}^2 g_{1} & = g_{4} (g_{4} g_{1})& \Longrightarrow & & t_{7}^{2}t_{10} & = 1 \\
g_{3}^2 g_{2} & = g_{3} (g_{3} g_{2})& \Longrightarrow & & t_{5}^{2}t_{11} & = 1 \\
g_{2}^2 g_{1} & = g_{2} (g_{2} g_{1})& \Longrightarrow & & t_{3}^{2}t_{6}t_{8}t_{10} & = 1 \\
g_{2} g_{1}^{2} & = (g_{2} g_{1}) g_{1}& \Longrightarrow & & t_{3}^{2}t_{6}t_{7}t_{10} & = 1 \\
\end{aligned}
\]
Scanning through the conjugacy class representatives of $G$ and the generators of their centralizers, we see that no new relations are imposed.
Collecting the coefficients of these relations into a matrix yields
\[
T = \bordermatrix{
{} & t_{1} & t_{2} & t_{3} & t_{4} & t_{5} & t_{6} & t_{7} & t_{8} & t_{9} & t_{10} & t_{11} \cr
{} &  &  & 2 &  &  & 1 &  & 1 &  & 1 &  \cr
{} &  &  &  &  & 2 &  &  &  &  &  & 1 \cr
{} &  &  &  &  &  &  & 1 & 1 &  & 1 &  \cr
{} &  &  &  &  &  &  &  & 2 &  & 1 &  \cr
}.
\]
It follows readily that the nontrivial elementary divisors of the Smith normal form of $T$ are all equal to $1$. The torsion subgroup of the group generated by the tails is thus trivial, thereby showing $\B_0(G) = 1$.


\item \label{number:16} 
Let the group $G$ be the representative of this family given by the presentation
\[
\begin{aligned}
\langle g_{1}, \,g_{2}, \,g_{3}, \,g_{4}, \,g_{5}, \,g_{6}, \,g_{7} & \mid & g_{1}^{2} &= g_{5}, \\ 
 & & g_{2}^{2} &= 1, & [g_{2}, g_{1}]  &= g_{4}, \\ 
 & & g_{3}^{2} &= 1, & [g_{3}, g_{1}]  &= g_{7}, & [g_{3}, g_{2}]  &= g_{6}g_{7}, \\ 
 & & g_{4}^{2} &= g_{6}, & [g_{4}, g_{1}]  &= g_{6}, & [g_{4}, g_{2}]  &= g_{6}, \\ 
 & & g_{5}^{2} &= g_{7}, \\ 
 & & g_{6}^{2} &= 1, \\ 
 & & g_{7}^{2} &= 1\rangle. \\ 
\end{aligned}
\]
We add 12 tails to the presentation as to form a quotient of the universal central extension of the system: 
$g_{1}^{2} = g_{5} t_{1}$,
$g_{2}^{2} =  t_{2}$,
$[g_{2}, g_{1}] = g_{4} t_{3}$,
$g_{3}^{2} =  t_{4}$,
$[g_{3}, g_{1}] = g_{7} t_{5}$,
$[g_{3}, g_{2}] = g_{6}g_{7} t_{6}$,
$g_{4}^{2} = g_{6} t_{7}$,
$[g_{4}, g_{1}] = g_{6} t_{8}$,
$[g_{4}, g_{2}] = g_{6} t_{9}$,
$g_{5}^{2} = g_{7} t_{10}$,
$g_{6}^{2} =  t_{11}$,
$g_{7}^{2} =  t_{12}$.
Carrying out consistency checks gives the following relations between the tails:
\[
\begin{aligned}
g_{4}^2 g_{2} & = g_{4} (g_{4} g_{2})& \Longrightarrow & & t_{9}^{2}t_{11} & = 1 \\
g_{4}^2 g_{1} & = g_{4} (g_{4} g_{1})& \Longrightarrow & & t_{8}^{2}t_{11} & = 1 \\
g_{3}^2 g_{2} & = g_{3} (g_{3} g_{2})& \Longrightarrow & & t_{6}^{2}t_{11}t_{12} & = 1 \\
g_{3}^2 g_{1} & = g_{3} (g_{3} g_{1})& \Longrightarrow & & t_{5}^{2}t_{12} & = 1 \\
g_{2}^2 g_{1} & = g_{2} (g_{2} g_{1})& \Longrightarrow & & t_{3}^{2}t_{7}t_{9}t_{11} & = 1 \\
g_{2} g_{1}^{2} & = (g_{2} g_{1}) g_{1}& \Longrightarrow & & t_{3}^{2}t_{7}t_{8}t_{11} & = 1 \\
\end{aligned}
\]
Scanning through the conjugacy class representatives of $G$ and the generators of their centralizers, we see that no new relations are imposed.
Collecting the coefficients of these relations into a matrix yields
\[
T = \bordermatrix{
{} & t_{1} & t_{2} & t_{3} & t_{4} & t_{5} & t_{6} & t_{7} & t_{8} & t_{9} & t_{10} & t_{11} & t_{12} \cr
{} &  &  & 2 &  &  &  & 1 &  & 1 &  & 1 &  \cr
{} &  &  &  &  & 2 &  &  &  &  &  &  & 1 \cr
{} &  &  &  &  &  & 2 &  &  &  &  & 1 & 1 \cr
{} &  &  &  &  &  &  &  & 1 & 1 &  & 1 &  \cr
{} &  &  &  &  &  &  &  &  & 2 &  & 1 &  \cr
}.
\]
A change of basis according to the transition matrix (specifying expansions of $t_i^{*}$ by $t_j$)
\[
\bordermatrix{
{} & t_{1}^{*} & t_{2}^{*} & t_{3}^{*} & t_{4}^{*} & t_{5}^{*} & t_{6}^{*} & t_{7}^{*} & t_{8}^{*} & t_{9}^{*} & t_{10}^{*} & t_{11}^{*} & t_{12}^{*} \cr
t_{1} &  &  &  &  &  & -1 & -1 & 1 &  &  & 1 &  \cr
t_{2} &  &  &  &  &  & -1 &  & -1 &  &  & -1 & -1 \cr
t_{3} & -2 &  &  &  & -2 &  &  & -1 &  &  & -1 &  \cr
t_{4} &  &  &  &  &  & 4 &  & -1 &  &  &  &  \cr
t_{5} & 4 &  &  &  & 3 & 1 &  &  &  &  &  &  \cr
t_{6} & -16 & 2 & -2 &  & -13 & -4 &  &  &  &  &  &  \cr
t_{7} & -1 &  &  &  & -1 &  & 1 &  &  &  &  &  \cr
t_{8} & 16 & -2 & 2 & 1 & 13 &  &  & 1 &  &  &  &  \cr
t_{9} & -27 & 4 & -2 & -3 & -21 &  &  &  & 1 &  &  &  \cr
t_{10} &  &  &  &  &  &  &  &  &  & 1 &  &  \cr
t_{11} & -14 & 2 & -1 & -1 & -11 &  &  &  &  &  & 1 &  \cr
t_{12} & -6 & 1 & -1 &  & -5 &  &  &  &  &  &  & 1 \cr
}
\]
shows that the nontrivial elementary divisors of the Smith normal form of $T$ are $1$, $1$, $1$, $1$, $2$.  The element corresponding to the divisor that is greater than $1$ is $t_{5}^{*}$. This already gives
\[
\B_0(G) \cong \langle t_{5}^{*}  \mid {t_{5}^{*}}^{2} \rangle.
\]

We now deal with explicitly identifying the nonuniversal commutator relation generating $\B_0(G)$.
First, factor out by the tails $t_{i}^{*}$ whose corresponding elementary divisors are either trivial or $1$. Transforming the situation back to the original tails $t_i$, this amounts to the nontrivial expansion 
$t_{6} = t_{5}^{*}$ and all the other tails $t_i$ are trivial. We thus obtain a commutativity preserving central extension of the group $G$, given by the presentation
\[
\begin{aligned}
\langle g_{1}, \,g_{2}, \,g_{3}, \,g_{4}, \,g_{5}, \,g_{6}, \,g_{7}, \,t_{5}^{*} & \mid & g_{1}^{2} &= g_{5}, \\ 
 & & g_{2}^{2} &= 1, & [g_{2}, g_{1}]  &= g_{4}, \\ 
 & & g_{3}^{2} &= 1, & [g_{3}, g_{1}]  &= g_{7}, & [g_{3}, g_{2}]  &= g_{6}g_{7}t_{5}^{*} , \\ 
 & & g_{4}^{2} &= g_{6}, & [g_{4}, g_{1}]  &= g_{6}, & [g_{4}, g_{2}]  &= g_{6}, \\ 
 & & g_{5}^{2} &= g_{7}, \\ 
 & & g_{6}^{2} &= 1, \\ 
 & & g_{7}^{2} &= 1, \\ 
 & & {t_{5}^{*}}^{2} &= 1  \rangle,
\end{aligned}
\]
whence the nonuniversal commutator relation is identified as
\[
t_{5}^{*}  = [g_{3}, g_{1}] [g_{3}, g_{2}]^{-1}[g_{4}, g_{2}] .  \quad 
\]


\item \label{number:17} 
Let the group $G$ be the representative of this family given by the presentation
\[
\begin{aligned}
\langle g_{1}, \,g_{2}, \,g_{3}, \,g_{4}, \,g_{5}, \,g_{6}, \,g_{7} & \mid & g_{1}^{2} &= g_{4}, \\ 
 & & g_{2}^{2} &= 1, & [g_{2}, g_{1}]  &= g_{3}, \\ 
 & & g_{3}^{2} &= g_{6}, & [g_{3}, g_{1}]  &= g_{5}, & [g_{3}, g_{2}]  &= g_{6}, \\ 
 & & g_{4}^{2} &= g_{7}, & [g_{4}, g_{2}]  &= g_{5}g_{6}, \\ 
 & & g_{5}^{2} &= 1, \\ 
 & & g_{6}^{2} &= 1, \\ 
 & & g_{7}^{2} &= 1\rangle. \\ 
\end{aligned}
\]
We add 11 tails to the presentation as to form a quotient of the universal central extension of the system: 
$g_{1}^{2} = g_{4} t_{1}$,
$g_{2}^{2} =  t_{2}$,
$[g_{2}, g_{1}] = g_{3} t_{3}$,
$g_{3}^{2} = g_{6} t_{4}$,
$[g_{3}, g_{1}] = g_{5} t_{5}$,
$[g_{3}, g_{2}] = g_{6} t_{6}$,
$g_{4}^{2} = g_{7} t_{7}$,
$[g_{4}, g_{2}] = g_{5}g_{6} t_{8}$,
$g_{5}^{2} =  t_{9}$,
$g_{6}^{2} =  t_{10}$,
$g_{7}^{2} =  t_{11}$.
Carrying out consistency checks gives the following relations between the tails:
\[
\begin{aligned}
g_{4}^2 g_{2} & = g_{4} (g_{4} g_{2})& \Longrightarrow & & t_{8}^{2}t_{9}t_{10} & = 1 \\
g_{3}^2 g_{2} & = g_{3} (g_{3} g_{2})& \Longrightarrow & & t_{6}^{2}t_{10} & = 1 \\
g_{3}^2 g_{1} & = g_{3} (g_{3} g_{1})& \Longrightarrow & & t_{5}^{2}t_{9} & = 1 \\
g_{2}^2 g_{1} & = g_{2} (g_{2} g_{1})& \Longrightarrow & & t_{3}^{2}t_{4}t_{6}t_{10} & = 1 \\
g_{2} g_{1}^{2} & = (g_{2} g_{1}) g_{1}& \Longrightarrow & & t_{3}^{2}t_{4}t_{5}t_{8}t_{9}t_{10} & = 1 \\
\end{aligned}
\]
Scanning through the conjugacy class representatives of $G$ and the generators of their centralizers, we see that no new relations are imposed.
Collecting the coefficients of these relations into a matrix yields
\[
T = \bordermatrix{
{} & t_{1} & t_{2} & t_{3} & t_{4} & t_{5} & t_{6} & t_{7} & t_{8} & t_{9} & t_{10} & t_{11} \cr
{} &  &  & 2 & 1 &  & 1 &  &  &  & 1 &  \cr
{} &  &  &  &  & 1 & 1 &  & 1 & 1 & 1 &  \cr
{} &  &  &  &  &  & 2 &  &  &  & 1 &  \cr
{} &  &  &  &  &  &  &  & 2 & 1 & 1 &  \cr
}.
\]
It follows readily that the nontrivial elementary divisors of the Smith normal form of $T$ are all equal to $1$. The torsion subgroup of the group generated by the tails is thus trivial, thereby showing $\B_0(G) = 1$.


\item \label{number:18} 
Let the group $G$ be the representative of this family given by the presentation
\[
\begin{aligned}
\langle g_{1}, \,g_{2}, \,g_{3}, \,g_{4}, \,g_{5}, \,g_{6}, \,g_{7} & \mid & g_{1}^{2} &= g_{6}, \\ 
 & & g_{2}^{2} &= 1, & [g_{2}, g_{1}]  &= g_{5}, \\ 
 & & g_{3}^{2} &= 1, \\ 
 & & g_{4}^{2} &= 1, & [g_{4}, g_{3}]  &= g_{7}, \\ 
 & & g_{5}^{2} &= g_{7}, & [g_{5}, g_{1}]  &= g_{7}, & [g_{5}, g_{2}]  &= g_{7}, \\ 
 & & g_{6}^{2} &= 1, \\ 
 & & g_{7}^{2} &= 1\rangle. \\ 
\end{aligned}
\]
We add 11 tails to the presentation as to form a quotient of the universal central extension of the system: 
$g_{1}^{2} = g_{6} t_{1}$,
$g_{2}^{2} =  t_{2}$,
$[g_{2}, g_{1}] = g_{5} t_{3}$,
$g_{3}^{2} =  t_{4}$,
$g_{4}^{2} =  t_{5}$,
$[g_{4}, g_{3}] = g_{7} t_{6}$,
$g_{5}^{2} = g_{7} t_{7}$,
$[g_{5}, g_{1}] = g_{7} t_{8}$,
$[g_{5}, g_{2}] = g_{7} t_{9}$,
$g_{6}^{2} =  t_{10}$,
$g_{7}^{2} =  t_{11}$.
Carrying out consistency checks gives the following relations between the tails:
\[
\begin{aligned}
g_{5}^2 g_{2} & = g_{5} (g_{5} g_{2})& \Longrightarrow & & t_{9}^{2}t_{11} & = 1 \\
g_{5}^2 g_{1} & = g_{5} (g_{5} g_{1})& \Longrightarrow & & t_{8}^{2}t_{11} & = 1 \\
g_{4}^2 g_{3} & = g_{4} (g_{4} g_{3})& \Longrightarrow & & t_{6}^{2}t_{11} & = 1 \\
g_{2}^2 g_{1} & = g_{2} (g_{2} g_{1})& \Longrightarrow & & t_{3}^{2}t_{7}t_{9}t_{11} & = 1 \\
g_{2} g_{1}^{2} & = (g_{2} g_{1}) g_{1}& \Longrightarrow & & t_{3}^{2}t_{7}t_{8}t_{11} & = 1 \\
\end{aligned}
\]
Scanning through the conjugacy class representatives of $G$ and the generators of their centralizers, we obtain the following relations induced on the tails:
\[
\begin{aligned}
{[g_{4} g_{5} g_{7} , \, g_{1} g_{3} ]}_G & = 1 & \Longrightarrow & & t_{6}t_{8}t_{11} & = 1 \\
\end{aligned}
\]
Collecting the coefficients of these relations into a matrix yields
\[
T = \bordermatrix{
{} & t_{1} & t_{2} & t_{3} & t_{4} & t_{5} & t_{6} & t_{7} & t_{8} & t_{9} & t_{10} & t_{11} \cr
{} &  &  & 2 &  &  &  & 1 &  & 1 &  & 1 \cr
{} &  &  &  &  &  & 1 &  &  & 1 &  & 1 \cr
{} &  &  &  &  &  &  &  & 1 & 1 &  & 1 \cr
{} &  &  &  &  &  &  &  &  & 2 &  & 1 \cr
}.
\]
It follows readily that the nontrivial elementary divisors of the Smith normal form of $T$ are all equal to $1$. The torsion subgroup of the group generated by the tails is thus trivial, thereby showing $\B_0(G) = 1$.


\item \label{number:19} 
Let the group $G$ be the representative of this family given by the presentation
\[
\begin{aligned}
\langle g_{1}, \,g_{2}, \,g_{3}, \,g_{4}, \,g_{5}, \,g_{6}, \,g_{7} & \mid & g_{1}^{2} &= g_{5}, \\ 
 & & g_{2}^{2} &= 1, & [g_{2}, g_{1}]  &= g_{4}, \\ 
 & & g_{3}^{2} &= 1, & [g_{3}, g_{2}]  &= g_{7}, \\ 
 & & g_{4}^{2} &= g_{6}g_{7}, & [g_{4}, g_{1}]  &= g_{6}, & [g_{4}, g_{2}]  &= g_{6}, \\ 
 & & g_{5}^{2} &= 1, \\ 
 & & g_{6}^{2} &= g_{7}, & [g_{6}, g_{1}]  &= g_{7}, & [g_{6}, g_{2}]  &= g_{7}, \\ 
 & & g_{7}^{2} &= 1\rangle. \\ 
\end{aligned}
\]
We add 13 tails to the presentation as to form a quotient of the universal central extension of the system: 
$g_{1}^{2} = g_{5} t_{1}$,
$g_{2}^{2} =  t_{2}$,
$[g_{2}, g_{1}] = g_{4} t_{3}$,
$g_{3}^{2} =  t_{4}$,
$[g_{3}, g_{2}] = g_{7} t_{5}$,
$g_{4}^{2} = g_{6}g_{7} t_{6}$,
$[g_{4}, g_{1}] = g_{6} t_{7}$,
$[g_{4}, g_{2}] = g_{6} t_{8}$,
$g_{5}^{2} =  t_{9}$,
$g_{6}^{2} = g_{7} t_{10}$,
$[g_{6}, g_{1}] = g_{7} t_{11}$,
$[g_{6}, g_{2}] = g_{7} t_{12}$,
$g_{7}^{2} =  t_{13}$.
Carrying out consistency checks gives the following relations between the tails:
\[
\begin{aligned}
g_{4}(g_{2} g_{1}) & = (g_{4} g_{2}) g_{1}  & \Longrightarrow & & t_{11}t_{12}^{-1} & = 1 \\
g_{6}^2 g_{2} & = g_{6} (g_{6} g_{2})& \Longrightarrow & & t_{12}^{2}t_{13} & = 1 \\
g_{4}^2 g_{2} & = g_{4} (g_{4} g_{2})& \Longrightarrow & & t_{8}^{2}t_{10}t_{12}^{-1} & = 1 \\
g_{4}^2 g_{1} & = g_{4} (g_{4} g_{1})& \Longrightarrow & & t_{7}^{2}t_{10}t_{11}^{-1} & = 1 \\
g_{3}^2 g_{2} & = g_{3} (g_{3} g_{2})& \Longrightarrow & & t_{5}^{2}t_{13} & = 1 \\
g_{2}^2 g_{1} & = g_{2} (g_{2} g_{1})& \Longrightarrow & & t_{3}^{2}t_{6}t_{8}t_{10}t_{13} & = 1 \\
g_{2} g_{1}^{2} & = (g_{2} g_{1}) g_{1}& \Longrightarrow & & t_{3}^{2}t_{6}t_{7}t_{10}t_{13} & = 1 \\
\end{aligned}
\]
Scanning through the conjugacy class representatives of $G$ and the generators of their centralizers, we obtain the following relations induced on the tails:
\[
\begin{aligned}
{[g_{3} g_{6} g_{7} , \, g_{2} g_{4} g_{6} ]}_G & = 1 & \Longrightarrow & & t_{5}t_{12}t_{13} & = 1 \\
\end{aligned}
\]
Collecting the coefficients of these relations into a matrix yields
\[
T = \bordermatrix{
{} & t_{1} & t_{2} & t_{3} & t_{4} & t_{5} & t_{6} & t_{7} & t_{8} & t_{9} & t_{10} & t_{11} & t_{12} & t_{13} \cr
{} &  &  & 2 &  &  & 1 &  & 1 &  & 1 &  &  & 1 \cr
{} &  &  &  &  & 1 &  &  &  &  &  &  & 1 & 1 \cr
{} &  &  &  &  &  &  & 1 & 1 &  & 1 &  & 1 & 1 \cr
{} &  &  &  &  &  &  &  & 2 &  & 1 &  & 1 & 1 \cr
{} &  &  &  &  &  &  &  &  &  &  & 1 & 1 & 1 \cr
{} &  &  &  &  &  &  &  &  &  &  &  & 2 & 1 \cr
}.
\]
It follows readily that the nontrivial elementary divisors of the Smith normal form of $T$ are all equal to $1$. The torsion subgroup of the group generated by the tails is thus trivial, thereby showing $\B_0(G) = 1$.


\item \label{number:20} 
Let the group $G$ be the representative of this family given by the presentation
\[
\begin{aligned}
\langle g_{1}, \,g_{2}, \,g_{3}, \,g_{4}, \,g_{5}, \,g_{6}, \,g_{7} & \mid & g_{1}^{2} &= g_{4}, \\ 
 & & g_{2}^{2} &= g_{3}, & [g_{2}, g_{1}]  &= g_{3}, \\ 
 & & g_{3}^{2} &= g_{5}, & [g_{3}, g_{1}]  &= g_{5}, \\ 
 & & g_{4}^{2} &= g_{6}, & [g_{4}, g_{2}]  &= g_{7}, \\ 
 & & g_{5}^{2} &= g_{7}, & [g_{5}, g_{1}]  &= g_{7}, \\ 
 & & g_{6}^{2} &= 1, \\ 
 & & g_{7}^{2} &= 1\rangle. \\ 
\end{aligned}
\]
We add 11 tails to the presentation as to form a quotient of the universal central extension of the system: 
$g_{1}^{2} = g_{4} t_{1}$,
$g_{2}^{2} = g_{3} t_{2}$,
$[g_{2}, g_{1}] = g_{3} t_{3}$,
$g_{3}^{2} = g_{5} t_{4}$,
$[g_{3}, g_{1}] = g_{5} t_{5}$,
$g_{4}^{2} = g_{6} t_{6}$,
$[g_{4}, g_{2}] = g_{7} t_{7}$,
$g_{5}^{2} = g_{7} t_{8}$,
$[g_{5}, g_{1}] = g_{7} t_{9}$,
$g_{6}^{2} =  t_{10}$,
$g_{7}^{2} =  t_{11}$.
Carrying out consistency checks gives the following relations between the tails:
\[
\begin{aligned}
g_{5}^2 g_{1} & = g_{5} (g_{5} g_{1})& \Longrightarrow & & t_{9}^{2}t_{11} & = 1 \\
g_{4}^2 g_{2} & = g_{4} (g_{4} g_{2})& \Longrightarrow & & t_{7}^{2}t_{11} & = 1 \\
g_{3}^2 g_{1} & = g_{3} (g_{3} g_{1})& \Longrightarrow & & t_{5}^{2}t_{8}t_{9}^{-1} & = 1 \\
g_{2}^2 g_{1} & = g_{2} (g_{2} g_{1})& \Longrightarrow & & t_{3}^{2}t_{4}t_{5}^{-1} & = 1 \\
g_{2} g_{1}^{2} & = (g_{2} g_{1}) g_{1}& \Longrightarrow & & t_{3}^{2}t_{4}t_{5}t_{7}t_{8}t_{11} & = 1 \\
\end{aligned}
\]
Scanning through the conjugacy class representatives of $G$ and the generators of their centralizers, we see that no new relations are imposed.
Collecting the coefficients of these relations into a matrix yields
\[
T = \bordermatrix{
{} & t_{1} & t_{2} & t_{3} & t_{4} & t_{5} & t_{6} & t_{7} & t_{8} & t_{9} & t_{10} & t_{11} \cr
{} &  &  & 2 & 1 & 1 &  &  & 1 & 1 &  & 1 \cr
{} &  &  &  &  & 2 &  &  & 1 & 1 &  & 1 \cr
{} &  &  &  &  &  &  & 1 &  & 1 &  & 1 \cr
{} &  &  &  &  &  &  &  &  & 2 &  & 1 \cr
}.
\]
It follows readily that the nontrivial elementary divisors of the Smith normal form of $T$ are all equal to $1$. The torsion subgroup of the group generated by the tails is thus trivial, thereby showing $\B_0(G) = 1$.


\item \label{number:21} 
Let the group $G$ be the representative of this family given by the presentation
\[
\begin{aligned}
\langle g_{1}, \,g_{2}, \,g_{3}, \,g_{4}, \,g_{5}, \,g_{6}, \,g_{7} & \mid & g_{1}^{2} &= g_{4}, \\ 
 & & g_{2}^{2} &= 1, & [g_{2}, g_{1}]  &= g_{3}, \\ 
 & & g_{3}^{2} &= g_{5}, & [g_{3}, g_{1}]  &= g_{5}, & [g_{3}, g_{2}]  &= g_{5}g_{7}, \\ 
 & & g_{4}^{2} &= g_{6}, & [g_{4}, g_{2}]  &= g_{7}, \\ 
 & & g_{5}^{2} &= g_{7}, & [g_{5}, g_{1}]  &= g_{7}, & [g_{5}, g_{2}]  &= g_{7}, \\ 
 & & g_{6}^{2} &= 1, \\ 
 & & g_{7}^{2} &= 1\rangle. \\ 
\end{aligned}
\]
We add 13 tails to the presentation as to form a quotient of the universal central extension of the system: 
$g_{1}^{2} = g_{4} t_{1}$,
$g_{2}^{2} =  t_{2}$,
$[g_{2}, g_{1}] = g_{3} t_{3}$,
$g_{3}^{2} = g_{5} t_{4}$,
$[g_{3}, g_{1}] = g_{5} t_{5}$,
$[g_{3}, g_{2}] = g_{5}g_{7} t_{6}$,
$g_{4}^{2} = g_{6} t_{7}$,
$[g_{4}, g_{2}] = g_{7} t_{8}$,
$g_{5}^{2} = g_{7} t_{9}$,
$[g_{5}, g_{1}] = g_{7} t_{10}$,
$[g_{5}, g_{2}] = g_{7} t_{11}$,
$g_{6}^{2} =  t_{12}$,
$g_{7}^{2} =  t_{13}$.
Carrying out consistency checks gives the following relations between the tails:
\[
\begin{aligned}
g_{3}(g_{2} g_{1}) & = (g_{3} g_{2}) g_{1}  & \Longrightarrow & & t_{10}t_{11}^{-1} & = 1 \\
g_{5}^2 g_{2} & = g_{5} (g_{5} g_{2})& \Longrightarrow & & t_{11}^{2}t_{13} & = 1 \\
g_{4}^2 g_{2} & = g_{4} (g_{4} g_{2})& \Longrightarrow & & t_{8}^{2}t_{13} & = 1 \\
g_{3}^2 g_{2} & = g_{3} (g_{3} g_{2})& \Longrightarrow & & t_{6}^{2}t_{9}t_{11}^{-1}t_{13} & = 1 \\
g_{3}^2 g_{1} & = g_{3} (g_{3} g_{1})& \Longrightarrow & & t_{5}^{2}t_{9}t_{10}^{-1} & = 1 \\
g_{2}^2 g_{1} & = g_{2} (g_{2} g_{1})& \Longrightarrow & & t_{3}^{2}t_{4}t_{6}t_{9}t_{13} & = 1 \\
g_{2} g_{1}^{2} & = (g_{2} g_{1}) g_{1}& \Longrightarrow & & t_{3}^{2}t_{4}t_{5}t_{8}t_{9}t_{13} & = 1 \\
\end{aligned}
\]
Scanning through the conjugacy class representatives of $G$ and the generators of their centralizers, we obtain the following relations induced on the tails:
\[
\begin{aligned}
{[g_{4} g_{5} g_{7} , \, g_{2} g_{4} g_{6} ]}_G & = 1 & \Longrightarrow & & t_{8}t_{11}t_{13} & = 1 \\
\end{aligned}
\]
Collecting the coefficients of these relations into a matrix yields
\[
T = \bordermatrix{
{} & t_{1} & t_{2} & t_{3} & t_{4} & t_{5} & t_{6} & t_{7} & t_{8} & t_{9} & t_{10} & t_{11} & t_{12} & t_{13} \cr
{} &  &  & 2 & 1 &  & 1 &  &  & 1 &  &  &  & 1 \cr
{} &  &  &  &  & 1 & 1 &  &  & 1 &  &  &  & 1 \cr
{} &  &  &  &  &  & 2 &  &  & 1 &  & 1 &  & 2 \cr
{} &  &  &  &  &  &  &  & 1 &  &  & 1 &  & 1 \cr
{} &  &  &  &  &  &  &  &  &  & 1 & 1 &  & 1 \cr
{} &  &  &  &  &  &  &  &  &  &  & 2 &  & 1 \cr
}.
\]
It follows readily that the nontrivial elementary divisors of the Smith normal form of $T$ are all equal to $1$. The torsion subgroup of the group generated by the tails is thus trivial, thereby showing $\B_0(G) = 1$.


\item \label{number:22} 
Let the group $G$ be the representative of this family given by the presentation
\[
\begin{aligned}
\langle g_{1}, \,g_{2}, \,g_{3}, \,g_{4}, \,g_{5}, \,g_{6}, \,g_{7} & \mid & g_{1}^{2} &= g_{4}, \\ 
 & & g_{2}^{2} &= 1, & [g_{2}, g_{1}]  &= g_{3}, \\ 
 & & g_{3}^{2} &= 1, & [g_{3}, g_{1}]  &= g_{5}, \\ 
 & & g_{4}^{2} &= g_{6}, & [g_{4}, g_{2}]  &= g_{5}, & [g_{4}, g_{3}]  &= g_{7}, \\ 
 & & g_{5}^{2} &= 1, & [g_{5}, g_{1}]  &= g_{7}, \\ 
 & & g_{6}^{2} &= 1, \\ 
 & & g_{7}^{2} &= 1\rangle. \\ 
\end{aligned}
\]
We add 12 tails to the presentation as to form a quotient of the universal central extension of the system: 
$g_{1}^{2} = g_{4} t_{1}$,
$g_{2}^{2} =  t_{2}$,
$[g_{2}, g_{1}] = g_{3} t_{3}$,
$g_{3}^{2} =  t_{4}$,
$[g_{3}, g_{1}] = g_{5} t_{5}$,
$g_{4}^{2} = g_{6} t_{6}$,
$[g_{4}, g_{2}] = g_{5} t_{7}$,
$[g_{4}, g_{3}] = g_{7} t_{8}$,
$g_{5}^{2} =  t_{9}$,
$[g_{5}, g_{1}] = g_{7} t_{10}$,
$g_{6}^{2} =  t_{11}$,
$g_{7}^{2} =  t_{12}$.
Carrying out consistency checks gives the following relations between the tails:
\[
\begin{aligned}
g_{4}(g_{2} g_{1}) & = (g_{4} g_{2}) g_{1}  & \Longrightarrow & & t_{8}^{-1}t_{10} & = 1 \\
g_{5}^2 g_{1} & = g_{5} (g_{5} g_{1})& \Longrightarrow & & t_{10}^{2}t_{12} & = 1 \\
g_{4}^2 g_{2} & = g_{4} (g_{4} g_{2})& \Longrightarrow & & t_{7}^{2}t_{9} & = 1 \\
g_{3}^2 g_{1} & = g_{3} (g_{3} g_{1})& \Longrightarrow & & t_{5}^{2}t_{9} & = 1 \\
g_{2}^2 g_{1} & = g_{2} (g_{2} g_{1})& \Longrightarrow & & t_{3}^{2}t_{4} & = 1 \\
g_{2} g_{1}^{2} & = (g_{2} g_{1}) g_{1}& \Longrightarrow & & t_{3}^{2}t_{4}t_{5}t_{7}t_{8}^{2}t_{9}t_{12} & = 1 \\
\end{aligned}
\]
Scanning through the conjugacy class representatives of $G$ and the generators of their centralizers, we see that no new relations are imposed.
Collecting the coefficients of these relations into a matrix yields
\[
T = \bordermatrix{
{} & t_{1} & t_{2} & t_{3} & t_{4} & t_{5} & t_{6} & t_{7} & t_{8} & t_{9} & t_{10} & t_{11} & t_{12} \cr
{} &  &  & 2 & 1 &  &  &  &  &  &  &  &  \cr
{} &  &  &  &  & 1 &  & 1 &  & 1 &  &  &  \cr
{} &  &  &  &  &  &  & 2 &  & 1 &  &  &  \cr
{} &  &  &  &  &  &  &  & 1 &  & 1 &  & 1 \cr
{} &  &  &  &  &  &  &  &  &  & 2 &  & 1 \cr
}.
\]
It follows readily that the nontrivial elementary divisors of the Smith normal form of $T$ are all equal to $1$. The torsion subgroup of the group generated by the tails is thus trivial, thereby showing $\B_0(G) = 1$.


\item \label{number:23} 
Let the group $G$ be the representative of this family given by the presentation
\[
\begin{aligned}
\langle g_{1}, \,g_{2}, \,g_{3}, \,g_{4}, \,g_{5}, \,g_{6}, \,g_{7} & \mid & g_{1}^{2} &= g_{4}, \\ 
 & & g_{2}^{2} &= 1, & [g_{2}, g_{1}]  &= g_{3}, \\ 
 & & g_{3}^{2} &= g_{7}, & [g_{3}, g_{1}]  &= g_{5}, & [g_{3}, g_{2}]  &= g_{7}, \\ 
 & & g_{4}^{2} &= g_{6}, & [g_{4}, g_{2}]  &= g_{5}g_{7}, & [g_{4}, g_{3}]  &= g_{7}, \\ 
 & & g_{5}^{2} &= 1, & [g_{5}, g_{1}]  &= g_{7}, \\ 
 & & g_{6}^{2} &= 1, \\ 
 & & g_{7}^{2} &= 1\rangle. \\ 
\end{aligned}
\]
We add 13 tails to the presentation as to form a quotient of the universal central extension of the system: 
$g_{1}^{2} = g_{4} t_{1}$,
$g_{2}^{2} =  t_{2}$,
$[g_{2}, g_{1}] = g_{3} t_{3}$,
$g_{3}^{2} = g_{7} t_{4}$,
$[g_{3}, g_{1}] = g_{5} t_{5}$,
$[g_{3}, g_{2}] = g_{7} t_{6}$,
$g_{4}^{2} = g_{6} t_{7}$,
$[g_{4}, g_{2}] = g_{5}g_{7} t_{8}$,
$[g_{4}, g_{3}] = g_{7} t_{9}$,
$g_{5}^{2} =  t_{10}$,
$[g_{5}, g_{1}] = g_{7} t_{11}$,
$g_{6}^{2} =  t_{12}$,
$g_{7}^{2} =  t_{13}$.
Carrying out consistency checks gives the following relations between the tails:
\[
\begin{aligned}
g_{4}(g_{2} g_{1}) & = (g_{4} g_{2}) g_{1}  & \Longrightarrow & & t_{9}^{-1}t_{11} & = 1 \\
g_{5}^2 g_{1} & = g_{5} (g_{5} g_{1})& \Longrightarrow & & t_{11}^{2}t_{13} & = 1 \\
g_{4}^2 g_{2} & = g_{4} (g_{4} g_{2})& \Longrightarrow & & t_{8}^{2}t_{10}t_{13} & = 1 \\
g_{3}^2 g_{2} & = g_{3} (g_{3} g_{2})& \Longrightarrow & & t_{6}^{2}t_{13} & = 1 \\
g_{3}^2 g_{1} & = g_{3} (g_{3} g_{1})& \Longrightarrow & & t_{5}^{2}t_{10} & = 1 \\
g_{2}^2 g_{1} & = g_{2} (g_{2} g_{1})& \Longrightarrow & & t_{3}^{2}t_{4}t_{6}t_{13} & = 1 \\
g_{2} g_{1}^{2} & = (g_{2} g_{1}) g_{1}& \Longrightarrow & & t_{3}^{2}t_{4}t_{5}t_{8}t_{9}^{2}t_{10}t_{13}^{2} & = 1 \\
\end{aligned}
\]
Scanning through the conjugacy class representatives of $G$ and the generators of their centralizers, we obtain the following relations induced on the tails:
\[
\begin{aligned}
{[g_{3} g_{5} g_{7} , \, g_{2} g_{4} g_{6} ]}_G & = 1 & \Longrightarrow & & t_{6}t_{9}^{-1} & = 1 \\
\end{aligned}
\]
Collecting the coefficients of these relations into a matrix yields
\[
T = \bordermatrix{
{} & t_{1} & t_{2} & t_{3} & t_{4} & t_{5} & t_{6} & t_{7} & t_{8} & t_{9} & t_{10} & t_{11} & t_{12} & t_{13} \cr
{} &  &  & 2 & 1 &  &  &  &  &  &  & 1 &  & 1 \cr
{} &  &  &  &  & 1 &  &  & 1 &  & 1 & 1 &  & 1 \cr
{} &  &  &  &  &  & 1 &  &  &  &  & 1 &  & 1 \cr
{} &  &  &  &  &  &  &  & 2 &  & 1 &  &  & 1 \cr
{} &  &  &  &  &  &  &  &  & 1 &  & 1 &  & 1 \cr
{} &  &  &  &  &  &  &  &  &  &  & 2 &  & 1 \cr
}.
\]
It follows readily that the nontrivial elementary divisors of the Smith normal form of $T$ are all equal to $1$. The torsion subgroup of the group generated by the tails is thus trivial, thereby showing $\B_0(G) = 1$.


\item \label{number:24} 
Let the group $G$ be the representative of this family given by the presentation
\[
\begin{aligned}
\langle g_{1}, \,g_{2}, \,g_{3}, \,g_{4}, \,g_{5}, \,g_{6}, \,g_{7} & \mid & g_{1}^{2} &= g_{6}, \\ 
 & & g_{2}^{2} &= g_{4}, & [g_{2}, g_{1}]  &= g_{4}, \\ 
 & & g_{3}^{2} &= 1, & [g_{3}, g_{1}]  &= g_{5}, \\ 
 & & g_{4}^{2} &= g_{7}, & [g_{4}, g_{1}]  &= g_{7}, \\ 
 & & g_{5}^{2} &= g_{7}, & [g_{5}, g_{3}]  &= g_{7}, \\ 
 & & g_{6}^{2} &= 1, & [g_{6}, g_{3}]  &= g_{7}, \\ 
 & & g_{7}^{2} &= 1\rangle. \\ 
\end{aligned}
\]
We add 12 tails to the presentation as to form a quotient of the universal central extension of the system: 
$g_{1}^{2} = g_{6} t_{1}$,
$g_{2}^{2} = g_{4} t_{2}$,
$[g_{2}, g_{1}] = g_{4} t_{3}$,
$g_{3}^{2} =  t_{4}$,
$[g_{3}, g_{1}] = g_{5} t_{5}$,
$g_{4}^{2} = g_{7} t_{6}$,
$[g_{4}, g_{1}] = g_{7} t_{7}$,
$g_{5}^{2} = g_{7} t_{8}$,
$[g_{5}, g_{3}] = g_{7} t_{9}$,
$g_{6}^{2} =  t_{10}$,
$[g_{6}, g_{3}] = g_{7} t_{11}$,
$g_{7}^{2} =  t_{12}$.
Carrying out consistency checks gives the following relations between the tails:
\[
\begin{aligned}
g_{6}^2 g_{3} & = g_{6} (g_{6} g_{3})& \Longrightarrow & & t_{11}^{2}t_{12} & = 1 \\
g_{5}^2 g_{3} & = g_{5} (g_{5} g_{3})& \Longrightarrow & & t_{9}^{2}t_{12} & = 1 \\
g_{4}^2 g_{1} & = g_{4} (g_{4} g_{1})& \Longrightarrow & & t_{7}^{2}t_{12} & = 1 \\
g_{3}^2 g_{1} & = g_{3} (g_{3} g_{1})& \Longrightarrow & & t_{5}^{2}t_{8}t_{9}t_{12} & = 1 \\
g_{2}^2 g_{1} & = g_{2} (g_{2} g_{1})& \Longrightarrow & & t_{3}^{2}t_{6}t_{7}^{-1} & = 1 \\
g_{3} g_{1}^{2} & = (g_{3} g_{1}) g_{1}& \Longrightarrow & & t_{5}^{2}t_{8}t_{11}t_{12} & = 1 \\
\end{aligned}
\]
Scanning through the conjugacy class representatives of $G$ and the generators of their centralizers, we obtain the following relations induced on the tails:
\[
\begin{aligned}
{[g_{4} g_{6} g_{7} , \, g_{1} g_{3} ]}_G & = 1 & \Longrightarrow & & t_{7}t_{11}t_{12} & = 1 \\
\end{aligned}
\]
Collecting the coefficients of these relations into a matrix yields
\[
T = \bordermatrix{
{} & t_{1} & t_{2} & t_{3} & t_{4} & t_{5} & t_{6} & t_{7} & t_{8} & t_{9} & t_{10} & t_{11} & t_{12} \cr
{} &  &  & 2 &  &  & 1 &  &  &  &  & 1 & 1 \cr
{} &  &  &  &  & 2 &  &  & 1 &  &  & 1 & 1 \cr
{} &  &  &  &  &  &  & 1 &  &  &  & 1 & 1 \cr
{} &  &  &  &  &  &  &  &  & 1 &  & 1 & 1 \cr
{} &  &  &  &  &  &  &  &  &  &  & 2 & 1 \cr
}.
\]
It follows readily that the nontrivial elementary divisors of the Smith normal form of $T$ are all equal to $1$. The torsion subgroup of the group generated by the tails is thus trivial, thereby showing $\B_0(G) = 1$.


\item \label{number:25} 
Let the group $G$ be the representative of this family given by the presentation
\[
\begin{aligned}
\langle g_{1}, \,g_{2}, \,g_{3}, \,g_{4}, \,g_{5}, \,g_{6}, \,g_{7} & \mid & g_{1}^{2} &= g_{6}, \\ 
 & & g_{2}^{2} &= 1, & [g_{2}, g_{1}]  &= g_{4}, \\ 
 & & g_{3}^{2} &= 1, & [g_{3}, g_{1}]  &= g_{5}, \\ 
 & & g_{4}^{2} &= 1, & [g_{4}, g_{1}]  &= g_{7}, & [g_{4}, g_{3}]  &= g_{7}, \\ 
 & & g_{5}^{2} &= 1, & [g_{5}, g_{2}]  &= g_{7}, \\ 
 & & g_{6}^{2} &= 1, & [g_{6}, g_{2}]  &= g_{7}, \\ 
 & & g_{7}^{2} &= 1\rangle. \\ 
\end{aligned}
\]
We add 13 tails to the presentation as to form a quotient of the universal central extension of the system: 
$g_{1}^{2} = g_{6} t_{1}$,
$g_{2}^{2} =  t_{2}$,
$[g_{2}, g_{1}] = g_{4} t_{3}$,
$g_{3}^{2} =  t_{4}$,
$[g_{3}, g_{1}] = g_{5} t_{5}$,
$g_{4}^{2} =  t_{6}$,
$[g_{4}, g_{1}] = g_{7} t_{7}$,
$[g_{4}, g_{3}] = g_{7} t_{8}$,
$g_{5}^{2} =  t_{9}$,
$[g_{5}, g_{2}] = g_{7} t_{10}$,
$g_{6}^{2} =  t_{11}$,
$[g_{6}, g_{2}] = g_{7} t_{12}$,
$g_{7}^{2} =  t_{13}$.
Carrying out consistency checks gives the following relations between the tails:
\[
\begin{aligned}
g_{3}(g_{2} g_{1}) & = (g_{3} g_{2}) g_{1}  & \Longrightarrow & & t_{8}t_{10}^{-1} & = 1 \\
g_{6}^2 g_{2} & = g_{6} (g_{6} g_{2})& \Longrightarrow & & t_{12}^{2}t_{13} & = 1 \\
g_{5}^2 g_{2} & = g_{5} (g_{5} g_{2})& \Longrightarrow & & t_{10}^{2}t_{13} & = 1 \\
g_{4}^2 g_{1} & = g_{4} (g_{4} g_{1})& \Longrightarrow & & t_{7}^{2}t_{13} & = 1 \\
g_{3}^2 g_{1} & = g_{3} (g_{3} g_{1})& \Longrightarrow & & t_{5}^{2}t_{9} & = 1 \\
g_{2}^2 g_{1} & = g_{2} (g_{2} g_{1})& \Longrightarrow & & t_{3}^{2}t_{6} & = 1 \\
g_{2} g_{1}^{2} & = (g_{2} g_{1}) g_{1}& \Longrightarrow & & t_{3}^{2}t_{6}t_{7}t_{12}t_{13} & = 1 \\
\end{aligned}
\]
Scanning through the conjugacy class representatives of $G$ and the generators of their centralizers, we obtain the following relations induced on the tails:
\[
\begin{aligned}
{[g_{5} g_{6} , \, g_{2} g_{4} ]}_G & = 1 & \Longrightarrow & & t_{10}t_{12}t_{13} & = 1 \\
\end{aligned}
\]
Collecting the coefficients of these relations into a matrix yields
\[
T = \bordermatrix{
{} & t_{1} & t_{2} & t_{3} & t_{4} & t_{5} & t_{6} & t_{7} & t_{8} & t_{9} & t_{10} & t_{11} & t_{12} & t_{13} \cr
{} &  &  & 2 &  &  & 1 &  &  &  &  &  &  &  \cr
{} &  &  &  &  & 2 &  &  &  & 1 &  &  &  &  \cr
{} &  &  &  &  &  &  & 1 &  &  &  &  & 1 & 1 \cr
{} &  &  &  &  &  &  &  & 1 &  &  &  & 1 & 1 \cr
{} &  &  &  &  &  &  &  &  &  & 1 &  & 1 & 1 \cr
{} &  &  &  &  &  &  &  &  &  &  &  & 2 & 1 \cr
}.
\]
It follows readily that the nontrivial elementary divisors of the Smith normal form of $T$ are all equal to $1$. The torsion subgroup of the group generated by the tails is thus trivial, thereby showing $\B_0(G) = 1$.


\item \label{number:26} 
Let the group $G$ be the representative of this family given by the presentation
\[
\begin{aligned}
\langle g_{1}, \,g_{2}, \,g_{3}, \,g_{4}, \,g_{5}, \,g_{6}, \,g_{7} & \mid & g_{1}^{2} &= g_{6}, \\ 
 & & g_{2}^{2} &= 1, & [g_{2}, g_{1}]  &= g_{4}, \\ 
 & & g_{3}^{2} &= 1, & [g_{3}, g_{1}]  &= g_{5}, \\ 
 & & g_{4}^{2} &= g_{7}, & [g_{4}, g_{1}]  &= g_{7}, & [g_{4}, g_{2}]  &= g_{7}, & [g_{4}, g_{3}]  &= g_{7}, \\ 
 & & g_{5}^{2} &= 1, & [g_{5}, g_{2}]  &= g_{7}, \\ 
 & & g_{6}^{2} &= 1, \\ 
 & & g_{7}^{2} &= 1\rangle. \\ 
\end{aligned}
\]
We add 13 tails to the presentation as to form a quotient of the universal central extension of the system: 
$g_{1}^{2} = g_{6} t_{1}$,
$g_{2}^{2} =  t_{2}$,
$[g_{2}, g_{1}] = g_{4} t_{3}$,
$g_{3}^{2} =  t_{4}$,
$[g_{3}, g_{1}] = g_{5} t_{5}$,
$g_{4}^{2} = g_{7} t_{6}$,
$[g_{4}, g_{1}] = g_{7} t_{7}$,
$[g_{4}, g_{2}] = g_{7} t_{8}$,
$[g_{4}, g_{3}] = g_{7} t_{9}$,
$g_{5}^{2} =  t_{10}$,
$[g_{5}, g_{2}] = g_{7} t_{11}$,
$g_{6}^{2} =  t_{12}$,
$g_{7}^{2} =  t_{13}$.
Carrying out consistency checks gives the following relations between the tails:
\[
\begin{aligned}
g_{3}(g_{2} g_{1}) & = (g_{3} g_{2}) g_{1}  & \Longrightarrow & & t_{9}t_{11}^{-1} & = 1 \\
g_{5}^2 g_{2} & = g_{5} (g_{5} g_{2})& \Longrightarrow & & t_{11}^{2}t_{13} & = 1 \\
g_{4}^2 g_{2} & = g_{4} (g_{4} g_{2})& \Longrightarrow & & t_{8}^{2}t_{13} & = 1 \\
g_{4}^2 g_{1} & = g_{4} (g_{4} g_{1})& \Longrightarrow & & t_{7}^{2}t_{13} & = 1 \\
g_{3}^2 g_{1} & = g_{3} (g_{3} g_{1})& \Longrightarrow & & t_{5}^{2}t_{10} & = 1 \\
g_{2}^2 g_{1} & = g_{2} (g_{2} g_{1})& \Longrightarrow & & t_{3}^{2}t_{6}t_{8}t_{13} & = 1 \\
g_{2} g_{1}^{2} & = (g_{2} g_{1}) g_{1}& \Longrightarrow & & t_{3}^{2}t_{6}t_{7}t_{13} & = 1 \\
\end{aligned}
\]
Scanning through the conjugacy class representatives of $G$ and the generators of their centralizers, we obtain the following relations induced on the tails:
\[
\begin{aligned}
{[g_{4} g_{7} , \, g_{1} g_{3} ]}_G & = 1 & \Longrightarrow & & t_{7}t_{9}t_{13} & = 1 \\
\end{aligned}
\]
Collecting the coefficients of these relations into a matrix yields
\[
T = \bordermatrix{
{} & t_{1} & t_{2} & t_{3} & t_{4} & t_{5} & t_{6} & t_{7} & t_{8} & t_{9} & t_{10} & t_{11} & t_{12} & t_{13} \cr
{} &  &  & 2 &  &  & 1 &  &  &  &  & 1 &  & 1 \cr
{} &  &  &  &  & 2 &  &  &  &  & 1 &  &  &  \cr
{} &  &  &  &  &  &  & 1 &  &  &  & 1 &  & 1 \cr
{} &  &  &  &  &  &  &  & 1 &  &  & 1 &  & 1 \cr
{} &  &  &  &  &  &  &  &  & 1 &  & 1 &  & 1 \cr
{} &  &  &  &  &  &  &  &  &  &  & 2 &  & 1 \cr
}.
\]
It follows readily that the nontrivial elementary divisors of the Smith normal form of $T$ are all equal to $1$. The torsion subgroup of the group generated by the tails is thus trivial, thereby showing $\B_0(G) = 1$.


\item \label{number:27} 
Let the group $G$ be the representative of this family given by the presentation
\[
\begin{aligned}
\langle g_{1}, \,g_{2}, \,g_{3}, \,g_{4}, \,g_{5}, \,g_{6}, \,g_{7} & \mid & g_{1}^{2} &= g_{4}, \\ 
 & & g_{2}^{2} &= 1, & [g_{2}, g_{1}]  &= g_{3}, \\ 
 & & g_{3}^{2} &= g_{5}g_{6}, & [g_{3}, g_{1}]  &= g_{5}, & [g_{3}, g_{2}]  &= g_{5}, \\ 
 & & g_{4}^{2} &= 1, \\ 
 & & g_{5}^{2} &= g_{6}g_{7}, & [g_{5}, g_{1}]  &= g_{6}, & [g_{5}, g_{2}]  &= g_{6}, \\ 
 & & g_{6}^{2} &= g_{7}, & [g_{6}, g_{1}]  &= g_{7}, & [g_{6}, g_{2}]  &= g_{7}, \\ 
 & & g_{7}^{2} &= 1\rangle. \\ 
\end{aligned}
\]
We add 14 tails to the presentation as to form a quotient of the universal central extension of the system: 
$g_{1}^{2} = g_{4} t_{1}$,
$g_{2}^{2} =  t_{2}$,
$[g_{2}, g_{1}] = g_{3} t_{3}$,
$g_{3}^{2} = g_{5}g_{6} t_{4}$,
$[g_{3}, g_{1}] = g_{5} t_{5}$,
$[g_{3}, g_{2}] = g_{5} t_{6}$,
$g_{4}^{2} =  t_{7}$,
$g_{5}^{2} = g_{6}g_{7} t_{8}$,
$[g_{5}, g_{1}] = g_{6} t_{9}$,
$[g_{5}, g_{2}] = g_{6} t_{10}$,
$g_{6}^{2} = g_{7} t_{11}$,
$[g_{6}, g_{1}] = g_{7} t_{12}$,
$[g_{6}, g_{2}] = g_{7} t_{13}$,
$g_{7}^{2} =  t_{14}$.
Carrying out consistency checks gives the following relations between the tails:
\[
\begin{aligned}
g_{5}(g_{2} g_{1}) & = (g_{5} g_{2}) g_{1}  & \Longrightarrow & & t_{12}t_{13}^{-1} & = 1 \\
g_{3}(g_{2} g_{1}) & = (g_{3} g_{2}) g_{1}  & \Longrightarrow & & t_{9}t_{10}^{-1} & = 1 \\
g_{6}^2 g_{2} & = g_{6} (g_{6} g_{2})& \Longrightarrow & & t_{13}^{2}t_{14} & = 1 \\
g_{5}^2 g_{2} & = g_{5} (g_{5} g_{2})& \Longrightarrow & & t_{10}^{2}t_{11}t_{13}^{-1} & = 1 \\
g_{3}^2 g_{2} & = g_{3} (g_{3} g_{2})& \Longrightarrow & & t_{6}^{2}t_{8}t_{10}^{-1}t_{13}^{-1} & = 1 \\
g_{3}^2 g_{1} & = g_{3} (g_{3} g_{1})& \Longrightarrow & & t_{5}^{2}t_{8}t_{9}^{-1}t_{12}^{-1} & = 1 \\
g_{2}^2 g_{1} & = g_{2} (g_{2} g_{1})& \Longrightarrow & & t_{3}^{2}t_{4}t_{6}t_{8}t_{11}t_{14} & = 1 \\
g_{2} g_{1}^{2} & = (g_{2} g_{1}) g_{1}& \Longrightarrow & & t_{3}^{2}t_{4}t_{5}t_{8}t_{11}t_{14} & = 1 \\
\end{aligned}
\]
Scanning through the conjugacy class representatives of $G$ and the generators of their centralizers, we see that no new relations are imposed.
Collecting the coefficients of these relations into a matrix yields
\[
T = \bordermatrix{
{} & t_{1} & t_{2} & t_{3} & t_{4} & t_{5} & t_{6} & t_{7} & t_{8} & t_{9} & t_{10} & t_{11} & t_{12} & t_{13} & t_{14} \cr
{} &  &  & 2 & 1 &  & 1 &  & 1 &  &  & 1 &  &  & 1 \cr
{} &  &  &  &  & 1 & 1 &  & 1 &  & 1 & 1 &  &  & 1 \cr
{} &  &  &  &  &  & 2 &  & 1 &  & 1 & 1 &  &  & 1 \cr
{} &  &  &  &  &  &  &  &  & 1 & 1 & 1 &  & 1 & 1 \cr
{} &  &  &  &  &  &  &  &  &  & 2 & 1 &  & 1 & 1 \cr
{} &  &  &  &  &  &  &  &  &  &  &  & 1 & 1 & 1 \cr
{} &  &  &  &  &  &  &  &  &  &  &  &  & 2 & 1 \cr
}.
\]
It follows readily that the nontrivial elementary divisors of the Smith normal form of $T$ are all equal to $1$. The torsion subgroup of the group generated by the tails is thus trivial, thereby showing $\B_0(G) = 1$.


\item \label{number:28} 
Let the group $G$ be the representative of this family given by the presentation
\[
\begin{aligned}
\langle g_{1}, \,g_{2}, \,g_{3}, \,g_{4}, \,g_{5}, \,g_{6}, \,g_{7} & \mid & g_{1}^{2} &= 1, \\ 
 & & g_{2}^{2} &= 1, & [g_{2}, g_{1}]  &= g_{5}, \\ 
 & & g_{3}^{2} &= 1, & [g_{3}, g_{1}]  &= g_{6}, & [g_{3}, g_{2}]  &= g_{7}, \\ 
 & & g_{4}^{2} &= 1, & [g_{4}, g_{1}]  &= g_{5}g_{6}, & [g_{4}, g_{2}]  &= g_{5}, \\ 
 & & g_{5}^{2} &= 1, \\ 
 & & g_{6}^{2} &= 1, \\ 
 & & g_{7}^{2} &= 1\rangle. \\ 
\end{aligned}
\]
We add 12 tails to the presentation as to form a quotient of the universal central extension of the system: 
$g_{1}^{2} =  t_{1}$,
$g_{2}^{2} =  t_{2}$,
$[g_{2}, g_{1}] = g_{5} t_{3}$,
$g_{3}^{2} =  t_{4}$,
$[g_{3}, g_{1}] = g_{6} t_{5}$,
$[g_{3}, g_{2}] = g_{7} t_{6}$,
$g_{4}^{2} =  t_{7}$,
$[g_{4}, g_{1}] = g_{5}g_{6} t_{8}$,
$[g_{4}, g_{2}] = g_{5} t_{9}$,
$g_{5}^{2} =  t_{10}$,
$g_{6}^{2} =  t_{11}$,
$g_{7}^{2} =  t_{12}$.
Carrying out consistency checks gives the following relations between the tails:
\[
\begin{aligned}
g_{4}^2 g_{2} & = g_{4} (g_{4} g_{2})& \Longrightarrow & & t_{9}^{2}t_{10} & = 1 \\
g_{4}^2 g_{1} & = g_{4} (g_{4} g_{1})& \Longrightarrow & & t_{8}^{2}t_{10}t_{11} & = 1 \\
g_{3}^2 g_{2} & = g_{3} (g_{3} g_{2})& \Longrightarrow & & t_{6}^{2}t_{12} & = 1 \\
g_{3}^2 g_{1} & = g_{3} (g_{3} g_{1})& \Longrightarrow & & t_{5}^{2}t_{11} & = 1 \\
g_{2}^2 g_{1} & = g_{2} (g_{2} g_{1})& \Longrightarrow & & t_{3}^{2}t_{10} & = 1 \\
\end{aligned}
\]
Scanning through the conjugacy class representatives of $G$ and the generators of their centralizers, we obtain the following relations induced on the tails:
\[
\begin{aligned}
{[g_{2} g_{5} g_{7} , \, g_{1} g_{4} ]}_G & = 1 & \Longrightarrow & & t_{3}t_{9}^{-1} & = 1 \\
{[g_{2} g_{3} g_{4} g_{5} g_{7} , \, g_{1} ]}_G & = 1 & \Longrightarrow & & t_{3}t_{5}t_{8}t_{10}t_{11} & = 1 \\
\end{aligned}
\]
Collecting the coefficients of these relations into a matrix yields
\[
T = \bordermatrix{
{} & t_{1} & t_{2} & t_{3} & t_{4} & t_{5} & t_{6} & t_{7} & t_{8} & t_{9} & t_{10} & t_{11} & t_{12} \cr
{} &  &  & 1 &  &  &  &  &  & 1 & 1 &  &  \cr
{} &  &  &  &  & 1 &  &  & 1 & 1 & 1 & 1 &  \cr
{} &  &  &  &  &  & 2 &  &  &  &  &  & 1 \cr
{} &  &  &  &  &  &  &  & 2 &  & 1 & 1 &  \cr
{} &  &  &  &  &  &  &  &  & 2 & 1 &  &  \cr
}.
\]
It follows readily that the nontrivial elementary divisors of the Smith normal form of $T$ are all equal to $1$. The torsion subgroup of the group generated by the tails is thus trivial, thereby showing $\B_0(G) = 1$.


\item \label{number:29} 
Let the group $G$ be the representative of this family given by the presentation
\[
\begin{aligned}
\langle g_{1}, \,g_{2}, \,g_{3}, \,g_{4}, \,g_{5}, \,g_{6}, \,g_{7} & \mid & g_{1}^{2} &= 1, \\ 
 & & g_{2}^{2} &= 1, & [g_{2}, g_{1}]  &= g_{5}, \\ 
 & & g_{3}^{2} &= 1, & [g_{3}, g_{1}]  &= g_{6}, & [g_{3}, g_{2}]  &= g_{7}, \\ 
 & & g_{4}^{2} &= 1, & [g_{4}, g_{1}]  &= g_{5}, \\ 
 & & g_{5}^{2} &= 1, \\ 
 & & g_{6}^{2} &= 1, \\ 
 & & g_{7}^{2} &= 1\rangle. \\ 
\end{aligned}
\]
We add 11 tails to the presentation as to form a quotient of the universal central extension of the system: 
$g_{1}^{2} =  t_{1}$,
$g_{2}^{2} =  t_{2}$,
$[g_{2}, g_{1}] = g_{5} t_{3}$,
$g_{3}^{2} =  t_{4}$,
$[g_{3}, g_{1}] = g_{6} t_{5}$,
$[g_{3}, g_{2}] = g_{7} t_{6}$,
$g_{4}^{2} =  t_{7}$,
$[g_{4}, g_{1}] = g_{5} t_{8}$,
$g_{5}^{2} =  t_{9}$,
$g_{6}^{2} =  t_{10}$,
$g_{7}^{2} =  t_{11}$.
Carrying out consistency checks gives the following relations between the tails:
\[
\begin{aligned}
g_{4}^2 g_{1} & = g_{4} (g_{4} g_{1})& \Longrightarrow & & t_{8}^{2}t_{9} & = 1 \\
g_{3}^2 g_{2} & = g_{3} (g_{3} g_{2})& \Longrightarrow & & t_{6}^{2}t_{11} & = 1 \\
g_{3}^2 g_{1} & = g_{3} (g_{3} g_{1})& \Longrightarrow & & t_{5}^{2}t_{10} & = 1 \\
g_{2}^2 g_{1} & = g_{2} (g_{2} g_{1})& \Longrightarrow & & t_{3}^{2}t_{9} & = 1 \\
\end{aligned}
\]
Scanning through the conjugacy class representatives of $G$ and the generators of their centralizers, we obtain the following relations induced on the tails:
\[
\begin{aligned}
{[g_{2} g_{4} g_{7} , \, g_{1} ]}_G & = 1 & \Longrightarrow & & t_{3}t_{8}t_{9} & = 1 \\
\end{aligned}
\]
Collecting the coefficients of these relations into a matrix yields
\[
T = \bordermatrix{
{} & t_{1} & t_{2} & t_{3} & t_{4} & t_{5} & t_{6} & t_{7} & t_{8} & t_{9} & t_{10} & t_{11} \cr
{} &  &  & 1 &  &  &  &  & 1 & 1 &  &  \cr
{} &  &  &  &  & 2 &  &  &  &  & 1 &  \cr
{} &  &  &  &  &  & 2 &  &  &  &  & 1 \cr
{} &  &  &  &  &  &  &  & 2 & 1 &  &  \cr
}.
\]
It follows readily that the nontrivial elementary divisors of the Smith normal form of $T$ are all equal to $1$. The torsion subgroup of the group generated by the tails is thus trivial, thereby showing $\B_0(G) = 1$.


\item \label{number:30} 
Let the group $G$ be the representative of this family given by the presentation
\[
\begin{aligned}
\langle g_{1}, \,g_{2}, \,g_{3}, \,g_{4}, \,g_{5}, \,g_{6}, \,g_{7} & \mid & g_{1}^{2} &= 1, \\ 
 & & g_{2}^{2} &= 1, & [g_{2}, g_{1}]  &= g_{5}, \\ 
 & & g_{3}^{2} &= 1, & [g_{3}, g_{1}]  &= g_{6}, & [g_{3}, g_{2}]  &= g_{7}, \\ 
 & & g_{4}^{2} &= 1, & [g_{4}, g_{2}]  &= g_{5}g_{6}, & [g_{4}, g_{3}]  &= g_{5}, \\ 
 & & g_{5}^{2} &= 1, \\ 
 & & g_{6}^{2} &= 1, \\ 
 & & g_{7}^{2} &= 1\rangle. \\ 
\end{aligned}
\]
We add 12 tails to the presentation as to form a quotient of the universal central extension of the system: 
$g_{1}^{2} =  t_{1}$,
$g_{2}^{2} =  t_{2}$,
$[g_{2}, g_{1}] = g_{5} t_{3}$,
$g_{3}^{2} =  t_{4}$,
$[g_{3}, g_{1}] = g_{6} t_{5}$,
$[g_{3}, g_{2}] = g_{7} t_{6}$,
$g_{4}^{2} =  t_{7}$,
$[g_{4}, g_{2}] = g_{5}g_{6} t_{8}$,
$[g_{4}, g_{3}] = g_{5} t_{9}$,
$g_{5}^{2} =  t_{10}$,
$g_{6}^{2} =  t_{11}$,
$g_{7}^{2} =  t_{12}$.
Carrying out consistency checks gives the following relations between the tails:
\[
\begin{aligned}
g_{4}^2 g_{3} & = g_{4} (g_{4} g_{3})& \Longrightarrow & & t_{9}^{2}t_{10} & = 1 \\
g_{4}^2 g_{2} & = g_{4} (g_{4} g_{2})& \Longrightarrow & & t_{8}^{2}t_{10}t_{11} & = 1 \\
g_{3}^2 g_{2} & = g_{3} (g_{3} g_{2})& \Longrightarrow & & t_{6}^{2}t_{12} & = 1 \\
g_{3}^2 g_{1} & = g_{3} (g_{3} g_{1})& \Longrightarrow & & t_{5}^{2}t_{11} & = 1 \\
g_{2}^2 g_{1} & = g_{2} (g_{2} g_{1})& \Longrightarrow & & t_{3}^{2}t_{10} & = 1 \\
\end{aligned}
\]
Scanning through the conjugacy class representatives of $G$ and the generators of their centralizers, we see that no new relations are imposed.
Collecting the coefficients of these relations into a matrix yields
\[
T = \bordermatrix{
{} & t_{1} & t_{2} & t_{3} & t_{4} & t_{5} & t_{6} & t_{7} & t_{8} & t_{9} & t_{10} & t_{11} & t_{12} \cr
{} &  &  & 2 &  &  &  &  &  &  & 1 &  &  \cr
{} &  &  &  &  & 2 &  &  &  &  &  & 1 &  \cr
{} &  &  &  &  &  & 2 &  &  &  &  &  & 1 \cr
{} &  &  &  &  &  &  &  & 2 &  & 1 & 1 &  \cr
{} &  &  &  &  &  &  &  &  & 2 & 1 &  &  \cr
}.
\]
A change of basis according to the transition matrix (specifying expansions of $t_i^{*}$ by $t_j$)
\[
\bordermatrix{
{} & t_{1}^{*} & t_{2}^{*} & t_{3}^{*} & t_{4}^{*} & t_{5}^{*} & t_{6}^{*} & t_{7}^{*} & t_{8}^{*} & t_{9}^{*} & t_{10}^{*} & t_{11}^{*} & t_{12}^{*} \cr
t_{1} &  &  &  &  &  &  &  &  &  & -1 & 1 &  \cr
t_{2} &  &  &  &  &  & -1 &  &  & -1 &  & -1 & -1 \cr
t_{3} &  &  & 2 & 1 & 1 &  &  &  &  &  & -1 &  \cr
t_{4} &  &  &  &  &  & -1 &  &  & -1 &  &  &  \cr
t_{5} &  &  &  & 1 &  & 1 &  &  & 1 &  &  &  \cr
t_{6} &  & 2 & 2 &  & 2 &  &  &  & -1 &  &  &  \cr
t_{7} &  &  &  &  &  &  & 1 &  &  &  &  &  \cr
t_{8} &  &  & -2 & -3 & -2 &  &  & 1 &  &  &  &  \cr
t_{9} & 2 &  & -4 &  & -1 &  &  &  & 1 &  &  &  \cr
t_{10} & 1 &  & -2 & -1 & -1 &  &  &  &  & 1 &  &  \cr
t_{11} &  &  & -1 & -1 & -1 &  &  &  &  &  & 1 &  \cr
t_{12} &  & 1 & 1 &  & 1 &  &  &  &  &  &  & 1 \cr
}
\]
shows that the nontrivial elementary divisors of the Smith normal form of $T$ are $1$, $1$, $1$, $2$, $2$.  The elements corresponding to the divisors that are greater than $1$ are $t_{4}^{*}$, $t_{5}^{*}$. This already gives
\[
\B_0(G) \cong \langle t_{4}^{*} , \,  t_{5}^{*}  \mid {t_{4}^{*}}^{2} , \, {t_{5}^{*}}^{2} \rangle.
\]

We now deal with explicitly identifying the nonuniversal commutator relations generating $\B_0(G)$.
First, factor out by the tails $t_{i}^{*}$ whose corresponding elementary divisors are either trivial or $1$. Transforming the situation back to the original tails $t_i$, this amounts to the nontrivial expansions given by
\[
\bordermatrix{
{} & t_{4} & t_{5} & t_{6} & t_{9} \cr
t_{4}^{*} & 1 & 1 & 0 & 0 \cr
t_{5}^{*} & 0 & 0 & 1 & 1 \cr
}
\]
and all the other tails $t_i$ are trivial. We thus obtain a commutativity preserving central extension of the group $G$, given by the presentation
\[
\begin{aligned}
\langle g_{1}, \,g_{2}, \,g_{3}, \,g_{4}, \,g_{5}, \,g_{6}, \,g_{7}, \,t_{4}^{*}, \,t_{5}^{*} & \mid & g_{1}^{2} &= 1, \\ 
 & & g_{2}^{2} &= 1, & [g_{2}, g_{1}]  &= g_{5}, \\ 
 & & g_{3}^{2} &= t_{4}^{*} , & [g_{3}, g_{1}]  &= g_{6}t_{4}^{*} , & [g_{3}, g_{2}]  &= g_{7}t_{5}^{*} , \\ 
 & & g_{4}^{2} &= 1, & [g_{4}, g_{2}]  &= g_{5}g_{6}, & [g_{4}, g_{3}]  &= g_{5}t_{5}^{*} , \\ 
 & & g_{5}^{2} &= 1, \\ 
 & & g_{6}^{2} &= 1, \\ 
 & & g_{7}^{2} &= 1, \\ 
 & & {t_{4}^{*}}^{2} &= 1 , \\
 & & {t_{5}^{*}}^{2} &= 1  \rangle,
\end{aligned}
\]
whence the nonuniversal commutator relations are identified as
\[
t_{4}^{*}  = [g_{2}, g_{1}] [g_{3}, g_{1}] [g_{4}, g_{2}]^{-1},  \quad 
t_{5}^{*}  = [g_{2}, g_{1}] [g_{4}, g_{3}]^{-1}.  \quad 
\]


\item \label{number:31} 
Let the group $G$ be the representative of this family given by the presentation
\[
\begin{aligned}
\langle g_{1}, \,g_{2}, \,g_{3}, \,g_{4}, \,g_{5}, \,g_{6}, \,g_{7} & \mid & g_{1}^{2} &= 1, \\ 
 & & g_{2}^{2} &= 1, & [g_{2}, g_{1}]  &= g_{5}, \\ 
 & & g_{3}^{2} &= 1, & [g_{3}, g_{1}]  &= g_{6}, & [g_{3}, g_{2}]  &= g_{7}, \\ 
 & & g_{4}^{2} &= 1, & [g_{4}, g_{3}]  &= g_{5}, \\ 
 & & g_{5}^{2} &= 1, \\ 
 & & g_{6}^{2} &= 1, \\ 
 & & g_{7}^{2} &= 1\rangle. \\ 
\end{aligned}
\]
We add 11 tails to the presentation as to form a quotient of the universal central extension of the system: 
$g_{1}^{2} =  t_{1}$,
$g_{2}^{2} =  t_{2}$,
$[g_{2}, g_{1}] = g_{5} t_{3}$,
$g_{3}^{2} =  t_{4}$,
$[g_{3}, g_{1}] = g_{6} t_{5}$,
$[g_{3}, g_{2}] = g_{7} t_{6}$,
$g_{4}^{2} =  t_{7}$,
$[g_{4}, g_{3}] = g_{5} t_{8}$,
$g_{5}^{2} =  t_{9}$,
$g_{6}^{2} =  t_{10}$,
$g_{7}^{2} =  t_{11}$.
Carrying out consistency checks gives the following relations between the tails:
\[
\begin{aligned}
g_{4}^2 g_{3} & = g_{4} (g_{4} g_{3})& \Longrightarrow & & t_{8}^{2}t_{9} & = 1 \\
g_{3}^2 g_{2} & = g_{3} (g_{3} g_{2})& \Longrightarrow & & t_{6}^{2}t_{11} & = 1 \\
g_{3}^2 g_{1} & = g_{3} (g_{3} g_{1})& \Longrightarrow & & t_{5}^{2}t_{10} & = 1 \\
g_{2}^2 g_{1} & = g_{2} (g_{2} g_{1})& \Longrightarrow & & t_{3}^{2}t_{9} & = 1 \\
\end{aligned}
\]
Scanning through the conjugacy class representatives of $G$ and the generators of their centralizers, we see that no new relations are imposed.
Collecting the coefficients of these relations into a matrix yields
\[
T = \bordermatrix{
{} & t_{1} & t_{2} & t_{3} & t_{4} & t_{5} & t_{6} & t_{7} & t_{8} & t_{9} & t_{10} & t_{11} \cr
{} &  &  & 2 &  &  &  &  &  & 1 &  &  \cr
{} &  &  &  &  & 2 &  &  &  &  & 1 &  \cr
{} &  &  &  &  &  & 2 &  &  &  &  & 1 \cr
{} &  &  &  &  &  &  &  & 2 & 1 &  &  \cr
}.
\]
A change of basis according to the transition matrix (specifying expansions of $t_i^{*}$ by $t_j$)
\[
\bordermatrix{
{} & t_{1}^{*} & t_{2}^{*} & t_{3}^{*} & t_{4}^{*} & t_{5}^{*} & t_{6}^{*} & t_{7}^{*} & t_{8}^{*} & t_{9}^{*} & t_{10}^{*} & t_{11}^{*} \cr
t_{1} &  &  &  &  &  &  &  &  & -1 &  & 1 \cr
t_{2} &  &  &  &  & -2 &  &  & 1 &  & -1 & 1 \cr
t_{3} &  & -2 &  & -3 &  &  &  &  &  &  & -1 \cr
t_{4} &  &  &  &  & 1 & -1 &  &  &  &  &  \cr
t_{5} &  & 2 &  & 2 & 2 &  &  & -1 &  &  &  \cr
t_{6} &  & 4 & 2 & 6 &  & 1 &  &  &  &  &  \cr
t_{7} &  &  &  &  &  &  & 1 &  &  &  &  \cr
t_{8} & 2 & 6 &  & 9 &  &  &  & 1 &  &  &  \cr
t_{9} & 1 & 2 &  & 3 &  &  &  &  & 1 &  &  \cr
t_{10} &  & 1 &  & 1 &  &  &  &  &  & 1 &  \cr
t_{11} &  & 2 & 1 & 3 &  &  &  &  &  &  & 1 \cr
}
\]
shows that the nontrivial elementary divisors of the Smith normal form of $T$ are $1$, $1$, $1$, $2$.  The element corresponding to the divisor that is greater than $1$ is $t_{4}^{*}$. This already gives
\[
\B_0(G) \cong \langle t_{4}^{*}  \mid {t_{4}^{*}}^{2} \rangle.
\]

We now deal with explicitly identifying the nonuniversal commutator relation generating $\B_0(G)$.
First, factor out by the tails $t_{i}^{*}$ whose corresponding elementary divisors are either trivial or $1$. Transforming the situation back to the original tails $t_i$, this amounts to the nontrivial expansions given by
\[
\bordermatrix{
{} & t_{5} & t_{8} \cr
t_{4}^{*} & 1 & 1 \cr
}
\]
and all the other tails $t_i$ are trivial. We thus obtain a commutativity preserving central extension of the group $G$, given by the presentation
\[
\begin{aligned}
\langle g_{1}, \,g_{2}, \,g_{3}, \,g_{4}, \,g_{5}, \,g_{6}, \,g_{7}, \,t_{4}^{*} & \mid & g_{1}^{2} &= 1, \\ 
 & & g_{2}^{2} &= 1, & [g_{2}, g_{1}]  &= g_{5}, \\ 
 & & g_{3}^{2} &= 1, & [g_{3}, g_{1}]  &= g_{6}t_{4}^{*} , & [g_{3}, g_{2}]  &= g_{7}, \\ 
 & & g_{4}^{2} &= 1, & [g_{4}, g_{3}]  &= g_{5}t_{4}^{*} , \\ 
 & & g_{5}^{2} &= 1, \\ 
 & & g_{6}^{2} &= 1, \\ 
 & & g_{7}^{2} &= 1, \\ 
 & & {t_{4}^{*}}^{2} &= 1  \rangle,
\end{aligned}
\]
whence the nonuniversal commutator relation is identified as
\[
t_{4}^{*}  = [g_{2}, g_{1}] [g_{4}, g_{3}]^{-1}.  \quad 
\]


\item \label{number:32} 
Let the group $G$ be the representative of this family given by the presentation
\[
\begin{aligned}
\langle g_{1}, \,g_{2}, \,g_{3}, \,g_{4}, \,g_{5}, \,g_{6}, \,g_{7} & \mid & g_{1}^{2} &= 1, \\ 
 & & g_{2}^{2} &= 1, & [g_{2}, g_{1}]  &= g_{5}, \\ 
 & & g_{3}^{2} &= 1, & [g_{3}, g_{1}]  &= g_{6}, \\ 
 & & g_{4}^{2} &= 1, & [g_{4}, g_{1}]  &= g_{7}, \\ 
 & & g_{5}^{2} &= 1, \\ 
 & & g_{6}^{2} &= 1, \\ 
 & & g_{7}^{2} &= 1\rangle. \\ 
\end{aligned}
\]
We add 10 tails to the presentation as to form a quotient of the universal central extension of the system: 
$g_{1}^{2} =  t_{1}$,
$g_{2}^{2} =  t_{2}$,
$[g_{2}, g_{1}] = g_{5} t_{3}$,
$g_{3}^{2} =  t_{4}$,
$[g_{3}, g_{1}] = g_{6} t_{5}$,
$g_{4}^{2} =  t_{6}$,
$[g_{4}, g_{1}] = g_{7} t_{7}$,
$g_{5}^{2} =  t_{8}$,
$g_{6}^{2} =  t_{9}$,
$g_{7}^{2} =  t_{10}$.
Carrying out consistency checks gives the following relations between the tails:
\[
\begin{aligned}
g_{4}^2 g_{1} & = g_{4} (g_{4} g_{1})& \Longrightarrow & & t_{7}^{2}t_{10} & = 1 \\
g_{3}^2 g_{1} & = g_{3} (g_{3} g_{1})& \Longrightarrow & & t_{5}^{2}t_{9} & = 1 \\
g_{2}^2 g_{1} & = g_{2} (g_{2} g_{1})& \Longrightarrow & & t_{3}^{2}t_{8} & = 1 \\
\end{aligned}
\]
Scanning through the conjugacy class representatives of $G$ and the generators of their centralizers, we see that no new relations are imposed.
Collecting the coefficients of these relations into a matrix yields
\[
T = \bordermatrix{
{} & t_{1} & t_{2} & t_{3} & t_{4} & t_{5} & t_{6} & t_{7} & t_{8} & t_{9} & t_{10} \cr
{} &  &  & 2 &  &  &  &  & 1 &  &  \cr
{} &  &  &  &  & 2 &  &  &  & 1 &  \cr
{} &  &  &  &  &  &  & 2 &  &  & 1 \cr
}.
\]
It follows readily that the nontrivial elementary divisors of the Smith normal form of $T$ are all equal to $1$. The torsion subgroup of the group generated by the tails is thus trivial, thereby showing $\B_0(G) = 1$.


\item \label{number:33} 
Let the group $G$ be the representative of this family given by the presentation
\[
\begin{aligned}
\langle g_{1}, \,g_{2}, \,g_{3}, \,g_{4}, \,g_{5}, \,g_{6}, \,g_{7} & \mid & g_{1}^{2} &= 1, \\ 
 & & g_{2}^{2} &= 1, & [g_{2}, g_{1}]  &= g_{6}, \\ 
 & & g_{3}^{2} &= 1, & [g_{3}, g_{1}]  &= g_{7}, & [g_{3}, g_{2}]  &= g_{7}, \\ 
 & & g_{4}^{2} &= 1, & [g_{4}, g_{1}]  &= g_{7}, & [g_{4}, g_{2}]  &= g_{6}, \\ 
 & & g_{5}^{2} &= 1, & [g_{5}, g_{1}]  &= g_{6}, \\ 
 & & g_{6}^{2} &= 1, \\ 
 & & g_{7}^{2} &= 1\rangle. \\ 
\end{aligned}
\]
We add 13 tails to the presentation as to form a quotient of the universal central extension of the system: 
$g_{1}^{2} =  t_{1}$,
$g_{2}^{2} =  t_{2}$,
$[g_{2}, g_{1}] = g_{6} t_{3}$,
$g_{3}^{2} =  t_{4}$,
$[g_{3}, g_{1}] = g_{7} t_{5}$,
$[g_{3}, g_{2}] = g_{7} t_{6}$,
$g_{4}^{2} =  t_{7}$,
$[g_{4}, g_{1}] = g_{7} t_{8}$,
$[g_{4}, g_{2}] = g_{6} t_{9}$,
$g_{5}^{2} =  t_{10}$,
$[g_{5}, g_{1}] = g_{6} t_{11}$,
$g_{6}^{2} =  t_{12}$,
$g_{7}^{2} =  t_{13}$.
Carrying out consistency checks gives the following relations between the tails:
\[
\begin{aligned}
g_{5}^2 g_{1} & = g_{5} (g_{5} g_{1})& \Longrightarrow & & t_{11}^{2}t_{12} & = 1 \\
g_{4}^2 g_{2} & = g_{4} (g_{4} g_{2})& \Longrightarrow & & t_{9}^{2}t_{12} & = 1 \\
g_{4}^2 g_{1} & = g_{4} (g_{4} g_{1})& \Longrightarrow & & t_{8}^{2}t_{13} & = 1 \\
g_{3}^2 g_{2} & = g_{3} (g_{3} g_{2})& \Longrightarrow & & t_{6}^{2}t_{13} & = 1 \\
g_{3}^2 g_{1} & = g_{3} (g_{3} g_{1})& \Longrightarrow & & t_{5}^{2}t_{13} & = 1 \\
g_{2}^2 g_{1} & = g_{2} (g_{2} g_{1})& \Longrightarrow & & t_{3}^{2}t_{12} & = 1 \\
\end{aligned}
\]
Scanning through the conjugacy class representatives of $G$ and the generators of their centralizers, we obtain the following relations induced on the tails:
\[
\begin{aligned}
{[g_{3} g_{7} , \, g_{1} g_{2} g_{3} ]}_G & = 1 & \Longrightarrow & & t_{5}t_{6}t_{13} & = 1 \\
{[g_{3} g_{4} g_{6} , \, g_{1} g_{4} ]}_G & = 1 & \Longrightarrow & & t_{5}t_{8}t_{13} & = 1 \\
{[g_{2} g_{6} g_{7} , \, g_{1} g_{4} ]}_G & = 1 & \Longrightarrow & & t_{3}t_{9}^{-1} & = 1 \\
{[g_{2} g_{5} g_{6} g_{7} , \, g_{1} ]}_G & = 1 & \Longrightarrow & & t_{3}t_{11}t_{12} & = 1 \\
\end{aligned}
\]
Collecting the coefficients of these relations into a matrix yields
\[
T = \bordermatrix{
{} & t_{1} & t_{2} & t_{3} & t_{4} & t_{5} & t_{6} & t_{7} & t_{8} & t_{9} & t_{10} & t_{11} & t_{12} & t_{13} \cr
{} &  &  & 1 &  &  &  &  &  &  &  & 1 & 1 &  \cr
{} &  &  &  &  & 1 &  &  & 1 &  &  &  &  & 1 \cr
{} &  &  &  &  &  & 1 &  & 1 &  &  &  &  & 1 \cr
{} &  &  &  &  &  &  &  & 2 &  &  &  &  & 1 \cr
{} &  &  &  &  &  &  &  &  & 1 &  & 1 & 1 &  \cr
{} &  &  &  &  &  &  &  &  &  &  & 2 & 1 &  \cr
}.
\]
It follows readily that the nontrivial elementary divisors of the Smith normal form of $T$ are all equal to $1$. The torsion subgroup of the group generated by the tails is thus trivial, thereby showing $\B_0(G) = 1$.


\item \label{number:34} 
Let the group $G$ be the representative of this family given by the presentation
\[
\begin{aligned}
\langle g_{1}, \,g_{2}, \,g_{3}, \,g_{4}, \,g_{5}, \,g_{6}, \,g_{7} & \mid & g_{1}^{2} &= 1, \\ 
 & & g_{2}^{2} &= 1, & [g_{2}, g_{1}]  &= g_{6}, \\ 
 & & g_{3}^{2} &= 1, & [g_{3}, g_{1}]  &= g_{7}, \\ 
 & & g_{4}^{2} &= 1, & [g_{4}, g_{2}]  &= g_{6}, \\ 
 & & g_{5}^{2} &= 1, & [g_{5}, g_{1}]  &= g_{6}, \\ 
 & & g_{6}^{2} &= 1, \\ 
 & & g_{7}^{2} &= 1\rangle. \\ 
\end{aligned}
\]
We add 11 tails to the presentation as to form a quotient of the universal central extension of the system: 
$g_{1}^{2} =  t_{1}$,
$g_{2}^{2} =  t_{2}$,
$[g_{2}, g_{1}] = g_{6} t_{3}$,
$g_{3}^{2} =  t_{4}$,
$[g_{3}, g_{1}] = g_{7} t_{5}$,
$g_{4}^{2} =  t_{6}$,
$[g_{4}, g_{2}] = g_{6} t_{7}$,
$g_{5}^{2} =  t_{8}$,
$[g_{5}, g_{1}] = g_{6} t_{9}$,
$g_{6}^{2} =  t_{10}$,
$g_{7}^{2} =  t_{11}$.
Carrying out consistency checks gives the following relations between the tails:
\[
\begin{aligned}
g_{5}^2 g_{1} & = g_{5} (g_{5} g_{1})& \Longrightarrow & & t_{9}^{2}t_{10} & = 1 \\
g_{4}^2 g_{2} & = g_{4} (g_{4} g_{2})& \Longrightarrow & & t_{7}^{2}t_{10} & = 1 \\
g_{3}^2 g_{1} & = g_{3} (g_{3} g_{1})& \Longrightarrow & & t_{5}^{2}t_{11} & = 1 \\
g_{2}^2 g_{1} & = g_{2} (g_{2} g_{1})& \Longrightarrow & & t_{3}^{2}t_{10} & = 1 \\
\end{aligned}
\]
Scanning through the conjugacy class representatives of $G$ and the generators of their centralizers, we obtain the following relations induced on the tails:
\[
\begin{aligned}
{[g_{4} g_{5} g_{6} , \, g_{1} g_{2} g_{4} ]}_G & = 1 & \Longrightarrow & & t_{7}t_{9}t_{10} & = 1 \\
{[g_{2} g_{6} , \, g_{1} g_{3} g_{4} ]}_G & = 1 & \Longrightarrow & & t_{3}t_{7}^{-1} & = 1 \\
\end{aligned}
\]
Collecting the coefficients of these relations into a matrix yields
\[
T = \bordermatrix{
{} & t_{1} & t_{2} & t_{3} & t_{4} & t_{5} & t_{6} & t_{7} & t_{8} & t_{9} & t_{10} & t_{11} \cr
{} &  &  & 1 &  &  &  &  &  & 1 & 1 &  \cr
{} &  &  &  &  & 2 &  &  &  &  &  & 1 \cr
{} &  &  &  &  &  &  & 1 &  & 1 & 1 &  \cr
{} &  &  &  &  &  &  &  &  & 2 & 1 &  \cr
}.
\]
It follows readily that the nontrivial elementary divisors of the Smith normal form of $T$ are all equal to $1$. The torsion subgroup of the group generated by the tails is thus trivial, thereby showing $\B_0(G) = 1$.


\item \label{number:35} 
Let the group $G$ be the representative of this family given by the presentation
\[
\begin{aligned}
\langle g_{1}, \,g_{2}, \,g_{3}, \,g_{4}, \,g_{5}, \,g_{6}, \,g_{7} & \mid & g_{1}^{2} &= 1, \\ 
 & & g_{2}^{2} &= 1, & [g_{2}, g_{1}]  &= g_{7}, \\ 
 & & g_{3}^{2} &= 1, \\ 
 & & g_{4}^{2} &= 1, & [g_{4}, g_{3}]  &= g_{7}, \\ 
 & & g_{5}^{2} &= 1, & [g_{5}, g_{2}]  &= g_{7}, \\ 
 & & g_{6}^{2} &= 1, & [g_{6}, g_{1}]  &= g_{7}, \\ 
 & & g_{7}^{2} &= 1\rangle. \\ 
\end{aligned}
\]
We add 11 tails to the presentation as to form a quotient of the universal central extension of the system: 
$g_{1}^{2} =  t_{1}$,
$g_{2}^{2} =  t_{2}$,
$[g_{2}, g_{1}] = g_{7} t_{3}$,
$g_{3}^{2} =  t_{4}$,
$g_{4}^{2} =  t_{5}$,
$[g_{4}, g_{3}] = g_{7} t_{6}$,
$g_{5}^{2} =  t_{7}$,
$[g_{5}, g_{2}] = g_{7} t_{8}$,
$g_{6}^{2} =  t_{9}$,
$[g_{6}, g_{1}] = g_{7} t_{10}$,
$g_{7}^{2} =  t_{11}$.
Carrying out consistency checks gives the following relations between the tails:
\[
\begin{aligned}
g_{6}^2 g_{1} & = g_{6} (g_{6} g_{1})& \Longrightarrow & & t_{10}^{2}t_{11} & = 1 \\
g_{5}^2 g_{2} & = g_{5} (g_{5} g_{2})& \Longrightarrow & & t_{8}^{2}t_{11} & = 1 \\
g_{4}^2 g_{3} & = g_{4} (g_{4} g_{3})& \Longrightarrow & & t_{6}^{2}t_{11} & = 1 \\
g_{2}^2 g_{1} & = g_{2} (g_{2} g_{1})& \Longrightarrow & & t_{3}^{2}t_{11} & = 1 \\
\end{aligned}
\]
Scanning through the conjugacy class representatives of $G$ and the generators of their centralizers, we obtain the following relations induced on the tails:
\[
\begin{aligned}
{[g_{5} g_{6} g_{7} , \, g_{1} g_{2} g_{3} ]}_G & = 1 & \Longrightarrow & & t_{8}t_{10}t_{11} & = 1 \\
{[g_{4} g_{6} g_{7} , \, g_{1} g_{3} ]}_G & = 1 & \Longrightarrow & & t_{6}t_{10}t_{11} & = 1 \\
{[g_{2} g_{7} , \, g_{1} g_{3} g_{5} ]}_G & = 1 & \Longrightarrow & & t_{3}t_{8}^{-1} & = 1 \\
\end{aligned}
\]
Collecting the coefficients of these relations into a matrix yields
\[
T = \bordermatrix{
{} & t_{1} & t_{2} & t_{3} & t_{4} & t_{5} & t_{6} & t_{7} & t_{8} & t_{9} & t_{10} & t_{11} \cr
{} &  &  & 1 &  &  &  &  &  &  & 1 & 1 \cr
{} &  &  &  &  &  & 1 &  &  &  & 1 & 1 \cr
{} &  &  &  &  &  &  &  & 1 &  & 1 & 1 \cr
{} &  &  &  &  &  &  &  &  &  & 2 & 1 \cr
}.
\]
It follows readily that the nontrivial elementary divisors of the Smith normal form of $T$ are all equal to $1$. The torsion subgroup of the group generated by the tails is thus trivial, thereby showing $\B_0(G) = 1$.


\item \label{number:36} 
Let the group $G$ be the representative of this family given by the presentation
\[
\begin{aligned}
\langle g_{1}, \,g_{2}, \,g_{3}, \,g_{4}, \,g_{5}, \,g_{6}, \,g_{7} & \mid & g_{1}^{2} &= 1, \\ 
 & & g_{2}^{2} &= 1, & [g_{2}, g_{1}]  &= g_{4}, \\ 
 & & g_{3}^{2} &= 1, & [g_{3}, g_{1}]  &= g_{5}, & [g_{3}, g_{2}]  &= g_{6}, \\ 
 & & g_{4}^{2} &= g_{7}, & [g_{4}, g_{1}]  &= g_{7}, & [g_{4}, g_{2}]  &= g_{7}, \\ 
 & & g_{5}^{2} &= 1, \\ 
 & & g_{6}^{2} &= 1, \\ 
 & & g_{7}^{2} &= 1\rangle. \\ 
\end{aligned}
\]
We add 12 tails to the presentation as to form a quotient of the universal central extension of the system: 
$g_{1}^{2} =  t_{1}$,
$g_{2}^{2} =  t_{2}$,
$[g_{2}, g_{1}] = g_{4} t_{3}$,
$g_{3}^{2} =  t_{4}$,
$[g_{3}, g_{1}] = g_{5} t_{5}$,
$[g_{3}, g_{2}] = g_{6} t_{6}$,
$g_{4}^{2} = g_{7} t_{7}$,
$[g_{4}, g_{1}] = g_{7} t_{8}$,
$[g_{4}, g_{2}] = g_{7} t_{9}$,
$g_{5}^{2} =  t_{10}$,
$g_{6}^{2} =  t_{11}$,
$g_{7}^{2} =  t_{12}$.
Carrying out consistency checks gives the following relations between the tails:
\[
\begin{aligned}
g_{4}^2 g_{2} & = g_{4} (g_{4} g_{2})& \Longrightarrow & & t_{9}^{2}t_{12} & = 1 \\
g_{4}^2 g_{1} & = g_{4} (g_{4} g_{1})& \Longrightarrow & & t_{8}^{2}t_{12} & = 1 \\
g_{3}^2 g_{2} & = g_{3} (g_{3} g_{2})& \Longrightarrow & & t_{6}^{2}t_{11} & = 1 \\
g_{3}^2 g_{1} & = g_{3} (g_{3} g_{1})& \Longrightarrow & & t_{5}^{2}t_{10} & = 1 \\
g_{2}^2 g_{1} & = g_{2} (g_{2} g_{1})& \Longrightarrow & & t_{3}^{2}t_{7}t_{9}t_{12} & = 1 \\
g_{2} g_{1}^{2} & = (g_{2} g_{1}) g_{1}& \Longrightarrow & & t_{3}^{2}t_{7}t_{8}t_{12} & = 1 \\
\end{aligned}
\]
Scanning through the conjugacy class representatives of $G$ and the generators of their centralizers, we see that no new relations are imposed.
Collecting the coefficients of these relations into a matrix yields
\[
T = \bordermatrix{
{} & t_{1} & t_{2} & t_{3} & t_{4} & t_{5} & t_{6} & t_{7} & t_{8} & t_{9} & t_{10} & t_{11} & t_{12} \cr
{} &  &  & 2 &  &  &  & 1 &  & 1 &  &  & 1 \cr
{} &  &  &  &  & 2 &  &  &  &  & 1 &  &  \cr
{} &  &  &  &  &  & 2 &  &  &  &  & 1 &  \cr
{} &  &  &  &  &  &  &  & 1 & 1 &  &  & 1 \cr
{} &  &  &  &  &  &  &  &  & 2 &  &  & 1 \cr
}.
\]
It follows readily that the nontrivial elementary divisors of the Smith normal form of $T$ are all equal to $1$. The torsion subgroup of the group generated by the tails is thus trivial, thereby showing $\B_0(G) = 1$.


\item \label{number:37} 
Let the group $G$ be the representative of this family given by the presentation
\[
\begin{aligned}
\langle g_{1}, \,g_{2}, \,g_{3}, \,g_{4}, \,g_{5}, \,g_{6}, \,g_{7} & \mid & g_{1}^{2} &= g_{5}, \\ 
 & & g_{2}^{2} &= 1, & [g_{2}, g_{1}]  &= g_{4}, \\ 
 & & g_{3}^{2} &= 1, & [g_{3}, g_{1}]  &= g_{7}, \\ 
 & & g_{4}^{2} &= g_{7}, & [g_{4}, g_{1}]  &= g_{6}, & [g_{4}, g_{2}]  &= g_{7}, \\ 
 & & g_{5}^{2} &= 1, & [g_{5}, g_{2}]  &= g_{6}g_{7}, \\ 
 & & g_{6}^{2} &= 1, \\ 
 & & g_{7}^{2} &= 1\rangle. \\ 
\end{aligned}
\]
We add 12 tails to the presentation as to form a quotient of the universal central extension of the system: 
$g_{1}^{2} = g_{5} t_{1}$,
$g_{2}^{2} =  t_{2}$,
$[g_{2}, g_{1}] = g_{4} t_{3}$,
$g_{3}^{2} =  t_{4}$,
$[g_{3}, g_{1}] = g_{7} t_{5}$,
$g_{4}^{2} = g_{7} t_{6}$,
$[g_{4}, g_{1}] = g_{6} t_{7}$,
$[g_{4}, g_{2}] = g_{7} t_{8}$,
$g_{5}^{2} =  t_{9}$,
$[g_{5}, g_{2}] = g_{6}g_{7} t_{10}$,
$g_{6}^{2} =  t_{11}$,
$g_{7}^{2} =  t_{12}$.
Carrying out consistency checks gives the following relations between the tails:
\[
\begin{aligned}
g_{5}^2 g_{2} & = g_{5} (g_{5} g_{2})& \Longrightarrow & & t_{10}^{2}t_{11}t_{12} & = 1 \\
g_{4}^2 g_{2} & = g_{4} (g_{4} g_{2})& \Longrightarrow & & t_{8}^{2}t_{12} & = 1 \\
g_{4}^2 g_{1} & = g_{4} (g_{4} g_{1})& \Longrightarrow & & t_{7}^{2}t_{11} & = 1 \\
g_{3}^2 g_{1} & = g_{3} (g_{3} g_{1})& \Longrightarrow & & t_{5}^{2}t_{12} & = 1 \\
g_{2}^2 g_{1} & = g_{2} (g_{2} g_{1})& \Longrightarrow & & t_{3}^{2}t_{6}t_{8}t_{12} & = 1 \\
g_{2} g_{1}^{2} & = (g_{2} g_{1}) g_{1}& \Longrightarrow & & t_{3}^{2}t_{6}t_{7}t_{10}t_{11}t_{12} & = 1 \\
\end{aligned}
\]
Scanning through the conjugacy class representatives of $G$ and the generators of their centralizers, we see that no new relations are imposed.
Collecting the coefficients of these relations into a matrix yields
\[
T = \bordermatrix{
{} & t_{1} & t_{2} & t_{3} & t_{4} & t_{5} & t_{6} & t_{7} & t_{8} & t_{9} & t_{10} & t_{11} & t_{12} \cr
{} &  &  & 2 &  &  & 1 &  & 1 &  &  &  & 1 \cr
{} &  &  &  &  & 2 &  &  &  &  &  &  & 1 \cr
{} &  &  &  &  &  &  & 1 & 1 &  & 1 & 1 & 1 \cr
{} &  &  &  &  &  &  &  & 2 &  &  &  & 1 \cr
{} &  &  &  &  &  &  &  &  &  & 2 & 1 & 1 \cr
}.
\]
A change of basis according to the transition matrix (specifying expansions of $t_i^{*}$ by $t_j$)
\[
\bordermatrix{
{} & t_{1}^{*} & t_{2}^{*} & t_{3}^{*} & t_{4}^{*} & t_{5}^{*} & t_{6}^{*} & t_{7}^{*} & t_{8}^{*} & t_{9}^{*} & t_{10}^{*} & t_{11}^{*} & t_{12}^{*} \cr
t_{1} &  &  &  &  &  & -1 &  &  &  &  & 1 & 1 \cr
t_{2} &  &  &  &  &  & -1 & -1 &  &  &  &  &  \cr
t_{3} & -2 &  &  &  & -2 &  &  &  &  &  & -1 & -1 \cr
t_{4} &  &  &  &  &  &  &  & -1 &  & -1 & -3 &  \cr
t_{5} & 4 &  &  &  & 3 & 1 &  &  &  &  &  &  \cr
t_{6} & -1 &  &  &  & -1 &  &  &  &  &  &  &  \cr
t_{7} & -4 & 1 &  &  & -3 &  & 1 &  &  &  &  &  \cr
t_{8} & 1 & 1 & 10 & -4 & 7 &  &  & 1 &  &  &  &  \cr
t_{9} &  &  &  &  &  &  &  &  & 1 &  &  &  \cr
t_{10} & 2 & -1 & -4 & 2 & -1 &  &  &  &  & 1 &  &  \cr
t_{11} & -1 &  & -2 & 1 & -2 &  &  &  &  &  & 1 &  \cr
t_{12} & 3 &  & 3 & -1 & 4 &  &  &  &  &  &  & 1 \cr
}
\]
shows that the nontrivial elementary divisors of the Smith normal form of $T$ are $1$, $1$, $1$, $1$, $2$.  The element corresponding to the divisor that is greater than $1$ is $t_{5}^{*}$. This already gives
\[
\B_0(G) \cong \langle t_{5}^{*}  \mid {t_{5}^{*}}^{2} \rangle.
\]

We now deal with explicitly identifying the nonuniversal commutator relation generating $\B_0(G)$.
First, factor out by the tails $t_{i}^{*}$ whose corresponding elementary divisors are either trivial or $1$. Transforming the situation back to the original tails $t_i$, this amounts to the nontrivial expansions given by
\[
\bordermatrix{
{} & t_{1} & t_{3} & t_{5} \cr
t_{5}^{*} & 1 & 1 & 1 \cr
}
\]
and all the other tails $t_i$ are trivial. We thus obtain a commutativity preserving central extension of the group $G$, given by the presentation
\[
\begin{aligned}
\langle g_{1}, \,g_{2}, \,g_{3}, \,g_{4}, \,g_{5}, \,g_{6}, \,g_{7}, \,t_{5}^{*} & \mid & g_{1}^{2} &= g_{5}t_{5}^{*} , \\ 
 & & g_{2}^{2} &= 1, & [g_{2}, g_{1}]  &= g_{4}t_{5}^{*} , \\ 
 & & g_{3}^{2} &= 1, & [g_{3}, g_{1}]  &= g_{7}t_{5}^{*} , \\ 
 & & g_{4}^{2} &= g_{7}, & [g_{4}, g_{1}]  &= g_{6}, & [g_{4}, g_{2}]  &= g_{7}, \\ 
 & & g_{5}^{2} &= 1, & [g_{5}, g_{2}]  &= g_{6}g_{7}, \\ 
 & & g_{6}^{2} &= 1, \\ 
 & & g_{7}^{2} &= 1, \\ 
 & & {t_{5}^{*}}^{2} &= 1  \rangle,
\end{aligned}
\]
whence the nonuniversal commutator relation is identified as
\[
t_{5}^{*}  = [g_{3}, g_{1}] [g_{4}, g_{2}]^{-1}.  \quad 
\]


\item \label{number:38} 
Let the group $G$ be the representative of this family given by the presentation
\[
\begin{aligned}
\langle g_{1}, \,g_{2}, \,g_{3}, \,g_{4}, \,g_{5}, \,g_{6}, \,g_{7} & \mid & g_{1}^{2} &= g_{5}, \\ 
 & & g_{2}^{2} &= 1, & [g_{2}, g_{1}]  &= g_{4}, \\ 
 & & g_{3}^{2} &= 1, & [g_{3}, g_{2}]  &= g_{6}, \\ 
 & & g_{4}^{2} &= g_{7}, & [g_{4}, g_{1}]  &= g_{6}, & [g_{4}, g_{2}]  &= g_{7}, \\ 
 & & g_{5}^{2} &= 1, & [g_{5}, g_{2}]  &= g_{6}g_{7}, \\ 
 & & g_{6}^{2} &= 1, \\ 
 & & g_{7}^{2} &= 1\rangle. \\ 
\end{aligned}
\]
We add 12 tails to the presentation as to form a quotient of the universal central extension of the system: 
$g_{1}^{2} = g_{5} t_{1}$,
$g_{2}^{2} =  t_{2}$,
$[g_{2}, g_{1}] = g_{4} t_{3}$,
$g_{3}^{2} =  t_{4}$,
$[g_{3}, g_{2}] = g_{6} t_{5}$,
$g_{4}^{2} = g_{7} t_{6}$,
$[g_{4}, g_{1}] = g_{6} t_{7}$,
$[g_{4}, g_{2}] = g_{7} t_{8}$,
$g_{5}^{2} =  t_{9}$,
$[g_{5}, g_{2}] = g_{6}g_{7} t_{10}$,
$g_{6}^{2} =  t_{11}$,
$g_{7}^{2} =  t_{12}$.
Carrying out consistency checks gives the following relations between the tails:
\[
\begin{aligned}
g_{5}^2 g_{2} & = g_{5} (g_{5} g_{2})& \Longrightarrow & & t_{10}^{2}t_{11}t_{12} & = 1 \\
g_{4}^2 g_{2} & = g_{4} (g_{4} g_{2})& \Longrightarrow & & t_{8}^{2}t_{12} & = 1 \\
g_{4}^2 g_{1} & = g_{4} (g_{4} g_{1})& \Longrightarrow & & t_{7}^{2}t_{11} & = 1 \\
g_{3}^2 g_{2} & = g_{3} (g_{3} g_{2})& \Longrightarrow & & t_{5}^{2}t_{11} & = 1 \\
g_{2}^2 g_{1} & = g_{2} (g_{2} g_{1})& \Longrightarrow & & t_{3}^{2}t_{6}t_{8}t_{12} & = 1 \\
g_{2} g_{1}^{2} & = (g_{2} g_{1}) g_{1}& \Longrightarrow & & t_{3}^{2}t_{6}t_{7}t_{10}t_{11}t_{12} & = 1 \\
\end{aligned}
\]
Scanning through the conjugacy class representatives of $G$ and the generators of their centralizers, we obtain the following relations induced on the tails:
\[
\begin{aligned}
{[g_{3} g_{4} g_{5} g_{6} , \, g_{2} g_{4} g_{7} ]}_G & = 1 & \Longrightarrow & & t_{5}t_{8}t_{10}t_{11}t_{12} & = 1 \\
\end{aligned}
\]
Collecting the coefficients of these relations into a matrix yields
\[
T = \bordermatrix{
{} & t_{1} & t_{2} & t_{3} & t_{4} & t_{5} & t_{6} & t_{7} & t_{8} & t_{9} & t_{10} & t_{11} & t_{12} \cr
{} &  &  & 2 &  &  & 1 &  & 1 &  &  &  & 1 \cr
{} &  &  &  &  & 1 &  &  & 1 &  & 1 & 1 & 1 \cr
{} &  &  &  &  &  &  & 1 & 1 &  & 1 & 1 & 1 \cr
{} &  &  &  &  &  &  &  & 2 &  &  &  & 1 \cr
{} &  &  &  &  &  &  &  &  &  & 2 & 1 & 1 \cr
}.
\]
It follows readily that the nontrivial elementary divisors of the Smith normal form of $T$ are all equal to $1$. The torsion subgroup of the group generated by the tails is thus trivial, thereby showing $\B_0(G) = 1$.


\item \label{number:39} 
Let the group $G$ be the representative of this family given by the presentation
\[
\begin{aligned}
\langle g_{1}, \,g_{2}, \,g_{3}, \,g_{4}, \,g_{5}, \,g_{6}, \,g_{7} & \mid & g_{1}^{2} &= g_{4}, \\ 
 & & g_{2}^{2} &= g_{5}, & [g_{2}, g_{1}]  &= g_{3}, \\ 
 & & g_{3}^{2} &= 1, & [g_{3}, g_{1}]  &= g_{6}, & [g_{3}, g_{2}]  &= g_{7}, \\ 
 & & g_{4}^{2} &= 1, & [g_{4}, g_{2}]  &= g_{6}, \\ 
 & & g_{5}^{2} &= 1, & [g_{5}, g_{1}]  &= g_{7}, \\ 
 & & g_{6}^{2} &= 1, \\ 
 & & g_{7}^{2} &= 1\rangle. \\ 
\end{aligned}
\]
We add 12 tails to the presentation as to form a quotient of the universal central extension of the system: 
$g_{1}^{2} = g_{4} t_{1}$,
$g_{2}^{2} = g_{5} t_{2}$,
$[g_{2}, g_{1}] = g_{3} t_{3}$,
$g_{3}^{2} =  t_{4}$,
$[g_{3}, g_{1}] = g_{6} t_{5}$,
$[g_{3}, g_{2}] = g_{7} t_{6}$,
$g_{4}^{2} =  t_{7}$,
$[g_{4}, g_{2}] = g_{6} t_{8}$,
$g_{5}^{2} =  t_{9}$,
$[g_{5}, g_{1}] = g_{7} t_{10}$,
$g_{6}^{2} =  t_{11}$,
$g_{7}^{2} =  t_{12}$.
Carrying out consistency checks gives the following relations between the tails:
\[
\begin{aligned}
g_{5}^2 g_{1} & = g_{5} (g_{5} g_{1})& \Longrightarrow & & t_{10}^{2}t_{12} & = 1 \\
g_{4}^2 g_{2} & = g_{4} (g_{4} g_{2})& \Longrightarrow & & t_{8}^{2}t_{11} & = 1 \\
g_{3}^2 g_{2} & = g_{3} (g_{3} g_{2})& \Longrightarrow & & t_{6}^{2}t_{12} & = 1 \\
g_{3}^2 g_{1} & = g_{3} (g_{3} g_{1})& \Longrightarrow & & t_{5}^{2}t_{11} & = 1 \\
g_{2}^2 g_{1} & = g_{2} (g_{2} g_{1})& \Longrightarrow & & t_{3}^{2}t_{4}t_{6}t_{10}^{-1} & = 1 \\
g_{2} g_{1}^{2} & = (g_{2} g_{1}) g_{1}& \Longrightarrow & & t_{3}^{2}t_{4}t_{5}t_{8}t_{11} & = 1 \\
\end{aligned}
\]
Scanning through the conjugacy class representatives of $G$ and the generators of their centralizers, we see that no new relations are imposed.
Collecting the coefficients of these relations into a matrix yields
\[
T = \bordermatrix{
{} & t_{1} & t_{2} & t_{3} & t_{4} & t_{5} & t_{6} & t_{7} & t_{8} & t_{9} & t_{10} & t_{11} & t_{12} \cr
{} &  &  & 2 & 1 &  & 1 &  &  &  & 1 &  & 1 \cr
{} &  &  &  &  & 1 & 1 &  & 1 &  & 1 & 1 & 1 \cr
{} &  &  &  &  &  & 2 &  &  &  &  &  & 1 \cr
{} &  &  &  &  &  &  &  & 2 &  &  & 1 &  \cr
{} &  &  &  &  &  &  &  &  &  & 2 &  & 1 \cr
}.
\]
A change of basis according to the transition matrix (specifying expansions of $t_i^{*}$ by $t_j$)
\[
\bordermatrix{
{} & t_{1}^{*} & t_{2}^{*} & t_{3}^{*} & t_{4}^{*} & t_{5}^{*} & t_{6}^{*} & t_{7}^{*} & t_{8}^{*} & t_{9}^{*} & t_{10}^{*} & t_{11}^{*} & t_{12}^{*} \cr
t_{1} &  &  &  &  &  & -1 &  &  &  &  & 3 & 1 \cr
t_{2} &  &  &  &  &  &  &  &  &  &  & 1 &  \cr
t_{3} &  & -6 &  & -4 & -6 &  &  &  &  &  & -3 & -1 \cr
t_{4} &  & -3 &  & -2 & -3 &  &  &  &  &  & -3 &  \cr
t_{5} & 1 & 6 &  & 4 & 6 & 1 &  &  &  &  &  &  \cr
t_{6} & -1 & -5 &  & -4 & -6 & -1 &  &  &  &  & -1 &  \cr
t_{7} &  &  &  &  &  &  & 1 &  &  &  &  &  \cr
t_{8} & -1 & -4 &  & -2 & -4 &  &  & 1 &  &  &  &  \cr
t_{9} &  &  &  &  &  &  &  &  & 1 &  &  &  \cr
t_{10} & 1 & 11 & 2 & 6 & 12 &  &  &  &  & 1 &  &  \cr
t_{11} &  & 1 &  & 1 & 1 &  &  &  &  &  & 1 &  \cr
t_{12} &  & 3 & 1 & 1 & 3 &  &  &  &  &  &  & 1 \cr
}
\]
shows that the nontrivial elementary divisors of the Smith normal form of $T$ are $1$, $1$, $1$, $1$, $2$.  The element corresponding to the divisor that is greater than $1$ is $t_{5}^{*}$. This already gives
\[
\B_0(G) \cong \langle t_{5}^{*}  \mid {t_{5}^{*}}^{2} \rangle.
\]

We now deal with explicitly identifying the nonuniversal commutator relation generating $\B_0(G)$.
First, factor out by the tails $t_{i}^{*}$ whose corresponding elementary divisors are either trivial or $1$. Transforming the situation back to the original tails $t_i$, this amounts to the nontrivial expansions given by
\[
\bordermatrix{
{} & t_{4} & t_{5} & t_{6} \cr
t_{5}^{*} & 1 & 1 & 1 \cr
}
\]
and all the other tails $t_i$ are trivial. We thus obtain a commutativity preserving central extension of the group $G$, given by the presentation
\[
\begin{aligned}
\langle g_{1}, \,g_{2}, \,g_{3}, \,g_{4}, \,g_{5}, \,g_{6}, \,g_{7}, \,t_{5}^{*} & \mid & g_{1}^{2} &= g_{4}, \\ 
 & & g_{2}^{2} &= g_{5}, & [g_{2}, g_{1}]  &= g_{3}, \\ 
 & & g_{3}^{2} &= t_{5}^{*} , & [g_{3}, g_{1}]  &= g_{6}t_{5}^{*} , & [g_{3}, g_{2}]  &= g_{7}t_{5}^{*} , \\ 
 & & g_{4}^{2} &= 1, & [g_{4}, g_{2}]  &= g_{6}, \\ 
 & & g_{5}^{2} &= 1, & [g_{5}, g_{1}]  &= g_{7}, \\ 
 & & g_{6}^{2} &= 1, \\ 
 & & g_{7}^{2} &= 1, \\ 
 & & {t_{5}^{*}}^{2} &= 1  \rangle,
\end{aligned}
\]
whence the nonuniversal commutator relation is identified as
\[
t_{5}^{*}  = [g_{3}, g_{2}] [g_{5}, g_{1}]^{-1}.  \quad 
\]


\item \label{number:40} 
Let the group $G$ be the representative of this family given by the presentation
\[
\begin{aligned}
\langle g_{1}, \,g_{2}, \,g_{3}, \,g_{4}, \,g_{5}, \,g_{6}, \,g_{7} & \mid & g_{1}^{2} &= 1, \\ 
 & & g_{2}^{2} &= 1, & [g_{2}, g_{1}]  &= g_{5}, \\ 
 & & g_{3}^{2} &= 1, & [g_{3}, g_{1}]  &= g_{6}, \\ 
 & & g_{4}^{2} &= 1, & [g_{4}, g_{2}]  &= g_{5}, \\ 
 & & g_{5}^{2} &= g_{7}, & [g_{5}, g_{1}]  &= g_{7}, & [g_{5}, g_{2}]  &= g_{7}, & [g_{5}, g_{4}]  &= g_{7}, \\ 
 & & g_{6}^{2} &= 1, \\ 
 & & g_{7}^{2} &= 1\rangle. \\ 
\end{aligned}
\]
We add 13 tails to the presentation as to form a quotient of the universal central extension of the system: 
$g_{1}^{2} =  t_{1}$,
$g_{2}^{2} =  t_{2}$,
$[g_{2}, g_{1}] = g_{5} t_{3}$,
$g_{3}^{2} =  t_{4}$,
$[g_{3}, g_{1}] = g_{6} t_{5}$,
$g_{4}^{2} =  t_{6}$,
$[g_{4}, g_{2}] = g_{5} t_{7}$,
$g_{5}^{2} = g_{7} t_{8}$,
$[g_{5}, g_{1}] = g_{7} t_{9}$,
$[g_{5}, g_{2}] = g_{7} t_{10}$,
$[g_{5}, g_{4}] = g_{7} t_{11}$,
$g_{6}^{2} =  t_{12}$,
$g_{7}^{2} =  t_{13}$.
Carrying out consistency checks gives the following relations between the tails:
\[
\begin{aligned}
g_{4}(g_{2} g_{1}) & = (g_{4} g_{2}) g_{1}  & \Longrightarrow & & t_{9}t_{11}t_{13} & = 1 \\
g_{5}^2 g_{4} & = g_{5} (g_{5} g_{4})& \Longrightarrow & & t_{11}^{2}t_{13} & = 1 \\
g_{5}^2 g_{2} & = g_{5} (g_{5} g_{2})& \Longrightarrow & & t_{10}^{2}t_{13} & = 1 \\
g_{4}^2 g_{2} & = g_{4} (g_{4} g_{2})& \Longrightarrow & & t_{7}^{2}t_{8}t_{11}t_{13} & = 1 \\
g_{3}^2 g_{1} & = g_{3} (g_{3} g_{1})& \Longrightarrow & & t_{5}^{2}t_{12} & = 1 \\
g_{2}^2 g_{1} & = g_{2} (g_{2} g_{1})& \Longrightarrow & & t_{3}^{2}t_{8}t_{10}t_{13} & = 1 \\
g_{4} g_{2}^{2} & = (g_{4} g_{2}) g_{2}& \Longrightarrow & & t_{7}^{2}t_{8}t_{10}t_{13} & = 1 \\
\end{aligned}
\]
Scanning through the conjugacy class representatives of $G$ and the generators of their centralizers, we obtain the following relations induced on the tails:
\[
\begin{aligned}
{[g_{2} g_{5} g_{7} , \, g_{1} g_{3} g_{4} g_{5} g_{7} ]}_G & = 1 & \Longrightarrow & & t_{3}t_{7}^{-1}t_{9}t_{10}^{-1} & = 1 \\
\end{aligned}
\]
Collecting the coefficients of these relations into a matrix yields
\[
T = \bordermatrix{
{} & t_{1} & t_{2} & t_{3} & t_{4} & t_{5} & t_{6} & t_{7} & t_{8} & t_{9} & t_{10} & t_{11} & t_{12} & t_{13} \cr
{} &  &  & 1 &  &  &  & 1 & 1 &  &  & 1 &  & 1 \cr
{} &  &  &  &  & 2 &  &  &  &  &  &  & 1 &  \cr
{} &  &  &  &  &  &  & 2 & 1 &  &  & 1 &  & 1 \cr
{} &  &  &  &  &  &  &  &  & 1 &  & 1 &  & 1 \cr
{} &  &  &  &  &  &  &  &  &  & 1 & 1 &  & 1 \cr
{} &  &  &  &  &  &  &  &  &  &  & 2 &  & 1 \cr
}.
\]
It follows readily that the nontrivial elementary divisors of the Smith normal form of $T$ are all equal to $1$. The torsion subgroup of the group generated by the tails is thus trivial, thereby showing $\B_0(G) = 1$.


\item \label{number:41} 
Let the group $G$ be the representative of this family given by the presentation
\[
\begin{aligned}
\langle g_{1}, \,g_{2}, \,g_{3}, \,g_{4}, \,g_{5}, \,g_{6}, \,g_{7} & \mid & g_{1}^{2} &= 1, \\ 
 & & g_{2}^{2} &= 1, & [g_{2}, g_{1}]  &= g_{5}, \\ 
 & & g_{3}^{2} &= 1, & [g_{3}, g_{1}]  &= g_{6}, \\ 
 & & g_{4}^{2} &= 1, & [g_{4}, g_{1}]  &= g_{7}, & [g_{4}, g_{2}]  &= g_{5}, \\ 
 & & g_{5}^{2} &= g_{7}, & [g_{5}, g_{1}]  &= g_{7}, & [g_{5}, g_{2}]  &= g_{7}, & [g_{5}, g_{4}]  &= g_{7}, \\ 
 & & g_{6}^{2} &= 1, \\ 
 & & g_{7}^{2} &= 1\rangle. \\ 
\end{aligned}
\]
We add 14 tails to the presentation as to form a quotient of the universal central extension of the system: 
$g_{1}^{2} =  t_{1}$,
$g_{2}^{2} =  t_{2}$,
$[g_{2}, g_{1}] = g_{5} t_{3}$,
$g_{3}^{2} =  t_{4}$,
$[g_{3}, g_{1}] = g_{6} t_{5}$,
$g_{4}^{2} =  t_{6}$,
$[g_{4}, g_{1}] = g_{7} t_{7}$,
$[g_{4}, g_{2}] = g_{5} t_{8}$,
$g_{5}^{2} = g_{7} t_{9}$,
$[g_{5}, g_{1}] = g_{7} t_{10}$,
$[g_{5}, g_{2}] = g_{7} t_{11}$,
$[g_{5}, g_{4}] = g_{7} t_{12}$,
$g_{6}^{2} =  t_{13}$,
$g_{7}^{2} =  t_{14}$.
Carrying out consistency checks gives the following relations between the tails:
\[
\begin{aligned}
g_{4}(g_{2} g_{1}) & = (g_{4} g_{2}) g_{1}  & \Longrightarrow & & t_{10}t_{12}t_{14} & = 1 \\
g_{5}^2 g_{4} & = g_{5} (g_{5} g_{4})& \Longrightarrow & & t_{12}^{2}t_{14} & = 1 \\
g_{5}^2 g_{2} & = g_{5} (g_{5} g_{2})& \Longrightarrow & & t_{11}^{2}t_{14} & = 1 \\
g_{4}^2 g_{2} & = g_{4} (g_{4} g_{2})& \Longrightarrow & & t_{8}^{2}t_{9}t_{12}t_{14} & = 1 \\
g_{4}^2 g_{1} & = g_{4} (g_{4} g_{1})& \Longrightarrow & & t_{7}^{2}t_{14} & = 1 \\
g_{3}^2 g_{1} & = g_{3} (g_{3} g_{1})& \Longrightarrow & & t_{5}^{2}t_{13} & = 1 \\
g_{2}^2 g_{1} & = g_{2} (g_{2} g_{1})& \Longrightarrow & & t_{3}^{2}t_{9}t_{11}t_{14} & = 1 \\
g_{4} g_{2}^{2} & = (g_{4} g_{2}) g_{2}& \Longrightarrow & & t_{8}^{2}t_{9}t_{11}t_{14} & = 1 \\
\end{aligned}
\]
Scanning through the conjugacy class representatives of $G$ and the generators of their centralizers, we obtain the following relations induced on the tails:
\[
\begin{aligned}
{[g_{4} g_{5} g_{7} , \, g_{1} ]}_G & = 1 & \Longrightarrow & & t_{7}t_{10}t_{14} & = 1 \\
{[g_{2} g_{5} g_{7} , \, g_{1} g_{3} g_{4} g_{5} g_{7} ]}_G & = 1 & \Longrightarrow & & t_{3}t_{8}^{-1}t_{10}t_{11}^{-1} & = 1 \\
\end{aligned}
\]
Collecting the coefficients of these relations into a matrix yields
\[
T = \bordermatrix{
{} & t_{1} & t_{2} & t_{3} & t_{4} & t_{5} & t_{6} & t_{7} & t_{8} & t_{9} & t_{10} & t_{11} & t_{12} & t_{13} & t_{14} \cr
{} &  &  & 1 &  &  &  &  & 1 & 1 &  &  & 1 &  & 1 \cr
{} &  &  &  &  & 2 &  &  &  &  &  &  &  & 1 &  \cr
{} &  &  &  &  &  &  & 1 &  &  &  &  & 1 &  & 1 \cr
{} &  &  &  &  &  &  &  & 2 & 1 &  &  & 1 &  & 1 \cr
{} &  &  &  &  &  &  &  &  &  & 1 &  & 1 &  & 1 \cr
{} &  &  &  &  &  &  &  &  &  &  & 1 & 1 &  & 1 \cr
{} &  &  &  &  &  &  &  &  &  &  &  & 2 &  & 1 \cr
}.
\]
It follows readily that the nontrivial elementary divisors of the Smith normal form of $T$ are all equal to $1$. The torsion subgroup of the group generated by the tails is thus trivial, thereby showing $\B_0(G) = 1$.


\item \label{number:42} 
Let the group $G$ be the representative of this family given by the presentation
\[
\begin{aligned}
\langle g_{1}, \,g_{2}, \,g_{3}, \,g_{4}, \,g_{5}, \,g_{6}, \,g_{7} & \mid & g_{1}^{2} &= 1, \\ 
 & & g_{2}^{2} &= 1, & [g_{2}, g_{1}]  &= g_{5}, \\ 
 & & g_{3}^{2} &= 1, & [g_{3}, g_{1}]  &= g_{6}, & [g_{3}, g_{2}]  &= g_{5}, \\ 
 & & g_{4}^{2} &= 1, & [g_{4}, g_{1}]  &= g_{5}, & [g_{4}, g_{2}]  &= g_{7}, \\ 
 & & g_{5}^{2} &= 1, \\ 
 & & g_{6}^{2} &= g_{7}, & [g_{6}, g_{1}]  &= g_{7}, & [g_{6}, g_{3}]  &= g_{7}, \\ 
 & & g_{7}^{2} &= 1\rangle. \\ 
\end{aligned}
\]
We add 14 tails to the presentation as to form a quotient of the universal central extension of the system: 
$g_{1}^{2} =  t_{1}$,
$g_{2}^{2} =  t_{2}$,
$[g_{2}, g_{1}] = g_{5} t_{3}$,
$g_{3}^{2} =  t_{4}$,
$[g_{3}, g_{1}] = g_{6} t_{5}$,
$[g_{3}, g_{2}] = g_{5} t_{6}$,
$g_{4}^{2} =  t_{7}$,
$[g_{4}, g_{1}] = g_{5} t_{8}$,
$[g_{4}, g_{2}] = g_{7} t_{9}$,
$g_{5}^{2} =  t_{10}$,
$g_{6}^{2} = g_{7} t_{11}$,
$[g_{6}, g_{1}] = g_{7} t_{12}$,
$[g_{6}, g_{3}] = g_{7} t_{13}$,
$g_{7}^{2} =  t_{14}$.
Carrying out consistency checks gives the following relations between the tails:
\[
\begin{aligned}
g_{6}^2 g_{3} & = g_{6} (g_{6} g_{3})& \Longrightarrow & & t_{13}^{2}t_{14} & = 1 \\
g_{6}^2 g_{1} & = g_{6} (g_{6} g_{1})& \Longrightarrow & & t_{12}^{2}t_{14} & = 1 \\
g_{4}^2 g_{2} & = g_{4} (g_{4} g_{2})& \Longrightarrow & & t_{9}^{2}t_{14} & = 1 \\
g_{4}^2 g_{1} & = g_{4} (g_{4} g_{1})& \Longrightarrow & & t_{8}^{2}t_{10} & = 1 \\
g_{3}^2 g_{2} & = g_{3} (g_{3} g_{2})& \Longrightarrow & & t_{6}^{2}t_{10} & = 1 \\
g_{3}^2 g_{1} & = g_{3} (g_{3} g_{1})& \Longrightarrow & & t_{5}^{2}t_{11}t_{13}t_{14} & = 1 \\
g_{2}^2 g_{1} & = g_{2} (g_{2} g_{1})& \Longrightarrow & & t_{3}^{2}t_{10} & = 1 \\
g_{3} g_{1}^{2} & = (g_{3} g_{1}) g_{1}& \Longrightarrow & & t_{5}^{2}t_{11}t_{12}t_{14} & = 1 \\
\end{aligned}
\]
Scanning through the conjugacy class representatives of $G$ and the generators of their centralizers, we obtain the following relations induced on the tails:
\[
\begin{aligned}
{[g_{4} g_{5} g_{6} g_{7} , \, g_{2} g_{3} g_{4} ]}_G & = 1 & \Longrightarrow & & t_{9}t_{13}t_{14} & = 1 \\
{[g_{2} g_{5} g_{7} , \, g_{1} g_{3} ]}_G & = 1 & \Longrightarrow & & t_{3}t_{6}^{-1} & = 1 \\
{[g_{2} g_{4} g_{5} g_{7} , \, g_{1} ]}_G & = 1 & \Longrightarrow & & t_{3}t_{8}t_{10} & = 1 \\
\end{aligned}
\]
Collecting the coefficients of these relations into a matrix yields
\[
T = \bordermatrix{
{} & t_{1} & t_{2} & t_{3} & t_{4} & t_{5} & t_{6} & t_{7} & t_{8} & t_{9} & t_{10} & t_{11} & t_{12} & t_{13} & t_{14} \cr
{} &  &  & 1 &  &  &  &  & 1 &  & 1 &  &  &  &  \cr
{} &  &  &  &  & 2 &  &  &  &  &  & 1 &  & 1 & 1 \cr
{} &  &  &  &  &  & 1 &  & 1 &  & 1 &  &  &  &  \cr
{} &  &  &  &  &  &  &  & 2 &  & 1 &  &  &  &  \cr
{} &  &  &  &  &  &  &  &  & 1 &  &  &  & 1 & 1 \cr
{} &  &  &  &  &  &  &  &  &  &  &  & 1 & 1 & 1 \cr
{} &  &  &  &  &  &  &  &  &  &  &  &  & 2 & 1 \cr
}.
\]
It follows readily that the nontrivial elementary divisors of the Smith normal form of $T$ are all equal to $1$. The torsion subgroup of the group generated by the tails is thus trivial, thereby showing $\B_0(G) = 1$.


\item \label{number:43} 
Let the group $G$ be the representative of this family given by the presentation
\[
\begin{aligned}
\langle g_{1}, \,g_{2}, \,g_{3}, \,g_{4}, \,g_{5}, \,g_{6}, \,g_{7} & \mid & g_{1}^{2} &= 1, \\ 
 & & g_{2}^{2} &= 1, & [g_{2}, g_{1}]  &= g_{5}, \\ 
 & & g_{3}^{2} &= 1, & [g_{3}, g_{1}]  &= g_{6}, & [g_{3}, g_{2}]  &= g_{5}g_{7}, \\ 
 & & g_{4}^{2} &= 1, & [g_{4}, g_{1}]  &= g_{5}, \\ 
 & & g_{5}^{2} &= 1, \\ 
 & & g_{6}^{2} &= g_{7}, & [g_{6}, g_{1}]  &= g_{7}, & [g_{6}, g_{3}]  &= g_{7}, \\ 
 & & g_{7}^{2} &= 1\rangle. \\ 
\end{aligned}
\]
We add 13 tails to the presentation as to form a quotient of the universal central extension of the system: 
$g_{1}^{2} =  t_{1}$,
$g_{2}^{2} =  t_{2}$,
$[g_{2}, g_{1}] = g_{5} t_{3}$,
$g_{3}^{2} =  t_{4}$,
$[g_{3}, g_{1}] = g_{6} t_{5}$,
$[g_{3}, g_{2}] = g_{5}g_{7} t_{6}$,
$g_{4}^{2} =  t_{7}$,
$[g_{4}, g_{1}] = g_{5} t_{8}$,
$g_{5}^{2} =  t_{9}$,
$g_{6}^{2} = g_{7} t_{10}$,
$[g_{6}, g_{1}] = g_{7} t_{11}$,
$[g_{6}, g_{3}] = g_{7} t_{12}$,
$g_{7}^{2} =  t_{13}$.
Carrying out consistency checks gives the following relations between the tails:
\[
\begin{aligned}
g_{6}^2 g_{3} & = g_{6} (g_{6} g_{3})& \Longrightarrow & & t_{12}^{2}t_{13} & = 1 \\
g_{6}^2 g_{1} & = g_{6} (g_{6} g_{1})& \Longrightarrow & & t_{11}^{2}t_{13} & = 1 \\
g_{4}^2 g_{1} & = g_{4} (g_{4} g_{1})& \Longrightarrow & & t_{8}^{2}t_{9} & = 1 \\
g_{3}^2 g_{2} & = g_{3} (g_{3} g_{2})& \Longrightarrow & & t_{6}^{2}t_{9}t_{13} & = 1 \\
g_{3}^2 g_{1} & = g_{3} (g_{3} g_{1})& \Longrightarrow & & t_{5}^{2}t_{10}t_{12}t_{13} & = 1 \\
g_{2}^2 g_{1} & = g_{2} (g_{2} g_{1})& \Longrightarrow & & t_{3}^{2}t_{9} & = 1 \\
g_{3} g_{1}^{2} & = (g_{3} g_{1}) g_{1}& \Longrightarrow & & t_{5}^{2}t_{10}t_{11}t_{13} & = 1 \\
\end{aligned}
\]
Scanning through the conjugacy class representatives of $G$ and the generators of their centralizers, we obtain the following relations induced on the tails:
\[
\begin{aligned}
{[g_{2} g_{4} g_{5} , \, g_{1} g_{4} ]}_G & = 1 & \Longrightarrow & & t_{3}t_{8}t_{9} & = 1 \\
\end{aligned}
\]
Collecting the coefficients of these relations into a matrix yields
\[
T = \bordermatrix{
{} & t_{1} & t_{2} & t_{3} & t_{4} & t_{5} & t_{6} & t_{7} & t_{8} & t_{9} & t_{10} & t_{11} & t_{12} & t_{13} \cr
{} &  &  & 1 &  &  &  &  & 1 & 1 &  &  &  &  \cr
{} &  &  &  &  & 2 &  &  &  &  & 1 &  & 1 & 1 \cr
{} &  &  &  &  &  & 2 &  &  & 1 &  &  &  & 1 \cr
{} &  &  &  &  &  &  &  & 2 & 1 &  &  &  &  \cr
{} &  &  &  &  &  &  &  &  &  &  & 1 & 1 & 1 \cr
{} &  &  &  &  &  &  &  &  &  &  &  & 2 & 1 \cr
}.
\]
A change of basis according to the transition matrix (specifying expansions of $t_i^{*}$ by $t_j$)
\[
\bordermatrix{
{} & t_{1}^{*} & t_{2}^{*} & t_{3}^{*} & t_{4}^{*} & t_{5}^{*} & t_{6}^{*} & t_{7}^{*} & t_{8}^{*} & t_{9}^{*} & t_{10}^{*} & t_{11}^{*} & t_{12}^{*} & t_{13}^{*} \cr
t_{1} &  &  &  &  &  &  & 1 &  & 1 &  &  &  &  \cr
t_{2} &  &  &  &  &  &  &  &  & -1 & -1 &  &  & 1 \cr
t_{3} & 1 &  &  &  &  &  & -1 &  & -1 &  &  &  &  \cr
t_{4} &  &  &  &  &  &  &  &  &  &  & -1 &  & 1 \cr
t_{5} & 2 & 2 & 2 &  &  & 2 &  &  &  &  &  &  & -1 \cr
t_{6} & 2 & 2 & 4 &  &  & 3 & 1 &  &  &  &  &  &  \cr
t_{7} &  &  &  &  &  &  &  &  &  & -1 &  &  &  \cr
t_{8} & -9 & -6 & -10 &  &  & -9 &  & 1 &  &  &  &  &  \cr
t_{9} & -3 & -2 & -3 &  &  & -3 &  &  & 1 &  &  &  &  \cr
t_{10} & 1 & 1 & 1 &  &  & 1 &  &  &  & 1 &  &  &  \cr
t_{11} & -2 & -2 & -2 & 1 &  & -2 &  &  &  &  & 1 &  &  \cr
t_{12} & -5 & -5 & -7 & -1 & 2 & -6 &  &  &  &  &  & 1 &  \cr
t_{13} & -2 & -2 & -2 &  & 1 & -2 &  &  &  &  &  &  & 1 \cr
}
\]
shows that the nontrivial elementary divisors of the Smith normal form of $T$ are $1$, $1$, $1$, $1$, $1$, $2$.  The element corresponding to the divisor that is greater than $1$ is $t_{6}^{*}$. This already gives
\[
\B_0(G) \cong \langle t_{6}^{*}  \mid {t_{6}^{*}}^{2} \rangle.
\]

We now deal with explicitly identifying the nonuniversal commutator relation generating $\B_0(G)$.
First, factor out by the tails $t_{i}^{*}$ whose corresponding elementary divisors are either trivial or $1$. Transforming the situation back to the original tails $t_i$, this amounts to the nontrivial expansions given by
\[
\bordermatrix{
{} & t_{1} & t_{2} & t_{5} & t_{6} & t_{7} \cr
t_{6}^{*} & 1 & 1 & 1 & 1 & 1 \cr
}
\]
and all the other tails $t_i$ are trivial. We thus obtain a commutativity preserving central extension of the group $G$, given by the presentation
\[
\begin{aligned}
\langle g_{1}, \,g_{2}, \,g_{3}, \,g_{4}, \,g_{5}, \,g_{6}, \,g_{7}, \,t_{6}^{*} & \mid & g_{1}^{2} &= t_{6}^{*} , \\ 
 & & g_{2}^{2} &= t_{6}^{*} , & [g_{2}, g_{1}]  &= g_{5}, \\ 
 & & g_{3}^{2} &= 1, & [g_{3}, g_{1}]  &= g_{6}t_{6}^{*} , & [g_{3}, g_{2}]  &= g_{5}g_{7}t_{6}^{*} , \\ 
 & & g_{4}^{2} &= t_{6}^{*} , & [g_{4}, g_{1}]  &= g_{5}, \\ 
 & & g_{5}^{2} &= 1, \\ 
 & & g_{6}^{2} &= g_{7}, & [g_{6}, g_{1}]  &= g_{7}, & [g_{6}, g_{3}]  &= g_{7}, \\ 
 & & g_{7}^{2} &= 1, \\ 
 & & {t_{6}^{*}}^{2} &= 1  \rangle,
\end{aligned}
\]
whence the nonuniversal commutator relation is identified as
\[
t_{6}^{*}  = [g_{3}, g_{2}] [g_{4}, g_{1}]^{-1}[g_{6}, g_{3}]^{-1}.  \quad 
\]


\item \label{number:44} 
Let the group $G$ be the representative of this family given by the presentation
\[
\begin{aligned}
\langle g_{1}, \,g_{2}, \,g_{3}, \,g_{4}, \,g_{5}, \,g_{6}, \,g_{7} & \mid & g_{1}^{2} &= 1, \\ 
 & & g_{2}^{2} &= 1, & [g_{2}, g_{1}]  &= g_{5}, \\ 
 & & g_{3}^{2} &= 1, & [g_{3}, g_{1}]  &= g_{6}, & [g_{3}, g_{2}]  &= g_{5}, \\ 
 & & g_{4}^{2} &= 1, & [g_{4}, g_{1}]  &= g_{5}, \\ 
 & & g_{5}^{2} &= 1, \\ 
 & & g_{6}^{2} &= g_{7}, & [g_{6}, g_{1}]  &= g_{7}, & [g_{6}, g_{3}]  &= g_{7}, \\ 
 & & g_{7}^{2} &= 1\rangle. \\ 
\end{aligned}
\]
We add 13 tails to the presentation as to form a quotient of the universal central extension of the system: 
$g_{1}^{2} =  t_{1}$,
$g_{2}^{2} =  t_{2}$,
$[g_{2}, g_{1}] = g_{5} t_{3}$,
$g_{3}^{2} =  t_{4}$,
$[g_{3}, g_{1}] = g_{6} t_{5}$,
$[g_{3}, g_{2}] = g_{5} t_{6}$,
$g_{4}^{2} =  t_{7}$,
$[g_{4}, g_{1}] = g_{5} t_{8}$,
$g_{5}^{2} =  t_{9}$,
$g_{6}^{2} = g_{7} t_{10}$,
$[g_{6}, g_{1}] = g_{7} t_{11}$,
$[g_{6}, g_{3}] = g_{7} t_{12}$,
$g_{7}^{2} =  t_{13}$.
Carrying out consistency checks gives the following relations between the tails:
\[
\begin{aligned}
g_{6}^2 g_{3} & = g_{6} (g_{6} g_{3})& \Longrightarrow & & t_{12}^{2}t_{13} & = 1 \\
g_{6}^2 g_{1} & = g_{6} (g_{6} g_{1})& \Longrightarrow & & t_{11}^{2}t_{13} & = 1 \\
g_{4}^2 g_{1} & = g_{4} (g_{4} g_{1})& \Longrightarrow & & t_{8}^{2}t_{9} & = 1 \\
g_{3}^2 g_{2} & = g_{3} (g_{3} g_{2})& \Longrightarrow & & t_{6}^{2}t_{9} & = 1 \\
g_{3}^2 g_{1} & = g_{3} (g_{3} g_{1})& \Longrightarrow & & t_{5}^{2}t_{10}t_{12}t_{13} & = 1 \\
g_{2}^2 g_{1} & = g_{2} (g_{2} g_{1})& \Longrightarrow & & t_{3}^{2}t_{9} & = 1 \\
g_{3} g_{1}^{2} & = (g_{3} g_{1}) g_{1}& \Longrightarrow & & t_{5}^{2}t_{10}t_{11}t_{13} & = 1 \\
\end{aligned}
\]
Scanning through the conjugacy class representatives of $G$ and the generators of their centralizers, we obtain the following relations induced on the tails:
\[
\begin{aligned}
{[g_{2} g_{5} , \, g_{1} g_{3} g_{4} ]}_G & = 1 & \Longrightarrow & & t_{3}t_{6}^{-1} & = 1 \\
{[g_{2} g_{4} g_{5} , \, g_{1} g_{4} ]}_G & = 1 & \Longrightarrow & & t_{3}t_{8}t_{9} & = 1 \\
\end{aligned}
\]
Collecting the coefficients of these relations into a matrix yields
\[
T = \bordermatrix{
{} & t_{1} & t_{2} & t_{3} & t_{4} & t_{5} & t_{6} & t_{7} & t_{8} & t_{9} & t_{10} & t_{11} & t_{12} & t_{13} \cr
{} &  &  & 1 &  &  &  &  & 1 & 1 &  &  &  &  \cr
{} &  &  &  &  & 2 &  &  &  &  & 1 &  & 1 & 1 \cr
{} &  &  &  &  &  & 1 &  & 1 & 1 &  &  &  &  \cr
{} &  &  &  &  &  &  &  & 2 & 1 &  &  &  &  \cr
{} &  &  &  &  &  &  &  &  &  &  & 1 & 1 & 1 \cr
{} &  &  &  &  &  &  &  &  &  &  &  & 2 & 1 \cr
}.
\]
It follows readily that the nontrivial elementary divisors of the Smith normal form of $T$ are all equal to $1$. The torsion subgroup of the group generated by the tails is thus trivial, thereby showing $\B_0(G) = 1$.


\item \label{number:45} 
Let the group $G$ be the representative of this family given by the presentation
\[
\begin{aligned}
\langle g_{1}, \,g_{2}, \,g_{3}, \,g_{4}, \,g_{5}, \,g_{6}, \,g_{7} & \mid & g_{1}^{2} &= g_{4}, \\ 
 & & g_{2}^{2} &= g_{5}, & [g_{2}, g_{1}]  &= g_{3}, \\ 
 & & g_{3}^{2} &= g_{6}g_{7}, & [g_{3}, g_{1}]  &= g_{6}, \\ 
 & & g_{4}^{2} &= 1, & [g_{4}, g_{2}]  &= g_{7}, \\ 
 & & g_{5}^{2} &= 1, & [g_{5}, g_{1}]  &= g_{6}g_{7}, \\ 
 & & g_{6}^{2} &= 1, \\ 
 & & g_{7}^{2} &= 1\rangle. \\ 
\end{aligned}
\]
We add 11 tails to the presentation as to form a quotient of the universal central extension of the system: 
$g_{1}^{2} = g_{4} t_{1}$,
$g_{2}^{2} = g_{5} t_{2}$,
$[g_{2}, g_{1}] = g_{3} t_{3}$,
$g_{3}^{2} = g_{6}g_{7} t_{4}$,
$[g_{3}, g_{1}] = g_{6} t_{5}$,
$g_{4}^{2} =  t_{6}$,
$[g_{4}, g_{2}] = g_{7} t_{7}$,
$g_{5}^{2} =  t_{8}$,
$[g_{5}, g_{1}] = g_{6}g_{7} t_{9}$,
$g_{6}^{2} =  t_{10}$,
$g_{7}^{2} =  t_{11}$.
Carrying out consistency checks gives the following relations between the tails:
\[
\begin{aligned}
g_{5}^2 g_{1} & = g_{5} (g_{5} g_{1})& \Longrightarrow & & t_{9}^{2}t_{10}t_{11} & = 1 \\
g_{4}^2 g_{2} & = g_{4} (g_{4} g_{2})& \Longrightarrow & & t_{7}^{2}t_{11} & = 1 \\
g_{3}^2 g_{1} & = g_{3} (g_{3} g_{1})& \Longrightarrow & & t_{5}^{2}t_{10} & = 1 \\
g_{2}^2 g_{1} & = g_{2} (g_{2} g_{1})& \Longrightarrow & & t_{3}^{2}t_{4}t_{9}^{-1} & = 1 \\
g_{2} g_{1}^{2} & = (g_{2} g_{1}) g_{1}& \Longrightarrow & & t_{3}^{2}t_{4}t_{5}t_{7}t_{10}t_{11} & = 1 \\
\end{aligned}
\]
Scanning through the conjugacy class representatives of $G$ and the generators of their centralizers, we see that no new relations are imposed.
Collecting the coefficients of these relations into a matrix yields
\[
T = \bordermatrix{
{} & t_{1} & t_{2} & t_{3} & t_{4} & t_{5} & t_{6} & t_{7} & t_{8} & t_{9} & t_{10} & t_{11} \cr
{} &  &  & 2 & 1 &  &  &  &  & 1 & 1 & 1 \cr
{} &  &  &  &  & 1 &  & 1 &  & 1 & 1 & 1 \cr
{} &  &  &  &  &  &  & 2 &  &  &  & 1 \cr
{} &  &  &  &  &  &  &  &  & 2 & 1 & 1 \cr
}.
\]
It follows readily that the nontrivial elementary divisors of the Smith normal form of $T$ are all equal to $1$. The torsion subgroup of the group generated by the tails is thus trivial, thereby showing $\B_0(G) = 1$.


\item \label{number:46} 
Let the group $G$ be the representative of this family given by the presentation
\[
\begin{aligned}
\langle g_{1}, \,g_{2}, \,g_{3}, \,g_{4}, \,g_{5}, \,g_{6}, \,g_{7} & \mid & g_{1}^{2} &= 1, \\ 
 & & g_{2}^{2} &= 1, & [g_{2}, g_{1}]  &= g_{5}, \\ 
 & & g_{3}^{2} &= 1, & [g_{3}, g_{1}]  &= g_{6}, \\ 
 & & g_{4}^{2} &= 1, & [g_{4}, g_{3}]  &= g_{7}, \\ 
 & & g_{5}^{2} &= g_{7}, & [g_{5}, g_{1}]  &= g_{7}, & [g_{5}, g_{2}]  &= g_{7}, \\ 
 & & g_{6}^{2} &= 1, \\ 
 & & g_{7}^{2} &= 1\rangle. \\ 
\end{aligned}
\]
We add 12 tails to the presentation as to form a quotient of the universal central extension of the system: 
$g_{1}^{2} =  t_{1}$,
$g_{2}^{2} =  t_{2}$,
$[g_{2}, g_{1}] = g_{5} t_{3}$,
$g_{3}^{2} =  t_{4}$,
$[g_{3}, g_{1}] = g_{6} t_{5}$,
$g_{4}^{2} =  t_{6}$,
$[g_{4}, g_{3}] = g_{7} t_{7}$,
$g_{5}^{2} = g_{7} t_{8}$,
$[g_{5}, g_{1}] = g_{7} t_{9}$,
$[g_{5}, g_{2}] = g_{7} t_{10}$,
$g_{6}^{2} =  t_{11}$,
$g_{7}^{2} =  t_{12}$.
Carrying out consistency checks gives the following relations between the tails:
\[
\begin{aligned}
g_{5}^2 g_{2} & = g_{5} (g_{5} g_{2})& \Longrightarrow & & t_{10}^{2}t_{12} & = 1 \\
g_{5}^2 g_{1} & = g_{5} (g_{5} g_{1})& \Longrightarrow & & t_{9}^{2}t_{12} & = 1 \\
g_{4}^2 g_{3} & = g_{4} (g_{4} g_{3})& \Longrightarrow & & t_{7}^{2}t_{12} & = 1 \\
g_{3}^2 g_{1} & = g_{3} (g_{3} g_{1})& \Longrightarrow & & t_{5}^{2}t_{11} & = 1 \\
g_{2}^2 g_{1} & = g_{2} (g_{2} g_{1})& \Longrightarrow & & t_{3}^{2}t_{8}t_{10}t_{12} & = 1 \\
g_{2} g_{1}^{2} & = (g_{2} g_{1}) g_{1}& \Longrightarrow & & t_{3}^{2}t_{8}t_{9}t_{12} & = 1 \\
\end{aligned}
\]
Scanning through the conjugacy class representatives of $G$ and the generators of their centralizers, we obtain the following relations induced on the tails:
\[
\begin{aligned}
{[g_{4} g_{5} g_{7} , \, g_{1} g_{3} ]}_G & = 1 & \Longrightarrow & & t_{7}t_{9}t_{12} & = 1 \\
\end{aligned}
\]
Collecting the coefficients of these relations into a matrix yields
\[
T = \bordermatrix{
{} & t_{1} & t_{2} & t_{3} & t_{4} & t_{5} & t_{6} & t_{7} & t_{8} & t_{9} & t_{10} & t_{11} & t_{12} \cr
{} &  &  & 2 &  &  &  &  & 1 &  & 1 &  & 1 \cr
{} &  &  &  &  & 2 &  &  &  &  &  & 1 &  \cr
{} &  &  &  &  &  &  & 1 &  &  & 1 &  & 1 \cr
{} &  &  &  &  &  &  &  &  & 1 & 1 &  & 1 \cr
{} &  &  &  &  &  &  &  &  &  & 2 &  & 1 \cr
}.
\]
It follows readily that the nontrivial elementary divisors of the Smith normal form of $T$ are all equal to $1$. The torsion subgroup of the group generated by the tails is thus trivial, thereby showing $\B_0(G) = 1$.


\item \label{number:47} 
Let the group $G$ be the representative of this family given by the presentation
\[
\begin{aligned}
\langle g_{1}, \,g_{2}, \,g_{3}, \,g_{4}, \,g_{5}, \,g_{6}, \,g_{7} & \mid & g_{1}^{2} &= g_{5}, \\ 
 & & g_{2}^{2} &= 1, & [g_{2}, g_{1}]  &= g_{4}, \\ 
 & & g_{3}^{2} &= 1, & [g_{3}, g_{2}]  &= g_{7}, \\ 
 & & g_{4}^{2} &= 1, & [g_{4}, g_{1}]  &= g_{6}, \\ 
 & & g_{5}^{2} &= g_{7}, & [g_{5}, g_{2}]  &= g_{6}, \\ 
 & & g_{6}^{2} &= 1, \\ 
 & & g_{7}^{2} &= 1\rangle. \\ 
\end{aligned}
\]
We add 11 tails to the presentation as to form a quotient of the universal central extension of the system: 
$g_{1}^{2} = g_{5} t_{1}$,
$g_{2}^{2} =  t_{2}$,
$[g_{2}, g_{1}] = g_{4} t_{3}$,
$g_{3}^{2} =  t_{4}$,
$[g_{3}, g_{2}] = g_{7} t_{5}$,
$g_{4}^{2} =  t_{6}$,
$[g_{4}, g_{1}] = g_{6} t_{7}$,
$g_{5}^{2} = g_{7} t_{8}$,
$[g_{5}, g_{2}] = g_{6} t_{9}$,
$g_{6}^{2} =  t_{10}$,
$g_{7}^{2} =  t_{11}$.
Carrying out consistency checks gives the following relations between the tails:
\[
\begin{aligned}
g_{5}^2 g_{2} & = g_{5} (g_{5} g_{2})& \Longrightarrow & & t_{9}^{2}t_{10} & = 1 \\
g_{4}^2 g_{1} & = g_{4} (g_{4} g_{1})& \Longrightarrow & & t_{7}^{2}t_{10} & = 1 \\
g_{3}^2 g_{2} & = g_{3} (g_{3} g_{2})& \Longrightarrow & & t_{5}^{2}t_{11} & = 1 \\
g_{2}^2 g_{1} & = g_{2} (g_{2} g_{1})& \Longrightarrow & & t_{3}^{2}t_{6} & = 1 \\
g_{2} g_{1}^{2} & = (g_{2} g_{1}) g_{1}& \Longrightarrow & & t_{3}^{2}t_{6}t_{7}t_{9}t_{10} & = 1 \\
\end{aligned}
\]
Scanning through the conjugacy class representatives of $G$ and the generators of their centralizers, we see that no new relations are imposed.
Collecting the coefficients of these relations into a matrix yields
\[
T = \bordermatrix{
{} & t_{1} & t_{2} & t_{3} & t_{4} & t_{5} & t_{6} & t_{7} & t_{8} & t_{9} & t_{10} & t_{11} \cr
{} &  &  & 2 &  &  & 1 &  &  &  &  &  \cr
{} &  &  &  &  & 2 &  &  &  &  &  & 1 \cr
{} &  &  &  &  &  &  & 1 &  & 1 & 1 &  \cr
{} &  &  &  &  &  &  &  &  & 2 & 1 &  \cr
}.
\]
It follows readily that the nontrivial elementary divisors of the Smith normal form of $T$ are all equal to $1$. The torsion subgroup of the group generated by the tails is thus trivial, thereby showing $\B_0(G) = 1$.


\item \label{number:48} 
Let the group $G$ be the representative of this family given by the presentation
\[
\begin{aligned}
\langle g_{1}, \,g_{2}, \,g_{3}, \,g_{4}, \,g_{5}, \,g_{6}, \,g_{7} & \mid & g_{1}^{2} &= 1, \\ 
 & & g_{2}^{2} &= 1, & [g_{2}, g_{1}]  &= g_{5}, \\ 
 & & g_{3}^{2} &= 1, & [g_{3}, g_{1}]  &= g_{6}, \\ 
 & & g_{4}^{2} &= 1, & [g_{4}, g_{1}]  &= g_{7}, \\ 
 & & g_{5}^{2} &= g_{7}, & [g_{5}, g_{1}]  &= g_{7}, & [g_{5}, g_{2}]  &= g_{7}, \\ 
 & & g_{6}^{2} &= 1, \\ 
 & & g_{7}^{2} &= 1\rangle. \\ 
\end{aligned}
\]
We add 12 tails to the presentation as to form a quotient of the universal central extension of the system: 
$g_{1}^{2} =  t_{1}$,
$g_{2}^{2} =  t_{2}$,
$[g_{2}, g_{1}] = g_{5} t_{3}$,
$g_{3}^{2} =  t_{4}$,
$[g_{3}, g_{1}] = g_{6} t_{5}$,
$g_{4}^{2} =  t_{6}$,
$[g_{4}, g_{1}] = g_{7} t_{7}$,
$g_{5}^{2} = g_{7} t_{8}$,
$[g_{5}, g_{1}] = g_{7} t_{9}$,
$[g_{5}, g_{2}] = g_{7} t_{10}$,
$g_{6}^{2} =  t_{11}$,
$g_{7}^{2} =  t_{12}$.
Carrying out consistency checks gives the following relations between the tails:
\[
\begin{aligned}
g_{5}^2 g_{2} & = g_{5} (g_{5} g_{2})& \Longrightarrow & & t_{10}^{2}t_{12} & = 1 \\
g_{5}^2 g_{1} & = g_{5} (g_{5} g_{1})& \Longrightarrow & & t_{9}^{2}t_{12} & = 1 \\
g_{4}^2 g_{1} & = g_{4} (g_{4} g_{1})& \Longrightarrow & & t_{7}^{2}t_{12} & = 1 \\
g_{3}^2 g_{1} & = g_{3} (g_{3} g_{1})& \Longrightarrow & & t_{5}^{2}t_{11} & = 1 \\
g_{2}^2 g_{1} & = g_{2} (g_{2} g_{1})& \Longrightarrow & & t_{3}^{2}t_{8}t_{10}t_{12} & = 1 \\
g_{2} g_{1}^{2} & = (g_{2} g_{1}) g_{1}& \Longrightarrow & & t_{3}^{2}t_{8}t_{9}t_{12} & = 1 \\
\end{aligned}
\]
Scanning through the conjugacy class representatives of $G$ and the generators of their centralizers, we obtain the following relations induced on the tails:
\[
\begin{aligned}
{[g_{4} g_{5} g_{7} , \, g_{1} ]}_G & = 1 & \Longrightarrow & & t_{7}t_{9}t_{12} & = 1 \\
\end{aligned}
\]
Collecting the coefficients of these relations into a matrix yields
\[
T = \bordermatrix{
{} & t_{1} & t_{2} & t_{3} & t_{4} & t_{5} & t_{6} & t_{7} & t_{8} & t_{9} & t_{10} & t_{11} & t_{12} \cr
{} &  &  & 2 &  &  &  &  & 1 &  & 1 &  & 1 \cr
{} &  &  &  &  & 2 &  &  &  &  &  & 1 &  \cr
{} &  &  &  &  &  &  & 1 &  &  & 1 &  & 1 \cr
{} &  &  &  &  &  &  &  &  & 1 & 1 &  & 1 \cr
{} &  &  &  &  &  &  &  &  &  & 2 &  & 1 \cr
}.
\]
It follows readily that the nontrivial elementary divisors of the Smith normal form of $T$ are all equal to $1$. The torsion subgroup of the group generated by the tails is thus trivial, thereby showing $\B_0(G) = 1$.


\item \label{number:49} 
Let the group $G$ be the representative of this family given by the presentation
\[
\begin{aligned}
\langle g_{1}, \,g_{2}, \,g_{3}, \,g_{4}, \,g_{5}, \,g_{6}, \,g_{7} & \mid & g_{1}^{2} &= 1, \\ 
 & & g_{2}^{2} &= g_{5}, & [g_{2}, g_{1}]  &= g_{5}, \\ 
 & & g_{3}^{2} &= 1, & [g_{3}, g_{1}]  &= g_{6}, \\ 
 & & g_{4}^{2} &= 1, & [g_{4}, g_{3}]  &= g_{7}, \\ 
 & & g_{5}^{2} &= g_{7}, & [g_{5}, g_{1}]  &= g_{7}, \\ 
 & & g_{6}^{2} &= 1, \\ 
 & & g_{7}^{2} &= 1\rangle. \\ 
\end{aligned}
\]
We add 11 tails to the presentation as to form a quotient of the universal central extension of the system: 
$g_{1}^{2} =  t_{1}$,
$g_{2}^{2} = g_{5} t_{2}$,
$[g_{2}, g_{1}] = g_{5} t_{3}$,
$g_{3}^{2} =  t_{4}$,
$[g_{3}, g_{1}] = g_{6} t_{5}$,
$g_{4}^{2} =  t_{6}$,
$[g_{4}, g_{3}] = g_{7} t_{7}$,
$g_{5}^{2} = g_{7} t_{8}$,
$[g_{5}, g_{1}] = g_{7} t_{9}$,
$g_{6}^{2} =  t_{10}$,
$g_{7}^{2} =  t_{11}$.
Carrying out consistency checks gives the following relations between the tails:
\[
\begin{aligned}
g_{5}^2 g_{1} & = g_{5} (g_{5} g_{1})& \Longrightarrow & & t_{9}^{2}t_{11} & = 1 \\
g_{4}^2 g_{3} & = g_{4} (g_{4} g_{3})& \Longrightarrow & & t_{7}^{2}t_{11} & = 1 \\
g_{3}^2 g_{1} & = g_{3} (g_{3} g_{1})& \Longrightarrow & & t_{5}^{2}t_{10} & = 1 \\
g_{2}^2 g_{1} & = g_{2} (g_{2} g_{1})& \Longrightarrow & & t_{3}^{2}t_{8}t_{9}^{-1} & = 1 \\
\end{aligned}
\]
Scanning through the conjugacy class representatives of $G$ and the generators of their centralizers, we obtain the following relations induced on the tails:
\[
\begin{aligned}
{[g_{4} g_{5} g_{7} , \, g_{1} g_{3} ]}_G & = 1 & \Longrightarrow & & t_{7}t_{9}t_{11} & = 1 \\
\end{aligned}
\]
Collecting the coefficients of these relations into a matrix yields
\[
T = \bordermatrix{
{} & t_{1} & t_{2} & t_{3} & t_{4} & t_{5} & t_{6} & t_{7} & t_{8} & t_{9} & t_{10} & t_{11} \cr
{} &  &  & 2 &  &  &  &  & 1 & 1 &  & 1 \cr
{} &  &  &  &  & 2 &  &  &  &  & 1 &  \cr
{} &  &  &  &  &  &  & 1 &  & 1 &  & 1 \cr
{} &  &  &  &  &  &  &  &  & 2 &  & 1 \cr
}.
\]
It follows readily that the nontrivial elementary divisors of the Smith normal form of $T$ are all equal to $1$. The torsion subgroup of the group generated by the tails is thus trivial, thereby showing $\B_0(G) = 1$.


\item \label{number:50} 
Let the group $G$ be the representative of this family given by the presentation
\[
\begin{aligned}
\langle g_{1}, \,g_{2}, \,g_{3}, \,g_{4}, \,g_{5}, \,g_{6}, \,g_{7} & \mid & g_{1}^{2} &= g_{6}, \\ 
 & & g_{2}^{2} &= g_{4}, & [g_{2}, g_{1}]  &= g_{4}, \\ 
 & & g_{3}^{2} &= 1, & [g_{3}, g_{1}]  &= g_{5}, \\ 
 & & g_{4}^{2} &= g_{7}, & [g_{4}, g_{1}]  &= g_{7}, \\ 
 & & g_{5}^{2} &= 1, & [g_{5}, g_{1}]  &= g_{7}, \\ 
 & & g_{6}^{2} &= 1, & [g_{6}, g_{3}]  &= g_{7}, \\ 
 & & g_{7}^{2} &= 1\rangle. \\ 
\end{aligned}
\]
We add 12 tails to the presentation as to form a quotient of the universal central extension of the system: 
$g_{1}^{2} = g_{6} t_{1}$,
$g_{2}^{2} = g_{4} t_{2}$,
$[g_{2}, g_{1}] = g_{4} t_{3}$,
$g_{3}^{2} =  t_{4}$,
$[g_{3}, g_{1}] = g_{5} t_{5}$,
$g_{4}^{2} = g_{7} t_{6}$,
$[g_{4}, g_{1}] = g_{7} t_{7}$,
$g_{5}^{2} =  t_{8}$,
$[g_{5}, g_{1}] = g_{7} t_{9}$,
$g_{6}^{2} =  t_{10}$,
$[g_{6}, g_{3}] = g_{7} t_{11}$,
$g_{7}^{2} =  t_{12}$.
Carrying out consistency checks gives the following relations between the tails:
\[
\begin{aligned}
g_{6}^2 g_{3} & = g_{6} (g_{6} g_{3})& \Longrightarrow & & t_{11}^{2}t_{12} & = 1 \\
g_{5}^2 g_{1} & = g_{5} (g_{5} g_{1})& \Longrightarrow & & t_{9}^{2}t_{12} & = 1 \\
g_{4}^2 g_{1} & = g_{4} (g_{4} g_{1})& \Longrightarrow & & t_{7}^{2}t_{12} & = 1 \\
g_{3}^2 g_{1} & = g_{3} (g_{3} g_{1})& \Longrightarrow & & t_{5}^{2}t_{8} & = 1 \\
g_{2}^2 g_{1} & = g_{2} (g_{2} g_{1})& \Longrightarrow & & t_{3}^{2}t_{6}t_{7}^{-1} & = 1 \\
g_{3} g_{1}^{2} & = (g_{3} g_{1}) g_{1}& \Longrightarrow & & t_{5}^{2}t_{8}t_{9}t_{11}t_{12} & = 1 \\
\end{aligned}
\]
Scanning through the conjugacy class representatives of $G$ and the generators of their centralizers, we obtain the following relations induced on the tails:
\[
\begin{aligned}
{[g_{4} g_{6} g_{7} , \, g_{1} g_{3} ]}_G & = 1 & \Longrightarrow & & t_{7}t_{11}t_{12} & = 1 \\
\end{aligned}
\]
Collecting the coefficients of these relations into a matrix yields
\[
T = \bordermatrix{
{} & t_{1} & t_{2} & t_{3} & t_{4} & t_{5} & t_{6} & t_{7} & t_{8} & t_{9} & t_{10} & t_{11} & t_{12} \cr
{} &  &  & 2 &  &  & 1 &  &  &  &  & 1 & 1 \cr
{} &  &  &  &  & 2 &  &  & 1 &  &  &  &  \cr
{} &  &  &  &  &  &  & 1 &  &  &  & 1 & 1 \cr
{} &  &  &  &  &  &  &  &  & 1 &  & 1 & 1 \cr
{} &  &  &  &  &  &  &  &  &  &  & 2 & 1 \cr
}.
\]
It follows readily that the nontrivial elementary divisors of the Smith normal form of $T$ are all equal to $1$. The torsion subgroup of the group generated by the tails is thus trivial, thereby showing $\B_0(G) = 1$.


\item \label{number:51} 
Let the group $G$ be the representative of this family given by the presentation
\[
\begin{aligned}
\langle g_{1}, \,g_{2}, \,g_{3}, \,g_{4}, \,g_{5}, \,g_{6}, \,g_{7} & \mid & g_{1}^{2} &= 1, \\ 
 & & g_{2}^{2} &= g_{5}, & [g_{2}, g_{1}]  &= g_{5}, \\ 
 & & g_{3}^{2} &= 1, & [g_{3}, g_{1}]  &= g_{6}, \\ 
 & & g_{4}^{2} &= 1, & [g_{4}, g_{2}]  &= g_{7}, \\ 
 & & g_{5}^{2} &= g_{7}, & [g_{5}, g_{1}]  &= g_{7}, \\ 
 & & g_{6}^{2} &= 1, \\ 
 & & g_{7}^{2} &= 1\rangle. \\ 
\end{aligned}
\]
We add 11 tails to the presentation as to form a quotient of the universal central extension of the system: 
$g_{1}^{2} =  t_{1}$,
$g_{2}^{2} = g_{5} t_{2}$,
$[g_{2}, g_{1}] = g_{5} t_{3}$,
$g_{3}^{2} =  t_{4}$,
$[g_{3}, g_{1}] = g_{6} t_{5}$,
$g_{4}^{2} =  t_{6}$,
$[g_{4}, g_{2}] = g_{7} t_{7}$,
$g_{5}^{2} = g_{7} t_{8}$,
$[g_{5}, g_{1}] = g_{7} t_{9}$,
$g_{6}^{2} =  t_{10}$,
$g_{7}^{2} =  t_{11}$.
Carrying out consistency checks gives the following relations between the tails:
\[
\begin{aligned}
g_{5}^2 g_{1} & = g_{5} (g_{5} g_{1})& \Longrightarrow & & t_{9}^{2}t_{11} & = 1 \\
g_{4}^2 g_{2} & = g_{4} (g_{4} g_{2})& \Longrightarrow & & t_{7}^{2}t_{11} & = 1 \\
g_{3}^2 g_{1} & = g_{3} (g_{3} g_{1})& \Longrightarrow & & t_{5}^{2}t_{10} & = 1 \\
g_{2}^2 g_{1} & = g_{2} (g_{2} g_{1})& \Longrightarrow & & t_{3}^{2}t_{8}t_{9}^{-1} & = 1 \\
\end{aligned}
\]
Scanning through the conjugacy class representatives of $G$ and the generators of their centralizers, we obtain the following relations induced on the tails:
\[
\begin{aligned}
{[g_{4} g_{5} g_{7} , \, g_{1} g_{2} g_{3} ]}_G & = 1 & \Longrightarrow & & t_{7}t_{9}t_{11} & = 1 \\
\end{aligned}
\]
Collecting the coefficients of these relations into a matrix yields
\[
T = \bordermatrix{
{} & t_{1} & t_{2} & t_{3} & t_{4} & t_{5} & t_{6} & t_{7} & t_{8} & t_{9} & t_{10} & t_{11} \cr
{} &  &  & 2 &  &  &  &  & 1 & 1 &  & 1 \cr
{} &  &  &  &  & 2 &  &  &  &  & 1 &  \cr
{} &  &  &  &  &  &  & 1 &  & 1 &  & 1 \cr
{} &  &  &  &  &  &  &  &  & 2 &  & 1 \cr
}.
\]
It follows readily that the nontrivial elementary divisors of the Smith normal form of $T$ are all equal to $1$. The torsion subgroup of the group generated by the tails is thus trivial, thereby showing $\B_0(G) = 1$.


\item \label{number:52} 
Let the group $G$ be the representative of this family given by the presentation
\[
\begin{aligned}
\langle g_{1}, \,g_{2}, \,g_{3}, \,g_{4}, \,g_{5}, \,g_{6}, \,g_{7} & \mid & g_{1}^{2} &= g_{5}, \\ 
 & & g_{2}^{2} &= 1, & [g_{2}, g_{1}]  &= g_{4}, \\ 
 & & g_{3}^{2} &= 1, & [g_{3}, g_{1}]  &= g_{7}, \\ 
 & & g_{4}^{2} &= 1, & [g_{4}, g_{1}]  &= g_{6}, \\ 
 & & g_{5}^{2} &= g_{7}, & [g_{5}, g_{2}]  &= g_{6}, \\ 
 & & g_{6}^{2} &= 1, \\ 
 & & g_{7}^{2} &= 1\rangle. \\ 
\end{aligned}
\]
We add 11 tails to the presentation as to form a quotient of the universal central extension of the system: 
$g_{1}^{2} = g_{5} t_{1}$,
$g_{2}^{2} =  t_{2}$,
$[g_{2}, g_{1}] = g_{4} t_{3}$,
$g_{3}^{2} =  t_{4}$,
$[g_{3}, g_{1}] = g_{7} t_{5}$,
$g_{4}^{2} =  t_{6}$,
$[g_{4}, g_{1}] = g_{6} t_{7}$,
$g_{5}^{2} = g_{7} t_{8}$,
$[g_{5}, g_{2}] = g_{6} t_{9}$,
$g_{6}^{2} =  t_{10}$,
$g_{7}^{2} =  t_{11}$.
Carrying out consistency checks gives the following relations between the tails:
\[
\begin{aligned}
g_{5}^2 g_{2} & = g_{5} (g_{5} g_{2})& \Longrightarrow & & t_{9}^{2}t_{10} & = 1 \\
g_{4}^2 g_{1} & = g_{4} (g_{4} g_{1})& \Longrightarrow & & t_{7}^{2}t_{10} & = 1 \\
g_{3}^2 g_{1} & = g_{3} (g_{3} g_{1})& \Longrightarrow & & t_{5}^{2}t_{11} & = 1 \\
g_{2}^2 g_{1} & = g_{2} (g_{2} g_{1})& \Longrightarrow & & t_{3}^{2}t_{6} & = 1 \\
g_{2} g_{1}^{2} & = (g_{2} g_{1}) g_{1}& \Longrightarrow & & t_{3}^{2}t_{6}t_{7}t_{9}t_{10} & = 1 \\
\end{aligned}
\]
Scanning through the conjugacy class representatives of $G$ and the generators of their centralizers, we see that no new relations are imposed.
Collecting the coefficients of these relations into a matrix yields
\[
T = \bordermatrix{
{} & t_{1} & t_{2} & t_{3} & t_{4} & t_{5} & t_{6} & t_{7} & t_{8} & t_{9} & t_{10} & t_{11} \cr
{} &  &  & 2 &  &  & 1 &  &  &  &  &  \cr
{} &  &  &  &  & 2 &  &  &  &  &  & 1 \cr
{} &  &  &  &  &  &  & 1 &  & 1 & 1 &  \cr
{} &  &  &  &  &  &  &  &  & 2 & 1 &  \cr
}.
\]
It follows readily that the nontrivial elementary divisors of the Smith normal form of $T$ are all equal to $1$. The torsion subgroup of the group generated by the tails is thus trivial, thereby showing $\B_0(G) = 1$.


\item \label{number:53} 
Let the group $G$ be the representative of this family given by the presentation
\[
\begin{aligned}
\langle g_{1}, \,g_{2}, \,g_{3}, \,g_{4}, \,g_{5}, \,g_{6}, \,g_{7} & \mid & g_{1}^{2} &= 1, \\ 
 & & g_{2}^{2} &= 1, & [g_{2}, g_{1}]  &= g_{6}, \\ 
 & & g_{3}^{2} &= 1, \\ 
 & & g_{4}^{2} &= 1, & [g_{4}, g_{3}]  &= g_{7}, \\ 
 & & g_{5}^{2} &= 1, & [g_{5}, g_{1}]  &= g_{7}, \\ 
 & & g_{6}^{2} &= g_{7}, & [g_{6}, g_{1}]  &= g_{7}, & [g_{6}, g_{2}]  &= g_{7}, \\ 
 & & g_{7}^{2} &= 1\rangle. \\ 
\end{aligned}
\]
We add 12 tails to the presentation as to form a quotient of the universal central extension of the system: 
$g_{1}^{2} =  t_{1}$,
$g_{2}^{2} =  t_{2}$,
$[g_{2}, g_{1}] = g_{6} t_{3}$,
$g_{3}^{2} =  t_{4}$,
$g_{4}^{2} =  t_{5}$,
$[g_{4}, g_{3}] = g_{7} t_{6}$,
$g_{5}^{2} =  t_{7}$,
$[g_{5}, g_{1}] = g_{7} t_{8}$,
$g_{6}^{2} = g_{7} t_{9}$,
$[g_{6}, g_{1}] = g_{7} t_{10}$,
$[g_{6}, g_{2}] = g_{7} t_{11}$,
$g_{7}^{2} =  t_{12}$.
Carrying out consistency checks gives the following relations between the tails:
\[
\begin{aligned}
g_{6}^2 g_{2} & = g_{6} (g_{6} g_{2})& \Longrightarrow & & t_{11}^{2}t_{12} & = 1 \\
g_{6}^2 g_{1} & = g_{6} (g_{6} g_{1})& \Longrightarrow & & t_{10}^{2}t_{12} & = 1 \\
g_{5}^2 g_{1} & = g_{5} (g_{5} g_{1})& \Longrightarrow & & t_{8}^{2}t_{12} & = 1 \\
g_{4}^2 g_{3} & = g_{4} (g_{4} g_{3})& \Longrightarrow & & t_{6}^{2}t_{12} & = 1 \\
g_{2}^2 g_{1} & = g_{2} (g_{2} g_{1})& \Longrightarrow & & t_{3}^{2}t_{9}t_{11}t_{12} & = 1 \\
g_{2} g_{1}^{2} & = (g_{2} g_{1}) g_{1}& \Longrightarrow & & t_{3}^{2}t_{9}t_{10}t_{12} & = 1 \\
\end{aligned}
\]
Scanning through the conjugacy class representatives of $G$ and the generators of their centralizers, we obtain the following relations induced on the tails:
\[
\begin{aligned}
{[g_{5} g_{6} g_{7} , \, g_{1} ]}_G & = 1 & \Longrightarrow & & t_{8}t_{10}t_{12} & = 1 \\
{[g_{4} g_{6} g_{7} , \, g_{1} g_{3} ]}_G & = 1 & \Longrightarrow & & t_{6}t_{10}t_{12} & = 1 \\
\end{aligned}
\]
Collecting the coefficients of these relations into a matrix yields
\[
T = \bordermatrix{
{} & t_{1} & t_{2} & t_{3} & t_{4} & t_{5} & t_{6} & t_{7} & t_{8} & t_{9} & t_{10} & t_{11} & t_{12} \cr
{} &  &  & 2 &  &  &  &  &  & 1 &  & 1 & 1 \cr
{} &  &  &  &  &  & 1 &  &  &  &  & 1 & 1 \cr
{} &  &  &  &  &  &  &  & 1 &  &  & 1 & 1 \cr
{} &  &  &  &  &  &  &  &  &  & 1 & 1 & 1 \cr
{} &  &  &  &  &  &  &  &  &  &  & 2 & 1 \cr
}.
\]
It follows readily that the nontrivial elementary divisors of the Smith normal form of $T$ are all equal to $1$. The torsion subgroup of the group generated by the tails is thus trivial, thereby showing $\B_0(G) = 1$.


\item \label{number:54} 
Let the group $G$ be the representative of this family given by the presentation
\[
\begin{aligned}
\langle g_{1}, \,g_{2}, \,g_{3}, \,g_{4}, \,g_{5}, \,g_{6}, \,g_{7} & \mid & g_{1}^{2} &= g_{6}, \\ 
 & & g_{2}^{2} &= 1, & [g_{2}, g_{1}]  &= g_{5}, \\ 
 & & g_{3}^{2} &= 1, \\ 
 & & g_{4}^{2} &= 1, & [g_{4}, g_{3}]  &= g_{7}, \\ 
 & & g_{5}^{2} &= 1, & [g_{5}, g_{1}]  &= g_{7}, \\ 
 & & g_{6}^{2} &= 1, & [g_{6}, g_{2}]  &= g_{7}, \\ 
 & & g_{7}^{2} &= 1\rangle. \\ 
\end{aligned}
\]
We add 11 tails to the presentation as to form a quotient of the universal central extension of the system: 
$g_{1}^{2} = g_{6} t_{1}$,
$g_{2}^{2} =  t_{2}$,
$[g_{2}, g_{1}] = g_{5} t_{3}$,
$g_{3}^{2} =  t_{4}$,
$g_{4}^{2} =  t_{5}$,
$[g_{4}, g_{3}] = g_{7} t_{6}$,
$g_{5}^{2} =  t_{7}$,
$[g_{5}, g_{1}] = g_{7} t_{8}$,
$g_{6}^{2} =  t_{9}$,
$[g_{6}, g_{2}] = g_{7} t_{10}$,
$g_{7}^{2} =  t_{11}$.
Carrying out consistency checks gives the following relations between the tails:
\[
\begin{aligned}
g_{6}^2 g_{2} & = g_{6} (g_{6} g_{2})& \Longrightarrow & & t_{10}^{2}t_{11} & = 1 \\
g_{5}^2 g_{1} & = g_{5} (g_{5} g_{1})& \Longrightarrow & & t_{8}^{2}t_{11} & = 1 \\
g_{4}^2 g_{3} & = g_{4} (g_{4} g_{3})& \Longrightarrow & & t_{6}^{2}t_{11} & = 1 \\
g_{2}^2 g_{1} & = g_{2} (g_{2} g_{1})& \Longrightarrow & & t_{3}^{2}t_{7} & = 1 \\
g_{2} g_{1}^{2} & = (g_{2} g_{1}) g_{1}& \Longrightarrow & & t_{3}^{2}t_{7}t_{8}t_{10}t_{11} & = 1 \\
\end{aligned}
\]
Scanning through the conjugacy class representatives of $G$ and the generators of their centralizers, we obtain the following relations induced on the tails:
\[
\begin{aligned}
{[g_{4} g_{6} g_{7} , \, g_{2} g_{3} g_{4} ]}_G & = 1 & \Longrightarrow & & t_{6}t_{10}t_{11} & = 1 \\
\end{aligned}
\]
Collecting the coefficients of these relations into a matrix yields
\[
T = \bordermatrix{
{} & t_{1} & t_{2} & t_{3} & t_{4} & t_{5} & t_{6} & t_{7} & t_{8} & t_{9} & t_{10} & t_{11} \cr
{} &  &  & 2 &  &  &  & 1 &  &  &  &  \cr
{} &  &  &  &  &  & 1 &  &  &  & 1 & 1 \cr
{} &  &  &  &  &  &  &  & 1 &  & 1 & 1 \cr
{} &  &  &  &  &  &  &  &  &  & 2 & 1 \cr
}.
\]
It follows readily that the nontrivial elementary divisors of the Smith normal form of $T$ are all equal to $1$. The torsion subgroup of the group generated by the tails is thus trivial, thereby showing $\B_0(G) = 1$.


\item \label{number:55} 
Let the group $G$ be the representative of this family given by the presentation
\[
\begin{aligned}
\langle g_{1}, \,g_{2}, \,g_{3}, \,g_{4}, \,g_{5}, \,g_{6}, \,g_{7} & \mid & g_{1}^{2} &= 1, \\ 
 & & g_{2}^{2} &= 1, & [g_{2}, g_{1}]  &= g_{4}, \\ 
 & & g_{3}^{2} &= 1, & [g_{3}, g_{1}]  &= g_{5}, \\ 
 & & g_{4}^{2} &= g_{6}, & [g_{4}, g_{1}]  &= g_{6}, & [g_{4}, g_{2}]  &= g_{6}, & [g_{4}, g_{3}]  &= g_{7}, \\ 
 & & g_{5}^{2} &= g_{6}g_{7}, & [g_{5}, g_{1}]  &= g_{6}g_{7}, & [g_{5}, g_{2}]  &= g_{7}, & [g_{5}, g_{3}]  &= g_{6}g_{7}, \\ 
 & & g_{6}^{2} &= 1, \\ 
 & & g_{7}^{2} &= 1\rangle. \\ 
\end{aligned}
\]
We add 15 tails to the presentation as to form a quotient of the universal central extension of the system: 
$g_{1}^{2} =  t_{1}$,
$g_{2}^{2} =  t_{2}$,
$[g_{2}, g_{1}] = g_{4} t_{3}$,
$g_{3}^{2} =  t_{4}$,
$[g_{3}, g_{1}] = g_{5} t_{5}$,
$g_{4}^{2} = g_{6} t_{6}$,
$[g_{4}, g_{1}] = g_{6} t_{7}$,
$[g_{4}, g_{2}] = g_{6} t_{8}$,
$[g_{4}, g_{3}] = g_{7} t_{9}$,
$g_{5}^{2} = g_{6}g_{7} t_{10}$,
$[g_{5}, g_{1}] = g_{6}g_{7} t_{11}$,
$[g_{5}, g_{2}] = g_{7} t_{12}$,
$[g_{5}, g_{3}] = g_{6}g_{7} t_{13}$,
$g_{6}^{2} =  t_{14}$,
$g_{7}^{2} =  t_{15}$.
Carrying out consistency checks gives the following relations between the tails:
\[
\begin{aligned}
g_{3}(g_{2} g_{1}) & = (g_{3} g_{2}) g_{1}  & \Longrightarrow & & t_{9}t_{12}^{-1} & = 1 \\
g_{5}^2 g_{3} & = g_{5} (g_{5} g_{3})& \Longrightarrow & & t_{13}^{2}t_{14}t_{15} & = 1 \\
g_{5}^2 g_{2} & = g_{5} (g_{5} g_{2})& \Longrightarrow & & t_{12}^{2}t_{15} & = 1 \\
g_{5}^2 g_{1} & = g_{5} (g_{5} g_{1})& \Longrightarrow & & t_{11}^{2}t_{14}t_{15} & = 1 \\
g_{4}^2 g_{2} & = g_{4} (g_{4} g_{2})& \Longrightarrow & & t_{8}^{2}t_{14} & = 1 \\
g_{4}^2 g_{1} & = g_{4} (g_{4} g_{1})& \Longrightarrow & & t_{7}^{2}t_{14} & = 1 \\
g_{3}^2 g_{1} & = g_{3} (g_{3} g_{1})& \Longrightarrow & & t_{5}^{2}t_{10}t_{13}t_{14}t_{15} & = 1 \\
g_{2}^2 g_{1} & = g_{2} (g_{2} g_{1})& \Longrightarrow & & t_{3}^{2}t_{6}t_{8}t_{14} & = 1 \\
g_{3} g_{1}^{2} & = (g_{3} g_{1}) g_{1}& \Longrightarrow & & t_{5}^{2}t_{10}t_{11}t_{14}t_{15} & = 1 \\
g_{2} g_{1}^{2} & = (g_{2} g_{1}) g_{1}& \Longrightarrow & & t_{3}^{2}t_{6}t_{7}t_{14} & = 1 \\
\end{aligned}
\]
Scanning through the conjugacy class representatives of $G$ and the generators of their centralizers, we obtain the following relations induced on the tails:
\[
\begin{aligned}
{[g_{3} g_{4} g_{5} g_{6} g_{7} , \, g_{2} g_{5} ]}_G & = 1 & \Longrightarrow & & t_{8}t_{12}t_{13}^{-1} & = 1 \\
\end{aligned}
\]
Collecting the coefficients of these relations into a matrix yields
\[
T = \bordermatrix{
{} & t_{1} & t_{2} & t_{3} & t_{4} & t_{5} & t_{6} & t_{7} & t_{8} & t_{9} & t_{10} & t_{11} & t_{12} & t_{13} & t_{14} & t_{15} \cr
{} &  &  & 2 &  &  & 1 &  &  &  &  &  & 1 & 1 & 1 & 1 \cr
{} &  &  &  &  & 2 &  &  &  &  & 1 &  &  & 1 & 1 & 1 \cr
{} &  &  &  &  &  &  & 1 &  &  &  &  & 1 & 1 & 1 & 1 \cr
{} &  &  &  &  &  &  &  & 1 &  &  &  & 1 & 1 & 1 & 1 \cr
{} &  &  &  &  &  &  &  &  & 1 &  &  & 1 &  &  & 1 \cr
{} &  &  &  &  &  &  &  &  &  &  & 1 &  & 1 & 1 & 1 \cr
{} &  &  &  &  &  &  &  &  &  &  &  & 2 &  &  & 1 \cr
{} &  &  &  &  &  &  &  &  &  &  &  &  & 2 & 1 & 1 \cr
}.
\]
It follows readily that the nontrivial elementary divisors of the Smith normal form of $T$ are all equal to $1$. The torsion subgroup of the group generated by the tails is thus trivial, thereby showing $\B_0(G) = 1$.


\item \label{number:56} 
Let the group $G$ be the representative of this family given by the presentation
\[
\begin{aligned}
\langle g_{1}, \,g_{2}, \,g_{3}, \,g_{4}, \,g_{5}, \,g_{6}, \,g_{7} & \mid & g_{1}^{2} &= 1, \\ 
 & & g_{2}^{2} &= 1, & [g_{2}, g_{1}]  &= g_{4}, \\ 
 & & g_{3}^{2} &= 1, & [g_{3}, g_{1}]  &= g_{5}, \\ 
 & & g_{4}^{2} &= g_{6}, & [g_{4}, g_{1}]  &= g_{6}, & [g_{4}, g_{2}]  &= g_{6}, & [g_{4}, g_{3}]  &= g_{7}, \\ 
 & & g_{5}^{2} &= g_{7}, & [g_{5}, g_{1}]  &= g_{7}, & [g_{5}, g_{2}]  &= g_{7}, & [g_{5}, g_{3}]  &= g_{7}, \\ 
 & & g_{6}^{2} &= 1, \\ 
 & & g_{7}^{2} &= 1\rangle. \\ 
\end{aligned}
\]
We add 15 tails to the presentation as to form a quotient of the universal central extension of the system: 
$g_{1}^{2} =  t_{1}$,
$g_{2}^{2} =  t_{2}$,
$[g_{2}, g_{1}] = g_{4} t_{3}$,
$g_{3}^{2} =  t_{4}$,
$[g_{3}, g_{1}] = g_{5} t_{5}$,
$g_{4}^{2} = g_{6} t_{6}$,
$[g_{4}, g_{1}] = g_{6} t_{7}$,
$[g_{4}, g_{2}] = g_{6} t_{8}$,
$[g_{4}, g_{3}] = g_{7} t_{9}$,
$g_{5}^{2} = g_{7} t_{10}$,
$[g_{5}, g_{1}] = g_{7} t_{11}$,
$[g_{5}, g_{2}] = g_{7} t_{12}$,
$[g_{5}, g_{3}] = g_{7} t_{13}$,
$g_{6}^{2} =  t_{14}$,
$g_{7}^{2} =  t_{15}$.
Carrying out consistency checks gives the following relations between the tails:
\[
\begin{aligned}
g_{3}(g_{2} g_{1}) & = (g_{3} g_{2}) g_{1}  & \Longrightarrow & & t_{9}t_{12}^{-1} & = 1 \\
g_{5}^2 g_{3} & = g_{5} (g_{5} g_{3})& \Longrightarrow & & t_{13}^{2}t_{15} & = 1 \\
g_{5}^2 g_{2} & = g_{5} (g_{5} g_{2})& \Longrightarrow & & t_{12}^{2}t_{15} & = 1 \\
g_{5}^2 g_{1} & = g_{5} (g_{5} g_{1})& \Longrightarrow & & t_{11}^{2}t_{15} & = 1 \\
g_{4}^2 g_{2} & = g_{4} (g_{4} g_{2})& \Longrightarrow & & t_{8}^{2}t_{14} & = 1 \\
g_{4}^2 g_{1} & = g_{4} (g_{4} g_{1})& \Longrightarrow & & t_{7}^{2}t_{14} & = 1 \\
g_{3}^2 g_{1} & = g_{3} (g_{3} g_{1})& \Longrightarrow & & t_{5}^{2}t_{10}t_{13}t_{15} & = 1 \\
g_{2}^2 g_{1} & = g_{2} (g_{2} g_{1})& \Longrightarrow & & t_{3}^{2}t_{6}t_{8}t_{14} & = 1 \\
g_{3} g_{1}^{2} & = (g_{3} g_{1}) g_{1}& \Longrightarrow & & t_{5}^{2}t_{10}t_{11}t_{15} & = 1 \\
g_{2} g_{1}^{2} & = (g_{2} g_{1}) g_{1}& \Longrightarrow & & t_{3}^{2}t_{6}t_{7}t_{14} & = 1 \\
\end{aligned}
\]
Scanning through the conjugacy class representatives of $G$ and the generators of their centralizers, we obtain the following relations induced on the tails:
\[
\begin{aligned}
{[g_{5} g_{7} , \, g_{2} g_{3} g_{4} ]}_G & = 1 & \Longrightarrow & & t_{12}t_{13}t_{15} & = 1 \\
\end{aligned}
\]
Collecting the coefficients of these relations into a matrix yields
\[
T = \bordermatrix{
{} & t_{1} & t_{2} & t_{3} & t_{4} & t_{5} & t_{6} & t_{7} & t_{8} & t_{9} & t_{10} & t_{11} & t_{12} & t_{13} & t_{14} & t_{15} \cr
{} &  &  & 2 &  &  & 1 &  & 1 &  &  &  &  &  & 1 &  \cr
{} &  &  &  &  & 2 &  &  &  &  & 1 &  &  & 1 &  & 1 \cr
{} &  &  &  &  &  &  & 1 & 1 &  &  &  &  &  & 1 &  \cr
{} &  &  &  &  &  &  &  & 2 &  &  &  &  &  & 1 &  \cr
{} &  &  &  &  &  &  &  &  & 1 &  &  &  & 1 &  & 1 \cr
{} &  &  &  &  &  &  &  &  &  &  & 1 &  & 1 &  & 1 \cr
{} &  &  &  &  &  &  &  &  &  &  &  & 1 & 1 &  & 1 \cr
{} &  &  &  &  &  &  &  &  &  &  &  &  & 2 &  & 1 \cr
}.
\]
It follows readily that the nontrivial elementary divisors of the Smith normal form of $T$ are all equal to $1$. The torsion subgroup of the group generated by the tails is thus trivial, thereby showing $\B_0(G) = 1$.


\item \label{number:57} 
Let the group $G$ be the representative of this family given by the presentation
\[
\begin{aligned}
\langle g_{1}, \,g_{2}, \,g_{3}, \,g_{4}, \,g_{5}, \,g_{6}, \,g_{7} & \mid & g_{1}^{2} &= 1, \\ 
 & & g_{2}^{2} &= 1, & [g_{2}, g_{1}]  &= g_{4}, \\ 
 & & g_{3}^{2} &= 1, & [g_{3}, g_{1}]  &= g_{5}, \\ 
 & & g_{4}^{2} &= g_{6}, & [g_{4}, g_{1}]  &= g_{6}, & [g_{4}, g_{2}]  &= g_{6}, & [g_{4}, g_{3}]  &= g_{7}, \\ 
 & & g_{5}^{2} &= 1, & [g_{5}, g_{2}]  &= g_{7}, \\ 
 & & g_{6}^{2} &= 1, \\ 
 & & g_{7}^{2} &= 1\rangle. \\ 
\end{aligned}
\]
We add 13 tails to the presentation as to form a quotient of the universal central extension of the system: 
$g_{1}^{2} =  t_{1}$,
$g_{2}^{2} =  t_{2}$,
$[g_{2}, g_{1}] = g_{4} t_{3}$,
$g_{3}^{2} =  t_{4}$,
$[g_{3}, g_{1}] = g_{5} t_{5}$,
$g_{4}^{2} = g_{6} t_{6}$,
$[g_{4}, g_{1}] = g_{6} t_{7}$,
$[g_{4}, g_{2}] = g_{6} t_{8}$,
$[g_{4}, g_{3}] = g_{7} t_{9}$,
$g_{5}^{2} =  t_{10}$,
$[g_{5}, g_{2}] = g_{7} t_{11}$,
$g_{6}^{2} =  t_{12}$,
$g_{7}^{2} =  t_{13}$.
Carrying out consistency checks gives the following relations between the tails:
\[
\begin{aligned}
g_{3}(g_{2} g_{1}) & = (g_{3} g_{2}) g_{1}  & \Longrightarrow & & t_{9}t_{11}^{-1} & = 1 \\
g_{5}^2 g_{2} & = g_{5} (g_{5} g_{2})& \Longrightarrow & & t_{11}^{2}t_{13} & = 1 \\
g_{4}^2 g_{2} & = g_{4} (g_{4} g_{2})& \Longrightarrow & & t_{8}^{2}t_{12} & = 1 \\
g_{4}^2 g_{1} & = g_{4} (g_{4} g_{1})& \Longrightarrow & & t_{7}^{2}t_{12} & = 1 \\
g_{3}^2 g_{1} & = g_{3} (g_{3} g_{1})& \Longrightarrow & & t_{5}^{2}t_{10} & = 1 \\
g_{2}^2 g_{1} & = g_{2} (g_{2} g_{1})& \Longrightarrow & & t_{3}^{2}t_{6}t_{8}t_{12} & = 1 \\
g_{2} g_{1}^{2} & = (g_{2} g_{1}) g_{1}& \Longrightarrow & & t_{3}^{2}t_{6}t_{7}t_{12} & = 1 \\
\end{aligned}
\]
Scanning through the conjugacy class representatives of $G$ and the generators of their centralizers, we see that no new relations are imposed.
Collecting the coefficients of these relations into a matrix yields
\[
T = \bordermatrix{
{} & t_{1} & t_{2} & t_{3} & t_{4} & t_{5} & t_{6} & t_{7} & t_{8} & t_{9} & t_{10} & t_{11} & t_{12} & t_{13} \cr
{} &  &  & 2 &  &  & 1 &  & 1 &  &  &  & 1 &  \cr
{} &  &  &  &  & 2 &  &  &  &  & 1 &  &  &  \cr
{} &  &  &  &  &  &  & 1 & 1 &  &  &  & 1 &  \cr
{} &  &  &  &  &  &  &  & 2 &  &  &  & 1 &  \cr
{} &  &  &  &  &  &  &  &  & 1 &  & 1 &  & 1 \cr
{} &  &  &  &  &  &  &  &  &  &  & 2 &  & 1 \cr
}.
\]
It follows readily that the nontrivial elementary divisors of the Smith normal form of $T$ are all equal to $1$. The torsion subgroup of the group generated by the tails is thus trivial, thereby showing $\B_0(G) = 1$.


\item \label{number:58} 
Let the group $G$ be the representative of this family given by the presentation
\[
\begin{aligned}
\langle g_{1}, \,g_{2}, \,g_{3}, \,g_{4}, \,g_{5}, \,g_{6}, \,g_{7} & \mid & g_{1}^{2} &= 1, \\ 
 & & g_{2}^{2} &= g_{4}, & [g_{2}, g_{1}]  &= g_{4}, \\ 
 & & g_{3}^{2} &= 1, & [g_{3}, g_{1}]  &= g_{5}, & [g_{3}, g_{2}]  &= g_{6}, \\ 
 & & g_{4}^{2} &= g_{6}, & [g_{4}, g_{1}]  &= g_{6}, \\ 
 & & g_{5}^{2} &= g_{7}, & [g_{5}, g_{1}]  &= g_{7}, & [g_{5}, g_{3}]  &= g_{7}, \\ 
 & & g_{6}^{2} &= 1, \\ 
 & & g_{7}^{2} &= 1\rangle. \\ 
\end{aligned}
\]
We add 13 tails to the presentation as to form a quotient of the universal central extension of the system: 
$g_{1}^{2} =  t_{1}$,
$g_{2}^{2} = g_{4} t_{2}$,
$[g_{2}, g_{1}] = g_{4} t_{3}$,
$g_{3}^{2} =  t_{4}$,
$[g_{3}, g_{1}] = g_{5} t_{5}$,
$[g_{3}, g_{2}] = g_{6} t_{6}$,
$g_{4}^{2} = g_{6} t_{7}$,
$[g_{4}, g_{1}] = g_{6} t_{8}$,
$g_{5}^{2} = g_{7} t_{9}$,
$[g_{5}, g_{1}] = g_{7} t_{10}$,
$[g_{5}, g_{3}] = g_{7} t_{11}$,
$g_{6}^{2} =  t_{12}$,
$g_{7}^{2} =  t_{13}$.
Carrying out consistency checks gives the following relations between the tails:
\[
\begin{aligned}
g_{5}^2 g_{3} & = g_{5} (g_{5} g_{3})& \Longrightarrow & & t_{11}^{2}t_{13} & = 1 \\
g_{5}^2 g_{1} & = g_{5} (g_{5} g_{1})& \Longrightarrow & & t_{10}^{2}t_{13} & = 1 \\
g_{4}^2 g_{1} & = g_{4} (g_{4} g_{1})& \Longrightarrow & & t_{8}^{2}t_{12} & = 1 \\
g_{3}^2 g_{2} & = g_{3} (g_{3} g_{2})& \Longrightarrow & & t_{6}^{2}t_{12} & = 1 \\
g_{3}^2 g_{1} & = g_{3} (g_{3} g_{1})& \Longrightarrow & & t_{5}^{2}t_{9}t_{11}t_{13} & = 1 \\
g_{2}^2 g_{1} & = g_{2} (g_{2} g_{1})& \Longrightarrow & & t_{3}^{2}t_{7}t_{8}^{-1} & = 1 \\
g_{3} g_{1}^{2} & = (g_{3} g_{1}) g_{1}& \Longrightarrow & & t_{5}^{2}t_{9}t_{10}t_{13} & = 1 \\
\end{aligned}
\]
Scanning through the conjugacy class representatives of $G$ and the generators of their centralizers, we see that no new relations are imposed.
Collecting the coefficients of these relations into a matrix yields
\[
T = \bordermatrix{
{} & t_{1} & t_{2} & t_{3} & t_{4} & t_{5} & t_{6} & t_{7} & t_{8} & t_{9} & t_{10} & t_{11} & t_{12} & t_{13} \cr
{} &  &  & 2 &  &  &  & 1 & 1 &  &  &  & 1 &  \cr
{} &  &  &  &  & 2 &  &  &  & 1 &  & 1 &  & 1 \cr
{} &  &  &  &  &  & 2 &  &  &  &  &  & 1 &  \cr
{} &  &  &  &  &  &  &  & 2 &  &  &  & 1 &  \cr
{} &  &  &  &  &  &  &  &  &  & 1 & 1 &  & 1 \cr
{} &  &  &  &  &  &  &  &  &  &  & 2 &  & 1 \cr
}.
\]
A change of basis according to the transition matrix (specifying expansions of $t_i^{*}$ by $t_j$)
\[
\bordermatrix{
{} & t_{1}^{*} & t_{2}^{*} & t_{3}^{*} & t_{4}^{*} & t_{5}^{*} & t_{6}^{*} & t_{7}^{*} & t_{8}^{*} & t_{9}^{*} & t_{10}^{*} & t_{11}^{*} & t_{12}^{*} & t_{13}^{*} \cr
t_{1} &  &  &  &  &  &  & 5 & 1 &  & 2 &  & 1 & 1 \cr
t_{2} &  &  &  &  &  &  & -8 &  & -1 & 1 &  &  &  \cr
t_{3} & 22 & 4 &  & 16 & 8 & 12 &  &  &  & -1 &  & -1 & -1 \cr
t_{4} &  &  &  &  &  &  &  & -1 &  & -1 &  & -1 &  \cr
t_{5} & -32 & -6 &  & -24 & -12 & -18 &  &  &  & -1 &  &  &  \cr
t_{6} & 48 & 10 &  & 36 & 18 & 27 & 8 &  &  &  &  &  &  \cr
t_{7} & 11 & 2 &  & 8 & 4 & 6 & -5 & -1 &  &  &  &  &  \cr
t_{8} & -41 & -8 &  & -30 & -16 & -23 &  & 1 &  &  &  &  &  \cr
t_{9} & -16 & -3 &  & -12 & -6 & -9 &  &  & 1 &  &  &  &  \cr
t_{10} & 32 & 6 &  & 24 & 13 & 18 &  &  &  & 1 &  &  &  \cr
t_{11} & -142 & -27 & 2 & -106 & -55 & -79 &  &  &  &  & 1 &  &  \cr
t_{12} & 9 & 2 &  & 7 & 3 & 5 &  &  &  &  &  & 1 &  \cr
t_{13} & -63 & -12 & 1 & -47 & -24 & -35 &  &  &  &  &  &  & 1 \cr
}
\]
shows that the nontrivial elementary divisors of the Smith normal form of $T$ are $1$, $1$, $1$, $1$, $1$, $2$.  The element corresponding to the divisor that is greater than $1$ is $t_{6}^{*}$. This already gives
\[
\B_0(G) \cong \langle t_{6}^{*}  \mid {t_{6}^{*}}^{2} \rangle.
\]

We now deal with explicitly identifying the nonuniversal commutator relation generating $\B_0(G)$.
First, factor out by the tails $t_{i}^{*}$ whose corresponding elementary divisors are either trivial or $1$. Transforming the situation back to the original tails $t_i$, this amounts to the nontrivial expansion 
$t_{6} = t_{6}^{*}$ and all the other tails $t_i$ are trivial. We thus obtain a commutativity preserving central extension of the group $G$, given by the presentation
\[
\begin{aligned}
\langle g_{1}, \,g_{2}, \,g_{3}, \,g_{4}, \,g_{5}, \,g_{6}, \,g_{7}, \,t_{6}^{*} & \mid & g_{1}^{2} &= 1, \\ 
 & & g_{2}^{2} &= g_{4}, & [g_{2}, g_{1}]  &= g_{4}, \\ 
 & & g_{3}^{2} &= 1, & [g_{3}, g_{1}]  &= g_{5}, & [g_{3}, g_{2}]  &= g_{6}t_{6}^{*} , \\ 
 & & g_{4}^{2} &= g_{6}, & [g_{4}, g_{1}]  &= g_{6}, \\ 
 & & g_{5}^{2} &= g_{7}, & [g_{5}, g_{1}]  &= g_{7}, & [g_{5}, g_{3}]  &= g_{7}, \\ 
 & & g_{6}^{2} &= 1, \\ 
 & & g_{7}^{2} &= 1, \\ 
 & & {t_{6}^{*}}^{2} &= 1  \rangle,
\end{aligned}
\]
whence the nonuniversal commutator relation is identified as
\[
t_{6}^{*}  = [g_{3}, g_{2}] [g_{4}, g_{1}]^{-1}.  \quad 
\]


\item \label{number:59} 
Let the group $G$ be the representative of this family given by the presentation
\[
\begin{aligned}
\langle g_{1}, \,g_{2}, \,g_{3}, \,g_{4}, \,g_{5}, \,g_{6}, \,g_{7} & \mid & g_{1}^{2} &= 1, \\ 
 & & g_{2}^{2} &= g_{4}, & [g_{2}, g_{1}]  &= g_{4}, \\ 
 & & g_{3}^{2} &= 1, & [g_{3}, g_{1}]  &= g_{5}, \\ 
 & & g_{4}^{2} &= g_{6}, & [g_{4}, g_{1}]  &= g_{6}, \\ 
 & & g_{5}^{2} &= g_{7}, & [g_{5}, g_{1}]  &= g_{7}, & [g_{5}, g_{3}]  &= g_{7}, \\ 
 & & g_{6}^{2} &= 1, \\ 
 & & g_{7}^{2} &= 1\rangle. \\ 
\end{aligned}
\]
We add 12 tails to the presentation as to form a quotient of the universal central extension of the system: 
$g_{1}^{2} =  t_{1}$,
$g_{2}^{2} = g_{4} t_{2}$,
$[g_{2}, g_{1}] = g_{4} t_{3}$,
$g_{3}^{2} =  t_{4}$,
$[g_{3}, g_{1}] = g_{5} t_{5}$,
$g_{4}^{2} = g_{6} t_{6}$,
$[g_{4}, g_{1}] = g_{6} t_{7}$,
$g_{5}^{2} = g_{7} t_{8}$,
$[g_{5}, g_{1}] = g_{7} t_{9}$,
$[g_{5}, g_{3}] = g_{7} t_{10}$,
$g_{6}^{2} =  t_{11}$,
$g_{7}^{2} =  t_{12}$.
Carrying out consistency checks gives the following relations between the tails:
\[
\begin{aligned}
g_{5}^2 g_{3} & = g_{5} (g_{5} g_{3})& \Longrightarrow & & t_{10}^{2}t_{12} & = 1 \\
g_{5}^2 g_{1} & = g_{5} (g_{5} g_{1})& \Longrightarrow & & t_{9}^{2}t_{12} & = 1 \\
g_{4}^2 g_{1} & = g_{4} (g_{4} g_{1})& \Longrightarrow & & t_{7}^{2}t_{11} & = 1 \\
g_{3}^2 g_{1} & = g_{3} (g_{3} g_{1})& \Longrightarrow & & t_{5}^{2}t_{8}t_{10}t_{12} & = 1 \\
g_{2}^2 g_{1} & = g_{2} (g_{2} g_{1})& \Longrightarrow & & t_{3}^{2}t_{6}t_{7}^{-1} & = 1 \\
g_{3} g_{1}^{2} & = (g_{3} g_{1}) g_{1}& \Longrightarrow & & t_{5}^{2}t_{8}t_{9}t_{12} & = 1 \\
\end{aligned}
\]
Scanning through the conjugacy class representatives of $G$ and the generators of their centralizers, we see that no new relations are imposed.
Collecting the coefficients of these relations into a matrix yields
\[
T = \bordermatrix{
{} & t_{1} & t_{2} & t_{3} & t_{4} & t_{5} & t_{6} & t_{7} & t_{8} & t_{9} & t_{10} & t_{11} & t_{12} \cr
{} &  &  & 2 &  &  & 1 & 1 &  &  &  & 1 &  \cr
{} &  &  &  &  & 2 &  &  & 1 &  & 1 &  & 1 \cr
{} &  &  &  &  &  &  & 2 &  &  &  & 1 &  \cr
{} &  &  &  &  &  &  &  &  & 1 & 1 &  & 1 \cr
{} &  &  &  &  &  &  &  &  &  & 2 &  & 1 \cr
}.
\]
It follows readily that the nontrivial elementary divisors of the Smith normal form of $T$ are all equal to $1$. The torsion subgroup of the group generated by the tails is thus trivial, thereby showing $\B_0(G) = 1$.


\item \label{number:60} 
Let the group $G$ be the representative of this family given by the presentation
\[
\begin{aligned}
\langle g_{1}, \,g_{2}, \,g_{3}, \,g_{4}, \,g_{5}, \,g_{6}, \,g_{7} & \mid & g_{1}^{2} &= 1, \\ 
 & & g_{2}^{2} &= g_{4}, & [g_{2}, g_{1}]  &= g_{4}, \\ 
 & & g_{3}^{2} &= g_{5}, & [g_{3}, g_{1}]  &= g_{5}, & [g_{3}, g_{2}]  &= g_{6}, \\ 
 & & g_{4}^{2} &= g_{6}, & [g_{4}, g_{1}]  &= g_{6}, \\ 
 & & g_{5}^{2} &= g_{7}, & [g_{5}, g_{1}]  &= g_{7}, \\ 
 & & g_{6}^{2} &= 1, \\ 
 & & g_{7}^{2} &= 1\rangle. \\ 
\end{aligned}
\]
We add 12 tails to the presentation as to form a quotient of the universal central extension of the system: 
$g_{1}^{2} =  t_{1}$,
$g_{2}^{2} = g_{4} t_{2}$,
$[g_{2}, g_{1}] = g_{4} t_{3}$,
$g_{3}^{2} = g_{5} t_{4}$,
$[g_{3}, g_{1}] = g_{5} t_{5}$,
$[g_{3}, g_{2}] = g_{6} t_{6}$,
$g_{4}^{2} = g_{6} t_{7}$,
$[g_{4}, g_{1}] = g_{6} t_{8}$,
$g_{5}^{2} = g_{7} t_{9}$,
$[g_{5}, g_{1}] = g_{7} t_{10}$,
$g_{6}^{2} =  t_{11}$,
$g_{7}^{2} =  t_{12}$.
Carrying out consistency checks gives the following relations between the tails:
\[
\begin{aligned}
g_{5}^2 g_{1} & = g_{5} (g_{5} g_{1})& \Longrightarrow & & t_{10}^{2}t_{12} & = 1 \\
g_{4}^2 g_{1} & = g_{4} (g_{4} g_{1})& \Longrightarrow & & t_{8}^{2}t_{11} & = 1 \\
g_{3}^2 g_{2} & = g_{3} (g_{3} g_{2})& \Longrightarrow & & t_{6}^{2}t_{11} & = 1 \\
g_{3}^2 g_{1} & = g_{3} (g_{3} g_{1})& \Longrightarrow & & t_{5}^{2}t_{9}t_{10}^{-1} & = 1 \\
g_{2}^2 g_{1} & = g_{2} (g_{2} g_{1})& \Longrightarrow & & t_{3}^{2}t_{7}t_{8}^{-1} & = 1 \\
\end{aligned}
\]
Scanning through the conjugacy class representatives of $G$ and the generators of their centralizers, we see that no new relations are imposed.
Collecting the coefficients of these relations into a matrix yields
\[
T = \bordermatrix{
{} & t_{1} & t_{2} & t_{3} & t_{4} & t_{5} & t_{6} & t_{7} & t_{8} & t_{9} & t_{10} & t_{11} & t_{12} \cr
{} &  &  & 2 &  &  &  & 1 & 1 &  &  & 1 &  \cr
{} &  &  &  &  & 2 &  &  &  & 1 & 1 &  & 1 \cr
{} &  &  &  &  &  & 2 &  &  &  &  & 1 &  \cr
{} &  &  &  &  &  &  &  & 2 &  &  & 1 &  \cr
{} &  &  &  &  &  &  &  &  &  & 2 &  & 1 \cr
}.
\]
A change of basis according to the transition matrix (specifying expansions of $t_i^{*}$ by $t_j$)
\[
\bordermatrix{
{} & t_{1}^{*} & t_{2}^{*} & t_{3}^{*} & t_{4}^{*} & t_{5}^{*} & t_{6}^{*} & t_{7}^{*} & t_{8}^{*} & t_{9}^{*} & t_{10}^{*} & t_{11}^{*} & t_{12}^{*} \cr
t_{1} &  &  &  &  &  & -1 & -1 &  &  &  & 1 &  \cr
t_{2} &  &  &  &  &  & 1 & 1 &  & -1 &  &  &  \cr
t_{3} & -2 & -4 &  &  & -4 &  &  &  &  &  & -1 &  \cr
t_{4} &  &  &  &  &  &  &  & -1 &  &  &  & -1 \cr
t_{5} & 4 & 6 &  &  & 6 & 1 &  &  &  &  &  &  \cr
t_{6} & -6 & -8 &  &  & -9 & -2 & -1 &  &  &  &  &  \cr
t_{7} & -1 & -2 &  &  & -2 &  & 1 &  &  &  &  &  \cr
t_{8} & 11 & 18 & 2 &  & 19 &  &  & 1 &  &  &  &  \cr
t_{9} & 2 & 3 &  &  & 3 &  &  &  & 1 &  &  &  \cr
t_{10} & -24 & -39 & -4 & 2 & -41 &  &  &  &  & 1 &  &  \cr
t_{11} & 2 & 4 & 1 &  & 4 &  &  &  &  &  & 1 &  \cr
t_{12} & -11 & -18 & -2 & 1 & -19 &  &  &  &  &  &  & 1 \cr
}
\]
shows that the nontrivial elementary divisors of the Smith normal form of $T$ are $1$, $1$, $1$, $1$, $2$.  The element corresponding to the divisor that is greater than $1$ is $t_{5}^{*}$. This already gives
\[
\B_0(G) \cong \langle t_{5}^{*}  \mid {t_{5}^{*}}^{2} \rangle.
\]

We now deal with explicitly identifying the nonuniversal commutator relation generating $\B_0(G)$.
First, factor out by the tails $t_{i}^{*}$ whose corresponding elementary divisors are either trivial or $1$. Transforming the situation back to the original tails $t_i$, this amounts to the nontrivial expansions given by
\[
\bordermatrix{
{} & t_{1} & t_{3} & t_{5} & t_{6} \cr
t_{5}^{*} & 1 & 1 & 1 & 1 \cr
}
\]
and all the other tails $t_i$ are trivial. We thus obtain a commutativity preserving central extension of the group $G$, given by the presentation
\[
\begin{aligned}
\langle g_{1}, \,g_{2}, \,g_{3}, \,g_{4}, \,g_{5}, \,g_{6}, \,g_{7}, \,t_{5}^{*} & \mid & g_{1}^{2} &= t_{5}^{*} , \\ 
 & & g_{2}^{2} &= g_{4}, & [g_{2}, g_{1}]  &= g_{4}t_{5}^{*} , \\ 
 & & g_{3}^{2} &= g_{5}, & [g_{3}, g_{1}]  &= g_{5}t_{5}^{*} , & [g_{3}, g_{2}]  &= g_{6}t_{5}^{*} , \\ 
 & & g_{4}^{2} &= g_{6}, & [g_{4}, g_{1}]  &= g_{6}, \\ 
 & & g_{5}^{2} &= g_{7}, & [g_{5}, g_{1}]  &= g_{7}, \\ 
 & & g_{6}^{2} &= 1, \\ 
 & & g_{7}^{2} &= 1, \\ 
 & & {t_{5}^{*}}^{2} &= 1  \rangle,
\end{aligned}
\]
whence the nonuniversal commutator relation is identified as
\[
t_{5}^{*}  = [g_{3}, g_{2}] [g_{4}, g_{1}]^{-1}.  \quad 
\]


\item \label{number:61} 
Let the group $G$ be the representative of this family given by the presentation
\[
\begin{aligned}
\langle g_{1}, \,g_{2}, \,g_{3}, \,g_{4}, \,g_{5}, \,g_{6}, \,g_{7} & \mid & g_{1}^{2} &= 1, \\ 
 & & g_{2}^{2} &= g_{4}, & [g_{2}, g_{1}]  &= g_{4}, \\ 
 & & g_{3}^{2} &= g_{5}, & [g_{3}, g_{1}]  &= g_{5}, \\ 
 & & g_{4}^{2} &= g_{6}, & [g_{4}, g_{1}]  &= g_{6}, \\ 
 & & g_{5}^{2} &= g_{7}, & [g_{5}, g_{1}]  &= g_{7}, \\ 
 & & g_{6}^{2} &= 1, \\ 
 & & g_{7}^{2} &= 1\rangle. \\ 
\end{aligned}
\]
We add 11 tails to the presentation as to form a quotient of the universal central extension of the system: 
$g_{1}^{2} =  t_{1}$,
$g_{2}^{2} = g_{4} t_{2}$,
$[g_{2}, g_{1}] = g_{4} t_{3}$,
$g_{3}^{2} = g_{5} t_{4}$,
$[g_{3}, g_{1}] = g_{5} t_{5}$,
$g_{4}^{2} = g_{6} t_{6}$,
$[g_{4}, g_{1}] = g_{6} t_{7}$,
$g_{5}^{2} = g_{7} t_{8}$,
$[g_{5}, g_{1}] = g_{7} t_{9}$,
$g_{6}^{2} =  t_{10}$,
$g_{7}^{2} =  t_{11}$.
Carrying out consistency checks gives the following relations between the tails:
\[
\begin{aligned}
g_{5}^2 g_{1} & = g_{5} (g_{5} g_{1})& \Longrightarrow & & t_{9}^{2}t_{11} & = 1 \\
g_{4}^2 g_{1} & = g_{4} (g_{4} g_{1})& \Longrightarrow & & t_{7}^{2}t_{10} & = 1 \\
g_{3}^2 g_{1} & = g_{3} (g_{3} g_{1})& \Longrightarrow & & t_{5}^{2}t_{8}t_{9}^{-1} & = 1 \\
g_{2}^2 g_{1} & = g_{2} (g_{2} g_{1})& \Longrightarrow & & t_{3}^{2}t_{6}t_{7}^{-1} & = 1 \\
\end{aligned}
\]
Scanning through the conjugacy class representatives of $G$ and the generators of their centralizers, we see that no new relations are imposed.
Collecting the coefficients of these relations into a matrix yields
\[
T = \bordermatrix{
{} & t_{1} & t_{2} & t_{3} & t_{4} & t_{5} & t_{6} & t_{7} & t_{8} & t_{9} & t_{10} & t_{11} \cr
{} &  &  & 2 &  &  & 1 & 1 &  &  & 1 &  \cr
{} &  &  &  &  & 2 &  &  & 1 & 1 &  & 1 \cr
{} &  &  &  &  &  &  & 2 &  &  & 1 &  \cr
{} &  &  &  &  &  &  &  &  & 2 &  & 1 \cr
}.
\]
It follows readily that the nontrivial elementary divisors of the Smith normal form of $T$ are all equal to $1$. The torsion subgroup of the group generated by the tails is thus trivial, thereby showing $\B_0(G) = 1$.


\item \label{number:62} 
Let the group $G$ be the representative of this family given by the presentation
\[
\begin{aligned}
\langle g_{1}, \,g_{2}, \,g_{3}, \,g_{4}, \,g_{5}, \,g_{6}, \,g_{7} & \mid & g_{1}^{2} &= 1, \\ 
 & & g_{2}^{2} &= 1, & [g_{2}, g_{1}]  &= g_{4}, \\ 
 & & g_{3}^{2} &= 1, & [g_{3}, g_{1}]  &= g_{5}, & [g_{3}, g_{2}]  &= g_{6}, \\ 
 & & g_{4}^{2} &= g_{7}, & [g_{4}, g_{1}]  &= g_{7}, & [g_{4}, g_{2}]  &= g_{7}, & [g_{4}, g_{3}]  &= g_{7}, \\ 
 & & g_{5}^{2} &= g_{7}, & [g_{5}, g_{1}]  &= g_{7}, & [g_{5}, g_{3}]  &= g_{7}, \\ 
 & & g_{6}^{2} &= g_{7}, & [g_{6}, g_{1}]  &= g_{7}, & [g_{6}, g_{2}]  &= g_{7}, & [g_{6}, g_{3}]  &= g_{7}, \\ 
 & & g_{7}^{2} &= 1\rangle. \\ 
\end{aligned}
\]
We add 18 tails to the presentation as to form a quotient of the universal central extension of the system: 
$g_{1}^{2} =  t_{1}$,
$g_{2}^{2} =  t_{2}$,
$[g_{2}, g_{1}] = g_{4} t_{3}$,
$g_{3}^{2} =  t_{4}$,
$[g_{3}, g_{1}] = g_{5} t_{5}$,
$[g_{3}, g_{2}] = g_{6} t_{6}$,
$g_{4}^{2} = g_{7} t_{7}$,
$[g_{4}, g_{1}] = g_{7} t_{8}$,
$[g_{4}, g_{2}] = g_{7} t_{9}$,
$[g_{4}, g_{3}] = g_{7} t_{10}$,
$g_{5}^{2} = g_{7} t_{11}$,
$[g_{5}, g_{1}] = g_{7} t_{12}$,
$[g_{5}, g_{3}] = g_{7} t_{13}$,
$g_{6}^{2} = g_{7} t_{14}$,
$[g_{6}, g_{1}] = g_{7} t_{15}$,
$[g_{6}, g_{2}] = g_{7} t_{16}$,
$[g_{6}, g_{3}] = g_{7} t_{17}$,
$g_{7}^{2} =  t_{18}$.
Carrying out consistency checks gives the following relations between the tails:
\[
\begin{aligned}
g_{3}(g_{2} g_{1}) & = (g_{3} g_{2}) g_{1}  & \Longrightarrow & & t_{10}t_{15}t_{18} & = 1 \\
g_{6}^2 g_{3} & = g_{6} (g_{6} g_{3})& \Longrightarrow & & t_{17}^{2}t_{18} & = 1 \\
g_{6}^2 g_{2} & = g_{6} (g_{6} g_{2})& \Longrightarrow & & t_{16}^{2}t_{18} & = 1 \\
g_{6}^2 g_{1} & = g_{6} (g_{6} g_{1})& \Longrightarrow & & t_{15}^{2}t_{18} & = 1 \\
g_{5}^2 g_{3} & = g_{5} (g_{5} g_{3})& \Longrightarrow & & t_{13}^{2}t_{18} & = 1 \\
g_{5}^2 g_{1} & = g_{5} (g_{5} g_{1})& \Longrightarrow & & t_{12}^{2}t_{18} & = 1 \\
g_{4}^2 g_{2} & = g_{4} (g_{4} g_{2})& \Longrightarrow & & t_{9}^{2}t_{18} & = 1 \\
g_{4}^2 g_{1} & = g_{4} (g_{4} g_{1})& \Longrightarrow & & t_{8}^{2}t_{18} & = 1 \\
g_{3}^2 g_{2} & = g_{3} (g_{3} g_{2})& \Longrightarrow & & t_{6}^{2}t_{14}t_{17}t_{18} & = 1 \\
g_{3}^2 g_{1} & = g_{3} (g_{3} g_{1})& \Longrightarrow & & t_{5}^{2}t_{11}t_{13}t_{18} & = 1 \\
g_{2}^2 g_{1} & = g_{2} (g_{2} g_{1})& \Longrightarrow & & t_{3}^{2}t_{7}t_{9}t_{18} & = 1 \\
g_{3} g_{2}^{2} & = (g_{3} g_{2}) g_{2}& \Longrightarrow & & t_{6}^{2}t_{14}t_{16}t_{18} & = 1 \\
g_{3} g_{1}^{2} & = (g_{3} g_{1}) g_{1}& \Longrightarrow & & t_{5}^{2}t_{11}t_{12}t_{18} & = 1 \\
g_{2} g_{1}^{2} & = (g_{2} g_{1}) g_{1}& \Longrightarrow & & t_{3}^{2}t_{7}t_{8}t_{18} & = 1 \\
\end{aligned}
\]
Scanning through the conjugacy class representatives of $G$ and the generators of their centralizers, we obtain the following relations induced on the tails:
\[
\begin{aligned}
{[g_{6} g_{7} , \, g_{1} g_{3} ]}_G & = 1 & \Longrightarrow & & t_{15}t_{17}t_{18} & = 1 \\
{[g_{5} g_{6} g_{7} , \, g_{1} ]}_G & = 1 & \Longrightarrow & & t_{12}t_{15}t_{18} & = 1 \\
{[g_{4} g_{7} , \, g_{1} g_{3} ]}_G & = 1 & \Longrightarrow & & t_{8}t_{10}t_{18} & = 1 \\
\end{aligned}
\]
Collecting the coefficients of these relations into a matrix yields
\[
T = \bordermatrix{
{} & t_{1} & t_{2} & t_{3} & t_{4} & t_{5} & t_{6} & t_{7} & t_{8} & t_{9} & t_{10} & t_{11} & t_{12} & t_{13} & t_{14} & t_{15} & t_{16} & t_{17} & t_{18} \cr
{} &  &  & 2 &  &  &  & 1 &  &  &  &  &  &  &  &  &  & 1 & 1 \cr
{} &  &  &  &  & 2 &  &  &  &  &  & 1 &  &  &  &  &  & 1 & 1 \cr
{} &  &  &  &  &  & 2 &  &  &  &  &  &  &  & 1 &  &  & 1 & 1 \cr
{} &  &  &  &  &  &  &  & 1 &  &  &  &  &  &  &  &  & 1 & 1 \cr
{} &  &  &  &  &  &  &  &  & 1 &  &  &  &  &  &  &  & 1 & 1 \cr
{} &  &  &  &  &  &  &  &  &  & 1 &  &  &  &  &  &  & 1 & 1 \cr
{} &  &  &  &  &  &  &  &  &  &  &  & 1 &  &  &  &  & 1 & 1 \cr
{} &  &  &  &  &  &  &  &  &  &  &  &  & 1 &  &  &  & 1 & 1 \cr
{} &  &  &  &  &  &  &  &  &  &  &  &  &  &  & 1 &  & 1 & 1 \cr
{} &  &  &  &  &  &  &  &  &  &  &  &  &  &  &  & 1 & 1 & 1 \cr
{} &  &  &  &  &  &  &  &  &  &  &  &  &  &  &  &  & 2 & 1 \cr
}.
\]
It follows readily that the nontrivial elementary divisors of the Smith normal form of $T$ are all equal to $1$. The torsion subgroup of the group generated by the tails is thus trivial, thereby showing $\B_0(G) = 1$.


\item \label{number:63} 
Let the group $G$ be the representative of this family given by the presentation
\[
\begin{aligned}
\langle g_{1}, \,g_{2}, \,g_{3}, \,g_{4}, \,g_{5}, \,g_{6}, \,g_{7} & \mid & g_{1}^{2} &= 1, \\ 
 & & g_{2}^{2} &= 1, & [g_{2}, g_{1}]  &= g_{4}, \\ 
 & & g_{3}^{2} &= 1, & [g_{3}, g_{1}]  &= g_{5}, & [g_{3}, g_{2}]  &= g_{6}, \\ 
 & & g_{4}^{2} &= g_{7}, & [g_{4}, g_{1}]  &= g_{7}, & [g_{4}, g_{2}]  &= g_{7}, & [g_{4}, g_{3}]  &= g_{7}, \\ 
 & & g_{5}^{2} &= 1, \\ 
 & & g_{6}^{2} &= 1, & [g_{6}, g_{1}]  &= g_{7}, \\ 
 & & g_{7}^{2} &= 1\rangle. \\ 
\end{aligned}
\]
We add 14 tails to the presentation as to form a quotient of the universal central extension of the system: 
$g_{1}^{2} =  t_{1}$,
$g_{2}^{2} =  t_{2}$,
$[g_{2}, g_{1}] = g_{4} t_{3}$,
$g_{3}^{2} =  t_{4}$,
$[g_{3}, g_{1}] = g_{5} t_{5}$,
$[g_{3}, g_{2}] = g_{6} t_{6}$,
$g_{4}^{2} = g_{7} t_{7}$,
$[g_{4}, g_{1}] = g_{7} t_{8}$,
$[g_{4}, g_{2}] = g_{7} t_{9}$,
$[g_{4}, g_{3}] = g_{7} t_{10}$,
$g_{5}^{2} =  t_{11}$,
$g_{6}^{2} =  t_{12}$,
$[g_{6}, g_{1}] = g_{7} t_{13}$,
$g_{7}^{2} =  t_{14}$.
Carrying out consistency checks gives the following relations between the tails:
\[
\begin{aligned}
g_{3}(g_{2} g_{1}) & = (g_{3} g_{2}) g_{1}  & \Longrightarrow & & t_{10}t_{13}t_{14} & = 1 \\
g_{6}^2 g_{1} & = g_{6} (g_{6} g_{1})& \Longrightarrow & & t_{13}^{2}t_{14} & = 1 \\
g_{4}^2 g_{2} & = g_{4} (g_{4} g_{2})& \Longrightarrow & & t_{9}^{2}t_{14} & = 1 \\
g_{4}^2 g_{1} & = g_{4} (g_{4} g_{1})& \Longrightarrow & & t_{8}^{2}t_{14} & = 1 \\
g_{3}^2 g_{2} & = g_{3} (g_{3} g_{2})& \Longrightarrow & & t_{6}^{2}t_{12} & = 1 \\
g_{3}^2 g_{1} & = g_{3} (g_{3} g_{1})& \Longrightarrow & & t_{5}^{2}t_{11} & = 1 \\
g_{2}^2 g_{1} & = g_{2} (g_{2} g_{1})& \Longrightarrow & & t_{3}^{2}t_{7}t_{9}t_{14} & = 1 \\
g_{2} g_{1}^{2} & = (g_{2} g_{1}) g_{1}& \Longrightarrow & & t_{3}^{2}t_{7}t_{8}t_{14} & = 1 \\
\end{aligned}
\]
Scanning through the conjugacy class representatives of $G$ and the generators of their centralizers, we obtain the following relations induced on the tails:
\[
\begin{aligned}
{[g_{4} g_{7} , \, g_{1} g_{3} ]}_G & = 1 & \Longrightarrow & & t_{8}t_{10}t_{14} & = 1 \\
\end{aligned}
\]
Collecting the coefficients of these relations into a matrix yields
\[
T = \bordermatrix{
{} & t_{1} & t_{2} & t_{3} & t_{4} & t_{5} & t_{6} & t_{7} & t_{8} & t_{9} & t_{10} & t_{11} & t_{12} & t_{13} & t_{14} \cr
{} &  &  & 2 &  &  &  & 1 &  &  &  &  &  & 1 & 1 \cr
{} &  &  &  &  & 2 &  &  &  &  &  & 1 &  &  &  \cr
{} &  &  &  &  &  & 2 &  &  &  &  &  & 1 &  &  \cr
{} &  &  &  &  &  &  &  & 1 &  &  &  &  & 1 & 1 \cr
{} &  &  &  &  &  &  &  &  & 1 &  &  &  & 1 & 1 \cr
{} &  &  &  &  &  &  &  &  &  & 1 &  &  & 1 & 1 \cr
{} &  &  &  &  &  &  &  &  &  &  &  &  & 2 & 1 \cr
}.
\]
It follows readily that the nontrivial elementary divisors of the Smith normal form of $T$ are all equal to $1$. The torsion subgroup of the group generated by the tails is thus trivial, thereby showing $\B_0(G) = 1$.


\item \label{number:64} 
Let the group $G$ be the representative of this family given by the presentation
\[
\begin{aligned}
\langle g_{1}, \,g_{2}, \,g_{3}, \,g_{4}, \,g_{5}, \,g_{6}, \,g_{7} & \mid & g_{1}^{2} &= g_{4}, \\ 
 & & g_{2}^{2} &= g_{4}, & [g_{2}, g_{1}]  &= g_{4}, \\ 
 & & g_{3}^{2} &= 1, & [g_{3}, g_{1}]  &= g_{5}, & [g_{3}, g_{2}]  &= g_{6}, \\ 
 & & g_{4}^{2} &= 1, & [g_{4}, g_{3}]  &= g_{7}, \\ 
 & & g_{5}^{2} &= g_{7}, & [g_{5}, g_{3}]  &= g_{7}, \\ 
 & & g_{6}^{2} &= g_{7}, & [g_{6}, g_{1}]  &= g_{7}, & [g_{6}, g_{3}]  &= g_{7}, \\ 
 & & g_{7}^{2} &= 1\rangle. \\ 
\end{aligned}
\]
We add 14 tails to the presentation as to form a quotient of the universal central extension of the system: 
$g_{1}^{2} = g_{4} t_{1}$,
$g_{2}^{2} = g_{4} t_{2}$,
$[g_{2}, g_{1}] = g_{4} t_{3}$,
$g_{3}^{2} =  t_{4}$,
$[g_{3}, g_{1}] = g_{5} t_{5}$,
$[g_{3}, g_{2}] = g_{6} t_{6}$,
$g_{4}^{2} =  t_{7}$,
$[g_{4}, g_{3}] = g_{7} t_{8}$,
$g_{5}^{2} = g_{7} t_{9}$,
$[g_{5}, g_{3}] = g_{7} t_{10}$,
$g_{6}^{2} = g_{7} t_{11}$,
$[g_{6}, g_{1}] = g_{7} t_{12}$,
$[g_{6}, g_{3}] = g_{7} t_{13}$,
$g_{7}^{2} =  t_{14}$.
Carrying out consistency checks gives the following relations between the tails:
\[
\begin{aligned}
g_{3}(g_{2} g_{1}) & = (g_{3} g_{2}) g_{1}  & \Longrightarrow & & t_{8}t_{12}t_{14} & = 1 \\
g_{6}^2 g_{3} & = g_{6} (g_{6} g_{3})& \Longrightarrow & & t_{13}^{2}t_{14} & = 1 \\
g_{6}^2 g_{1} & = g_{6} (g_{6} g_{1})& \Longrightarrow & & t_{12}^{2}t_{14} & = 1 \\
g_{5}^2 g_{3} & = g_{5} (g_{5} g_{3})& \Longrightarrow & & t_{10}^{2}t_{14} & = 1 \\
g_{3}^2 g_{2} & = g_{3} (g_{3} g_{2})& \Longrightarrow & & t_{6}^{2}t_{11}t_{13}t_{14} & = 1 \\
g_{3}^2 g_{1} & = g_{3} (g_{3} g_{1})& \Longrightarrow & & t_{5}^{2}t_{9}t_{10}t_{14} & = 1 \\
g_{2}^2 g_{1} & = g_{2} (g_{2} g_{1})& \Longrightarrow & & t_{3}^{2}t_{7} & = 1 \\
g_{3} g_{2}^{2} & = (g_{3} g_{2}) g_{2}& \Longrightarrow & & t_{6}^{2}t_{8}t_{11}t_{14} & = 1 \\
g_{3} g_{1}^{2} & = (g_{3} g_{1}) g_{1}& \Longrightarrow & & t_{5}^{2}t_{8}t_{9}t_{14} & = 1 \\
\end{aligned}
\]
Scanning through the conjugacy class representatives of $G$ and the generators of their centralizers, we see that no new relations are imposed.
Collecting the coefficients of these relations into a matrix yields
\[
T = \bordermatrix{
{} & t_{1} & t_{2} & t_{3} & t_{4} & t_{5} & t_{6} & t_{7} & t_{8} & t_{9} & t_{10} & t_{11} & t_{12} & t_{13} & t_{14} \cr
{} &  &  & 2 &  &  &  & 1 &  &  &  &  &  &  &  \cr
{} &  &  &  &  & 2 &  &  &  & 1 &  &  &  & 1 & 1 \cr
{} &  &  &  &  &  & 2 &  &  &  &  & 1 &  & 1 & 1 \cr
{} &  &  &  &  &  &  &  & 1 &  &  &  &  & 1 & 1 \cr
{} &  &  &  &  &  &  &  &  &  & 1 &  &  & 1 & 1 \cr
{} &  &  &  &  &  &  &  &  &  &  &  & 1 & 1 & 1 \cr
{} &  &  &  &  &  &  &  &  &  &  &  &  & 2 & 1 \cr
}.
\]
It follows readily that the nontrivial elementary divisors of the Smith normal form of $T$ are all equal to $1$. The torsion subgroup of the group generated by the tails is thus trivial, thereby showing $\B_0(G) = 1$.


\item \label{number:65} 
Let the group $G$ be the representative of this family given by the presentation
\[
\begin{aligned}
\langle g_{1}, \,g_{2}, \,g_{3}, \,g_{4}, \,g_{5}, \,g_{6}, \,g_{7} & \mid & g_{1}^{2} &= 1, \\ 
 & & g_{2}^{2} &= 1, & [g_{2}, g_{1}]  &= g_{4}, \\ 
 & & g_{3}^{2} &= 1, & [g_{3}, g_{1}]  &= g_{5}, & [g_{3}, g_{2}]  &= g_{6}, \\ 
 & & g_{4}^{2} &= 1, & [g_{4}, g_{3}]  &= g_{7}, \\ 
 & & g_{5}^{2} &= 1, \\ 
 & & g_{6}^{2} &= 1, & [g_{6}, g_{1}]  &= g_{7}, \\ 
 & & g_{7}^{2} &= 1\rangle. \\ 
\end{aligned}
\]
We add 12 tails to the presentation as to form a quotient of the universal central extension of the system: 
$g_{1}^{2} =  t_{1}$,
$g_{2}^{2} =  t_{2}$,
$[g_{2}, g_{1}] = g_{4} t_{3}$,
$g_{3}^{2} =  t_{4}$,
$[g_{3}, g_{1}] = g_{5} t_{5}$,
$[g_{3}, g_{2}] = g_{6} t_{6}$,
$g_{4}^{2} =  t_{7}$,
$[g_{4}, g_{3}] = g_{7} t_{8}$,
$g_{5}^{2} =  t_{9}$,
$g_{6}^{2} =  t_{10}$,
$[g_{6}, g_{1}] = g_{7} t_{11}$,
$g_{7}^{2} =  t_{12}$.
Carrying out consistency checks gives the following relations between the tails:
\[
\begin{aligned}
g_{3}(g_{2} g_{1}) & = (g_{3} g_{2}) g_{1}  & \Longrightarrow & & t_{8}t_{11}t_{12} & = 1 \\
g_{6}^2 g_{1} & = g_{6} (g_{6} g_{1})& \Longrightarrow & & t_{11}^{2}t_{12} & = 1 \\
g_{3}^2 g_{2} & = g_{3} (g_{3} g_{2})& \Longrightarrow & & t_{6}^{2}t_{10} & = 1 \\
g_{3}^2 g_{1} & = g_{3} (g_{3} g_{1})& \Longrightarrow & & t_{5}^{2}t_{9} & = 1 \\
g_{2}^2 g_{1} & = g_{2} (g_{2} g_{1})& \Longrightarrow & & t_{3}^{2}t_{7} & = 1 \\
\end{aligned}
\]
Scanning through the conjugacy class representatives of $G$ and the generators of their centralizers, we see that no new relations are imposed.
Collecting the coefficients of these relations into a matrix yields
\[
T = \bordermatrix{
{} & t_{1} & t_{2} & t_{3} & t_{4} & t_{5} & t_{6} & t_{7} & t_{8} & t_{9} & t_{10} & t_{11} & t_{12} \cr
{} &  &  & 2 &  &  &  & 1 &  &  &  &  &  \cr
{} &  &  &  &  & 2 &  &  &  & 1 &  &  &  \cr
{} &  &  &  &  &  & 2 &  &  &  & 1 &  &  \cr
{} &  &  &  &  &  &  &  & 1 &  &  & 1 & 1 \cr
{} &  &  &  &  &  &  &  &  &  &  & 2 & 1 \cr
}.
\]
It follows readily that the nontrivial elementary divisors of the Smith normal form of $T$ are all equal to $1$. The torsion subgroup of the group generated by the tails is thus trivial, thereby showing $\B_0(G) = 1$.


\item \label{number:66} 
Let the group $G$ be the representative of this family given by the presentation
\[
\begin{aligned}
\langle g_{1}, \,g_{2}, \,g_{3}, \,g_{4}, \,g_{5}, \,g_{6}, \,g_{7} & \mid & g_{1}^{2} &= 1, \\ 
 & & g_{2}^{2} &= 1, & [g_{2}, g_{1}]  &= g_{4}, \\ 
 & & g_{3}^{2} &= 1, & [g_{3}, g_{1}]  &= g_{5}, & [g_{3}, g_{2}]  &= g_{6}, \\ 
 & & g_{4}^{2} &= g_{7}, & [g_{4}, g_{1}]  &= g_{7}, & [g_{4}, g_{2}]  &= g_{7}, \\ 
 & & g_{5}^{2} &= 1, & [g_{5}, g_{2}]  &= g_{7}, \\ 
 & & g_{6}^{2} &= 1, & [g_{6}, g_{1}]  &= g_{7}, \\ 
 & & g_{7}^{2} &= 1\rangle. \\ 
\end{aligned}
\]
We add 14 tails to the presentation as to form a quotient of the universal central extension of the system: 
$g_{1}^{2} =  t_{1}$,
$g_{2}^{2} =  t_{2}$,
$[g_{2}, g_{1}] = g_{4} t_{3}$,
$g_{3}^{2} =  t_{4}$,
$[g_{3}, g_{1}] = g_{5} t_{5}$,
$[g_{3}, g_{2}] = g_{6} t_{6}$,
$g_{4}^{2} = g_{7} t_{7}$,
$[g_{4}, g_{1}] = g_{7} t_{8}$,
$[g_{4}, g_{2}] = g_{7} t_{9}$,
$g_{5}^{2} =  t_{10}$,
$[g_{5}, g_{2}] = g_{7} t_{11}$,
$g_{6}^{2} =  t_{12}$,
$[g_{6}, g_{1}] = g_{7} t_{13}$,
$g_{7}^{2} =  t_{14}$.
Carrying out consistency checks gives the following relations between the tails:
\[
\begin{aligned}
g_{3}(g_{2} g_{1}) & = (g_{3} g_{2}) g_{1}  & \Longrightarrow & & t_{11}^{-1}t_{13} & = 1 \\
g_{6}^2 g_{1} & = g_{6} (g_{6} g_{1})& \Longrightarrow & & t_{13}^{2}t_{14} & = 1 \\
g_{4}^2 g_{2} & = g_{4} (g_{4} g_{2})& \Longrightarrow & & t_{9}^{2}t_{14} & = 1 \\
g_{4}^2 g_{1} & = g_{4} (g_{4} g_{1})& \Longrightarrow & & t_{8}^{2}t_{14} & = 1 \\
g_{3}^2 g_{2} & = g_{3} (g_{3} g_{2})& \Longrightarrow & & t_{6}^{2}t_{12} & = 1 \\
g_{3}^2 g_{1} & = g_{3} (g_{3} g_{1})& \Longrightarrow & & t_{5}^{2}t_{10} & = 1 \\
g_{2}^2 g_{1} & = g_{2} (g_{2} g_{1})& \Longrightarrow & & t_{3}^{2}t_{7}t_{9}t_{14} & = 1 \\
g_{2} g_{1}^{2} & = (g_{2} g_{1}) g_{1}& \Longrightarrow & & t_{3}^{2}t_{7}t_{8}t_{14} & = 1 \\
\end{aligned}
\]
Scanning through the conjugacy class representatives of $G$ and the generators of their centralizers, we obtain the following relations induced on the tails:
\[
\begin{aligned}
{[g_{4} g_{6} g_{7} , \, g_{1} ]}_G & = 1 & \Longrightarrow & & t_{8}t_{13}t_{14} & = 1 \\
\end{aligned}
\]
Collecting the coefficients of these relations into a matrix yields
\[
T = \bordermatrix{
{} & t_{1} & t_{2} & t_{3} & t_{4} & t_{5} & t_{6} & t_{7} & t_{8} & t_{9} & t_{10} & t_{11} & t_{12} & t_{13} & t_{14} \cr
{} &  &  & 2 &  &  &  & 1 &  &  &  &  &  & 1 & 1 \cr
{} &  &  &  &  & 2 &  &  &  &  & 1 &  &  &  &  \cr
{} &  &  &  &  &  & 2 &  &  &  &  &  & 1 &  &  \cr
{} &  &  &  &  &  &  &  & 1 &  &  &  &  & 1 & 1 \cr
{} &  &  &  &  &  &  &  &  & 1 &  &  &  & 1 & 1 \cr
{} &  &  &  &  &  &  &  &  &  &  & 1 &  & 1 & 1 \cr
{} &  &  &  &  &  &  &  &  &  &  &  &  & 2 & 1 \cr
}.
\]
It follows readily that the nontrivial elementary divisors of the Smith normal form of $T$ are all equal to $1$. The torsion subgroup of the group generated by the tails is thus trivial, thereby showing $\B_0(G) = 1$.


\item \label{number:67} 
Let the group $G$ be the representative of this family given by the presentation
\[
\begin{aligned}
\langle g_{1}, \,g_{2}, \,g_{3}, \,g_{4}, \,g_{5}, \,g_{6}, \,g_{7} & \mid & g_{1}^{2} &= 1, \\ 
 & & g_{2}^{2} &= 1, & [g_{2}, g_{1}]  &= g_{5}, \\ 
 & & g_{3}^{2} &= 1, & [g_{3}, g_{1}]  &= g_{6}, \\ 
 & & g_{4}^{2} &= 1, & [g_{4}, g_{2}]  &= g_{5}, \\ 
 & & g_{5}^{2} &= g_{7}, & [g_{5}, g_{1}]  &= g_{7}, & [g_{5}, g_{2}]  &= g_{7}, & [g_{5}, g_{4}]  &= g_{7}, \\ 
 & & g_{6}^{2} &= g_{7}, & [g_{6}, g_{1}]  &= g_{7}, & [g_{6}, g_{3}]  &= g_{7}, \\ 
 & & g_{7}^{2} &= 1\rangle. \\ 
\end{aligned}
\]
We add 15 tails to the presentation as to form a quotient of the universal central extension of the system: 
$g_{1}^{2} =  t_{1}$,
$g_{2}^{2} =  t_{2}$,
$[g_{2}, g_{1}] = g_{5} t_{3}$,
$g_{3}^{2} =  t_{4}$,
$[g_{3}, g_{1}] = g_{6} t_{5}$,
$g_{4}^{2} =  t_{6}$,
$[g_{4}, g_{2}] = g_{5} t_{7}$,
$g_{5}^{2} = g_{7} t_{8}$,
$[g_{5}, g_{1}] = g_{7} t_{9}$,
$[g_{5}, g_{2}] = g_{7} t_{10}$,
$[g_{5}, g_{4}] = g_{7} t_{11}$,
$g_{6}^{2} = g_{7} t_{12}$,
$[g_{6}, g_{1}] = g_{7} t_{13}$,
$[g_{6}, g_{3}] = g_{7} t_{14}$,
$g_{7}^{2} =  t_{15}$.
Carrying out consistency checks gives the following relations between the tails:
\[
\begin{aligned}
g_{4}(g_{2} g_{1}) & = (g_{4} g_{2}) g_{1}  & \Longrightarrow & & t_{9}t_{11}t_{15} & = 1 \\
g_{6}^2 g_{3} & = g_{6} (g_{6} g_{3})& \Longrightarrow & & t_{14}^{2}t_{15} & = 1 \\
g_{6}^2 g_{1} & = g_{6} (g_{6} g_{1})& \Longrightarrow & & t_{13}^{2}t_{15} & = 1 \\
g_{5}^2 g_{4} & = g_{5} (g_{5} g_{4})& \Longrightarrow & & t_{11}^{2}t_{15} & = 1 \\
g_{5}^2 g_{2} & = g_{5} (g_{5} g_{2})& \Longrightarrow & & t_{10}^{2}t_{15} & = 1 \\
g_{4}^2 g_{2} & = g_{4} (g_{4} g_{2})& \Longrightarrow & & t_{7}^{2}t_{8}t_{11}t_{15} & = 1 \\
g_{3}^2 g_{1} & = g_{3} (g_{3} g_{1})& \Longrightarrow & & t_{5}^{2}t_{12}t_{14}t_{15} & = 1 \\
g_{2}^2 g_{1} & = g_{2} (g_{2} g_{1})& \Longrightarrow & & t_{3}^{2}t_{8}t_{10}t_{15} & = 1 \\
g_{4} g_{2}^{2} & = (g_{4} g_{2}) g_{2}& \Longrightarrow & & t_{7}^{2}t_{8}t_{10}t_{15} & = 1 \\
g_{3} g_{1}^{2} & = (g_{3} g_{1}) g_{1}& \Longrightarrow & & t_{5}^{2}t_{12}t_{13}t_{15} & = 1 \\
\end{aligned}
\]
Scanning through the conjugacy class representatives of $G$ and the generators of their centralizers, we obtain the following relations induced on the tails:
\[
\begin{aligned}
{[g_{5} g_{6} g_{7} , \, g_{1} ]}_G & = 1 & \Longrightarrow & & t_{9}t_{13}t_{15} & = 1 \\
{[g_{2} g_{5} g_{7} , \, g_{1} g_{3} g_{4} g_{5} g_{7} ]}_G & = 1 & \Longrightarrow & & t_{3}t_{7}^{-1}t_{9}t_{10}^{-1} & = 1 \\
\end{aligned}
\]
Collecting the coefficients of these relations into a matrix yields
\[
T = \bordermatrix{
{} & t_{1} & t_{2} & t_{3} & t_{4} & t_{5} & t_{6} & t_{7} & t_{8} & t_{9} & t_{10} & t_{11} & t_{12} & t_{13} & t_{14} & t_{15} \cr
{} &  &  & 1 &  &  &  & 1 & 1 &  &  &  &  &  & 1 & 1 \cr
{} &  &  &  &  & 2 &  &  &  &  &  &  & 1 &  & 1 & 1 \cr
{} &  &  &  &  &  &  & 2 & 1 &  &  &  &  &  & 1 & 1 \cr
{} &  &  &  &  &  &  &  &  & 1 &  &  &  &  & 1 & 1 \cr
{} &  &  &  &  &  &  &  &  &  & 1 &  &  &  & 1 & 1 \cr
{} &  &  &  &  &  &  &  &  &  &  & 1 &  &  & 1 & 1 \cr
{} &  &  &  &  &  &  &  &  &  &  &  &  & 1 & 1 & 1 \cr
{} &  &  &  &  &  &  &  &  &  &  &  &  &  & 2 & 1 \cr
}.
\]
It follows readily that the nontrivial elementary divisors of the Smith normal form of $T$ are all equal to $1$. The torsion subgroup of the group generated by the tails is thus trivial, thereby showing $\B_0(G) = 1$.


\item \label{number:68} 
Let the group $G$ be the representative of this family given by the presentation
\[
\begin{aligned}
\langle g_{1}, \,g_{2}, \,g_{3}, \,g_{4}, \,g_{5}, \,g_{6}, \,g_{7} & \mid & g_{1}^{2} &= 1, \\ 
 & & g_{2}^{2} &= 1, & [g_{2}, g_{1}]  &= g_{5}, \\ 
 & & g_{3}^{2} &= 1, & [g_{3}, g_{1}]  &= g_{6}, & [g_{3}, g_{2}]  &= g_{7}, \\ 
 & & g_{4}^{2} &= 1, & [g_{4}, g_{2}]  &= g_{5}, \\ 
 & & g_{5}^{2} &= g_{7}, & [g_{5}, g_{1}]  &= g_{7}, & [g_{5}, g_{2}]  &= g_{7}, & [g_{5}, g_{4}]  &= g_{7}, \\ 
 & & g_{6}^{2} &= g_{7}, & [g_{6}, g_{1}]  &= g_{7}, & [g_{6}, g_{3}]  &= g_{7}, \\ 
 & & g_{7}^{2} &= 1\rangle. \\ 
\end{aligned}
\]
We add 16 tails to the presentation as to form a quotient of the universal central extension of the system: 
$g_{1}^{2} =  t_{1}$,
$g_{2}^{2} =  t_{2}$,
$[g_{2}, g_{1}] = g_{5} t_{3}$,
$g_{3}^{2} =  t_{4}$,
$[g_{3}, g_{1}] = g_{6} t_{5}$,
$[g_{3}, g_{2}] = g_{7} t_{6}$,
$g_{4}^{2} =  t_{7}$,
$[g_{4}, g_{2}] = g_{5} t_{8}$,
$g_{5}^{2} = g_{7} t_{9}$,
$[g_{5}, g_{1}] = g_{7} t_{10}$,
$[g_{5}, g_{2}] = g_{7} t_{11}$,
$[g_{5}, g_{4}] = g_{7} t_{12}$,
$g_{6}^{2} = g_{7} t_{13}$,
$[g_{6}, g_{1}] = g_{7} t_{14}$,
$[g_{6}, g_{3}] = g_{7} t_{15}$,
$g_{7}^{2} =  t_{16}$.
Carrying out consistency checks gives the following relations between the tails:
\[
\begin{aligned}
g_{4}(g_{2} g_{1}) & = (g_{4} g_{2}) g_{1}  & \Longrightarrow & & t_{10}t_{12}t_{16} & = 1 \\
g_{6}^2 g_{3} & = g_{6} (g_{6} g_{3})& \Longrightarrow & & t_{15}^{2}t_{16} & = 1 \\
g_{6}^2 g_{1} & = g_{6} (g_{6} g_{1})& \Longrightarrow & & t_{14}^{2}t_{16} & = 1 \\
g_{5}^2 g_{4} & = g_{5} (g_{5} g_{4})& \Longrightarrow & & t_{12}^{2}t_{16} & = 1 \\
g_{5}^2 g_{2} & = g_{5} (g_{5} g_{2})& \Longrightarrow & & t_{11}^{2}t_{16} & = 1 \\
g_{4}^2 g_{2} & = g_{4} (g_{4} g_{2})& \Longrightarrow & & t_{8}^{2}t_{9}t_{12}t_{16} & = 1 \\
g_{3}^2 g_{2} & = g_{3} (g_{3} g_{2})& \Longrightarrow & & t_{6}^{2}t_{16} & = 1 \\
g_{3}^2 g_{1} & = g_{3} (g_{3} g_{1})& \Longrightarrow & & t_{5}^{2}t_{13}t_{15}t_{16} & = 1 \\
g_{2}^2 g_{1} & = g_{2} (g_{2} g_{1})& \Longrightarrow & & t_{3}^{2}t_{9}t_{11}t_{16} & = 1 \\
g_{4} g_{2}^{2} & = (g_{4} g_{2}) g_{2}& \Longrightarrow & & t_{8}^{2}t_{9}t_{11}t_{16} & = 1 \\
g_{3} g_{1}^{2} & = (g_{3} g_{1}) g_{1}& \Longrightarrow & & t_{5}^{2}t_{13}t_{14}t_{16} & = 1 \\
\end{aligned}
\]
Scanning through the conjugacy class representatives of $G$ and the generators of their centralizers, we obtain the following relations induced on the tails:
\[
\begin{aligned}
{[g_{5} g_{6} g_{7} , \, g_{1} ]}_G & = 1 & \Longrightarrow & & t_{10}t_{14}t_{16} & = 1 \\
{[g_{3} g_{6} g_{7} , \, g_{2} g_{4} g_{6} g_{7} ]}_G & = 1 & \Longrightarrow & & t_{6}t_{15}^{-1} & = 1 \\
{[g_{2} g_{5} g_{7} , \, g_{1} g_{3} g_{4} ]}_G & = 1 & \Longrightarrow & & t_{3}t_{6}^{-1}t_{8}^{-1}t_{10} & = 1 \\
\end{aligned}
\]
Collecting the coefficients of these relations into a matrix yields
\[
T = \bordermatrix{
{} & t_{1} & t_{2} & t_{3} & t_{4} & t_{5} & t_{6} & t_{7} & t_{8} & t_{9} & t_{10} & t_{11} & t_{12} & t_{13} & t_{14} & t_{15} & t_{16} \cr
{} &  &  & 1 &  &  &  &  & 1 & 1 &  &  &  &  &  & 1 & 1 \cr
{} &  &  &  &  & 2 &  &  &  &  &  &  &  & 1 &  & 1 & 1 \cr
{} &  &  &  &  &  & 1 &  &  &  &  &  &  &  &  & 1 & 1 \cr
{} &  &  &  &  &  &  &  & 2 & 1 &  &  &  &  &  & 1 & 1 \cr
{} &  &  &  &  &  &  &  &  &  & 1 &  &  &  &  & 1 & 1 \cr
{} &  &  &  &  &  &  &  &  &  &  & 1 &  &  &  & 1 & 1 \cr
{} &  &  &  &  &  &  &  &  &  &  &  & 1 &  &  & 1 & 1 \cr
{} &  &  &  &  &  &  &  &  &  &  &  &  &  & 1 & 1 & 1 \cr
{} &  &  &  &  &  &  &  &  &  &  &  &  &  &  & 2 & 1 \cr
}.
\]
It follows readily that the nontrivial elementary divisors of the Smith normal form of $T$ are all equal to $1$. The torsion subgroup of the group generated by the tails is thus trivial, thereby showing $\B_0(G) = 1$.


\item \label{number:69} 
Let the group $G$ be the representative of this family given by the presentation
\[
\begin{aligned}
\langle g_{1}, \,g_{2}, \,g_{3}, \,g_{4}, \,g_{5}, \,g_{6}, \,g_{7} & \mid & g_{1}^{2} &= 1, \\ 
 & & g_{2}^{2} &= 1, & [g_{2}, g_{1}]  &= g_{5}, \\ 
 & & g_{3}^{2} &= 1, & [g_{3}, g_{1}]  &= g_{6}, \\ 
 & & g_{4}^{2} &= 1, & [g_{4}, g_{1}]  &= g_{7}, \\ 
 & & g_{5}^{2} &= g_{7}, & [g_{5}, g_{1}]  &= g_{7}, & [g_{5}, g_{2}]  &= g_{7}, & [g_{5}, g_{3}]  &= g_{7}, \\ 
 & & g_{6}^{2} &= 1, & [g_{6}, g_{2}]  &= g_{7}, \\ 
 & & g_{7}^{2} &= 1\rangle. \\ 
\end{aligned}
\]
We add 14 tails to the presentation as to form a quotient of the universal central extension of the system: 
$g_{1}^{2} =  t_{1}$,
$g_{2}^{2} =  t_{2}$,
$[g_{2}, g_{1}] = g_{5} t_{3}$,
$g_{3}^{2} =  t_{4}$,
$[g_{3}, g_{1}] = g_{6} t_{5}$,
$g_{4}^{2} =  t_{6}$,
$[g_{4}, g_{1}] = g_{7} t_{7}$,
$g_{5}^{2} = g_{7} t_{8}$,
$[g_{5}, g_{1}] = g_{7} t_{9}$,
$[g_{5}, g_{2}] = g_{7} t_{10}$,
$[g_{5}, g_{3}] = g_{7} t_{11}$,
$g_{6}^{2} =  t_{12}$,
$[g_{6}, g_{2}] = g_{7} t_{13}$,
$g_{7}^{2} =  t_{14}$.
Carrying out consistency checks gives the following relations between the tails:
\[
\begin{aligned}
g_{3}(g_{2} g_{1}) & = (g_{3} g_{2}) g_{1}  & \Longrightarrow & & t_{11}t_{13}^{-1} & = 1 \\
g_{6}^2 g_{2} & = g_{6} (g_{6} g_{2})& \Longrightarrow & & t_{13}^{2}t_{14} & = 1 \\
g_{5}^2 g_{2} & = g_{5} (g_{5} g_{2})& \Longrightarrow & & t_{10}^{2}t_{14} & = 1 \\
g_{5}^2 g_{1} & = g_{5} (g_{5} g_{1})& \Longrightarrow & & t_{9}^{2}t_{14} & = 1 \\
g_{4}^2 g_{1} & = g_{4} (g_{4} g_{1})& \Longrightarrow & & t_{7}^{2}t_{14} & = 1 \\
g_{3}^2 g_{1} & = g_{3} (g_{3} g_{1})& \Longrightarrow & & t_{5}^{2}t_{12} & = 1 \\
g_{2}^2 g_{1} & = g_{2} (g_{2} g_{1})& \Longrightarrow & & t_{3}^{2}t_{8}t_{10}t_{14} & = 1 \\
g_{2} g_{1}^{2} & = (g_{2} g_{1}) g_{1}& \Longrightarrow & & t_{3}^{2}t_{8}t_{9}t_{14} & = 1 \\
\end{aligned}
\]
Scanning through the conjugacy class representatives of $G$ and the generators of their centralizers, we obtain the following relations induced on the tails:
\[
\begin{aligned}
{[g_{5} g_{7} , \, g_{1} g_{3} ]}_G & = 1 & \Longrightarrow & & t_{9}t_{11}t_{14} & = 1 \\
{[g_{4} g_{6} g_{7} , \, g_{1} g_{2} g_{3} ]}_G & = 1 & \Longrightarrow & & t_{7}t_{13}t_{14} & = 1 \\
\end{aligned}
\]
Collecting the coefficients of these relations into a matrix yields
\[
T = \bordermatrix{
{} & t_{1} & t_{2} & t_{3} & t_{4} & t_{5} & t_{6} & t_{7} & t_{8} & t_{9} & t_{10} & t_{11} & t_{12} & t_{13} & t_{14} \cr
{} &  &  & 2 &  &  &  &  & 1 &  &  &  &  & 1 & 1 \cr
{} &  &  &  &  & 2 &  &  &  &  &  &  & 1 &  &  \cr
{} &  &  &  &  &  &  & 1 &  &  &  &  &  & 1 & 1 \cr
{} &  &  &  &  &  &  &  &  & 1 &  &  &  & 1 & 1 \cr
{} &  &  &  &  &  &  &  &  &  & 1 &  &  & 1 & 1 \cr
{} &  &  &  &  &  &  &  &  &  &  & 1 &  & 1 & 1 \cr
{} &  &  &  &  &  &  &  &  &  &  &  &  & 2 & 1 \cr
}.
\]
It follows readily that the nontrivial elementary divisors of the Smith normal form of $T$ are all equal to $1$. The torsion subgroup of the group generated by the tails is thus trivial, thereby showing $\B_0(G) = 1$.


\item \label{number:70} 
Let the group $G$ be the representative of this family given by the presentation
\[
\begin{aligned}
\langle g_{1}, \,g_{2}, \,g_{3}, \,g_{4}, \,g_{5}, \,g_{6}, \,g_{7} & \mid & g_{1}^{2} &= g_{6}, \\ 
 & & g_{2}^{2} &= g_{4}g_{6}, & [g_{2}, g_{1}]  &= g_{4}, \\ 
 & & g_{3}^{2} &= 1, & [g_{3}, g_{1}]  &= g_{5}, \\ 
 & & g_{4}^{2} &= g_{7}, & [g_{4}, g_{1}]  &= g_{7}, & [g_{4}, g_{3}]  &= g_{7}, \\ 
 & & g_{5}^{2} &= g_{7}, & [g_{5}, g_{2}]  &= g_{7}, & [g_{5}, g_{3}]  &= g_{7}, \\ 
 & & g_{6}^{2} &= 1, & [g_{6}, g_{3}]  &= g_{7}, \\ 
 & & g_{7}^{2} &= 1\rangle. \\ 
\end{aligned}
\]
We add 14 tails to the presentation as to form a quotient of the universal central extension of the system: 
$g_{1}^{2} = g_{6} t_{1}$,
$g_{2}^{2} = g_{4}g_{6} t_{2}$,
$[g_{2}, g_{1}] = g_{4} t_{3}$,
$g_{3}^{2} =  t_{4}$,
$[g_{3}, g_{1}] = g_{5} t_{5}$,
$g_{4}^{2} = g_{7} t_{6}$,
$[g_{4}, g_{1}] = g_{7} t_{7}$,
$[g_{4}, g_{3}] = g_{7} t_{8}$,
$g_{5}^{2} = g_{7} t_{9}$,
$[g_{5}, g_{2}] = g_{7} t_{10}$,
$[g_{5}, g_{3}] = g_{7} t_{11}$,
$g_{6}^{2} =  t_{12}$,
$[g_{6}, g_{3}] = g_{7} t_{13}$,
$g_{7}^{2} =  t_{14}$.
Carrying out consistency checks gives the following relations between the tails:
\[
\begin{aligned}
g_{3}(g_{2} g_{1}) & = (g_{3} g_{2}) g_{1}  & \Longrightarrow & & t_{8}t_{10}^{-1} & = 1 \\
g_{6}^2 g_{3} & = g_{6} (g_{6} g_{3})& \Longrightarrow & & t_{13}^{2}t_{14} & = 1 \\
g_{5}^2 g_{3} & = g_{5} (g_{5} g_{3})& \Longrightarrow & & t_{11}^{2}t_{14} & = 1 \\
g_{5}^2 g_{2} & = g_{5} (g_{5} g_{2})& \Longrightarrow & & t_{10}^{2}t_{14} & = 1 \\
g_{4}^2 g_{1} & = g_{4} (g_{4} g_{1})& \Longrightarrow & & t_{7}^{2}t_{14} & = 1 \\
g_{3}^2 g_{1} & = g_{3} (g_{3} g_{1})& \Longrightarrow & & t_{5}^{2}t_{9}t_{11}t_{14} & = 1 \\
g_{2}^2 g_{1} & = g_{2} (g_{2} g_{1})& \Longrightarrow & & t_{3}^{2}t_{6}t_{7}^{-1} & = 1 \\
g_{3} g_{2}^{2} & = (g_{3} g_{2}) g_{2}& \Longrightarrow & & t_{8}t_{13}t_{14} & = 1 \\
g_{3} g_{1}^{2} & = (g_{3} g_{1}) g_{1}& \Longrightarrow & & t_{5}^{2}t_{9}t_{13}t_{14} & = 1 \\
\end{aligned}
\]
Scanning through the conjugacy class representatives of $G$ and the generators of their centralizers, we obtain the following relations induced on the tails:
\[
\begin{aligned}
{[g_{4} g_{7} , \, g_{1} g_{3} ]}_G & = 1 & \Longrightarrow & & t_{7}t_{8}t_{14} & = 1 \\
\end{aligned}
\]
Collecting the coefficients of these relations into a matrix yields
\[
T = \bordermatrix{
{} & t_{1} & t_{2} & t_{3} & t_{4} & t_{5} & t_{6} & t_{7} & t_{8} & t_{9} & t_{10} & t_{11} & t_{12} & t_{13} & t_{14} \cr
{} &  &  & 2 &  &  & 1 &  &  &  &  &  &  & 1 & 1 \cr
{} &  &  &  &  & 2 &  &  &  & 1 &  &  &  & 1 & 1 \cr
{} &  &  &  &  &  &  & 1 &  &  &  &  &  & 1 & 1 \cr
{} &  &  &  &  &  &  &  & 1 &  &  &  &  & 1 & 1 \cr
{} &  &  &  &  &  &  &  &  &  & 1 &  &  & 1 & 1 \cr
{} &  &  &  &  &  &  &  &  &  &  & 1 &  & 1 & 1 \cr
{} &  &  &  &  &  &  &  &  &  &  &  &  & 2 & 1 \cr
}.
\]
It follows readily that the nontrivial elementary divisors of the Smith normal form of $T$ are all equal to $1$. The torsion subgroup of the group generated by the tails is thus trivial, thereby showing $\B_0(G) = 1$.


\item \label{number:71} 
Let the group $G$ be the representative of this family given by the presentation
\[
\begin{aligned}
\langle g_{1}, \,g_{2}, \,g_{3}, \,g_{4}, \,g_{5}, \,g_{6}, \,g_{7} & \mid & g_{1}^{2} &= g_{6}, \\ 
 & & g_{2}^{2} &= g_{4}g_{6}, & [g_{2}, g_{1}]  &= g_{4}, \\ 
 & & g_{3}^{2} &= 1, & [g_{3}, g_{1}]  &= g_{5}, \\ 
 & & g_{4}^{2} &= 1, & [g_{4}, g_{1}]  &= g_{7}, & [g_{4}, g_{2}]  &= g_{7}, & [g_{4}, g_{3}]  &= g_{7}, \\ 
 & & g_{5}^{2} &= 1, & [g_{5}, g_{1}]  &= g_{7}, & [g_{5}, g_{2}]  &= g_{7}, \\ 
 & & g_{6}^{2} &= 1, & [g_{6}, g_{2}]  &= g_{7}, & [g_{6}, g_{3}]  &= g_{7}, \\ 
 & & g_{7}^{2} &= 1\rangle. \\ 
\end{aligned}
\]
We add 16 tails to the presentation as to form a quotient of the universal central extension of the system: 
$g_{1}^{2} = g_{6} t_{1}$,
$g_{2}^{2} = g_{4}g_{6} t_{2}$,
$[g_{2}, g_{1}] = g_{4} t_{3}$,
$g_{3}^{2} =  t_{4}$,
$[g_{3}, g_{1}] = g_{5} t_{5}$,
$g_{4}^{2} =  t_{6}$,
$[g_{4}, g_{1}] = g_{7} t_{7}$,
$[g_{4}, g_{2}] = g_{7} t_{8}$,
$[g_{4}, g_{3}] = g_{7} t_{9}$,
$g_{5}^{2} =  t_{10}$,
$[g_{5}, g_{1}] = g_{7} t_{11}$,
$[g_{5}, g_{2}] = g_{7} t_{12}$,
$g_{6}^{2} =  t_{13}$,
$[g_{6}, g_{2}] = g_{7} t_{14}$,
$[g_{6}, g_{3}] = g_{7} t_{15}$,
$g_{7}^{2} =  t_{16}$.
Carrying out consistency checks gives the following relations between the tails:
\[
\begin{aligned}
g_{3}(g_{2} g_{1}) & = (g_{3} g_{2}) g_{1}  & \Longrightarrow & & t_{9}t_{12}^{-1} & = 1 \\
g_{6}^2 g_{3} & = g_{6} (g_{6} g_{3})& \Longrightarrow & & t_{15}^{2}t_{16} & = 1 \\
g_{6}^2 g_{2} & = g_{6} (g_{6} g_{2})& \Longrightarrow & & t_{14}^{2}t_{16} & = 1 \\
g_{5}^2 g_{2} & = g_{5} (g_{5} g_{2})& \Longrightarrow & & t_{12}^{2}t_{16} & = 1 \\
g_{5}^2 g_{1} & = g_{5} (g_{5} g_{1})& \Longrightarrow & & t_{11}^{2}t_{16} & = 1 \\
g_{4}^2 g_{2} & = g_{4} (g_{4} g_{2})& \Longrightarrow & & t_{8}^{2}t_{16} & = 1 \\
g_{4}^2 g_{1} & = g_{4} (g_{4} g_{1})& \Longrightarrow & & t_{7}^{2}t_{16} & = 1 \\
g_{3}^2 g_{1} & = g_{3} (g_{3} g_{1})& \Longrightarrow & & t_{5}^{2}t_{10} & = 1 \\
g_{2}^2 g_{1} & = g_{2} (g_{2} g_{1})& \Longrightarrow & & t_{3}^{2}t_{6}t_{7}^{-1}t_{8} & = 1 \\
g_{3} g_{2}^{2} & = (g_{3} g_{2}) g_{2}& \Longrightarrow & & t_{9}t_{15}t_{16} & = 1 \\
g_{3} g_{1}^{2} & = (g_{3} g_{1}) g_{1}& \Longrightarrow & & t_{5}^{2}t_{10}t_{11}t_{15}t_{16} & = 1 \\
g_{2} g_{1}^{2} & = (g_{2} g_{1}) g_{1}& \Longrightarrow & & t_{3}^{2}t_{6}t_{7}t_{14}t_{16} & = 1 \\
\end{aligned}
\]
Scanning through the conjugacy class representatives of $G$ and the generators of their centralizers, we obtain the following relations induced on the tails:
\[
\begin{aligned}
{[g_{6} g_{7} , \, g_{2} g_{3} g_{4} ]}_G & = 1 & \Longrightarrow & & t_{14}t_{15}t_{16} & = 1 \\
{[g_{4} g_{7} , \, g_{1} g_{3} ]}_G & = 1 & \Longrightarrow & & t_{7}t_{9}t_{16} & = 1 \\
\end{aligned}
\]
Collecting the coefficients of these relations into a matrix yields
\[
T = \bordermatrix{
{} & t_{1} & t_{2} & t_{3} & t_{4} & t_{5} & t_{6} & t_{7} & t_{8} & t_{9} & t_{10} & t_{11} & t_{12} & t_{13} & t_{14} & t_{15} & t_{16} \cr
{} &  &  & 2 &  &  & 1 &  &  &  &  &  &  &  &  &  &  \cr
{} &  &  &  &  & 2 &  &  &  &  & 1 &  &  &  &  &  &  \cr
{} &  &  &  &  &  &  & 1 &  &  &  &  &  &  &  & 1 & 1 \cr
{} &  &  &  &  &  &  &  & 1 &  &  &  &  &  &  & 1 & 1 \cr
{} &  &  &  &  &  &  &  &  & 1 &  &  &  &  &  & 1 & 1 \cr
{} &  &  &  &  &  &  &  &  &  &  & 1 &  &  &  & 1 & 1 \cr
{} &  &  &  &  &  &  &  &  &  &  &  & 1 &  &  & 1 & 1 \cr
{} &  &  &  &  &  &  &  &  &  &  &  &  &  & 1 & 1 & 1 \cr
{} &  &  &  &  &  &  &  &  &  &  &  &  &  &  & 2 & 1 \cr
}.
\]
It follows readily that the nontrivial elementary divisors of the Smith normal form of $T$ are all equal to $1$. The torsion subgroup of the group generated by the tails is thus trivial, thereby showing $\B_0(G) = 1$.


\item \label{number:72} 
Let the group $G$ be the representative of this family given by the presentation
\[
\begin{aligned}
\langle g_{1}, \,g_{2}, \,g_{3}, \,g_{4}, \,g_{5}, \,g_{6}, \,g_{7} & \mid & g_{1}^{2} &= 1, \\ 
 & & g_{2}^{2} &= 1, & [g_{2}, g_{1}]  &= g_{5}, \\ 
 & & g_{3}^{2} &= 1, & [g_{3}, g_{1}]  &= g_{6}, \\ 
 & & g_{4}^{2} &= 1, & [g_{4}, g_{1}]  &= g_{7}, \\ 
 & & g_{5}^{2} &= 1, & [g_{5}, g_{3}]  &= g_{7}, \\ 
 & & g_{6}^{2} &= 1, & [g_{6}, g_{2}]  &= g_{7}, \\ 
 & & g_{7}^{2} &= 1\rangle. \\ 
\end{aligned}
\]
We add 12 tails to the presentation as to form a quotient of the universal central extension of the system: 
$g_{1}^{2} =  t_{1}$,
$g_{2}^{2} =  t_{2}$,
$[g_{2}, g_{1}] = g_{5} t_{3}$,
$g_{3}^{2} =  t_{4}$,
$[g_{3}, g_{1}] = g_{6} t_{5}$,
$g_{4}^{2} =  t_{6}$,
$[g_{4}, g_{1}] = g_{7} t_{7}$,
$g_{5}^{2} =  t_{8}$,
$[g_{5}, g_{3}] = g_{7} t_{9}$,
$g_{6}^{2} =  t_{10}$,
$[g_{6}, g_{2}] = g_{7} t_{11}$,
$g_{7}^{2} =  t_{12}$.
Carrying out consistency checks gives the following relations between the tails:
\[
\begin{aligned}
g_{3}(g_{2} g_{1}) & = (g_{3} g_{2}) g_{1}  & \Longrightarrow & & t_{9}t_{11}^{-1} & = 1 \\
g_{6}^2 g_{2} & = g_{6} (g_{6} g_{2})& \Longrightarrow & & t_{11}^{2}t_{12} & = 1 \\
g_{4}^2 g_{1} & = g_{4} (g_{4} g_{1})& \Longrightarrow & & t_{7}^{2}t_{12} & = 1 \\
g_{3}^2 g_{1} & = g_{3} (g_{3} g_{1})& \Longrightarrow & & t_{5}^{2}t_{10} & = 1 \\
g_{2}^2 g_{1} & = g_{2} (g_{2} g_{1})& \Longrightarrow & & t_{3}^{2}t_{8} & = 1 \\
\end{aligned}
\]
Scanning through the conjugacy class representatives of $G$ and the generators of their centralizers, we obtain the following relations induced on the tails:
\[
\begin{aligned}
{[g_{4} g_{6} g_{7} , \, g_{1} g_{2} g_{3} ]}_G & = 1 & \Longrightarrow & & t_{7}t_{11}t_{12} & = 1 \\
\end{aligned}
\]
Collecting the coefficients of these relations into a matrix yields
\[
T = \bordermatrix{
{} & t_{1} & t_{2} & t_{3} & t_{4} & t_{5} & t_{6} & t_{7} & t_{8} & t_{9} & t_{10} & t_{11} & t_{12} \cr
{} &  &  & 2 &  &  &  &  & 1 &  &  &  &  \cr
{} &  &  &  &  & 2 &  &  &  &  & 1 &  &  \cr
{} &  &  &  &  &  &  & 1 &  &  &  & 1 & 1 \cr
{} &  &  &  &  &  &  &  &  & 1 &  & 1 & 1 \cr
{} &  &  &  &  &  &  &  &  &  &  & 2 & 1 \cr
}.
\]
It follows readily that the nontrivial elementary divisors of the Smith normal form of $T$ are all equal to $1$. The torsion subgroup of the group generated by the tails is thus trivial, thereby showing $\B_0(G) = 1$.


\item \label{number:73} 
Let the group $G$ be the representative of this family given by the presentation
\[
\begin{aligned}
\langle g_{1}, \,g_{2}, \,g_{3}, \,g_{4}, \,g_{5}, \,g_{6}, \,g_{7} & \mid & g_{1}^{2} &= g_{6}, \\ 
 & & g_{2}^{2} &= g_{4}g_{6}, & [g_{2}, g_{1}]  &= g_{4}, \\ 
 & & g_{3}^{2} &= 1, & [g_{3}, g_{1}]  &= g_{5}, \\ 
 & & g_{4}^{2} &= g_{7}, & [g_{4}, g_{1}]  &= g_{7}, & [g_{4}, g_{3}]  &= g_{7}, \\ 
 & & g_{5}^{2} &= 1, & [g_{5}, g_{1}]  &= g_{7}, & [g_{5}, g_{2}]  &= g_{7}, \\ 
 & & g_{6}^{2} &= 1, & [g_{6}, g_{3}]  &= g_{7}, \\ 
 & & g_{7}^{2} &= 1\rangle. \\ 
\end{aligned}
\]
We add 14 tails to the presentation as to form a quotient of the universal central extension of the system: 
$g_{1}^{2} = g_{6} t_{1}$,
$g_{2}^{2} = g_{4}g_{6} t_{2}$,
$[g_{2}, g_{1}] = g_{4} t_{3}$,
$g_{3}^{2} =  t_{4}$,
$[g_{3}, g_{1}] = g_{5} t_{5}$,
$g_{4}^{2} = g_{7} t_{6}$,
$[g_{4}, g_{1}] = g_{7} t_{7}$,
$[g_{4}, g_{3}] = g_{7} t_{8}$,
$g_{5}^{2} =  t_{9}$,
$[g_{5}, g_{1}] = g_{7} t_{10}$,
$[g_{5}, g_{2}] = g_{7} t_{11}$,
$g_{6}^{2} =  t_{12}$,
$[g_{6}, g_{3}] = g_{7} t_{13}$,
$g_{7}^{2} =  t_{14}$.
Carrying out consistency checks gives the following relations between the tails:
\[
\begin{aligned}
g_{3}(g_{2} g_{1}) & = (g_{3} g_{2}) g_{1}  & \Longrightarrow & & t_{8}t_{11}^{-1} & = 1 \\
g_{6}^2 g_{3} & = g_{6} (g_{6} g_{3})& \Longrightarrow & & t_{13}^{2}t_{14} & = 1 \\
g_{5}^2 g_{2} & = g_{5} (g_{5} g_{2})& \Longrightarrow & & t_{11}^{2}t_{14} & = 1 \\
g_{5}^2 g_{1} & = g_{5} (g_{5} g_{1})& \Longrightarrow & & t_{10}^{2}t_{14} & = 1 \\
g_{4}^2 g_{1} & = g_{4} (g_{4} g_{1})& \Longrightarrow & & t_{7}^{2}t_{14} & = 1 \\
g_{3}^2 g_{1} & = g_{3} (g_{3} g_{1})& \Longrightarrow & & t_{5}^{2}t_{9} & = 1 \\
g_{2}^2 g_{1} & = g_{2} (g_{2} g_{1})& \Longrightarrow & & t_{3}^{2}t_{6}t_{7}^{-1} & = 1 \\
g_{3} g_{2}^{2} & = (g_{3} g_{2}) g_{2}& \Longrightarrow & & t_{8}t_{13}t_{14} & = 1 \\
g_{3} g_{1}^{2} & = (g_{3} g_{1}) g_{1}& \Longrightarrow & & t_{5}^{2}t_{9}t_{10}t_{13}t_{14} & = 1 \\
\end{aligned}
\]
Scanning through the conjugacy class representatives of $G$ and the generators of their centralizers, we obtain the following relations induced on the tails:
\[
\begin{aligned}
{[g_{4} g_{7} , \, g_{1} g_{3} ]}_G & = 1 & \Longrightarrow & & t_{7}t_{8}t_{14} & = 1 \\
\end{aligned}
\]
Collecting the coefficients of these relations into a matrix yields
\[
T = \bordermatrix{
{} & t_{1} & t_{2} & t_{3} & t_{4} & t_{5} & t_{6} & t_{7} & t_{8} & t_{9} & t_{10} & t_{11} & t_{12} & t_{13} & t_{14} \cr
{} &  &  & 2 &  &  & 1 &  &  &  &  &  &  & 1 & 1 \cr
{} &  &  &  &  & 2 &  &  &  & 1 &  &  &  &  &  \cr
{} &  &  &  &  &  &  & 1 &  &  &  &  &  & 1 & 1 \cr
{} &  &  &  &  &  &  &  & 1 &  &  &  &  & 1 & 1 \cr
{} &  &  &  &  &  &  &  &  &  & 1 &  &  & 1 & 1 \cr
{} &  &  &  &  &  &  &  &  &  &  & 1 &  & 1 & 1 \cr
{} &  &  &  &  &  &  &  &  &  &  &  &  & 2 & 1 \cr
}.
\]
It follows readily that the nontrivial elementary divisors of the Smith normal form of $T$ are all equal to $1$. The torsion subgroup of the group generated by the tails is thus trivial, thereby showing $\B_0(G) = 1$.


\item \label{number:74} 
Let the group $G$ be the representative of this family given by the presentation
\[
\begin{aligned}
\langle g_{1}, \,g_{2}, \,g_{3}, \,g_{4}, \,g_{5}, \,g_{6}, \,g_{7} & \mid & g_{1}^{2} &= 1, \\ 
 & & g_{2}^{2} &= g_{5}, & [g_{2}, g_{1}]  &= g_{5}, \\ 
 & & g_{3}^{2} &= 1, & [g_{3}, g_{1}]  &= g_{6}, \\ 
 & & g_{4}^{2} &= 1, & [g_{4}, g_{2}]  &= g_{7}, \\ 
 & & g_{5}^{2} &= g_{7}, & [g_{5}, g_{1}]  &= g_{7}, \\ 
 & & g_{6}^{2} &= g_{7}, & [g_{6}, g_{1}]  &= g_{7}, & [g_{6}, g_{3}]  &= g_{7}, \\ 
 & & g_{7}^{2} &= 1\rangle. \\ 
\end{aligned}
\]
We add 13 tails to the presentation as to form a quotient of the universal central extension of the system: 
$g_{1}^{2} =  t_{1}$,
$g_{2}^{2} = g_{5} t_{2}$,
$[g_{2}, g_{1}] = g_{5} t_{3}$,
$g_{3}^{2} =  t_{4}$,
$[g_{3}, g_{1}] = g_{6} t_{5}$,
$g_{4}^{2} =  t_{6}$,
$[g_{4}, g_{2}] = g_{7} t_{7}$,
$g_{5}^{2} = g_{7} t_{8}$,
$[g_{5}, g_{1}] = g_{7} t_{9}$,
$g_{6}^{2} = g_{7} t_{10}$,
$[g_{6}, g_{1}] = g_{7} t_{11}$,
$[g_{6}, g_{3}] = g_{7} t_{12}$,
$g_{7}^{2} =  t_{13}$.
Carrying out consistency checks gives the following relations between the tails:
\[
\begin{aligned}
g_{6}^2 g_{3} & = g_{6} (g_{6} g_{3})& \Longrightarrow & & t_{12}^{2}t_{13} & = 1 \\
g_{6}^2 g_{1} & = g_{6} (g_{6} g_{1})& \Longrightarrow & & t_{11}^{2}t_{13} & = 1 \\
g_{5}^2 g_{1} & = g_{5} (g_{5} g_{1})& \Longrightarrow & & t_{9}^{2}t_{13} & = 1 \\
g_{4}^2 g_{2} & = g_{4} (g_{4} g_{2})& \Longrightarrow & & t_{7}^{2}t_{13} & = 1 \\
g_{3}^2 g_{1} & = g_{3} (g_{3} g_{1})& \Longrightarrow & & t_{5}^{2}t_{10}t_{12}t_{13} & = 1 \\
g_{2}^2 g_{1} & = g_{2} (g_{2} g_{1})& \Longrightarrow & & t_{3}^{2}t_{8}t_{9}^{-1} & = 1 \\
g_{3} g_{1}^{2} & = (g_{3} g_{1}) g_{1}& \Longrightarrow & & t_{5}^{2}t_{10}t_{11}t_{13} & = 1 \\
\end{aligned}
\]
Scanning through the conjugacy class representatives of $G$ and the generators of their centralizers, we obtain the following relations induced on the tails:
\[
\begin{aligned}
{[g_{5} g_{6} g_{7} , \, g_{1} ]}_G & = 1 & \Longrightarrow & & t_{9}t_{11}t_{13} & = 1 \\
{[g_{4} g_{6} g_{7} , \, g_{2} g_{3} g_{4} ]}_G & = 1 & \Longrightarrow & & t_{7}t_{12}t_{13} & = 1 \\
\end{aligned}
\]
Collecting the coefficients of these relations into a matrix yields
\[
T = \bordermatrix{
{} & t_{1} & t_{2} & t_{3} & t_{4} & t_{5} & t_{6} & t_{7} & t_{8} & t_{9} & t_{10} & t_{11} & t_{12} & t_{13} \cr
{} &  &  & 2 &  &  &  &  & 1 &  &  &  & 1 & 1 \cr
{} &  &  &  &  & 2 &  &  &  &  & 1 &  & 1 & 1 \cr
{} &  &  &  &  &  &  & 1 &  &  &  &  & 1 & 1 \cr
{} &  &  &  &  &  &  &  &  & 1 &  &  & 1 & 1 \cr
{} &  &  &  &  &  &  &  &  &  &  & 1 & 1 & 1 \cr
{} &  &  &  &  &  &  &  &  &  &  &  & 2 & 1 \cr
}.
\]
It follows readily that the nontrivial elementary divisors of the Smith normal form of $T$ are all equal to $1$. The torsion subgroup of the group generated by the tails is thus trivial, thereby showing $\B_0(G) = 1$.


\item \label{number:75} 
Let the group $G$ be the representative of this family given by the presentation
\[
\begin{aligned}
\langle g_{1}, \,g_{2}, \,g_{3}, \,g_{4}, \,g_{5}, \,g_{6}, \,g_{7} & \mid & g_{1}^{2} &= g_{6}, \\ 
 & & g_{2}^{2} &= 1, & [g_{2}, g_{1}]  &= g_{4}, \\ 
 & & g_{3}^{2} &= 1, & [g_{3}, g_{1}]  &= g_{5}, \\ 
 & & g_{4}^{2} &= g_{7}, & [g_{4}, g_{1}]  &= g_{7}, & [g_{4}, g_{2}]  &= g_{7}, \\ 
 & & g_{5}^{2} &= 1, & [g_{5}, g_{1}]  &= g_{7}, \\ 
 & & g_{6}^{2} &= 1, & [g_{6}, g_{3}]  &= g_{7}, \\ 
 & & g_{7}^{2} &= 1\rangle. \\ 
\end{aligned}
\]
We add 13 tails to the presentation as to form a quotient of the universal central extension of the system: 
$g_{1}^{2} = g_{6} t_{1}$,
$g_{2}^{2} =  t_{2}$,
$[g_{2}, g_{1}] = g_{4} t_{3}$,
$g_{3}^{2} =  t_{4}$,
$[g_{3}, g_{1}] = g_{5} t_{5}$,
$g_{4}^{2} = g_{7} t_{6}$,
$[g_{4}, g_{1}] = g_{7} t_{7}$,
$[g_{4}, g_{2}] = g_{7} t_{8}$,
$g_{5}^{2} =  t_{9}$,
$[g_{5}, g_{1}] = g_{7} t_{10}$,
$g_{6}^{2} =  t_{11}$,
$[g_{6}, g_{3}] = g_{7} t_{12}$,
$g_{7}^{2} =  t_{13}$.
Carrying out consistency checks gives the following relations between the tails:
\[
\begin{aligned}
g_{6}^2 g_{3} & = g_{6} (g_{6} g_{3})& \Longrightarrow & & t_{12}^{2}t_{13} & = 1 \\
g_{5}^2 g_{1} & = g_{5} (g_{5} g_{1})& \Longrightarrow & & t_{10}^{2}t_{13} & = 1 \\
g_{4}^2 g_{2} & = g_{4} (g_{4} g_{2})& \Longrightarrow & & t_{8}^{2}t_{13} & = 1 \\
g_{4}^2 g_{1} & = g_{4} (g_{4} g_{1})& \Longrightarrow & & t_{7}^{2}t_{13} & = 1 \\
g_{3}^2 g_{1} & = g_{3} (g_{3} g_{1})& \Longrightarrow & & t_{5}^{2}t_{9} & = 1 \\
g_{2}^2 g_{1} & = g_{2} (g_{2} g_{1})& \Longrightarrow & & t_{3}^{2}t_{6}t_{8}t_{13} & = 1 \\
g_{3} g_{1}^{2} & = (g_{3} g_{1}) g_{1}& \Longrightarrow & & t_{5}^{2}t_{9}t_{10}t_{12}t_{13} & = 1 \\
g_{2} g_{1}^{2} & = (g_{2} g_{1}) g_{1}& \Longrightarrow & & t_{3}^{2}t_{6}t_{7}t_{13} & = 1 \\
\end{aligned}
\]
Scanning through the conjugacy class representatives of $G$ and the generators of their centralizers, we obtain the following relations induced on the tails:
\[
\begin{aligned}
{[g_{4} g_{6} g_{7} , \, g_{1} g_{3} ]}_G & = 1 & \Longrightarrow & & t_{7}t_{12}t_{13} & = 1 \\
\end{aligned}
\]
Collecting the coefficients of these relations into a matrix yields
\[
T = \bordermatrix{
{} & t_{1} & t_{2} & t_{3} & t_{4} & t_{5} & t_{6} & t_{7} & t_{8} & t_{9} & t_{10} & t_{11} & t_{12} & t_{13} \cr
{} &  &  & 2 &  &  & 1 &  &  &  &  &  & 1 & 1 \cr
{} &  &  &  &  & 2 &  &  &  & 1 &  &  &  &  \cr
{} &  &  &  &  &  &  & 1 &  &  &  &  & 1 & 1 \cr
{} &  &  &  &  &  &  &  & 1 &  &  &  & 1 & 1 \cr
{} &  &  &  &  &  &  &  &  &  & 1 &  & 1 & 1 \cr
{} &  &  &  &  &  &  &  &  &  &  &  & 2 & 1 \cr
}.
\]
It follows readily that the nontrivial elementary divisors of the Smith normal form of $T$ are all equal to $1$. The torsion subgroup of the group generated by the tails is thus trivial, thereby showing $\B_0(G) = 1$.


\item \label{number:76} 
Let the group $G$ be the representative of this family given by the presentation
\[
\begin{aligned}
\langle g_{1}, \,g_{2}, \,g_{3}, \,g_{4}, \,g_{5}, \,g_{6}, \,g_{7} & \mid & g_{1}^{2} &= g_{6}, \\ 
 & & g_{2}^{2} &= 1, & [g_{2}, g_{1}]  &= g_{4}, \\ 
 & & g_{3}^{2} &= g_{4}, & [g_{3}, g_{1}]  &= g_{5}, \\ 
 & & g_{4}^{2} &= 1, & [g_{4}, g_{1}]  &= g_{7}, \\ 
 & & g_{5}^{2} &= 1, & [g_{5}, g_{3}]  &= g_{7}, \\ 
 & & g_{6}^{2} &= 1, & [g_{6}, g_{2}]  &= g_{7}, \\ 
 & & g_{7}^{2} &= 1\rangle. \\ 
\end{aligned}
\]
We add 12 tails to the presentation as to form a quotient of the universal central extension of the system: 
$g_{1}^{2} = g_{6} t_{1}$,
$g_{2}^{2} =  t_{2}$,
$[g_{2}, g_{1}] = g_{4} t_{3}$,
$g_{3}^{2} = g_{4} t_{4}$,
$[g_{3}, g_{1}] = g_{5} t_{5}$,
$g_{4}^{2} =  t_{6}$,
$[g_{4}, g_{1}] = g_{7} t_{7}$,
$g_{5}^{2} =  t_{8}$,
$[g_{5}, g_{3}] = g_{7} t_{9}$,
$g_{6}^{2} =  t_{10}$,
$[g_{6}, g_{2}] = g_{7} t_{11}$,
$g_{7}^{2} =  t_{12}$.
Carrying out consistency checks gives the following relations between the tails:
\[
\begin{aligned}
g_{6}^2 g_{2} & = g_{6} (g_{6} g_{2})& \Longrightarrow & & t_{11}^{2}t_{12} & = 1 \\
g_{5}^2 g_{3} & = g_{5} (g_{5} g_{3})& \Longrightarrow & & t_{9}^{2}t_{12} & = 1 \\
g_{4}^2 g_{1} & = g_{4} (g_{4} g_{1})& \Longrightarrow & & t_{7}^{2}t_{12} & = 1 \\
g_{3}^2 g_{1} & = g_{3} (g_{3} g_{1})& \Longrightarrow & & t_{5}^{2}t_{7}^{-1}t_{8}t_{9} & = 1 \\
g_{2}^2 g_{1} & = g_{2} (g_{2} g_{1})& \Longrightarrow & & t_{3}^{2}t_{6} & = 1 \\
g_{3} g_{1}^{2} & = (g_{3} g_{1}) g_{1}& \Longrightarrow & & t_{5}^{2}t_{8} & = 1 \\
g_{2} g_{1}^{2} & = (g_{2} g_{1}) g_{1}& \Longrightarrow & & t_{3}^{2}t_{6}t_{7}t_{11}t_{12} & = 1 \\
\end{aligned}
\]
Scanning through the conjugacy class representatives of $G$ and the generators of their centralizers, we see that no new relations are imposed.
Collecting the coefficients of these relations into a matrix yields
\[
T = \bordermatrix{
{} & t_{1} & t_{2} & t_{3} & t_{4} & t_{5} & t_{6} & t_{7} & t_{8} & t_{9} & t_{10} & t_{11} & t_{12} \cr
{} &  &  & 2 &  &  & 1 &  &  &  &  &  &  \cr
{} &  &  &  &  & 2 &  &  & 1 &  &  &  &  \cr
{} &  &  &  &  &  &  & 1 &  &  &  & 1 & 1 \cr
{} &  &  &  &  &  &  &  &  & 1 &  & 1 & 1 \cr
{} &  &  &  &  &  &  &  &  &  &  & 2 & 1 \cr
}.
\]
It follows readily that the nontrivial elementary divisors of the Smith normal form of $T$ are all equal to $1$. The torsion subgroup of the group generated by the tails is thus trivial, thereby showing $\B_0(G) = 1$.


\item \label{number:77} 
Let the group $G$ be the representative of this family given by the presentation
\[
\begin{aligned}
\langle g_{1}, \,g_{2}, \,g_{3}, \,g_{4}, \,g_{5}, \,g_{6}, \,g_{7} & \mid & g_{1}^{2} &= g_{4}, \\ 
 & & g_{2}^{2} &= 1, & [g_{2}, g_{1}]  &= g_{3}, \\ 
 & & g_{3}^{2} &= g_{6}, & [g_{3}, g_{1}]  &= g_{5}, & [g_{3}, g_{2}]  &= g_{6}, \\ 
 & & g_{4}^{2} &= 1, & [g_{4}, g_{2}]  &= g_{5}g_{6}, & [g_{4}, g_{3}]  &= g_{7}, \\ 
 & & g_{5}^{2} &= 1, & [g_{5}, g_{1}]  &= g_{7}, \\ 
 & & g_{6}^{2} &= 1, \\ 
 & & g_{7}^{2} &= 1\rangle. \\ 
\end{aligned}
\]
We add 13 tails to the presentation as to form a quotient of the universal central extension of the system: 
$g_{1}^{2} = g_{4} t_{1}$,
$g_{2}^{2} =  t_{2}$,
$[g_{2}, g_{1}] = g_{3} t_{3}$,
$g_{3}^{2} = g_{6} t_{4}$,
$[g_{3}, g_{1}] = g_{5} t_{5}$,
$[g_{3}, g_{2}] = g_{6} t_{6}$,
$g_{4}^{2} =  t_{7}$,
$[g_{4}, g_{2}] = g_{5}g_{6} t_{8}$,
$[g_{4}, g_{3}] = g_{7} t_{9}$,
$g_{5}^{2} =  t_{10}$,
$[g_{5}, g_{1}] = g_{7} t_{11}$,
$g_{6}^{2} =  t_{12}$,
$g_{7}^{2} =  t_{13}$.
Carrying out consistency checks gives the following relations between the tails:
\[
\begin{aligned}
g_{4}(g_{2} g_{1}) & = (g_{4} g_{2}) g_{1}  & \Longrightarrow & & t_{9}^{-1}t_{11} & = 1 \\
g_{5}^2 g_{1} & = g_{5} (g_{5} g_{1})& \Longrightarrow & & t_{11}^{2}t_{13} & = 1 \\
g_{4}^2 g_{2} & = g_{4} (g_{4} g_{2})& \Longrightarrow & & t_{8}^{2}t_{10}t_{12} & = 1 \\
g_{3}^2 g_{2} & = g_{3} (g_{3} g_{2})& \Longrightarrow & & t_{6}^{2}t_{12} & = 1 \\
g_{3}^2 g_{1} & = g_{3} (g_{3} g_{1})& \Longrightarrow & & t_{5}^{2}t_{10} & = 1 \\
g_{2}^2 g_{1} & = g_{2} (g_{2} g_{1})& \Longrightarrow & & t_{3}^{2}t_{4}t_{6}t_{12} & = 1 \\
g_{2} g_{1}^{2} & = (g_{2} g_{1}) g_{1}& \Longrightarrow & & t_{3}^{2}t_{4}t_{5}t_{8}t_{9}^{2}t_{10}t_{12}t_{13} & = 1 \\
\end{aligned}
\]
Scanning through the conjugacy class representatives of $G$ and the generators of their centralizers, we see that no new relations are imposed.
Collecting the coefficients of these relations into a matrix yields
\[
T = \bordermatrix{
{} & t_{1} & t_{2} & t_{3} & t_{4} & t_{5} & t_{6} & t_{7} & t_{8} & t_{9} & t_{10} & t_{11} & t_{12} & t_{13} \cr
{} &  &  & 2 & 1 &  & 1 &  &  &  &  &  & 1 &  \cr
{} &  &  &  &  & 1 & 1 &  & 1 &  & 1 &  & 1 &  \cr
{} &  &  &  &  &  & 2 &  &  &  &  &  & 1 &  \cr
{} &  &  &  &  &  &  &  & 2 &  & 1 &  & 1 &  \cr
{} &  &  &  &  &  &  &  &  & 1 &  & 1 &  & 1 \cr
{} &  &  &  &  &  &  &  &  &  &  & 2 &  & 1 \cr
}.
\]
It follows readily that the nontrivial elementary divisors of the Smith normal form of $T$ are all equal to $1$. The torsion subgroup of the group generated by the tails is thus trivial, thereby showing $\B_0(G) = 1$.


\item \label{number:78} 
Let the group $G$ be the representative of this family given by the presentation
\[
\begin{aligned}
\langle g_{1}, \,g_{2}, \,g_{3}, \,g_{4}, \,g_{5}, \,g_{6}, \,g_{7} & \mid & g_{1}^{2} &= g_{4}, \\ 
 & & g_{2}^{2} &= 1, & [g_{2}, g_{1}]  &= g_{3}, \\ 
 & & g_{3}^{2} &= g_{6}g_{7}, & [g_{3}, g_{1}]  &= g_{5}, & [g_{3}, g_{2}]  &= g_{6}, \\ 
 & & g_{4}^{2} &= 1, & [g_{4}, g_{2}]  &= g_{5}g_{6}g_{7}, \\ 
 & & g_{5}^{2} &= g_{7}, & [g_{5}, g_{1}]  &= g_{7}, & [g_{5}, g_{2}]  &= g_{7}, \\ 
 & & g_{6}^{2} &= g_{7}, & [g_{6}, g_{1}]  &= g_{7}, & [g_{6}, g_{2}]  &= g_{7}, \\ 
 & & g_{7}^{2} &= 1\rangle. \\ 
\end{aligned}
\]
We add 15 tails to the presentation as to form a quotient of the universal central extension of the system: 
$g_{1}^{2} = g_{4} t_{1}$,
$g_{2}^{2} =  t_{2}$,
$[g_{2}, g_{1}] = g_{3} t_{3}$,
$g_{3}^{2} = g_{6}g_{7} t_{4}$,
$[g_{3}, g_{1}] = g_{5} t_{5}$,
$[g_{3}, g_{2}] = g_{6} t_{6}$,
$g_{4}^{2} =  t_{7}$,
$[g_{4}, g_{2}] = g_{5}g_{6}g_{7} t_{8}$,
$g_{5}^{2} = g_{7} t_{9}$,
$[g_{5}, g_{1}] = g_{7} t_{10}$,
$[g_{5}, g_{2}] = g_{7} t_{11}$,
$g_{6}^{2} = g_{7} t_{12}$,
$[g_{6}, g_{1}] = g_{7} t_{13}$,
$[g_{6}, g_{2}] = g_{7} t_{14}$,
$g_{7}^{2} =  t_{15}$.
Carrying out consistency checks gives the following relations between the tails:
\[
\begin{aligned}
g_{4}(g_{2} g_{1}) & = (g_{4} g_{2}) g_{1}  & \Longrightarrow & & t_{10}t_{13}t_{15} & = 1 \\
g_{3}(g_{2} g_{1}) & = (g_{3} g_{2}) g_{1}  & \Longrightarrow & & t_{11}^{-1}t_{13} & = 1 \\
g_{6}^2 g_{2} & = g_{6} (g_{6} g_{2})& \Longrightarrow & & t_{14}^{2}t_{15} & = 1 \\
g_{6}^2 g_{1} & = g_{6} (g_{6} g_{1})& \Longrightarrow & & t_{13}^{2}t_{15} & = 1 \\
g_{4}^2 g_{2} & = g_{4} (g_{4} g_{2})& \Longrightarrow & & t_{8}^{2}t_{9}t_{12}t_{15}^{2} & = 1 \\
g_{3}^2 g_{2} & = g_{3} (g_{3} g_{2})& \Longrightarrow & & t_{6}^{2}t_{12}t_{14}^{-1} & = 1 \\
g_{3}^2 g_{1} & = g_{3} (g_{3} g_{1})& \Longrightarrow & & t_{5}^{2}t_{9}t_{13}^{-1} & = 1 \\
g_{2}^2 g_{1} & = g_{2} (g_{2} g_{1})& \Longrightarrow & & t_{3}^{2}t_{4}t_{6}t_{12}t_{15} & = 1 \\
g_{4} g_{2}^{2} & = (g_{4} g_{2}) g_{2}& \Longrightarrow & & t_{8}^{2}t_{9}t_{11}t_{12}t_{14}t_{15}^{3} & = 1 \\
g_{2} g_{1}^{2} & = (g_{2} g_{1}) g_{1}& \Longrightarrow & & t_{3}^{2}t_{4}t_{5}t_{8}t_{9}t_{12}t_{15}^{2} & = 1 \\
\end{aligned}
\]
Scanning through the conjugacy class representatives of $G$ and the generators of their centralizers, we see that no new relations are imposed.
Collecting the coefficients of these relations into a matrix yields
\[
T = \bordermatrix{
{} & t_{1} & t_{2} & t_{3} & t_{4} & t_{5} & t_{6} & t_{7} & t_{8} & t_{9} & t_{10} & t_{11} & t_{12} & t_{13} & t_{14} & t_{15} \cr
{} &  &  & 2 & 1 &  & 1 &  &  &  &  &  & 1 &  &  & 1 \cr
{} &  &  &  &  & 1 & 1 &  & 1 & 1 &  &  & 1 &  & 1 & 2 \cr
{} &  &  &  &  &  & 2 &  &  &  &  &  & 1 &  & 1 & 1 \cr
{} &  &  &  &  &  &  &  & 2 & 1 &  &  & 1 &  &  & 2 \cr
{} &  &  &  &  &  &  &  &  &  & 1 &  &  &  & 1 & 1 \cr
{} &  &  &  &  &  &  &  &  &  &  & 1 &  &  & 1 & 1 \cr
{} &  &  &  &  &  &  &  &  &  &  &  &  & 1 & 1 & 1 \cr
{} &  &  &  &  &  &  &  &  &  &  &  &  &  & 2 & 1 \cr
}.
\]
It follows readily that the nontrivial elementary divisors of the Smith normal form of $T$ are all equal to $1$. The torsion subgroup of the group generated by the tails is thus trivial, thereby showing $\B_0(G) = 1$.


\item \label{number:79} 
Let the group $G$ be the representative of this family given by the presentation
\[
\begin{aligned}
\langle g_{1}, \,g_{2}, \,g_{3}, \,g_{4}, \,g_{5}, \,g_{6}, \,g_{7} & \mid & g_{1}^{2} &= 1, \\ 
 & & g_{2}^{2} &= 1, & [g_{2}, g_{1}]  &= g_{4}, \\ 
 & & g_{3}^{2} &= 1, & [g_{3}, g_{1}]  &= g_{5}, \\ 
 & & g_{4}^{2} &= g_{6}g_{7}, & [g_{4}, g_{1}]  &= g_{6}, & [g_{4}, g_{2}]  &= g_{6}, \\ 
 & & g_{5}^{2} &= 1, \\ 
 & & g_{6}^{2} &= g_{7}, & [g_{6}, g_{1}]  &= g_{7}, & [g_{6}, g_{2}]  &= g_{7}, \\ 
 & & g_{7}^{2} &= 1\rangle. \\ 
\end{aligned}
\]
We add 13 tails to the presentation as to form a quotient of the universal central extension of the system: 
$g_{1}^{2} =  t_{1}$,
$g_{2}^{2} =  t_{2}$,
$[g_{2}, g_{1}] = g_{4} t_{3}$,
$g_{3}^{2} =  t_{4}$,
$[g_{3}, g_{1}] = g_{5} t_{5}$,
$g_{4}^{2} = g_{6}g_{7} t_{6}$,
$[g_{4}, g_{1}] = g_{6} t_{7}$,
$[g_{4}, g_{2}] = g_{6} t_{8}$,
$g_{5}^{2} =  t_{9}$,
$g_{6}^{2} = g_{7} t_{10}$,
$[g_{6}, g_{1}] = g_{7} t_{11}$,
$[g_{6}, g_{2}] = g_{7} t_{12}$,
$g_{7}^{2} =  t_{13}$.
Carrying out consistency checks gives the following relations between the tails:
\[
\begin{aligned}
g_{4}(g_{2} g_{1}) & = (g_{4} g_{2}) g_{1}  & \Longrightarrow & & t_{11}t_{12}^{-1} & = 1 \\
g_{6}^2 g_{2} & = g_{6} (g_{6} g_{2})& \Longrightarrow & & t_{12}^{2}t_{13} & = 1 \\
g_{4}^2 g_{2} & = g_{4} (g_{4} g_{2})& \Longrightarrow & & t_{8}^{2}t_{10}t_{12}^{-1} & = 1 \\
g_{4}^2 g_{1} & = g_{4} (g_{4} g_{1})& \Longrightarrow & & t_{7}^{2}t_{10}t_{11}^{-1} & = 1 \\
g_{3}^2 g_{1} & = g_{3} (g_{3} g_{1})& \Longrightarrow & & t_{5}^{2}t_{9} & = 1 \\
g_{2}^2 g_{1} & = g_{2} (g_{2} g_{1})& \Longrightarrow & & t_{3}^{2}t_{6}t_{8}t_{10}t_{13} & = 1 \\
g_{2} g_{1}^{2} & = (g_{2} g_{1}) g_{1}& \Longrightarrow & & t_{3}^{2}t_{6}t_{7}t_{10}t_{13} & = 1 \\
\end{aligned}
\]
Scanning through the conjugacy class representatives of $G$ and the generators of their centralizers, we see that no new relations are imposed.
Collecting the coefficients of these relations into a matrix yields
\[
T = \bordermatrix{
{} & t_{1} & t_{2} & t_{3} & t_{4} & t_{5} & t_{6} & t_{7} & t_{8} & t_{9} & t_{10} & t_{11} & t_{12} & t_{13} \cr
{} &  &  & 2 &  &  & 1 &  & 1 &  & 1 &  &  & 1 \cr
{} &  &  &  &  & 2 &  &  &  & 1 &  &  &  &  \cr
{} &  &  &  &  &  &  & 1 & 1 &  & 1 &  & 1 & 1 \cr
{} &  &  &  &  &  &  &  & 2 &  & 1 &  & 1 & 1 \cr
{} &  &  &  &  &  &  &  &  &  &  & 1 & 1 & 1 \cr
{} &  &  &  &  &  &  &  &  &  &  &  & 2 & 1 \cr
}.
\]
It follows readily that the nontrivial elementary divisors of the Smith normal form of $T$ are all equal to $1$. The torsion subgroup of the group generated by the tails is thus trivial, thereby showing $\B_0(G) = 1$.


\item \label{number:80} 
Let the group $G$ be the representative of this family given by the presentation
\[
\begin{aligned}
\langle g_{1}, \,g_{2}, \,g_{3}, \,g_{4}, \,g_{5}, \,g_{6}, \,g_{7} & \mid & g_{1}^{2} &= 1, \\ 
 & & g_{2}^{2} &= g_{4}g_{6}, & [g_{2}, g_{1}]  &= g_{4}, \\ 
 & & g_{3}^{2} &= 1, & [g_{3}, g_{1}]  &= g_{5}, & [g_{3}, g_{2}]  &= g_{7}, \\ 
 & & g_{4}^{2} &= g_{6}g_{7}, & [g_{4}, g_{1}]  &= g_{6}, \\ 
 & & g_{5}^{2} &= 1, \\ 
 & & g_{6}^{2} &= g_{7}, & [g_{6}, g_{1}]  &= g_{7}, \\ 
 & & g_{7}^{2} &= 1\rangle. \\ 
\end{aligned}
\]
We add 12 tails to the presentation as to form a quotient of the universal central extension of the system: 
$g_{1}^{2} =  t_{1}$,
$g_{2}^{2} = g_{4}g_{6} t_{2}$,
$[g_{2}, g_{1}] = g_{4} t_{3}$,
$g_{3}^{2} =  t_{4}$,
$[g_{3}, g_{1}] = g_{5} t_{5}$,
$[g_{3}, g_{2}] = g_{7} t_{6}$,
$g_{4}^{2} = g_{6}g_{7} t_{7}$,
$[g_{4}, g_{1}] = g_{6} t_{8}$,
$g_{5}^{2} =  t_{9}$,
$g_{6}^{2} = g_{7} t_{10}$,
$[g_{6}, g_{1}] = g_{7} t_{11}$,
$g_{7}^{2} =  t_{12}$.
Carrying out consistency checks gives the following relations between the tails:
\[
\begin{aligned}
g_{6}^2 g_{1} & = g_{6} (g_{6} g_{1})& \Longrightarrow & & t_{11}^{2}t_{12} & = 1 \\
g_{4}^2 g_{1} & = g_{4} (g_{4} g_{1})& \Longrightarrow & & t_{8}^{2}t_{10}t_{11}^{-1} & = 1 \\
g_{3}^2 g_{2} & = g_{3} (g_{3} g_{2})& \Longrightarrow & & t_{6}^{2}t_{12} & = 1 \\
g_{3}^2 g_{1} & = g_{3} (g_{3} g_{1})& \Longrightarrow & & t_{5}^{2}t_{9} & = 1 \\
g_{2}^2 g_{1} & = g_{2} (g_{2} g_{1})& \Longrightarrow & & t_{3}^{2}t_{7}t_{8}^{-1}t_{11}^{-1} & = 1 \\
\end{aligned}
\]
Scanning through the conjugacy class representatives of $G$ and the generators of their centralizers, we see that no new relations are imposed.
Collecting the coefficients of these relations into a matrix yields
\[
T = \bordermatrix{
{} & t_{1} & t_{2} & t_{3} & t_{4} & t_{5} & t_{6} & t_{7} & t_{8} & t_{9} & t_{10} & t_{11} & t_{12} \cr
{} &  &  & 2 &  &  &  & 1 & 1 &  & 1 &  & 1 \cr
{} &  &  &  &  & 2 &  &  &  & 1 &  &  &  \cr
{} &  &  &  &  &  & 2 &  &  &  &  &  & 1 \cr
{} &  &  &  &  &  &  &  & 2 &  & 1 & 1 & 1 \cr
{} &  &  &  &  &  &  &  &  &  &  & 2 & 1 \cr
}.
\]
A change of basis according to the transition matrix (specifying expansions of $t_i^{*}$ by $t_j$)
\[
\bordermatrix{
{} & t_{1}^{*} & t_{2}^{*} & t_{3}^{*} & t_{4}^{*} & t_{5}^{*} & t_{6}^{*} & t_{7}^{*} & t_{8}^{*} & t_{9}^{*} & t_{10}^{*} & t_{11}^{*} & t_{12}^{*} \cr
t_{1} &  &  &  &  &  & -7 & -1 &  &  & 1 &  & 1 \cr
t_{2} &  &  &  &  &  & -3 &  & 1 & -1 &  &  &  \cr
t_{3} & -26 & -36 &  & -24 & -20 &  &  &  &  & -1 &  & -1 \cr
t_{4} &  &  &  &  &  & 4 & -1 &  &  & -3 &  &  \cr
t_{5} & 40 & 54 &  & 36 & 30 & 7 &  &  &  &  &  &  \cr
t_{6} & -60 & -80 &  & -54 & -45 & -4 &  & -1 &  &  &  &  \cr
t_{7} & -13 & -18 &  & -12 & -10 &  & 1 &  &  &  &  &  \cr
t_{8} & 37 & 50 &  & 34 & 28 &  &  & 1 &  &  &  &  \cr
t_{9} & 20 & 27 &  & 18 & 15 &  &  &  & 1 &  &  &  \cr
t_{10} & 12 & 16 &  & 11 & 9 &  &  &  &  & 1 &  &  \cr
t_{11} & 89 & 122 & 2 & 81 & 68 &  &  &  &  &  & 1 &  \cr
t_{12} & 14 & 20 & 1 & 13 & 11 &  &  &  &  &  &  & 1 \cr
}
\]
shows that the nontrivial elementary divisors of the Smith normal form of $T$ are $1$, $1$, $1$, $1$, $2$.  The element corresponding to the divisor that is greater than $1$ is $t_{5}^{*}$. This already gives
\[
\B_0(G) \cong \langle t_{5}^{*}  \mid {t_{5}^{*}}^{2} \rangle.
\]

We now deal with explicitly identifying the nonuniversal commutator relation generating $\B_0(G)$.
First, factor out by the tails $t_{i}^{*}$ whose corresponding elementary divisors are either trivial or $1$. Transforming the situation back to the original tails $t_i$, this amounts to the nontrivial expansions given by
\[
\bordermatrix{
{} & t_{1} & t_{3} & t_{5} & t_{6} & t_{7} & t_{8} \cr
t_{5}^{*} & 1 & 1 & 1 & 1 & 1 & 1 \cr
}
\]
and all the other tails $t_i$ are trivial. We thus obtain a commutativity preserving central extension of the group $G$, given by the presentation
\[
\begin{aligned}
\langle g_{1}, \,g_{2}, \,g_{3}, \,g_{4}, \,g_{5}, \,g_{6}, \,g_{7}, \,t_{5}^{*} & \mid & g_{1}^{2} &= t_{5}^{*} , \\ 
 & & g_{2}^{2} &= g_{4}g_{6}, & [g_{2}, g_{1}]  &= g_{4}t_{5}^{*} , \\ 
 & & g_{3}^{2} &= 1, & [g_{3}, g_{1}]  &= g_{5}t_{5}^{*} , & [g_{3}, g_{2}]  &= g_{7}t_{5}^{*} , \\ 
 & & g_{4}^{2} &= g_{6}g_{7}t_{5}^{*} , & [g_{4}, g_{1}]  &= g_{6}t_{5}^{*} , \\ 
 & & g_{5}^{2} &= 1, \\ 
 & & g_{6}^{2} &= g_{7}, & [g_{6}, g_{1}]  &= g_{7}, \\ 
 & & g_{7}^{2} &= 1, \\ 
 & & {t_{5}^{*}}^{2} &= 1  \rangle,
\end{aligned}
\]
whence the nonuniversal commutator relation is identified as
\[
t_{5}^{*}  = [g_{3}, g_{2}] [g_{6}, g_{1}]^{-1}.  \quad 
\]


\item \label{number:81} 
Let the group $G$ be the representative of this family given by the presentation
\[
\begin{aligned}
\langle g_{1}, \,g_{2}, \,g_{3}, \,g_{4}, \,g_{5}, \,g_{6}, \,g_{7} & \mid & g_{1}^{2} &= 1, \\ 
 & & g_{2}^{2} &= g_{4}g_{6}, & [g_{2}, g_{1}]  &= g_{4}, \\ 
 & & g_{3}^{2} &= 1, & [g_{3}, g_{1}]  &= g_{5}, \\ 
 & & g_{4}^{2} &= g_{6}g_{7}, & [g_{4}, g_{1}]  &= g_{6}, \\ 
 & & g_{5}^{2} &= 1, \\ 
 & & g_{6}^{2} &= g_{7}, & [g_{6}, g_{1}]  &= g_{7}, \\ 
 & & g_{7}^{2} &= 1\rangle. \\ 
\end{aligned}
\]
We add 11 tails to the presentation as to form a quotient of the universal central extension of the system: 
$g_{1}^{2} =  t_{1}$,
$g_{2}^{2} = g_{4}g_{6} t_{2}$,
$[g_{2}, g_{1}] = g_{4} t_{3}$,
$g_{3}^{2} =  t_{4}$,
$[g_{3}, g_{1}] = g_{5} t_{5}$,
$g_{4}^{2} = g_{6}g_{7} t_{6}$,
$[g_{4}, g_{1}] = g_{6} t_{7}$,
$g_{5}^{2} =  t_{8}$,
$g_{6}^{2} = g_{7} t_{9}$,
$[g_{6}, g_{1}] = g_{7} t_{10}$,
$g_{7}^{2} =  t_{11}$.
Carrying out consistency checks gives the following relations between the tails:
\[
\begin{aligned}
g_{6}^2 g_{1} & = g_{6} (g_{6} g_{1})& \Longrightarrow & & t_{10}^{2}t_{11} & = 1 \\
g_{4}^2 g_{1} & = g_{4} (g_{4} g_{1})& \Longrightarrow & & t_{7}^{2}t_{9}t_{10}^{-1} & = 1 \\
g_{3}^2 g_{1} & = g_{3} (g_{3} g_{1})& \Longrightarrow & & t_{5}^{2}t_{8} & = 1 \\
g_{2}^2 g_{1} & = g_{2} (g_{2} g_{1})& \Longrightarrow & & t_{3}^{2}t_{6}t_{7}^{-1}t_{10}^{-1} & = 1 \\
\end{aligned}
\]
Scanning through the conjugacy class representatives of $G$ and the generators of their centralizers, we see that no new relations are imposed.
Collecting the coefficients of these relations into a matrix yields
\[
T = \bordermatrix{
{} & t_{1} & t_{2} & t_{3} & t_{4} & t_{5} & t_{6} & t_{7} & t_{8} & t_{9} & t_{10} & t_{11} \cr
{} &  &  & 2 &  &  & 1 & 1 &  & 1 &  & 1 \cr
{} &  &  &  &  & 2 &  &  & 1 &  &  &  \cr
{} &  &  &  &  &  &  & 2 &  & 1 & 1 & 1 \cr
{} &  &  &  &  &  &  &  &  &  & 2 & 1 \cr
}.
\]
It follows readily that the nontrivial elementary divisors of the Smith normal form of $T$ are all equal to $1$. The torsion subgroup of the group generated by the tails is thus trivial, thereby showing $\B_0(G) = 1$.


\item \label{number:82} 
Let the group $G$ be the representative of this family given by the presentation
\[
\begin{aligned}
\langle g_{1}, \,g_{2}, \,g_{3}, \,g_{4}, \,g_{5}, \,g_{6}, \,g_{7} & \mid & g_{1}^{2} &= g_{5}, \\ 
 & & g_{2}^{2} &= 1, & [g_{2}, g_{1}]  &= g_{4}, \\ 
 & & g_{3}^{2} &= g_{5}g_{6}, & [g_{3}, g_{1}]  &= g_{6}, \\ 
 & & g_{4}^{2} &= g_{6}, & [g_{4}, g_{1}]  &= g_{6}, & [g_{4}, g_{2}]  &= g_{6}g_{7}, & [g_{4}, g_{3}]  &= g_{7}, \\ 
 & & g_{5}^{2} &= 1, & [g_{5}, g_{2}]  &= g_{7}, \\ 
 & & g_{6}^{2} &= g_{7}, & [g_{6}, g_{1}]  &= g_{7}, & [g_{6}, g_{2}]  &= g_{7}, \\ 
 & & g_{7}^{2} &= 1\rangle. \\ 
\end{aligned}
\]
We add 15 tails to the presentation as to form a quotient of the universal central extension of the system: 
$g_{1}^{2} = g_{5} t_{1}$,
$g_{2}^{2} =  t_{2}$,
$[g_{2}, g_{1}] = g_{4} t_{3}$,
$g_{3}^{2} = g_{5}g_{6} t_{4}$,
$[g_{3}, g_{1}] = g_{6} t_{5}$,
$g_{4}^{2} = g_{6} t_{6}$,
$[g_{4}, g_{1}] = g_{6} t_{7}$,
$[g_{4}, g_{2}] = g_{6}g_{7} t_{8}$,
$[g_{4}, g_{3}] = g_{7} t_{9}$,
$g_{5}^{2} =  t_{10}$,
$[g_{5}, g_{2}] = g_{7} t_{11}$,
$g_{6}^{2} = g_{7} t_{12}$,
$[g_{6}, g_{1}] = g_{7} t_{13}$,
$[g_{6}, g_{2}] = g_{7} t_{14}$,
$g_{7}^{2} =  t_{15}$.
Carrying out consistency checks gives the following relations between the tails:
\[
\begin{aligned}
g_{4}(g_{2} g_{1}) & = (g_{4} g_{2}) g_{1}  & \Longrightarrow & & t_{13}t_{14}^{-1} & = 1 \\
g_{3}(g_{2} g_{1}) & = (g_{3} g_{2}) g_{1}  & \Longrightarrow & & t_{9}t_{14}^{-1} & = 1 \\
g_{6}^2 g_{2} & = g_{6} (g_{6} g_{2})& \Longrightarrow & & t_{14}^{2}t_{15} & = 1 \\
g_{5}^2 g_{2} & = g_{5} (g_{5} g_{2})& \Longrightarrow & & t_{11}^{2}t_{15} & = 1 \\
g_{4}^2 g_{2} & = g_{4} (g_{4} g_{2})& \Longrightarrow & & t_{8}^{2}t_{12}t_{14}^{-1}t_{15} & = 1 \\
g_{4}^2 g_{1} & = g_{4} (g_{4} g_{1})& \Longrightarrow & & t_{7}^{2}t_{12}t_{13}^{-1} & = 1 \\
g_{3}^2 g_{2} & = g_{3} (g_{3} g_{2})& \Longrightarrow & & t_{11}^{-1}t_{14}^{-1}t_{15}^{-1} & = 1 \\
g_{3}^2 g_{1} & = g_{3} (g_{3} g_{1})& \Longrightarrow & & t_{5}^{2}t_{12}t_{13}^{-1} & = 1 \\
g_{2}^2 g_{1} & = g_{2} (g_{2} g_{1})& \Longrightarrow & & t_{3}^{2}t_{6}t_{8}t_{12}t_{15} & = 1 \\
g_{2} g_{1}^{2} & = (g_{2} g_{1}) g_{1}& \Longrightarrow & & t_{3}^{2}t_{6}t_{7}t_{11}t_{12}t_{15} & = 1 \\
\end{aligned}
\]
Scanning through the conjugacy class representatives of $G$ and the generators of their centralizers, we obtain the following relations induced on the tails:
\[
\begin{aligned}
{[g_{3} g_{4} g_{6} g_{7} , \, g_{1} ]}_G & = 1 & \Longrightarrow & & t_{5}t_{7}t_{12}t_{13}t_{15} & = 1 \\
\end{aligned}
\]
Collecting the coefficients of these relations into a matrix yields
\[
T = \bordermatrix{
{} & t_{1} & t_{2} & t_{3} & t_{4} & t_{5} & t_{6} & t_{7} & t_{8} & t_{9} & t_{10} & t_{11} & t_{12} & t_{13} & t_{14} & t_{15} \cr
{} &  &  & 2 &  &  & 1 &  & 1 &  &  &  & 1 &  &  & 1 \cr
{} &  &  &  &  & 1 &  &  & 1 &  &  &  & 1 &  &  & 1 \cr
{} &  &  &  &  &  &  & 1 & 1 &  &  &  & 1 &  &  & 1 \cr
{} &  &  &  &  &  &  &  & 2 &  &  &  & 1 &  & 1 & 2 \cr
{} &  &  &  &  &  &  &  &  & 1 &  &  &  &  & 1 & 1 \cr
{} &  &  &  &  &  &  &  &  &  &  & 1 &  &  & 1 & 1 \cr
{} &  &  &  &  &  &  &  &  &  &  &  &  & 1 & 1 & 1 \cr
{} &  &  &  &  &  &  &  &  &  &  &  &  &  & 2 & 1 \cr
}.
\]
It follows readily that the nontrivial elementary divisors of the Smith normal form of $T$ are all equal to $1$. The torsion subgroup of the group generated by the tails is thus trivial, thereby showing $\B_0(G) = 1$.


\item \label{number:83} 
Let the group $G$ be the representative of this family given by the presentation
\[
\begin{aligned}
\langle g_{1}, \,g_{2}, \,g_{3}, \,g_{4}, \,g_{5}, \,g_{6}, \,g_{7} & \mid & g_{1}^{2} &= g_{5}, \\ 
 & & g_{2}^{2} &= 1, & [g_{2}, g_{1}]  &= g_{4}, \\ 
 & & g_{3}^{2} &= 1, & [g_{3}, g_{2}]  &= g_{7}, \\ 
 & & g_{4}^{2} &= 1, & [g_{4}, g_{1}]  &= g_{6}, \\ 
 & & g_{5}^{2} &= 1, & [g_{5}, g_{2}]  &= g_{6}, & [g_{5}, g_{4}]  &= g_{7}, \\ 
 & & g_{6}^{2} &= 1, & [g_{6}, g_{1}]  &= g_{7}, \\ 
 & & g_{7}^{2} &= 1\rangle. \\ 
\end{aligned}
\]
We add 13 tails to the presentation as to form a quotient of the universal central extension of the system: 
$g_{1}^{2} = g_{5} t_{1}$,
$g_{2}^{2} =  t_{2}$,
$[g_{2}, g_{1}] = g_{4} t_{3}$,
$g_{3}^{2} =  t_{4}$,
$[g_{3}, g_{2}] = g_{7} t_{5}$,
$g_{4}^{2} =  t_{6}$,
$[g_{4}, g_{1}] = g_{6} t_{7}$,
$g_{5}^{2} =  t_{8}$,
$[g_{5}, g_{2}] = g_{6} t_{9}$,
$[g_{5}, g_{4}] = g_{7} t_{10}$,
$g_{6}^{2} =  t_{11}$,
$[g_{6}, g_{1}] = g_{7} t_{12}$,
$g_{7}^{2} =  t_{13}$.
Carrying out consistency checks gives the following relations between the tails:
\[
\begin{aligned}
g_{5}(g_{2} g_{1}) & = (g_{5} g_{2}) g_{1}  & \Longrightarrow & & t_{10}^{-1}t_{12} & = 1 \\
g_{6}^2 g_{1} & = g_{6} (g_{6} g_{1})& \Longrightarrow & & t_{12}^{2}t_{13} & = 1 \\
g_{5}^2 g_{2} & = g_{5} (g_{5} g_{2})& \Longrightarrow & & t_{9}^{2}t_{11} & = 1 \\
g_{4}^2 g_{1} & = g_{4} (g_{4} g_{1})& \Longrightarrow & & t_{7}^{2}t_{11} & = 1 \\
g_{3}^2 g_{2} & = g_{3} (g_{3} g_{2})& \Longrightarrow & & t_{5}^{2}t_{13} & = 1 \\
g_{2}^2 g_{1} & = g_{2} (g_{2} g_{1})& \Longrightarrow & & t_{3}^{2}t_{6} & = 1 \\
g_{2} g_{1}^{2} & = (g_{2} g_{1}) g_{1}& \Longrightarrow & & t_{3}^{2}t_{6}t_{7}t_{9}t_{10}^{2}t_{11}t_{13} & = 1 \\
\end{aligned}
\]
Scanning through the conjugacy class representatives of $G$ and the generators of their centralizers, we obtain the following relations induced on the tails:
\[
\begin{aligned}
{[g_{3} g_{6} g_{7} , \, g_{1} g_{2} g_{3} ]}_G & = 1 & \Longrightarrow & & t_{5}t_{12}t_{13} & = 1 \\
\end{aligned}
\]
Collecting the coefficients of these relations into a matrix yields
\[
T = \bordermatrix{
{} & t_{1} & t_{2} & t_{3} & t_{4} & t_{5} & t_{6} & t_{7} & t_{8} & t_{9} & t_{10} & t_{11} & t_{12} & t_{13} \cr
{} &  &  & 2 &  &  & 1 &  &  &  &  &  &  &  \cr
{} &  &  &  &  & 1 &  &  &  &  &  &  & 1 & 1 \cr
{} &  &  &  &  &  &  & 1 &  & 1 &  & 1 &  &  \cr
{} &  &  &  &  &  &  &  &  & 2 &  & 1 &  &  \cr
{} &  &  &  &  &  &  &  &  &  & 1 &  & 1 & 1 \cr
{} &  &  &  &  &  &  &  &  &  &  &  & 2 & 1 \cr
}.
\]
It follows readily that the nontrivial elementary divisors of the Smith normal form of $T$ are all equal to $1$. The torsion subgroup of the group generated by the tails is thus trivial, thereby showing $\B_0(G) = 1$.


\item \label{number:84} 
Let the group $G$ be the representative of this family given by the presentation
\[
\begin{aligned}
\langle g_{1}, \,g_{2}, \,g_{3}, \,g_{4}, \,g_{5}, \,g_{6}, \,g_{7} & \mid & g_{1}^{2} &= g_{5}, \\ 
 & & g_{2}^{2} &= 1, & [g_{2}, g_{1}]  &= g_{4}, \\ 
 & & g_{3}^{2} &= 1, & [g_{3}, g_{2}]  &= g_{7}, \\ 
 & & g_{4}^{2} &= g_{7}, & [g_{4}, g_{1}]  &= g_{6}, & [g_{4}, g_{2}]  &= g_{7}, \\ 
 & & g_{5}^{2} &= 1, & [g_{5}, g_{2}]  &= g_{6}g_{7}, & [g_{5}, g_{4}]  &= g_{7}, \\ 
 & & g_{6}^{2} &= 1, & [g_{6}, g_{1}]  &= g_{7}, \\ 
 & & g_{7}^{2} &= 1\rangle. \\ 
\end{aligned}
\]
We add 14 tails to the presentation as to form a quotient of the universal central extension of the system: 
$g_{1}^{2} = g_{5} t_{1}$,
$g_{2}^{2} =  t_{2}$,
$[g_{2}, g_{1}] = g_{4} t_{3}$,
$g_{3}^{2} =  t_{4}$,
$[g_{3}, g_{2}] = g_{7} t_{5}$,
$g_{4}^{2} = g_{7} t_{6}$,
$[g_{4}, g_{1}] = g_{6} t_{7}$,
$[g_{4}, g_{2}] = g_{7} t_{8}$,
$g_{5}^{2} =  t_{9}$,
$[g_{5}, g_{2}] = g_{6}g_{7} t_{10}$,
$[g_{5}, g_{4}] = g_{7} t_{11}$,
$g_{6}^{2} =  t_{12}$,
$[g_{6}, g_{1}] = g_{7} t_{13}$,
$g_{7}^{2} =  t_{14}$.
Carrying out consistency checks gives the following relations between the tails:
\[
\begin{aligned}
g_{5}(g_{2} g_{1}) & = (g_{5} g_{2}) g_{1}  & \Longrightarrow & & t_{11}^{-1}t_{13} & = 1 \\
g_{6}^2 g_{1} & = g_{6} (g_{6} g_{1})& \Longrightarrow & & t_{13}^{2}t_{14} & = 1 \\
g_{5}^2 g_{2} & = g_{5} (g_{5} g_{2})& \Longrightarrow & & t_{10}^{2}t_{12}t_{14} & = 1 \\
g_{4}^2 g_{2} & = g_{4} (g_{4} g_{2})& \Longrightarrow & & t_{8}^{2}t_{14} & = 1 \\
g_{4}^2 g_{1} & = g_{4} (g_{4} g_{1})& \Longrightarrow & & t_{7}^{2}t_{12} & = 1 \\
g_{3}^2 g_{2} & = g_{3} (g_{3} g_{2})& \Longrightarrow & & t_{5}^{2}t_{14} & = 1 \\
g_{2}^2 g_{1} & = g_{2} (g_{2} g_{1})& \Longrightarrow & & t_{3}^{2}t_{6}t_{8}t_{14} & = 1 \\
g_{2} g_{1}^{2} & = (g_{2} g_{1}) g_{1}& \Longrightarrow & & t_{3}^{2}t_{6}t_{7}t_{10}t_{11}^{2}t_{12}t_{14}^{2} & = 1 \\
\end{aligned}
\]
Scanning through the conjugacy class representatives of $G$ and the generators of their centralizers, we obtain the following relations induced on the tails:
\[
\begin{aligned}
{[g_{4} g_{6} g_{7} , \, g_{2} g_{4} g_{5} g_{7} ]}_G & = 1 & \Longrightarrow & & t_{8}t_{11}^{-1} & = 1 \\
{[g_{3} g_{6} g_{7} , \, g_{1} g_{2} g_{3} ]}_G & = 1 & \Longrightarrow & & t_{5}t_{13}t_{14} & = 1 \\
\end{aligned}
\]
Collecting the coefficients of these relations into a matrix yields
\[
T = \bordermatrix{
{} & t_{1} & t_{2} & t_{3} & t_{4} & t_{5} & t_{6} & t_{7} & t_{8} & t_{9} & t_{10} & t_{11} & t_{12} & t_{13} & t_{14} \cr
{} &  &  & 2 &  &  & 1 &  &  &  &  &  &  & 1 & 1 \cr
{} &  &  &  &  & 1 &  &  &  &  &  &  &  & 1 & 1 \cr
{} &  &  &  &  &  &  & 1 &  &  & 1 &  & 1 & 1 & 1 \cr
{} &  &  &  &  &  &  &  & 1 &  &  &  &  & 1 & 1 \cr
{} &  &  &  &  &  &  &  &  &  & 2 &  & 1 &  & 1 \cr
{} &  &  &  &  &  &  &  &  &  &  & 1 &  & 1 & 1 \cr
{} &  &  &  &  &  &  &  &  &  &  &  &  & 2 & 1 \cr
}.
\]
It follows readily that the nontrivial elementary divisors of the Smith normal form of $T$ are all equal to $1$. The torsion subgroup of the group generated by the tails is thus trivial, thereby showing $\B_0(G) = 1$.


\item \label{number:85} 
Let the group $G$ be the representative of this family given by the presentation
\[
\begin{aligned}
\langle g_{1}, \,g_{2}, \,g_{3}, \,g_{4}, \,g_{5}, \,g_{6}, \,g_{7} & \mid & g_{1}^{2} &= 1, \\ 
 & & g_{2}^{2} &= 1, & [g_{2}, g_{1}]  &= g_{5}, \\ 
 & & g_{3}^{2} &= 1, \\ 
 & & g_{4}^{2} &= 1, & [g_{4}, g_{3}]  &= g_{7}, \\ 
 & & g_{5}^{2} &= g_{6}g_{7}, & [g_{5}, g_{1}]  &= g_{6}, & [g_{5}, g_{2}]  &= g_{6}, \\ 
 & & g_{6}^{2} &= g_{7}, & [g_{6}, g_{1}]  &= g_{7}, & [g_{6}, g_{2}]  &= g_{7}, \\ 
 & & g_{7}^{2} &= 1\rangle. \\ 
\end{aligned}
\]
We add 13 tails to the presentation as to form a quotient of the universal central extension of the system: 
$g_{1}^{2} =  t_{1}$,
$g_{2}^{2} =  t_{2}$,
$[g_{2}, g_{1}] = g_{5} t_{3}$,
$g_{3}^{2} =  t_{4}$,
$g_{4}^{2} =  t_{5}$,
$[g_{4}, g_{3}] = g_{7} t_{6}$,
$g_{5}^{2} = g_{6}g_{7} t_{7}$,
$[g_{5}, g_{1}] = g_{6} t_{8}$,
$[g_{5}, g_{2}] = g_{6} t_{9}$,
$g_{6}^{2} = g_{7} t_{10}$,
$[g_{6}, g_{1}] = g_{7} t_{11}$,
$[g_{6}, g_{2}] = g_{7} t_{12}$,
$g_{7}^{2} =  t_{13}$.
Carrying out consistency checks gives the following relations between the tails:
\[
\begin{aligned}
g_{5}(g_{2} g_{1}) & = (g_{5} g_{2}) g_{1}  & \Longrightarrow & & t_{11}t_{12}^{-1} & = 1 \\
g_{6}^2 g_{2} & = g_{6} (g_{6} g_{2})& \Longrightarrow & & t_{12}^{2}t_{13} & = 1 \\
g_{5}^2 g_{2} & = g_{5} (g_{5} g_{2})& \Longrightarrow & & t_{9}^{2}t_{10}t_{12}^{-1} & = 1 \\
g_{5}^2 g_{1} & = g_{5} (g_{5} g_{1})& \Longrightarrow & & t_{8}^{2}t_{10}t_{11}^{-1} & = 1 \\
g_{4}^2 g_{3} & = g_{4} (g_{4} g_{3})& \Longrightarrow & & t_{6}^{2}t_{13} & = 1 \\
g_{2}^2 g_{1} & = g_{2} (g_{2} g_{1})& \Longrightarrow & & t_{3}^{2}t_{7}t_{9}t_{10}t_{13} & = 1 \\
g_{2} g_{1}^{2} & = (g_{2} g_{1}) g_{1}& \Longrightarrow & & t_{3}^{2}t_{7}t_{8}t_{10}t_{13} & = 1 \\
\end{aligned}
\]
Scanning through the conjugacy class representatives of $G$ and the generators of their centralizers, we obtain the following relations induced on the tails:
\[
\begin{aligned}
{[g_{4} g_{6} g_{7} , \, g_{1} g_{3} ]}_G & = 1 & \Longrightarrow & & t_{6}t_{11}t_{13} & = 1 \\
\end{aligned}
\]
Collecting the coefficients of these relations into a matrix yields
\[
T = \bordermatrix{
{} & t_{1} & t_{2} & t_{3} & t_{4} & t_{5} & t_{6} & t_{7} & t_{8} & t_{9} & t_{10} & t_{11} & t_{12} & t_{13} \cr
{} &  &  & 2 &  &  &  & 1 &  & 1 & 1 &  &  & 1 \cr
{} &  &  &  &  &  & 1 &  &  &  &  &  & 1 & 1 \cr
{} &  &  &  &  &  &  &  & 1 & 1 & 1 &  & 1 & 1 \cr
{} &  &  &  &  &  &  &  &  & 2 & 1 &  & 1 & 1 \cr
{} &  &  &  &  &  &  &  &  &  &  & 1 & 1 & 1 \cr
{} &  &  &  &  &  &  &  &  &  &  &  & 2 & 1 \cr
}.
\]
It follows readily that the nontrivial elementary divisors of the Smith normal form of $T$ are all equal to $1$. The torsion subgroup of the group generated by the tails is thus trivial, thereby showing $\B_0(G) = 1$.


\item \label{number:86} 
Let the group $G$ be the representative of this family given by the presentation
\[
\begin{aligned}
\langle g_{1}, \,g_{2}, \,g_{3}, \,g_{4}, \,g_{5}, \,g_{6}, \,g_{7} & \mid & g_{1}^{2} &= 1, \\ 
 & & g_{2}^{2} &= 1, & [g_{2}, g_{1}]  &= g_{4}, \\ 
 & & g_{3}^{2} &= g_{4}, & [g_{3}, g_{1}]  &= g_{5}, & [g_{3}, g_{2}]  &= g_{6}, \\ 
 & & g_{4}^{2} &= g_{7}, & [g_{4}, g_{1}]  &= g_{7}, & [g_{4}, g_{2}]  &= g_{7}, \\ 
 & & g_{5}^{2} &= g_{7}, & [g_{5}, g_{1}]  &= g_{7}, \\ 
 & & g_{6}^{2} &= 1, & [g_{6}, g_{3}]  &= g_{7}, \\ 
 & & g_{7}^{2} &= 1\rangle. \\ 
\end{aligned}
\]
We add 14 tails to the presentation as to form a quotient of the universal central extension of the system: 
$g_{1}^{2} =  t_{1}$,
$g_{2}^{2} =  t_{2}$,
$[g_{2}, g_{1}] = g_{4} t_{3}$,
$g_{3}^{2} = g_{4} t_{4}$,
$[g_{3}, g_{1}] = g_{5} t_{5}$,
$[g_{3}, g_{2}] = g_{6} t_{6}$,
$g_{4}^{2} = g_{7} t_{7}$,
$[g_{4}, g_{1}] = g_{7} t_{8}$,
$[g_{4}, g_{2}] = g_{7} t_{9}$,
$g_{5}^{2} = g_{7} t_{10}$,
$[g_{5}, g_{1}] = g_{7} t_{11}$,
$g_{6}^{2} =  t_{12}$,
$[g_{6}, g_{3}] = g_{7} t_{13}$,
$g_{7}^{2} =  t_{14}$.
Carrying out consistency checks gives the following relations between the tails:
\[
\begin{aligned}
g_{6}^2 g_{3} & = g_{6} (g_{6} g_{3})& \Longrightarrow & & t_{13}^{2}t_{14} & = 1 \\
g_{5}^2 g_{1} & = g_{5} (g_{5} g_{1})& \Longrightarrow & & t_{11}^{2}t_{14} & = 1 \\
g_{4}^2 g_{2} & = g_{4} (g_{4} g_{2})& \Longrightarrow & & t_{9}^{2}t_{14} & = 1 \\
g_{4}^2 g_{1} & = g_{4} (g_{4} g_{1})& \Longrightarrow & & t_{8}^{2}t_{14} & = 1 \\
g_{3}^2 g_{2} & = g_{3} (g_{3} g_{2})& \Longrightarrow & & t_{6}^{2}t_{9}^{-1}t_{12}t_{13} & = 1 \\
g_{3}^2 g_{1} & = g_{3} (g_{3} g_{1})& \Longrightarrow & & t_{5}^{2}t_{8}^{-1}t_{10} & = 1 \\
g_{2}^2 g_{1} & = g_{2} (g_{2} g_{1})& \Longrightarrow & & t_{3}^{2}t_{7}t_{9}t_{14} & = 1 \\
g_{3} g_{2}^{2} & = (g_{3} g_{2}) g_{2}& \Longrightarrow & & t_{6}^{2}t_{12} & = 1 \\
g_{3} g_{1}^{2} & = (g_{3} g_{1}) g_{1}& \Longrightarrow & & t_{5}^{2}t_{10}t_{11}t_{14} & = 1 \\
g_{2} g_{1}^{2} & = (g_{2} g_{1}) g_{1}& \Longrightarrow & & t_{3}^{2}t_{7}t_{8}t_{14} & = 1 \\
\end{aligned}
\]
Scanning through the conjugacy class representatives of $G$ and the generators of their centralizers, we see that no new relations are imposed.
Collecting the coefficients of these relations into a matrix yields
\[
T = \bordermatrix{
{} & t_{1} & t_{2} & t_{3} & t_{4} & t_{5} & t_{6} & t_{7} & t_{8} & t_{9} & t_{10} & t_{11} & t_{12} & t_{13} & t_{14} \cr
{} &  &  & 2 &  &  &  & 1 &  &  &  &  &  & 1 & 1 \cr
{} &  &  &  &  & 2 &  &  &  &  & 1 &  &  & 1 & 1 \cr
{} &  &  &  &  &  & 2 &  &  &  &  &  & 1 &  &  \cr
{} &  &  &  &  &  &  &  & 1 &  &  &  &  & 1 & 1 \cr
{} &  &  &  &  &  &  &  &  & 1 &  &  &  & 1 & 1 \cr
{} &  &  &  &  &  &  &  &  &  &  & 1 &  & 1 & 1 \cr
{} &  &  &  &  &  &  &  &  &  &  &  &  & 2 & 1 \cr
}.
\]
It follows readily that the nontrivial elementary divisors of the Smith normal form of $T$ are all equal to $1$. The torsion subgroup of the group generated by the tails is thus trivial, thereby showing $\B_0(G) = 1$.


\item \label{number:87} 
Let the group $G$ be the representative of this family given by the presentation
\[
\begin{aligned}
\langle g_{1}, \,g_{2}, \,g_{3}, \,g_{4}, \,g_{5}, \,g_{6}, \,g_{7} & \mid & g_{1}^{2} &= g_{4}g_{5}g_{6}, \\ 
 & & g_{2}^{2} &= g_{4}g_{5}, & [g_{2}, g_{1}]  &= g_{4}, \\ 
 & & g_{3}^{2} &= g_{4}, & [g_{3}, g_{1}]  &= g_{5}, & [g_{3}, g_{2}]  &= g_{6}, \\ 
 & & g_{4}^{2} &= 1, & [g_{4}, g_{1}]  &= g_{7}, \\ 
 & & g_{5}^{2} &= 1, & [g_{5}, g_{1}]  &= g_{7}, & [g_{5}, g_{3}]  &= g_{7}, \\ 
 & & g_{6}^{2} &= 1, & [g_{6}, g_{2}]  &= g_{7}, \\ 
 & & g_{7}^{2} &= 1\rangle. \\ 
\end{aligned}
\]
We add 14 tails to the presentation as to form a quotient of the universal central extension of the system: 
$g_{1}^{2} = g_{4}g_{5}g_{6} t_{1}$,
$g_{2}^{2} = g_{4}g_{5} t_{2}$,
$[g_{2}, g_{1}] = g_{4} t_{3}$,
$g_{3}^{2} = g_{4} t_{4}$,
$[g_{3}, g_{1}] = g_{5} t_{5}$,
$[g_{3}, g_{2}] = g_{6} t_{6}$,
$g_{4}^{2} =  t_{7}$,
$[g_{4}, g_{1}] = g_{7} t_{8}$,
$g_{5}^{2} =  t_{9}$,
$[g_{5}, g_{1}] = g_{7} t_{10}$,
$[g_{5}, g_{3}] = g_{7} t_{11}$,
$g_{6}^{2} =  t_{12}$,
$[g_{6}, g_{2}] = g_{7} t_{13}$,
$g_{7}^{2} =  t_{14}$.
Carrying out consistency checks gives the following relations between the tails:
\[
\begin{aligned}
g_{6}^2 g_{2} & = g_{6} (g_{6} g_{2})& \Longrightarrow & & t_{13}^{2}t_{14} & = 1 \\
g_{5}^2 g_{3} & = g_{5} (g_{5} g_{3})& \Longrightarrow & & t_{11}^{2}t_{14} & = 1 \\
g_{5}^2 g_{1} & = g_{5} (g_{5} g_{1})& \Longrightarrow & & t_{10}^{2}t_{14} & = 1 \\
g_{4}^2 g_{1} & = g_{4} (g_{4} g_{1})& \Longrightarrow & & t_{8}^{2}t_{14} & = 1 \\
g_{3}^2 g_{2} & = g_{3} (g_{3} g_{2})& \Longrightarrow & & t_{6}^{2}t_{12} & = 1 \\
g_{3}^2 g_{1} & = g_{3} (g_{3} g_{1})& \Longrightarrow & & t_{5}^{2}t_{8}^{-1}t_{9}t_{11} & = 1 \\
g_{2}^2 g_{1} & = g_{2} (g_{2} g_{1})& \Longrightarrow & & t_{3}^{2}t_{7}t_{8}^{-1}t_{10}^{-1}t_{14}^{-1} & = 1 \\
g_{3} g_{2}^{2} & = (g_{3} g_{2}) g_{2}& \Longrightarrow & & t_{6}^{2}t_{11}t_{12}t_{13}t_{14} & = 1 \\
g_{3} g_{1}^{2} & = (g_{3} g_{1}) g_{1}& \Longrightarrow & & t_{5}^{2}t_{9}t_{10}t_{11}t_{14} & = 1 \\
g_{2} g_{1}^{2} & = (g_{2} g_{1}) g_{1}& \Longrightarrow & & t_{3}^{2}t_{7}t_{8}t_{13}t_{14} & = 1 \\
\end{aligned}
\]
Scanning through the conjugacy class representatives of $G$ and the generators of their centralizers, we see that no new relations are imposed.
Collecting the coefficients of these relations into a matrix yields
\[
T = \bordermatrix{
{} & t_{1} & t_{2} & t_{3} & t_{4} & t_{5} & t_{6} & t_{7} & t_{8} & t_{9} & t_{10} & t_{11} & t_{12} & t_{13} & t_{14} \cr
{} &  &  & 2 &  &  &  & 1 &  &  &  &  &  &  &  \cr
{} &  &  &  &  & 2 &  &  &  & 1 &  &  &  &  &  \cr
{} &  &  &  &  &  & 2 &  &  &  &  &  & 1 &  &  \cr
{} &  &  &  &  &  &  &  & 1 &  &  &  &  & 1 & 1 \cr
{} &  &  &  &  &  &  &  &  &  & 1 &  &  & 1 & 1 \cr
{} &  &  &  &  &  &  &  &  &  &  & 1 &  & 1 & 1 \cr
{} &  &  &  &  &  &  &  &  &  &  &  &  & 2 & 1 \cr
}.
\]
It follows readily that the nontrivial elementary divisors of the Smith normal form of $T$ are all equal to $1$. The torsion subgroup of the group generated by the tails is thus trivial, thereby showing $\B_0(G) = 1$.


\item \label{number:88} 
Let the group $G$ be the representative of this family given by the presentation
\[
\begin{aligned}
\langle g_{1}, \,g_{2}, \,g_{3}, \,g_{4}, \,g_{5}, \,g_{6}, \,g_{7} & \mid & g_{1}^{2} &= g_{4}g_{5}, \\ 
 & & g_{2}^{2} &= g_{4}, & [g_{2}, g_{1}]  &= g_{4}, \\ 
 & & g_{3}^{2} &= 1, & [g_{3}, g_{1}]  &= g_{5}, & [g_{3}, g_{2}]  &= g_{6}, \\ 
 & & g_{4}^{2} &= g_{7}, & [g_{4}, g_{1}]  &= g_{7}, & [g_{4}, g_{3}]  &= g_{7}, \\ 
 & & g_{5}^{2} &= 1, & [g_{5}, g_{1}]  &= g_{7}, \\ 
 & & g_{6}^{2} &= 1, & [g_{6}, g_{1}]  &= g_{7}, & [g_{6}, g_{2}]  &= g_{7}, \\ 
 & & g_{7}^{2} &= 1\rangle. \\ 
\end{aligned}
\]
We add 15 tails to the presentation as to form a quotient of the universal central extension of the system: 
$g_{1}^{2} = g_{4}g_{5} t_{1}$,
$g_{2}^{2} = g_{4} t_{2}$,
$[g_{2}, g_{1}] = g_{4} t_{3}$,
$g_{3}^{2} =  t_{4}$,
$[g_{3}, g_{1}] = g_{5} t_{5}$,
$[g_{3}, g_{2}] = g_{6} t_{6}$,
$g_{4}^{2} = g_{7} t_{7}$,
$[g_{4}, g_{1}] = g_{7} t_{8}$,
$[g_{4}, g_{3}] = g_{7} t_{9}$,
$g_{5}^{2} =  t_{10}$,
$[g_{5}, g_{1}] = g_{7} t_{11}$,
$g_{6}^{2} =  t_{12}$,
$[g_{6}, g_{1}] = g_{7} t_{13}$,
$[g_{6}, g_{2}] = g_{7} t_{14}$,
$g_{7}^{2} =  t_{15}$.
Carrying out consistency checks gives the following relations between the tails:
\[
\begin{aligned}
g_{3}(g_{2} g_{1}) & = (g_{3} g_{2}) g_{1}  & \Longrightarrow & & t_{9}t_{13}t_{15} & = 1 \\
g_{6}^2 g_{2} & = g_{6} (g_{6} g_{2})& \Longrightarrow & & t_{14}^{2}t_{15} & = 1 \\
g_{6}^2 g_{1} & = g_{6} (g_{6} g_{1})& \Longrightarrow & & t_{13}^{2}t_{15} & = 1 \\
g_{5}^2 g_{1} & = g_{5} (g_{5} g_{1})& \Longrightarrow & & t_{11}^{2}t_{15} & = 1 \\
g_{4}^2 g_{1} & = g_{4} (g_{4} g_{1})& \Longrightarrow & & t_{8}^{2}t_{15} & = 1 \\
g_{3}^2 g_{2} & = g_{3} (g_{3} g_{2})& \Longrightarrow & & t_{6}^{2}t_{12} & = 1 \\
g_{3}^2 g_{1} & = g_{3} (g_{3} g_{1})& \Longrightarrow & & t_{5}^{2}t_{10} & = 1 \\
g_{2}^2 g_{1} & = g_{2} (g_{2} g_{1})& \Longrightarrow & & t_{3}^{2}t_{7}t_{8}^{-1} & = 1 \\
g_{3} g_{2}^{2} & = (g_{3} g_{2}) g_{2}& \Longrightarrow & & t_{6}^{2}t_{9}t_{12}t_{14}t_{15} & = 1 \\
g_{3} g_{1}^{2} & = (g_{3} g_{1}) g_{1}& \Longrightarrow & & t_{5}^{2}t_{9}t_{10}t_{11}t_{15} & = 1 \\
g_{1}^{2} g_{1} & = g_{1} g_{1}^{2}& \Longrightarrow & & t_{8}^{-1}t_{11}^{-1}t_{15}^{-1} & = 1 \\
\end{aligned}
\]
Scanning through the conjugacy class representatives of $G$ and the generators of their centralizers, we see that no new relations are imposed.
Collecting the coefficients of these relations into a matrix yields
\[
T = \bordermatrix{
{} & t_{1} & t_{2} & t_{3} & t_{4} & t_{5} & t_{6} & t_{7} & t_{8} & t_{9} & t_{10} & t_{11} & t_{12} & t_{13} & t_{14} & t_{15} \cr
{} &  &  & 2 &  &  &  & 1 &  &  &  &  &  &  & 1 & 1 \cr
{} &  &  &  &  & 2 &  &  &  &  & 1 &  &  &  &  &  \cr
{} &  &  &  &  &  & 2 &  &  &  &  &  & 1 &  &  &  \cr
{} &  &  &  &  &  &  &  & 1 &  &  &  &  &  & 1 & 1 \cr
{} &  &  &  &  &  &  &  &  & 1 &  &  &  &  & 1 & 1 \cr
{} &  &  &  &  &  &  &  &  &  &  & 1 &  &  & 1 & 1 \cr
{} &  &  &  &  &  &  &  &  &  &  &  &  & 1 & 1 & 1 \cr
{} &  &  &  &  &  &  &  &  &  &  &  &  &  & 2 & 1 \cr
}.
\]
It follows readily that the nontrivial elementary divisors of the Smith normal form of $T$ are all equal to $1$. The torsion subgroup of the group generated by the tails is thus trivial, thereby showing $\B_0(G) = 1$.


\item \label{number:89} 
Let the group $G$ be the representative of this family given by the presentation
\[
\begin{aligned}
\langle g_{1}, \,g_{2}, \,g_{3}, \,g_{4}, \,g_{5}, \,g_{6}, \,g_{7} & \mid & g_{1}^{2} &= 1, \\ 
 & & g_{2}^{2} &= 1, & [g_{2}, g_{1}]  &= g_{4}, \\ 
 & & g_{3}^{2} &= 1, & [g_{3}, g_{1}]  &= g_{5}, & [g_{3}, g_{2}]  &= g_{6}, \\ 
 & & g_{4}^{2} &= g_{7}, & [g_{4}, g_{1}]  &= g_{7}, & [g_{4}, g_{2}]  &= g_{7}, & [g_{4}, g_{3}]  &= g_{7}, \\ 
 & & g_{5}^{2} &= g_{7}, & [g_{5}, g_{1}]  &= g_{7}, & [g_{5}, g_{3}]  &= g_{7}, \\ 
 & & g_{6}^{2} &= 1, & [g_{6}, g_{1}]  &= g_{7}, \\ 
 & & g_{7}^{2} &= 1\rangle. \\ 
\end{aligned}
\]
We add 16 tails to the presentation as to form a quotient of the universal central extension of the system: 
$g_{1}^{2} =  t_{1}$,
$g_{2}^{2} =  t_{2}$,
$[g_{2}, g_{1}] = g_{4} t_{3}$,
$g_{3}^{2} =  t_{4}$,
$[g_{3}, g_{1}] = g_{5} t_{5}$,
$[g_{3}, g_{2}] = g_{6} t_{6}$,
$g_{4}^{2} = g_{7} t_{7}$,
$[g_{4}, g_{1}] = g_{7} t_{8}$,
$[g_{4}, g_{2}] = g_{7} t_{9}$,
$[g_{4}, g_{3}] = g_{7} t_{10}$,
$g_{5}^{2} = g_{7} t_{11}$,
$[g_{5}, g_{1}] = g_{7} t_{12}$,
$[g_{5}, g_{3}] = g_{7} t_{13}$,
$g_{6}^{2} =  t_{14}$,
$[g_{6}, g_{1}] = g_{7} t_{15}$,
$g_{7}^{2} =  t_{16}$.
Carrying out consistency checks gives the following relations between the tails:
\[
\begin{aligned}
g_{3}(g_{2} g_{1}) & = (g_{3} g_{2}) g_{1}  & \Longrightarrow & & t_{10}t_{15}t_{16} & = 1 \\
g_{6}^2 g_{1} & = g_{6} (g_{6} g_{1})& \Longrightarrow & & t_{15}^{2}t_{16} & = 1 \\
g_{5}^2 g_{3} & = g_{5} (g_{5} g_{3})& \Longrightarrow & & t_{13}^{2}t_{16} & = 1 \\
g_{5}^2 g_{1} & = g_{5} (g_{5} g_{1})& \Longrightarrow & & t_{12}^{2}t_{16} & = 1 \\
g_{4}^2 g_{2} & = g_{4} (g_{4} g_{2})& \Longrightarrow & & t_{9}^{2}t_{16} & = 1 \\
g_{4}^2 g_{1} & = g_{4} (g_{4} g_{1})& \Longrightarrow & & t_{8}^{2}t_{16} & = 1 \\
g_{3}^2 g_{2} & = g_{3} (g_{3} g_{2})& \Longrightarrow & & t_{6}^{2}t_{14} & = 1 \\
g_{3}^2 g_{1} & = g_{3} (g_{3} g_{1})& \Longrightarrow & & t_{5}^{2}t_{11}t_{13}t_{16} & = 1 \\
g_{2}^2 g_{1} & = g_{2} (g_{2} g_{1})& \Longrightarrow & & t_{3}^{2}t_{7}t_{9}t_{16} & = 1 \\
g_{3} g_{1}^{2} & = (g_{3} g_{1}) g_{1}& \Longrightarrow & & t_{5}^{2}t_{11}t_{12}t_{16} & = 1 \\
g_{2} g_{1}^{2} & = (g_{2} g_{1}) g_{1}& \Longrightarrow & & t_{3}^{2}t_{7}t_{8}t_{16} & = 1 \\
\end{aligned}
\]
Scanning through the conjugacy class representatives of $G$ and the generators of their centralizers, we obtain the following relations induced on the tails:
\[
\begin{aligned}
{[g_{5} g_{6} g_{7} , \, g_{1} ]}_G & = 1 & \Longrightarrow & & t_{12}t_{15}t_{16} & = 1 \\
{[g_{4} g_{7} , \, g_{1} g_{3} ]}_G & = 1 & \Longrightarrow & & t_{8}t_{10}t_{16} & = 1 \\
\end{aligned}
\]
Collecting the coefficients of these relations into a matrix yields
\[
T = \bordermatrix{
{} & t_{1} & t_{2} & t_{3} & t_{4} & t_{5} & t_{6} & t_{7} & t_{8} & t_{9} & t_{10} & t_{11} & t_{12} & t_{13} & t_{14} & t_{15} & t_{16} \cr
{} &  &  & 2 &  &  &  & 1 &  &  &  &  &  &  &  & 1 & 1 \cr
{} &  &  &  &  & 2 &  &  &  &  &  & 1 &  &  &  & 1 & 1 \cr
{} &  &  &  &  &  & 2 &  &  &  &  &  &  &  & 1 &  &  \cr
{} &  &  &  &  &  &  &  & 1 &  &  &  &  &  &  & 1 & 1 \cr
{} &  &  &  &  &  &  &  &  & 1 &  &  &  &  &  & 1 & 1 \cr
{} &  &  &  &  &  &  &  &  &  & 1 &  &  &  &  & 1 & 1 \cr
{} &  &  &  &  &  &  &  &  &  &  &  & 1 &  &  & 1 & 1 \cr
{} &  &  &  &  &  &  &  &  &  &  &  &  & 1 &  & 1 & 1 \cr
{} &  &  &  &  &  &  &  &  &  &  &  &  &  &  & 2 & 1 \cr
}.
\]
It follows readily that the nontrivial elementary divisors of the Smith normal form of $T$ are all equal to $1$. The torsion subgroup of the group generated by the tails is thus trivial, thereby showing $\B_0(G) = 1$.


\item \label{number:90} 
Let the group $G$ be the representative of this family given by the presentation
\[
\begin{aligned}
\langle g_{1}, \,g_{2}, \,g_{3}, \,g_{4}, \,g_{5}, \,g_{6}, \,g_{7} & \mid & g_{1}^{2} &= g_{5}, \\ 
 & & g_{2}^{2} &= 1, & [g_{2}, g_{1}]  &= g_{4}, \\ 
 & & g_{3}^{2} &= g_{4}, & [g_{3}, g_{1}]  &= g_{5}, & [g_{3}, g_{2}]  &= g_{6}, \\ 
 & & g_{4}^{2} &= g_{7}, & [g_{4}, g_{2}]  &= g_{7}, \\ 
 & & g_{5}^{2} &= g_{7}, & [g_{5}, g_{2}]  &= g_{7}, & [g_{5}, g_{3}]  &= g_{7}, \\ 
 & & g_{6}^{2} &= 1, & [g_{6}, g_{1}]  &= g_{7}, & [g_{6}, g_{3}]  &= g_{7}, \\ 
 & & g_{7}^{2} &= 1\rangle. \\ 
\end{aligned}
\]
We add 15 tails to the presentation as to form a quotient of the universal central extension of the system: 
$g_{1}^{2} = g_{5} t_{1}$,
$g_{2}^{2} =  t_{2}$,
$[g_{2}, g_{1}] = g_{4} t_{3}$,
$g_{3}^{2} = g_{4} t_{4}$,
$[g_{3}, g_{1}] = g_{5} t_{5}$,
$[g_{3}, g_{2}] = g_{6} t_{6}$,
$g_{4}^{2} = g_{7} t_{7}$,
$[g_{4}, g_{2}] = g_{7} t_{8}$,
$g_{5}^{2} = g_{7} t_{9}$,
$[g_{5}, g_{2}] = g_{7} t_{10}$,
$[g_{5}, g_{3}] = g_{7} t_{11}$,
$g_{6}^{2} =  t_{12}$,
$[g_{6}, g_{1}] = g_{7} t_{13}$,
$[g_{6}, g_{3}] = g_{7} t_{14}$,
$g_{7}^{2} =  t_{15}$.
Carrying out consistency checks gives the following relations between the tails:
\[
\begin{aligned}
g_{3}(g_{2} g_{1}) & = (g_{3} g_{2}) g_{1}  & \Longrightarrow & & t_{10}^{-1}t_{13} & = 1 \\
g_{6}^2 g_{3} & = g_{6} (g_{6} g_{3})& \Longrightarrow & & t_{14}^{2}t_{15} & = 1 \\
g_{6}^2 g_{1} & = g_{6} (g_{6} g_{1})& \Longrightarrow & & t_{13}^{2}t_{15} & = 1 \\
g_{5}^2 g_{3} & = g_{5} (g_{5} g_{3})& \Longrightarrow & & t_{11}^{2}t_{15} & = 1 \\
g_{4}^2 g_{2} & = g_{4} (g_{4} g_{2})& \Longrightarrow & & t_{8}^{2}t_{15} & = 1 \\
g_{3}^2 g_{2} & = g_{3} (g_{3} g_{2})& \Longrightarrow & & t_{6}^{2}t_{8}^{-1}t_{12}t_{14} & = 1 \\
g_{3}^2 g_{1} & = g_{3} (g_{3} g_{1})& \Longrightarrow & & t_{5}^{2}t_{9}t_{11}t_{15} & = 1 \\
g_{2}^2 g_{1} & = g_{2} (g_{2} g_{1})& \Longrightarrow & & t_{3}^{2}t_{7}t_{8}t_{15} & = 1 \\
g_{3} g_{2}^{2} & = (g_{3} g_{2}) g_{2}& \Longrightarrow & & t_{6}^{2}t_{12} & = 1 \\
g_{2} g_{1}^{2} & = (g_{2} g_{1}) g_{1}& \Longrightarrow & & t_{3}^{2}t_{7}t_{10}t_{15} & = 1 \\
\end{aligned}
\]
Scanning through the conjugacy class representatives of $G$ and the generators of their centralizers, we obtain the following relations induced on the tails:
\[
\begin{aligned}
{[g_{5} g_{7} , \, g_{2} g_{3} g_{4} ]}_G & = 1 & \Longrightarrow & & t_{10}t_{11}t_{15} & = 1 \\
\end{aligned}
\]
Collecting the coefficients of these relations into a matrix yields
\[
T = \bordermatrix{
{} & t_{1} & t_{2} & t_{3} & t_{4} & t_{5} & t_{6} & t_{7} & t_{8} & t_{9} & t_{10} & t_{11} & t_{12} & t_{13} & t_{14} & t_{15} \cr
{} &  &  & 2 &  &  &  & 1 &  &  &  &  &  &  & 1 & 1 \cr
{} &  &  &  &  & 2 &  &  &  & 1 &  &  &  &  & 1 & 1 \cr
{} &  &  &  &  &  & 2 &  &  &  &  &  & 1 &  &  &  \cr
{} &  &  &  &  &  &  &  & 1 &  &  &  &  &  & 1 & 1 \cr
{} &  &  &  &  &  &  &  &  &  & 1 &  &  &  & 1 & 1 \cr
{} &  &  &  &  &  &  &  &  &  &  & 1 &  &  & 1 & 1 \cr
{} &  &  &  &  &  &  &  &  &  &  &  &  & 1 & 1 & 1 \cr
{} &  &  &  &  &  &  &  &  &  &  &  &  &  & 2 & 1 \cr
}.
\]
It follows readily that the nontrivial elementary divisors of the Smith normal form of $T$ are all equal to $1$. The torsion subgroup of the group generated by the tails is thus trivial, thereby showing $\B_0(G) = 1$.


\item \label{number:91} 
Let the group $G$ be the representative of this family given by the presentation
\[
\begin{aligned}
\langle g_{1}, \,g_{2}, \,g_{3}, \,g_{4}, \,g_{5}, \,g_{6}, \,g_{7} & \mid & g_{1}^{2} &= g_{4}, \\ 
 & & g_{2}^{2} &= g_{4}, & [g_{2}, g_{1}]  &= g_{4}, \\ 
 & & g_{3}^{2} &= 1, & [g_{3}, g_{1}]  &= g_{5}, & [g_{3}, g_{2}]  &= g_{6}, \\ 
 & & g_{4}^{2} &= 1, & [g_{4}, g_{3}]  &= g_{7}, \\ 
 & & g_{5}^{2} &= 1, & [g_{5}, g_{1}]  &= g_{7}, \\ 
 & & g_{6}^{2} &= 1, & [g_{6}, g_{1}]  &= g_{7}, & [g_{6}, g_{2}]  &= g_{7}, \\ 
 & & g_{7}^{2} &= 1\rangle. \\ 
\end{aligned}
\]
We add 14 tails to the presentation as to form a quotient of the universal central extension of the system: 
$g_{1}^{2} = g_{4} t_{1}$,
$g_{2}^{2} = g_{4} t_{2}$,
$[g_{2}, g_{1}] = g_{4} t_{3}$,
$g_{3}^{2} =  t_{4}$,
$[g_{3}, g_{1}] = g_{5} t_{5}$,
$[g_{3}, g_{2}] = g_{6} t_{6}$,
$g_{4}^{2} =  t_{7}$,
$[g_{4}, g_{3}] = g_{7} t_{8}$,
$g_{5}^{2} =  t_{9}$,
$[g_{5}, g_{1}] = g_{7} t_{10}$,
$g_{6}^{2} =  t_{11}$,
$[g_{6}, g_{1}] = g_{7} t_{12}$,
$[g_{6}, g_{2}] = g_{7} t_{13}$,
$g_{7}^{2} =  t_{14}$.
Carrying out consistency checks gives the following relations between the tails:
\[
\begin{aligned}
g_{3}(g_{2} g_{1}) & = (g_{3} g_{2}) g_{1}  & \Longrightarrow & & t_{8}t_{12}t_{14} & = 1 \\
g_{6}^2 g_{2} & = g_{6} (g_{6} g_{2})& \Longrightarrow & & t_{13}^{2}t_{14} & = 1 \\
g_{6}^2 g_{1} & = g_{6} (g_{6} g_{1})& \Longrightarrow & & t_{12}^{2}t_{14} & = 1 \\
g_{5}^2 g_{1} & = g_{5} (g_{5} g_{1})& \Longrightarrow & & t_{10}^{2}t_{14} & = 1 \\
g_{3}^2 g_{2} & = g_{3} (g_{3} g_{2})& \Longrightarrow & & t_{6}^{2}t_{11} & = 1 \\
g_{3}^2 g_{1} & = g_{3} (g_{3} g_{1})& \Longrightarrow & & t_{5}^{2}t_{9} & = 1 \\
g_{2}^2 g_{1} & = g_{2} (g_{2} g_{1})& \Longrightarrow & & t_{3}^{2}t_{7} & = 1 \\
g_{3} g_{2}^{2} & = (g_{3} g_{2}) g_{2}& \Longrightarrow & & t_{6}^{2}t_{8}t_{11}t_{13}t_{14} & = 1 \\
g_{3} g_{1}^{2} & = (g_{3} g_{1}) g_{1}& \Longrightarrow & & t_{5}^{2}t_{8}t_{9}t_{10}t_{14} & = 1 \\
\end{aligned}
\]
Scanning through the conjugacy class representatives of $G$ and the generators of their centralizers, we see that no new relations are imposed.
Collecting the coefficients of these relations into a matrix yields
\[
T = \bordermatrix{
{} & t_{1} & t_{2} & t_{3} & t_{4} & t_{5} & t_{6} & t_{7} & t_{8} & t_{9} & t_{10} & t_{11} & t_{12} & t_{13} & t_{14} \cr
{} &  &  & 2 &  &  &  & 1 &  &  &  &  &  &  &  \cr
{} &  &  &  &  & 2 &  &  &  & 1 &  &  &  &  &  \cr
{} &  &  &  &  &  & 2 &  &  &  &  & 1 &  &  &  \cr
{} &  &  &  &  &  &  &  & 1 &  &  &  &  & 1 & 1 \cr
{} &  &  &  &  &  &  &  &  &  & 1 &  &  & 1 & 1 \cr
{} &  &  &  &  &  &  &  &  &  &  &  & 1 & 1 & 1 \cr
{} &  &  &  &  &  &  &  &  &  &  &  &  & 2 & 1 \cr
}.
\]
It follows readily that the nontrivial elementary divisors of the Smith normal form of $T$ are all equal to $1$. The torsion subgroup of the group generated by the tails is thus trivial, thereby showing $\B_0(G) = 1$.


\item \label{number:92} 
Let the group $G$ be the representative of this family given by the presentation
\[
\begin{aligned}
\langle g_{1}, \,g_{2}, \,g_{3}, \,g_{4}, \,g_{5}, \,g_{6}, \,g_{7} & \mid & g_{1}^{2} &= 1, \\ 
 & & g_{2}^{2} &= 1, & [g_{2}, g_{1}]  &= g_{4}, \\ 
 & & g_{3}^{2} &= 1, & [g_{3}, g_{1}]  &= g_{5}, \\ 
 & & g_{4}^{2} &= g_{7}, & [g_{4}, g_{1}]  &= g_{7}, & [g_{4}, g_{2}]  &= g_{7}, & [g_{4}, g_{3}]  &= g_{6}, \\ 
 & & g_{5}^{2} &= g_{7}, & [g_{5}, g_{1}]  &= g_{7}, & [g_{5}, g_{2}]  &= g_{6}g_{7}, & [g_{5}, g_{3}]  &= g_{7}, & [g_{5}, g_{4}]  &= g_{7}, \\ 
 & & g_{6}^{2} &= 1, & [g_{6}, g_{1}]  &= g_{7}, \\ 
 & & g_{7}^{2} &= 1\rangle. \\ 
\end{aligned}
\]
We add 17 tails to the presentation as to form a quotient of the universal central extension of the system: 
$g_{1}^{2} =  t_{1}$,
$g_{2}^{2} =  t_{2}$,
$[g_{2}, g_{1}] = g_{4} t_{3}$,
$g_{3}^{2} =  t_{4}$,
$[g_{3}, g_{1}] = g_{5} t_{5}$,
$g_{4}^{2} = g_{7} t_{6}$,
$[g_{4}, g_{1}] = g_{7} t_{7}$,
$[g_{4}, g_{2}] = g_{7} t_{8}$,
$[g_{4}, g_{3}] = g_{6} t_{9}$,
$g_{5}^{2} = g_{7} t_{10}$,
$[g_{5}, g_{1}] = g_{7} t_{11}$,
$[g_{5}, g_{2}] = g_{6}g_{7} t_{12}$,
$[g_{5}, g_{3}] = g_{7} t_{13}$,
$[g_{5}, g_{4}] = g_{7} t_{14}$,
$g_{6}^{2} =  t_{15}$,
$[g_{6}, g_{1}] = g_{7} t_{16}$,
$g_{7}^{2} =  t_{17}$.
Carrying out consistency checks gives the following relations between the tails:
\[
\begin{aligned}
g_{5}(g_{2} g_{1}) & = (g_{5} g_{2}) g_{1}  & \Longrightarrow & & t_{14}^{-1}t_{16} & = 1 \\
g_{4}(g_{3} g_{1}) & = (g_{4} g_{3}) g_{1}  & \Longrightarrow & & t_{14}t_{16}t_{17} & = 1 \\
g_{3}(g_{2} g_{1}) & = (g_{3} g_{2}) g_{1}  & \Longrightarrow & & t_{9}t_{12}^{-1}t_{14}^{-1}t_{17}^{-1} & = 1 \\
g_{5}^2 g_{3} & = g_{5} (g_{5} g_{3})& \Longrightarrow & & t_{13}^{2}t_{17} & = 1 \\
g_{5}^2 g_{2} & = g_{5} (g_{5} g_{2})& \Longrightarrow & & t_{12}^{2}t_{15}t_{17} & = 1 \\
g_{5}^2 g_{1} & = g_{5} (g_{5} g_{1})& \Longrightarrow & & t_{11}^{2}t_{17} & = 1 \\
g_{4}^2 g_{2} & = g_{4} (g_{4} g_{2})& \Longrightarrow & & t_{8}^{2}t_{17} & = 1 \\
g_{4}^2 g_{1} & = g_{4} (g_{4} g_{1})& \Longrightarrow & & t_{7}^{2}t_{17} & = 1 \\
g_{3}^2 g_{1} & = g_{3} (g_{3} g_{1})& \Longrightarrow & & t_{5}^{2}t_{10}t_{13}t_{17} & = 1 \\
g_{2}^2 g_{1} & = g_{2} (g_{2} g_{1})& \Longrightarrow & & t_{3}^{2}t_{6}t_{8}t_{17} & = 1 \\
g_{3} g_{1}^{2} & = (g_{3} g_{1}) g_{1}& \Longrightarrow & & t_{5}^{2}t_{10}t_{11}t_{17} & = 1 \\
g_{2} g_{1}^{2} & = (g_{2} g_{1}) g_{1}& \Longrightarrow & & t_{3}^{2}t_{6}t_{7}t_{17} & = 1 \\
\end{aligned}
\]
Scanning through the conjugacy class representatives of $G$ and the generators of their centralizers, we obtain the following relations induced on the tails:
\[
\begin{aligned}
{[g_{5} g_{6} g_{7} , \, g_{1} ]}_G & = 1 & \Longrightarrow & & t_{11}t_{16}t_{17} & = 1 \\
{[g_{4} g_{6} g_{7} , \, g_{1} ]}_G & = 1 & \Longrightarrow & & t_{7}t_{16}t_{17} & = 1 \\
\end{aligned}
\]
Collecting the coefficients of these relations into a matrix yields
\[
T = \bordermatrix{
{} & t_{1} & t_{2} & t_{3} & t_{4} & t_{5} & t_{6} & t_{7} & t_{8} & t_{9} & t_{10} & t_{11} & t_{12} & t_{13} & t_{14} & t_{15} & t_{16} & t_{17} \cr
{} &  &  & 2 &  &  & 1 &  &  &  &  &  &  &  &  &  & 1 & 1 \cr
{} &  &  &  &  & 2 &  &  &  &  & 1 &  &  &  &  &  & 1 & 1 \cr
{} &  &  &  &  &  &  & 1 &  &  &  &  &  &  &  &  & 1 & 1 \cr
{} &  &  &  &  &  &  &  & 1 &  &  &  &  &  &  &  & 1 & 1 \cr
{} &  &  &  &  &  &  &  &  & 1 &  &  & 1 &  &  & 1 & 1 & 1 \cr
{} &  &  &  &  &  &  &  &  &  &  & 1 &  &  &  &  & 1 & 1 \cr
{} &  &  &  &  &  &  &  &  &  &  &  & 2 &  &  & 1 &  & 1 \cr
{} &  &  &  &  &  &  &  &  &  &  &  &  & 1 &  &  & 1 & 1 \cr
{} &  &  &  &  &  &  &  &  &  &  &  &  &  & 1 &  & 1 & 1 \cr
{} &  &  &  &  &  &  &  &  &  &  &  &  &  &  &  & 2 & 1 \cr
}.
\]
It follows readily that the nontrivial elementary divisors of the Smith normal form of $T$ are all equal to $1$. The torsion subgroup of the group generated by the tails is thus trivial, thereby showing $\B_0(G) = 1$.


\item \label{number:93} 
Let the group $G$ be the representative of this family given by the presentation
\[
\begin{aligned}
\langle g_{1}, \,g_{2}, \,g_{3}, \,g_{4}, \,g_{5}, \,g_{6}, \,g_{7} & \mid & g_{1}^{2} &= 1, \\ 
 & & g_{2}^{2} &= 1, & [g_{2}, g_{1}]  &= g_{4}, \\ 
 & & g_{3}^{2} &= 1, & [g_{3}, g_{1}]  &= g_{5}, & [g_{3}, g_{2}]  &= g_{7}, \\ 
 & & g_{4}^{2} &= g_{7}, & [g_{4}, g_{1}]  &= g_{7}, & [g_{4}, g_{2}]  &= g_{7}, & [g_{4}, g_{3}]  &= g_{6}, \\ 
 & & g_{5}^{2} &= g_{7}, & [g_{5}, g_{1}]  &= g_{7}, & [g_{5}, g_{2}]  &= g_{6}g_{7}, & [g_{5}, g_{3}]  &= g_{7}, & [g_{5}, g_{4}]  &= g_{7}, \\ 
 & & g_{6}^{2} &= 1, & [g_{6}, g_{1}]  &= g_{7}, \\ 
 & & g_{7}^{2} &= 1\rangle. \\ 
\end{aligned}
\]
We add 18 tails to the presentation as to form a quotient of the universal central extension of the system: 
$g_{1}^{2} =  t_{1}$,
$g_{2}^{2} =  t_{2}$,
$[g_{2}, g_{1}] = g_{4} t_{3}$,
$g_{3}^{2} =  t_{4}$,
$[g_{3}, g_{1}] = g_{5} t_{5}$,
$[g_{3}, g_{2}] = g_{7} t_{6}$,
$g_{4}^{2} = g_{7} t_{7}$,
$[g_{4}, g_{1}] = g_{7} t_{8}$,
$[g_{4}, g_{2}] = g_{7} t_{9}$,
$[g_{4}, g_{3}] = g_{6} t_{10}$,
$g_{5}^{2} = g_{7} t_{11}$,
$[g_{5}, g_{1}] = g_{7} t_{12}$,
$[g_{5}, g_{2}] = g_{6}g_{7} t_{13}$,
$[g_{5}, g_{3}] = g_{7} t_{14}$,
$[g_{5}, g_{4}] = g_{7} t_{15}$,
$g_{6}^{2} =  t_{16}$,
$[g_{6}, g_{1}] = g_{7} t_{17}$,
$g_{7}^{2} =  t_{18}$.
Carrying out consistency checks gives the following relations between the tails:
\[
\begin{aligned}
g_{5}(g_{2} g_{1}) & = (g_{5} g_{2}) g_{1}  & \Longrightarrow & & t_{15}^{-1}t_{17} & = 1 \\
g_{4}(g_{3} g_{1}) & = (g_{4} g_{3}) g_{1}  & \Longrightarrow & & t_{15}t_{17}t_{18} & = 1 \\
g_{3}(g_{2} g_{1}) & = (g_{3} g_{2}) g_{1}  & \Longrightarrow & & t_{10}t_{13}^{-1}t_{15}^{-1}t_{18}^{-1} & = 1 \\
g_{5}^2 g_{3} & = g_{5} (g_{5} g_{3})& \Longrightarrow & & t_{14}^{2}t_{18} & = 1 \\
g_{5}^2 g_{2} & = g_{5} (g_{5} g_{2})& \Longrightarrow & & t_{13}^{2}t_{16}t_{18} & = 1 \\
g_{5}^2 g_{1} & = g_{5} (g_{5} g_{1})& \Longrightarrow & & t_{12}^{2}t_{18} & = 1 \\
g_{4}^2 g_{2} & = g_{4} (g_{4} g_{2})& \Longrightarrow & & t_{9}^{2}t_{18} & = 1 \\
g_{4}^2 g_{1} & = g_{4} (g_{4} g_{1})& \Longrightarrow & & t_{8}^{2}t_{18} & = 1 \\
g_{3}^2 g_{2} & = g_{3} (g_{3} g_{2})& \Longrightarrow & & t_{6}^{2}t_{18} & = 1 \\
g_{3}^2 g_{1} & = g_{3} (g_{3} g_{1})& \Longrightarrow & & t_{5}^{2}t_{11}t_{14}t_{18} & = 1 \\
g_{2}^2 g_{1} & = g_{2} (g_{2} g_{1})& \Longrightarrow & & t_{3}^{2}t_{7}t_{9}t_{18} & = 1 \\
g_{3} g_{1}^{2} & = (g_{3} g_{1}) g_{1}& \Longrightarrow & & t_{5}^{2}t_{11}t_{12}t_{18} & = 1 \\
g_{2} g_{1}^{2} & = (g_{2} g_{1}) g_{1}& \Longrightarrow & & t_{3}^{2}t_{7}t_{8}t_{18} & = 1 \\
\end{aligned}
\]
Scanning through the conjugacy class representatives of $G$ and the generators of their centralizers, we obtain the following relations induced on the tails:
\[
\begin{aligned}
{[g_{5} g_{6} g_{7} , \, g_{1} ]}_G & = 1 & \Longrightarrow & & t_{12}t_{17}t_{18} & = 1 \\
{[g_{4} g_{6} g_{7} , \, g_{1} ]}_G & = 1 & \Longrightarrow & & t_{8}t_{17}t_{18} & = 1 \\
{[g_{3} g_{5} g_{6} g_{7} , \, g_{2} g_{4} g_{5} g_{7} ]}_G & = 1 & \Longrightarrow & & t_{6}t_{10}^{-1}t_{13}t_{14}^{-1}t_{15}t_{18} & = 1 \\
\end{aligned}
\]
Collecting the coefficients of these relations into a matrix yields
\[
T = \bordermatrix{
{} & t_{1} & t_{2} & t_{3} & t_{4} & t_{5} & t_{6} & t_{7} & t_{8} & t_{9} & t_{10} & t_{11} & t_{12} & t_{13} & t_{14} & t_{15} & t_{16} & t_{17} & t_{18} \cr
{} &  &  & 2 &  &  &  & 1 &  &  &  &  &  &  &  &  &  & 1 & 1 \cr
{} &  &  &  &  & 2 &  &  &  &  &  & 1 &  &  &  &  &  & 1 & 1 \cr
{} &  &  &  &  &  & 1 &  &  &  &  &  &  &  &  &  &  & 1 & 1 \cr
{} &  &  &  &  &  &  &  & 1 &  &  &  &  &  &  &  &  & 1 & 1 \cr
{} &  &  &  &  &  &  &  &  & 1 &  &  &  &  &  &  &  & 1 & 1 \cr
{} &  &  &  &  &  &  &  &  &  & 1 &  &  & 1 &  &  & 1 & 1 & 1 \cr
{} &  &  &  &  &  &  &  &  &  &  &  & 1 &  &  &  &  & 1 & 1 \cr
{} &  &  &  &  &  &  &  &  &  &  &  &  & 2 &  &  & 1 &  & 1 \cr
{} &  &  &  &  &  &  &  &  &  &  &  &  &  & 1 &  &  & 1 & 1 \cr
{} &  &  &  &  &  &  &  &  &  &  &  &  &  &  & 1 &  & 1 & 1 \cr
{} &  &  &  &  &  &  &  &  &  &  &  &  &  &  &  &  & 2 & 1 \cr
}.
\]
It follows readily that the nontrivial elementary divisors of the Smith normal form of $T$ are all equal to $1$. The torsion subgroup of the group generated by the tails is thus trivial, thereby showing $\B_0(G) = 1$.


\item \label{number:94} 
Let the group $G$ be the representative of this family given by the presentation
\[
\begin{aligned}
\langle g_{1}, \,g_{2}, \,g_{3}, \,g_{4}, \,g_{5}, \,g_{6}, \,g_{7} & \mid & g_{1}^{2} &= 1, \\ 
 & & g_{2}^{2} &= 1, & [g_{2}, g_{1}]  &= g_{4}, \\ 
 & & g_{3}^{2} &= 1, & [g_{3}, g_{1}]  &= g_{5}, & [g_{3}, g_{2}]  &= g_{7}, \\ 
 & & g_{4}^{2} &= 1, & [g_{4}, g_{3}]  &= g_{6}, \\ 
 & & g_{5}^{2} &= 1, & [g_{5}, g_{2}]  &= g_{6}g_{7}, & [g_{5}, g_{4}]  &= g_{7}, \\ 
 & & g_{6}^{2} &= 1, & [g_{6}, g_{1}]  &= g_{7}, \\ 
 & & g_{7}^{2} &= 1\rangle. \\ 
\end{aligned}
\]
We add 14 tails to the presentation as to form a quotient of the universal central extension of the system: 
$g_{1}^{2} =  t_{1}$,
$g_{2}^{2} =  t_{2}$,
$[g_{2}, g_{1}] = g_{4} t_{3}$,
$g_{3}^{2} =  t_{4}$,
$[g_{3}, g_{1}] = g_{5} t_{5}$,
$[g_{3}, g_{2}] = g_{7} t_{6}$,
$g_{4}^{2} =  t_{7}$,
$[g_{4}, g_{3}] = g_{6} t_{8}$,
$g_{5}^{2} =  t_{9}$,
$[g_{5}, g_{2}] = g_{6}g_{7} t_{10}$,
$[g_{5}, g_{4}] = g_{7} t_{11}$,
$g_{6}^{2} =  t_{12}$,
$[g_{6}, g_{1}] = g_{7} t_{13}$,
$g_{7}^{2} =  t_{14}$.
Carrying out consistency checks gives the following relations between the tails:
\[
\begin{aligned}
g_{5}(g_{2} g_{1}) & = (g_{5} g_{2}) g_{1}  & \Longrightarrow & & t_{11}^{-1}t_{13} & = 1 \\
g_{4}(g_{3} g_{1}) & = (g_{4} g_{3}) g_{1}  & \Longrightarrow & & t_{11}t_{13}t_{14} & = 1 \\
g_{3}(g_{2} g_{1}) & = (g_{3} g_{2}) g_{1}  & \Longrightarrow & & t_{8}t_{10}^{-1}t_{11}^{-1}t_{14}^{-1} & = 1 \\
g_{5}^2 g_{2} & = g_{5} (g_{5} g_{2})& \Longrightarrow & & t_{10}^{2}t_{12}t_{14} & = 1 \\
g_{3}^2 g_{2} & = g_{3} (g_{3} g_{2})& \Longrightarrow & & t_{6}^{2}t_{14} & = 1 \\
g_{3}^2 g_{1} & = g_{3} (g_{3} g_{1})& \Longrightarrow & & t_{5}^{2}t_{9} & = 1 \\
g_{2}^2 g_{1} & = g_{2} (g_{2} g_{1})& \Longrightarrow & & t_{3}^{2}t_{7} & = 1 \\
\end{aligned}
\]
Scanning through the conjugacy class representatives of $G$ and the generators of their centralizers, we obtain the following relations induced on the tails:
\[
\begin{aligned}
{[g_{3} g_{4} g_{5} g_{6} g_{7} , \, g_{2} g_{4} g_{5} ]}_G & = 1 & \Longrightarrow & & t_{6}t_{8}^{-1}t_{10}t_{14} & = 1 \\
\end{aligned}
\]
Collecting the coefficients of these relations into a matrix yields
\[
T = \bordermatrix{
{} & t_{1} & t_{2} & t_{3} & t_{4} & t_{5} & t_{6} & t_{7} & t_{8} & t_{9} & t_{10} & t_{11} & t_{12} & t_{13} & t_{14} \cr
{} &  &  & 2 &  &  &  & 1 &  &  &  &  &  &  &  \cr
{} &  &  &  &  & 2 &  &  &  & 1 &  &  &  &  &  \cr
{} &  &  &  &  &  & 1 &  &  &  &  &  &  & 1 & 1 \cr
{} &  &  &  &  &  &  &  & 1 &  & 1 &  & 1 & 1 & 1 \cr
{} &  &  &  &  &  &  &  &  &  & 2 &  & 1 &  & 1 \cr
{} &  &  &  &  &  &  &  &  &  &  & 1 &  & 1 & 1 \cr
{} &  &  &  &  &  &  &  &  &  &  &  &  & 2 & 1 \cr
}.
\]
It follows readily that the nontrivial elementary divisors of the Smith normal form of $T$ are all equal to $1$. The torsion subgroup of the group generated by the tails is thus trivial, thereby showing $\B_0(G) = 1$.


\item \label{number:95} 
Let the group $G$ be the representative of this family given by the presentation
\[
\begin{aligned}
\langle g_{1}, \,g_{2}, \,g_{3}, \,g_{4}, \,g_{5}, \,g_{6}, \,g_{7} & \mid & g_{1}^{2} &= 1, \\ 
 & & g_{2}^{2} &= 1, & [g_{2}, g_{1}]  &= g_{4}, \\ 
 & & g_{3}^{2} &= 1, & [g_{3}, g_{1}]  &= g_{5}, \\ 
 & & g_{4}^{2} &= 1, & [g_{4}, g_{3}]  &= g_{6}, \\ 
 & & g_{5}^{2} &= 1, & [g_{5}, g_{2}]  &= g_{6}g_{7}, & [g_{5}, g_{4}]  &= g_{7}, \\ 
 & & g_{6}^{2} &= 1, & [g_{6}, g_{1}]  &= g_{7}, \\ 
 & & g_{7}^{2} &= 1\rangle. \\ 
\end{aligned}
\]
We add 13 tails to the presentation as to form a quotient of the universal central extension of the system: 
$g_{1}^{2} =  t_{1}$,
$g_{2}^{2} =  t_{2}$,
$[g_{2}, g_{1}] = g_{4} t_{3}$,
$g_{3}^{2} =  t_{4}$,
$[g_{3}, g_{1}] = g_{5} t_{5}$,
$g_{4}^{2} =  t_{6}$,
$[g_{4}, g_{3}] = g_{6} t_{7}$,
$g_{5}^{2} =  t_{8}$,
$[g_{5}, g_{2}] = g_{6}g_{7} t_{9}$,
$[g_{5}, g_{4}] = g_{7} t_{10}$,
$g_{6}^{2} =  t_{11}$,
$[g_{6}, g_{1}] = g_{7} t_{12}$,
$g_{7}^{2} =  t_{13}$.
Carrying out consistency checks gives the following relations between the tails:
\[
\begin{aligned}
g_{5}(g_{2} g_{1}) & = (g_{5} g_{2}) g_{1}  & \Longrightarrow & & t_{10}^{-1}t_{12} & = 1 \\
g_{4}(g_{3} g_{1}) & = (g_{4} g_{3}) g_{1}  & \Longrightarrow & & t_{10}t_{12}t_{13} & = 1 \\
g_{3}(g_{2} g_{1}) & = (g_{3} g_{2}) g_{1}  & \Longrightarrow & & t_{7}t_{9}^{-1}t_{10}^{-1}t_{13}^{-1} & = 1 \\
g_{5}^2 g_{2} & = g_{5} (g_{5} g_{2})& \Longrightarrow & & t_{9}^{2}t_{11}t_{13} & = 1 \\
g_{3}^2 g_{1} & = g_{3} (g_{3} g_{1})& \Longrightarrow & & t_{5}^{2}t_{8} & = 1 \\
g_{2}^2 g_{1} & = g_{2} (g_{2} g_{1})& \Longrightarrow & & t_{3}^{2}t_{6} & = 1 \\
\end{aligned}
\]
Scanning through the conjugacy class representatives of $G$ and the generators of their centralizers, we see that no new relations are imposed.
Collecting the coefficients of these relations into a matrix yields
\[
T = \bordermatrix{
{} & t_{1} & t_{2} & t_{3} & t_{4} & t_{5} & t_{6} & t_{7} & t_{8} & t_{9} & t_{10} & t_{11} & t_{12} & t_{13} \cr
{} &  &  & 2 &  &  & 1 &  &  &  &  &  &  &  \cr
{} &  &  &  &  & 2 &  &  & 1 &  &  &  &  &  \cr
{} &  &  &  &  &  &  & 1 &  & 1 &  & 1 & 1 & 1 \cr
{} &  &  &  &  &  &  &  &  & 2 &  & 1 &  & 1 \cr
{} &  &  &  &  &  &  &  &  &  & 1 &  & 1 & 1 \cr
{} &  &  &  &  &  &  &  &  &  &  &  & 2 & 1 \cr
}.
\]
It follows readily that the nontrivial elementary divisors of the Smith normal form of $T$ are all equal to $1$. The torsion subgroup of the group generated by the tails is thus trivial, thereby showing $\B_0(G) = 1$.


\item \label{number:96} 
Let the group $G$ be the representative of this family given by the presentation
\[
\begin{aligned}
\langle g_{1}, \,g_{2}, \,g_{3}, \,g_{4}, \,g_{5}, \,g_{6}, \,g_{7} & \mid & g_{1}^{2} &= g_{4}, \\ 
 & & g_{2}^{2} &= 1, & [g_{2}, g_{1}]  &= g_{3}, \\ 
 & & g_{3}^{2} &= g_{6}g_{7}, & [g_{3}, g_{1}]  &= g_{5}, & [g_{3}, g_{2}]  &= g_{6}, \\ 
 & & g_{4}^{2} &= 1, & [g_{4}, g_{2}]  &= g_{5}g_{6}g_{7}, & [g_{4}, g_{3}]  &= g_{7}, \\ 
 & & g_{5}^{2} &= g_{7}, & [g_{5}, g_{2}]  &= g_{7}, \\ 
 & & g_{6}^{2} &= g_{7}, & [g_{6}, g_{1}]  &= g_{7}, & [g_{6}, g_{2}]  &= g_{7}, \\ 
 & & g_{7}^{2} &= 1\rangle. \\ 
\end{aligned}
\]
We add 15 tails to the presentation as to form a quotient of the universal central extension of the system: 
$g_{1}^{2} = g_{4} t_{1}$,
$g_{2}^{2} =  t_{2}$,
$[g_{2}, g_{1}] = g_{3} t_{3}$,
$g_{3}^{2} = g_{6}g_{7} t_{4}$,
$[g_{3}, g_{1}] = g_{5} t_{5}$,
$[g_{3}, g_{2}] = g_{6} t_{6}$,
$g_{4}^{2} =  t_{7}$,
$[g_{4}, g_{2}] = g_{5}g_{6}g_{7} t_{8}$,
$[g_{4}, g_{3}] = g_{7} t_{9}$,
$g_{5}^{2} = g_{7} t_{10}$,
$[g_{5}, g_{2}] = g_{7} t_{11}$,
$g_{6}^{2} = g_{7} t_{12}$,
$[g_{6}, g_{1}] = g_{7} t_{13}$,
$[g_{6}, g_{2}] = g_{7} t_{14}$,
$g_{7}^{2} =  t_{15}$.
Carrying out consistency checks gives the following relations between the tails:
\[
\begin{aligned}
g_{4}(g_{2} g_{1}) & = (g_{4} g_{2}) g_{1}  & \Longrightarrow & & t_{9}^{-1}t_{13} & = 1 \\
g_{3}(g_{2} g_{1}) & = (g_{3} g_{2}) g_{1}  & \Longrightarrow & & t_{11}^{-1}t_{13} & = 1 \\
g_{6}^2 g_{2} & = g_{6} (g_{6} g_{2})& \Longrightarrow & & t_{14}^{2}t_{15} & = 1 \\
g_{6}^2 g_{1} & = g_{6} (g_{6} g_{1})& \Longrightarrow & & t_{13}^{2}t_{15} & = 1 \\
g_{4}^2 g_{2} & = g_{4} (g_{4} g_{2})& \Longrightarrow & & t_{8}^{2}t_{10}t_{12}t_{15}^{2} & = 1 \\
g_{3}^2 g_{2} & = g_{3} (g_{3} g_{2})& \Longrightarrow & & t_{6}^{2}t_{12}t_{14}^{-1} & = 1 \\
g_{3}^2 g_{1} & = g_{3} (g_{3} g_{1})& \Longrightarrow & & t_{5}^{2}t_{10}t_{13}^{-1} & = 1 \\
g_{2}^2 g_{1} & = g_{2} (g_{2} g_{1})& \Longrightarrow & & t_{3}^{2}t_{4}t_{6}t_{12}t_{15} & = 1 \\
g_{4} g_{2}^{2} & = (g_{4} g_{2}) g_{2}& \Longrightarrow & & t_{8}^{2}t_{10}t_{11}t_{12}t_{14}t_{15}^{3} & = 1 \\
g_{2} g_{1}^{2} & = (g_{2} g_{1}) g_{1}& \Longrightarrow & & t_{3}^{2}t_{4}t_{5}t_{8}t_{9}^{2}t_{10}t_{12}t_{15}^{3} & = 1 \\
\end{aligned}
\]
Scanning through the conjugacy class representatives of $G$ and the generators of their centralizers, we see that no new relations are imposed.
Collecting the coefficients of these relations into a matrix yields
\[
T = \bordermatrix{
{} & t_{1} & t_{2} & t_{3} & t_{4} & t_{5} & t_{6} & t_{7} & t_{8} & t_{9} & t_{10} & t_{11} & t_{12} & t_{13} & t_{14} & t_{15} \cr
{} &  &  & 2 & 1 &  & 1 &  &  &  &  &  & 1 &  &  & 1 \cr
{} &  &  &  &  & 1 & 1 &  & 1 &  & 1 &  & 1 &  & 1 & 2 \cr
{} &  &  &  &  &  & 2 &  &  &  &  &  & 1 &  & 1 & 1 \cr
{} &  &  &  &  &  &  &  & 2 &  & 1 &  & 1 &  &  & 2 \cr
{} &  &  &  &  &  &  &  &  & 1 &  &  &  &  & 1 & 1 \cr
{} &  &  &  &  &  &  &  &  &  &  & 1 &  &  & 1 & 1 \cr
{} &  &  &  &  &  &  &  &  &  &  &  &  & 1 & 1 & 1 \cr
{} &  &  &  &  &  &  &  &  &  &  &  &  &  & 2 & 1 \cr
}.
\]
It follows readily that the nontrivial elementary divisors of the Smith normal form of $T$ are all equal to $1$. The torsion subgroup of the group generated by the tails is thus trivial, thereby showing $\B_0(G) = 1$.


\item \label{number:97} 
Let the group $G$ be the representative of this family given by the presentation
\[
\begin{aligned}
\langle g_{1}, \,g_{2}, \,g_{3}, \,g_{4}, \,g_{5}, \,g_{6}, \,g_{7} & \mid & g_{1}^{2} &= g_{4}, \\ 
 & & g_{2}^{2} &= g_{6}, & [g_{2}, g_{1}]  &= g_{3}, \\ 
 & & g_{3}^{2} &= g_{6}g_{7}, & [g_{3}, g_{1}]  &= g_{5}, & [g_{3}, g_{2}]  &= g_{6}, \\ 
 & & g_{4}^{2} &= g_{5}g_{6}, & [g_{4}, g_{2}]  &= g_{5}g_{6}, \\ 
 & & g_{5}^{2} &= g_{7}, & [g_{5}, g_{1}]  &= g_{7}, & [g_{5}, g_{2}]  &= g_{7}, \\ 
 & & g_{6}^{2} &= 1, & [g_{6}, g_{1}]  &= g_{7}, \\ 
 & & g_{7}^{2} &= 1\rangle. \\ 
\end{aligned}
\]
We add 14 tails to the presentation as to form a quotient of the universal central extension of the system: 
$g_{1}^{2} = g_{4} t_{1}$,
$g_{2}^{2} = g_{6} t_{2}$,
$[g_{2}, g_{1}] = g_{3} t_{3}$,
$g_{3}^{2} = g_{6}g_{7} t_{4}$,
$[g_{3}, g_{1}] = g_{5} t_{5}$,
$[g_{3}, g_{2}] = g_{6} t_{6}$,
$g_{4}^{2} = g_{5}g_{6} t_{7}$,
$[g_{4}, g_{2}] = g_{5}g_{6} t_{8}$,
$g_{5}^{2} = g_{7} t_{9}$,
$[g_{5}, g_{1}] = g_{7} t_{10}$,
$[g_{5}, g_{2}] = g_{7} t_{11}$,
$g_{6}^{2} =  t_{12}$,
$[g_{6}, g_{1}] = g_{7} t_{13}$,
$g_{7}^{2} =  t_{14}$.
Carrying out consistency checks gives the following relations between the tails:
\[
\begin{aligned}
g_{4}(g_{2} g_{1}) & = (g_{4} g_{2}) g_{1}  & \Longrightarrow & & t_{10}t_{13}t_{14} & = 1 \\
g_{3}(g_{2} g_{1}) & = (g_{3} g_{2}) g_{1}  & \Longrightarrow & & t_{11}^{-1}t_{13} & = 1 \\
g_{6}^2 g_{1} & = g_{6} (g_{6} g_{1})& \Longrightarrow & & t_{13}^{2}t_{14} & = 1 \\
g_{4}^2 g_{2} & = g_{4} (g_{4} g_{2})& \Longrightarrow & & t_{8}^{2}t_{9}t_{11}^{-1}t_{12} & = 1 \\
g_{3}^2 g_{2} & = g_{3} (g_{3} g_{2})& \Longrightarrow & & t_{6}^{2}t_{12} & = 1 \\
g_{3}^2 g_{1} & = g_{3} (g_{3} g_{1})& \Longrightarrow & & t_{5}^{2}t_{9}t_{13}^{-1} & = 1 \\
g_{2}^2 g_{1} & = g_{2} (g_{2} g_{1})& \Longrightarrow & & t_{3}^{2}t_{4}t_{6}t_{12}t_{13}^{-1} & = 1 \\
g_{2} g_{1}^{2} & = (g_{2} g_{1}) g_{1}& \Longrightarrow & & t_{3}^{2}t_{4}t_{5}t_{8}t_{9}t_{12}t_{14} & = 1 \\
\end{aligned}
\]
Scanning through the conjugacy class representatives of $G$ and the generators of their centralizers, we see that no new relations are imposed.
Collecting the coefficients of these relations into a matrix yields
\[
T = \bordermatrix{
{} & t_{1} & t_{2} & t_{3} & t_{4} & t_{5} & t_{6} & t_{7} & t_{8} & t_{9} & t_{10} & t_{11} & t_{12} & t_{13} & t_{14} \cr
{} &  &  & 2 & 1 &  & 1 &  &  &  &  &  & 1 & 1 & 1 \cr
{} &  &  &  &  & 1 & 1 &  & 1 & 1 &  &  & 1 & 1 & 1 \cr
{} &  &  &  &  &  & 2 &  &  &  &  &  & 1 &  &  \cr
{} &  &  &  &  &  &  &  & 2 & 1 &  &  & 1 & 1 & 1 \cr
{} &  &  &  &  &  &  &  &  &  & 1 &  &  & 1 & 1 \cr
{} &  &  &  &  &  &  &  &  &  &  & 1 &  & 1 & 1 \cr
{} &  &  &  &  &  &  &  &  &  &  &  &  & 2 & 1 \cr
}.
\]
It follows readily that the nontrivial elementary divisors of the Smith normal form of $T$ are all equal to $1$. The torsion subgroup of the group generated by the tails is thus trivial, thereby showing $\B_0(G) = 1$.


\item \label{number:98} 
Let the group $G$ be the representative of this family given by the presentation
\[
\begin{aligned}
\langle g_{1}, \,g_{2}, \,g_{3}, \,g_{4}, \,g_{5}, \,g_{6}, \,g_{7} & \mid & g_{1}^{2} &= g_{4}, \\ 
 & & g_{2}^{2} &= 1, & [g_{2}, g_{1}]  &= g_{3}, \\ 
 & & g_{3}^{2} &= g_{6}, & [g_{3}, g_{1}]  &= g_{5}, & [g_{3}, g_{2}]  &= g_{6}, \\ 
 & & g_{4}^{2} &= g_{5}g_{6}, & [g_{4}, g_{2}]  &= g_{5}g_{6}g_{7}, \\ 
 & & g_{5}^{2} &= g_{7}, & [g_{5}, g_{1}]  &= g_{7}, & [g_{5}, g_{2}]  &= g_{7}, \\ 
 & & g_{6}^{2} &= 1, & [g_{6}, g_{1}]  &= g_{7}, \\ 
 & & g_{7}^{2} &= 1\rangle. \\ 
\end{aligned}
\]
We add 14 tails to the presentation as to form a quotient of the universal central extension of the system: 
$g_{1}^{2} = g_{4} t_{1}$,
$g_{2}^{2} =  t_{2}$,
$[g_{2}, g_{1}] = g_{3} t_{3}$,
$g_{3}^{2} = g_{6} t_{4}$,
$[g_{3}, g_{1}] = g_{5} t_{5}$,
$[g_{3}, g_{2}] = g_{6} t_{6}$,
$g_{4}^{2} = g_{5}g_{6} t_{7}$,
$[g_{4}, g_{2}] = g_{5}g_{6}g_{7} t_{8}$,
$g_{5}^{2} = g_{7} t_{9}$,
$[g_{5}, g_{1}] = g_{7} t_{10}$,
$[g_{5}, g_{2}] = g_{7} t_{11}$,
$g_{6}^{2} =  t_{12}$,
$[g_{6}, g_{1}] = g_{7} t_{13}$,
$g_{7}^{2} =  t_{14}$.
Carrying out consistency checks gives the following relations between the tails:
\[
\begin{aligned}
g_{4}(g_{2} g_{1}) & = (g_{4} g_{2}) g_{1}  & \Longrightarrow & & t_{10}t_{13}t_{14} & = 1 \\
g_{3}(g_{2} g_{1}) & = (g_{3} g_{2}) g_{1}  & \Longrightarrow & & t_{11}^{-1}t_{13} & = 1 \\
g_{6}^2 g_{1} & = g_{6} (g_{6} g_{1})& \Longrightarrow & & t_{13}^{2}t_{14} & = 1 \\
g_{4}^2 g_{2} & = g_{4} (g_{4} g_{2})& \Longrightarrow & & t_{8}^{2}t_{9}t_{11}^{-1}t_{12}t_{14} & = 1 \\
g_{3}^2 g_{2} & = g_{3} (g_{3} g_{2})& \Longrightarrow & & t_{6}^{2}t_{12} & = 1 \\
g_{3}^2 g_{1} & = g_{3} (g_{3} g_{1})& \Longrightarrow & & t_{5}^{2}t_{9}t_{13}^{-1} & = 1 \\
g_{2}^2 g_{1} & = g_{2} (g_{2} g_{1})& \Longrightarrow & & t_{3}^{2}t_{4}t_{6}t_{12} & = 1 \\
g_{2} g_{1}^{2} & = (g_{2} g_{1}) g_{1}& \Longrightarrow & & t_{3}^{2}t_{4}t_{5}t_{8}t_{9}t_{12}t_{14} & = 1 \\
\end{aligned}
\]
Scanning through the conjugacy class representatives of $G$ and the generators of their centralizers, we see that no new relations are imposed.
Collecting the coefficients of these relations into a matrix yields
\[
T = \bordermatrix{
{} & t_{1} & t_{2} & t_{3} & t_{4} & t_{5} & t_{6} & t_{7} & t_{8} & t_{9} & t_{10} & t_{11} & t_{12} & t_{13} & t_{14} \cr
{} &  &  & 2 & 1 &  & 1 &  &  &  &  &  & 1 &  &  \cr
{} &  &  &  &  & 1 & 1 &  & 1 & 1 &  &  & 1 &  & 1 \cr
{} &  &  &  &  &  & 2 &  &  &  &  &  & 1 &  &  \cr
{} &  &  &  &  &  &  &  & 2 & 1 &  &  & 1 & 1 & 2 \cr
{} &  &  &  &  &  &  &  &  &  & 1 &  &  & 1 & 1 \cr
{} &  &  &  &  &  &  &  &  &  &  & 1 &  & 1 & 1 \cr
{} &  &  &  &  &  &  &  &  &  &  &  &  & 2 & 1 \cr
}.
\]
It follows readily that the nontrivial elementary divisors of the Smith normal form of $T$ are all equal to $1$. The torsion subgroup of the group generated by the tails is thus trivial, thereby showing $\B_0(G) = 1$.


\item \label{number:99} 
Let the group $G$ be the representative of this family given by the presentation
\[
\begin{aligned}
\langle g_{1}, \,g_{2}, \,g_{3}, \,g_{4}, \,g_{5}, \,g_{6}, \,g_{7} & \mid & g_{1}^{2} &= g_{4}, \\ 
 & & g_{2}^{2} &= 1, & [g_{2}, g_{1}]  &= g_{3}, \\ 
 & & g_{3}^{2} &= g_{6}, & [g_{3}, g_{1}]  &= g_{5}, & [g_{3}, g_{2}]  &= g_{6}, \\ 
 & & g_{4}^{2} &= g_{5}, & [g_{4}, g_{2}]  &= g_{5}g_{6}g_{7}, & [g_{4}, g_{3}]  &= g_{7}, \\ 
 & & g_{5}^{2} &= g_{7}, & [g_{5}, g_{2}]  &= g_{7}, \\ 
 & & g_{6}^{2} &= 1, & [g_{6}, g_{1}]  &= g_{7}, \\ 
 & & g_{7}^{2} &= 1\rangle. \\ 
\end{aligned}
\]
We add 14 tails to the presentation as to form a quotient of the universal central extension of the system: 
$g_{1}^{2} = g_{4} t_{1}$,
$g_{2}^{2} =  t_{2}$,
$[g_{2}, g_{1}] = g_{3} t_{3}$,
$g_{3}^{2} = g_{6} t_{4}$,
$[g_{3}, g_{1}] = g_{5} t_{5}$,
$[g_{3}, g_{2}] = g_{6} t_{6}$,
$g_{4}^{2} = g_{5} t_{7}$,
$[g_{4}, g_{2}] = g_{5}g_{6}g_{7} t_{8}$,
$[g_{4}, g_{3}] = g_{7} t_{9}$,
$g_{5}^{2} = g_{7} t_{10}$,
$[g_{5}, g_{2}] = g_{7} t_{11}$,
$g_{6}^{2} =  t_{12}$,
$[g_{6}, g_{1}] = g_{7} t_{13}$,
$g_{7}^{2} =  t_{14}$.
Carrying out consistency checks gives the following relations between the tails:
\[
\begin{aligned}
g_{4}(g_{2} g_{1}) & = (g_{4} g_{2}) g_{1}  & \Longrightarrow & & t_{9}^{-1}t_{13} & = 1 \\
g_{3}(g_{2} g_{1}) & = (g_{3} g_{2}) g_{1}  & \Longrightarrow & & t_{11}^{-1}t_{13} & = 1 \\
g_{6}^2 g_{1} & = g_{6} (g_{6} g_{1})& \Longrightarrow & & t_{13}^{2}t_{14} & = 1 \\
g_{4}^2 g_{2} & = g_{4} (g_{4} g_{2})& \Longrightarrow & & t_{8}^{2}t_{10}t_{11}^{-1}t_{12}t_{14} & = 1 \\
g_{3}^2 g_{2} & = g_{3} (g_{3} g_{2})& \Longrightarrow & & t_{6}^{2}t_{12} & = 1 \\
g_{3}^2 g_{1} & = g_{3} (g_{3} g_{1})& \Longrightarrow & & t_{5}^{2}t_{10}t_{13}^{-1} & = 1 \\
g_{2}^2 g_{1} & = g_{2} (g_{2} g_{1})& \Longrightarrow & & t_{3}^{2}t_{4}t_{6}t_{12} & = 1 \\
g_{2} g_{1}^{2} & = (g_{2} g_{1}) g_{1}& \Longrightarrow & & t_{3}^{2}t_{4}t_{5}t_{8}t_{9}^{2}t_{10}t_{12}t_{14}^{2} & = 1 \\
\end{aligned}
\]
Scanning through the conjugacy class representatives of $G$ and the generators of their centralizers, we see that no new relations are imposed.
Collecting the coefficients of these relations into a matrix yields
\[
T = \bordermatrix{
{} & t_{1} & t_{2} & t_{3} & t_{4} & t_{5} & t_{6} & t_{7} & t_{8} & t_{9} & t_{10} & t_{11} & t_{12} & t_{13} & t_{14} \cr
{} &  &  & 2 & 1 &  & 1 &  &  &  &  &  & 1 &  &  \cr
{} &  &  &  &  & 1 & 1 &  & 1 &  & 1 &  & 1 &  & 1 \cr
{} &  &  &  &  &  & 2 &  &  &  &  &  & 1 &  &  \cr
{} &  &  &  &  &  &  &  & 2 &  & 1 &  & 1 & 1 & 2 \cr
{} &  &  &  &  &  &  &  &  & 1 &  &  &  & 1 & 1 \cr
{} &  &  &  &  &  &  &  &  &  &  & 1 &  & 1 & 1 \cr
{} &  &  &  &  &  &  &  &  &  &  &  &  & 2 & 1 \cr
}.
\]
It follows readily that the nontrivial elementary divisors of the Smith normal form of $T$ are all equal to $1$. The torsion subgroup of the group generated by the tails is thus trivial, thereby showing $\B_0(G) = 1$.


\item \label{number:100} 
Let the group $G$ be the representative of this family given by the presentation
\[
\begin{aligned}
\langle g_{1}, \,g_{2}, \,g_{3}, \,g_{4}, \,g_{5}, \,g_{6}, \,g_{7} & \mid & g_{1}^{2} &= 1, \\ 
 & & g_{2}^{2} &= 1, & [g_{2}, g_{1}]  &= g_{4}, \\ 
 & & g_{3}^{2} &= 1, & [g_{3}, g_{1}]  &= g_{5}, \\ 
 & & g_{4}^{2} &= g_{6}g_{7}, & [g_{4}, g_{1}]  &= g_{6}, & [g_{4}, g_{2}]  &= g_{6}, \\ 
 & & g_{5}^{2} &= g_{7}, & [g_{5}, g_{1}]  &= g_{7}, & [g_{5}, g_{3}]  &= g_{7}, \\ 
 & & g_{6}^{2} &= g_{7}, & [g_{6}, g_{1}]  &= g_{7}, & [g_{6}, g_{2}]  &= g_{7}, \\ 
 & & g_{7}^{2} &= 1\rangle. \\ 
\end{aligned}
\]
We add 15 tails to the presentation as to form a quotient of the universal central extension of the system: 
$g_{1}^{2} =  t_{1}$,
$g_{2}^{2} =  t_{2}$,
$[g_{2}, g_{1}] = g_{4} t_{3}$,
$g_{3}^{2} =  t_{4}$,
$[g_{3}, g_{1}] = g_{5} t_{5}$,
$g_{4}^{2} = g_{6}g_{7} t_{6}$,
$[g_{4}, g_{1}] = g_{6} t_{7}$,
$[g_{4}, g_{2}] = g_{6} t_{8}$,
$g_{5}^{2} = g_{7} t_{9}$,
$[g_{5}, g_{1}] = g_{7} t_{10}$,
$[g_{5}, g_{3}] = g_{7} t_{11}$,
$g_{6}^{2} = g_{7} t_{12}$,
$[g_{6}, g_{1}] = g_{7} t_{13}$,
$[g_{6}, g_{2}] = g_{7} t_{14}$,
$g_{7}^{2} =  t_{15}$.
Carrying out consistency checks gives the following relations between the tails:
\[
\begin{aligned}
g_{4}(g_{2} g_{1}) & = (g_{4} g_{2}) g_{1}  & \Longrightarrow & & t_{13}t_{14}^{-1} & = 1 \\
g_{6}^2 g_{2} & = g_{6} (g_{6} g_{2})& \Longrightarrow & & t_{14}^{2}t_{15} & = 1 \\
g_{5}^2 g_{3} & = g_{5} (g_{5} g_{3})& \Longrightarrow & & t_{11}^{2}t_{15} & = 1 \\
g_{5}^2 g_{1} & = g_{5} (g_{5} g_{1})& \Longrightarrow & & t_{10}^{2}t_{15} & = 1 \\
g_{4}^2 g_{2} & = g_{4} (g_{4} g_{2})& \Longrightarrow & & t_{8}^{2}t_{12}t_{14}^{-1} & = 1 \\
g_{4}^2 g_{1} & = g_{4} (g_{4} g_{1})& \Longrightarrow & & t_{7}^{2}t_{12}t_{13}^{-1} & = 1 \\
g_{3}^2 g_{1} & = g_{3} (g_{3} g_{1})& \Longrightarrow & & t_{5}^{2}t_{9}t_{11}t_{15} & = 1 \\
g_{2}^2 g_{1} & = g_{2} (g_{2} g_{1})& \Longrightarrow & & t_{3}^{2}t_{6}t_{8}t_{12}t_{15} & = 1 \\
g_{3} g_{1}^{2} & = (g_{3} g_{1}) g_{1}& \Longrightarrow & & t_{5}^{2}t_{9}t_{10}t_{15} & = 1 \\
g_{2} g_{1}^{2} & = (g_{2} g_{1}) g_{1}& \Longrightarrow & & t_{3}^{2}t_{6}t_{7}t_{12}t_{15} & = 1 \\
\end{aligned}
\]
Scanning through the conjugacy class representatives of $G$ and the generators of their centralizers, we obtain the following relations induced on the tails:
\[
\begin{aligned}
{[g_{5} g_{6} g_{7} , \, g_{1} ]}_G & = 1 & \Longrightarrow & & t_{10}t_{13}t_{15} & = 1 \\
\end{aligned}
\]
Collecting the coefficients of these relations into a matrix yields
\[
T = \bordermatrix{
{} & t_{1} & t_{2} & t_{3} & t_{4} & t_{5} & t_{6} & t_{7} & t_{8} & t_{9} & t_{10} & t_{11} & t_{12} & t_{13} & t_{14} & t_{15} \cr
{} &  &  & 2 &  &  & 1 &  & 1 &  &  &  & 1 &  &  & 1 \cr
{} &  &  &  &  & 2 &  &  &  & 1 &  &  &  &  & 1 & 1 \cr
{} &  &  &  &  &  &  & 1 & 1 &  &  &  & 1 &  & 1 & 1 \cr
{} &  &  &  &  &  &  &  & 2 &  &  &  & 1 &  & 1 & 1 \cr
{} &  &  &  &  &  &  &  &  &  & 1 &  &  &  & 1 & 1 \cr
{} &  &  &  &  &  &  &  &  &  &  & 1 &  &  & 1 & 1 \cr
{} &  &  &  &  &  &  &  &  &  &  &  &  & 1 & 1 & 1 \cr
{} &  &  &  &  &  &  &  &  &  &  &  &  &  & 2 & 1 \cr
}.
\]
It follows readily that the nontrivial elementary divisors of the Smith normal form of $T$ are all equal to $1$. The torsion subgroup of the group generated by the tails is thus trivial, thereby showing $\B_0(G) = 1$.


\item \label{number:101} 
Let the group $G$ be the representative of this family given by the presentation
\[
\begin{aligned}
\langle g_{1}, \,g_{2}, \,g_{3}, \,g_{4}, \,g_{5}, \,g_{6}, \,g_{7} & \mid & g_{1}^{2} &= 1, \\ 
 & & g_{2}^{2} &= 1, & [g_{2}, g_{1}]  &= g_{4}, \\ 
 & & g_{3}^{2} &= 1, & [g_{3}, g_{1}]  &= g_{5}, \\ 
 & & g_{4}^{2} &= g_{6}g_{7}, & [g_{4}, g_{1}]  &= g_{6}, & [g_{4}, g_{2}]  &= g_{6}, & [g_{4}, g_{3}]  &= g_{7}, \\ 
 & & g_{5}^{2} &= 1, & [g_{5}, g_{2}]  &= g_{7}, \\ 
 & & g_{6}^{2} &= g_{7}, & [g_{6}, g_{1}]  &= g_{7}, & [g_{6}, g_{2}]  &= g_{7}, \\ 
 & & g_{7}^{2} &= 1\rangle. \\ 
\end{aligned}
\]
We add 15 tails to the presentation as to form a quotient of the universal central extension of the system: 
$g_{1}^{2} =  t_{1}$,
$g_{2}^{2} =  t_{2}$,
$[g_{2}, g_{1}] = g_{4} t_{3}$,
$g_{3}^{2} =  t_{4}$,
$[g_{3}, g_{1}] = g_{5} t_{5}$,
$g_{4}^{2} = g_{6}g_{7} t_{6}$,
$[g_{4}, g_{1}] = g_{6} t_{7}$,
$[g_{4}, g_{2}] = g_{6} t_{8}$,
$[g_{4}, g_{3}] = g_{7} t_{9}$,
$g_{5}^{2} =  t_{10}$,
$[g_{5}, g_{2}] = g_{7} t_{11}$,
$g_{6}^{2} = g_{7} t_{12}$,
$[g_{6}, g_{1}] = g_{7} t_{13}$,
$[g_{6}, g_{2}] = g_{7} t_{14}$,
$g_{7}^{2} =  t_{15}$.
Carrying out consistency checks gives the following relations between the tails:
\[
\begin{aligned}
g_{4}(g_{2} g_{1}) & = (g_{4} g_{2}) g_{1}  & \Longrightarrow & & t_{13}t_{14}^{-1} & = 1 \\
g_{3}(g_{2} g_{1}) & = (g_{3} g_{2}) g_{1}  & \Longrightarrow & & t_{9}t_{11}^{-1} & = 1 \\
g_{6}^2 g_{2} & = g_{6} (g_{6} g_{2})& \Longrightarrow & & t_{14}^{2}t_{15} & = 1 \\
g_{5}^2 g_{2} & = g_{5} (g_{5} g_{2})& \Longrightarrow & & t_{11}^{2}t_{15} & = 1 \\
g_{4}^2 g_{2} & = g_{4} (g_{4} g_{2})& \Longrightarrow & & t_{8}^{2}t_{12}t_{14}^{-1} & = 1 \\
g_{4}^2 g_{1} & = g_{4} (g_{4} g_{1})& \Longrightarrow & & t_{7}^{2}t_{12}t_{13}^{-1} & = 1 \\
g_{3}^2 g_{1} & = g_{3} (g_{3} g_{1})& \Longrightarrow & & t_{5}^{2}t_{10} & = 1 \\
g_{2}^2 g_{1} & = g_{2} (g_{2} g_{1})& \Longrightarrow & & t_{3}^{2}t_{6}t_{8}t_{12}t_{15} & = 1 \\
g_{2} g_{1}^{2} & = (g_{2} g_{1}) g_{1}& \Longrightarrow & & t_{3}^{2}t_{6}t_{7}t_{12}t_{15} & = 1 \\
\end{aligned}
\]
Scanning through the conjugacy class representatives of $G$ and the generators of their centralizers, we obtain the following relations induced on the tails:
\[
\begin{aligned}
{[g_{5} g_{6} g_{7} , \, g_{2} g_{4} g_{6} ]}_G & = 1 & \Longrightarrow & & t_{11}t_{14}t_{15} & = 1 \\
\end{aligned}
\]
Collecting the coefficients of these relations into a matrix yields
\[
T = \bordermatrix{
{} & t_{1} & t_{2} & t_{3} & t_{4} & t_{5} & t_{6} & t_{7} & t_{8} & t_{9} & t_{10} & t_{11} & t_{12} & t_{13} & t_{14} & t_{15} \cr
{} &  &  & 2 &  &  & 1 &  & 1 &  &  &  & 1 &  &  & 1 \cr
{} &  &  &  &  & 2 &  &  &  &  & 1 &  &  &  &  &  \cr
{} &  &  &  &  &  &  & 1 & 1 &  &  &  & 1 &  & 1 & 1 \cr
{} &  &  &  &  &  &  &  & 2 &  &  &  & 1 &  & 1 & 1 \cr
{} &  &  &  &  &  &  &  &  & 1 &  &  &  &  & 1 & 1 \cr
{} &  &  &  &  &  &  &  &  &  &  & 1 &  &  & 1 & 1 \cr
{} &  &  &  &  &  &  &  &  &  &  &  &  & 1 & 1 & 1 \cr
{} &  &  &  &  &  &  &  &  &  &  &  &  &  & 2 & 1 \cr
}.
\]
It follows readily that the nontrivial elementary divisors of the Smith normal form of $T$ are all equal to $1$. The torsion subgroup of the group generated by the tails is thus trivial, thereby showing $\B_0(G) = 1$.


\item \label{number:102} 
Let the group $G$ be the representative of this family given by the presentation
\[
\begin{aligned}
\langle g_{1}, \,g_{2}, \,g_{3}, \,g_{4}, \,g_{5}, \,g_{6}, \,g_{7} & \mid & g_{1}^{2} &= 1, \\ 
 & & g_{2}^{2} &= 1, & [g_{2}, g_{1}]  &= g_{4}, \\ 
 & & g_{3}^{2} &= g_{4}g_{6}, & [g_{3}, g_{1}]  &= g_{5}, \\ 
 & & g_{4}^{2} &= 1, & [g_{4}, g_{3}]  &= g_{7}, \\ 
 & & g_{5}^{2} &= g_{6}g_{7}, & [g_{5}, g_{1}]  &= g_{6}, & [g_{5}, g_{2}]  &= g_{7}, & [g_{5}, g_{3}]  &= g_{6}g_{7}, \\ 
 & & g_{6}^{2} &= g_{7}, & [g_{6}, g_{1}]  &= g_{7}, & [g_{6}, g_{3}]  &= g_{7}, \\ 
 & & g_{7}^{2} &= 1\rangle. \\ 
\end{aligned}
\]
We add 15 tails to the presentation as to form a quotient of the universal central extension of the system: 
$g_{1}^{2} =  t_{1}$,
$g_{2}^{2} =  t_{2}$,
$[g_{2}, g_{1}] = g_{4} t_{3}$,
$g_{3}^{2} = g_{4}g_{6} t_{4}$,
$[g_{3}, g_{1}] = g_{5} t_{5}$,
$g_{4}^{2} =  t_{6}$,
$[g_{4}, g_{3}] = g_{7} t_{7}$,
$g_{5}^{2} = g_{6}g_{7} t_{8}$,
$[g_{5}, g_{1}] = g_{6} t_{9}$,
$[g_{5}, g_{2}] = g_{7} t_{10}$,
$[g_{5}, g_{3}] = g_{6}g_{7} t_{11}$,
$g_{6}^{2} = g_{7} t_{12}$,
$[g_{6}, g_{1}] = g_{7} t_{13}$,
$[g_{6}, g_{3}] = g_{7} t_{14}$,
$g_{7}^{2} =  t_{15}$.
Carrying out consistency checks gives the following relations between the tails:
\[
\begin{aligned}
g_{5}(g_{3} g_{1}) & = (g_{5} g_{3}) g_{1}  & \Longrightarrow & & t_{13}t_{14}^{-1} & = 1 \\
g_{3}(g_{2} g_{1}) & = (g_{3} g_{2}) g_{1}  & \Longrightarrow & & t_{7}t_{10}^{-1} & = 1 \\
g_{6}^2 g_{3} & = g_{6} (g_{6} g_{3})& \Longrightarrow & & t_{14}^{2}t_{15} & = 1 \\
g_{5}^2 g_{3} & = g_{5} (g_{5} g_{3})& \Longrightarrow & & t_{11}^{2}t_{12}t_{14}^{-1}t_{15} & = 1 \\
g_{5}^2 g_{2} & = g_{5} (g_{5} g_{2})& \Longrightarrow & & t_{10}^{2}t_{15} & = 1 \\
g_{5}^2 g_{1} & = g_{5} (g_{5} g_{1})& \Longrightarrow & & t_{9}^{2}t_{12}t_{13}^{-1} & = 1 \\
g_{3}^2 g_{1} & = g_{3} (g_{3} g_{1})& \Longrightarrow & & t_{5}^{2}t_{8}t_{11}t_{12}t_{13}^{-1}t_{15} & = 1 \\
g_{2}^2 g_{1} & = g_{2} (g_{2} g_{1})& \Longrightarrow & & t_{3}^{2}t_{6} & = 1 \\
g_{3} g_{1}^{2} & = (g_{3} g_{1}) g_{1}& \Longrightarrow & & t_{5}^{2}t_{8}t_{9}t_{12}t_{15} & = 1 \\
g_{3}^{2} g_{3} & = g_{3} g_{3}^{2}& \Longrightarrow & & t_{7}^{-1}t_{14}^{-1}t_{15}^{-1} & = 1 \\
\end{aligned}
\]
Scanning through the conjugacy class representatives of $G$ and the generators of their centralizers, we see that no new relations are imposed.
Collecting the coefficients of these relations into a matrix yields
\[
T = \bordermatrix{
{} & t_{1} & t_{2} & t_{3} & t_{4} & t_{5} & t_{6} & t_{7} & t_{8} & t_{9} & t_{10} & t_{11} & t_{12} & t_{13} & t_{14} & t_{15} \cr
{} &  &  & 2 &  &  & 1 &  &  &  &  &  &  &  &  &  \cr
{} &  &  &  &  & 2 &  &  & 1 &  &  & 1 & 1 &  & 1 & 2 \cr
{} &  &  &  &  &  &  & 1 &  &  &  &  &  &  & 1 & 1 \cr
{} &  &  &  &  &  &  &  &  & 1 &  & 1 & 1 &  &  & 1 \cr
{} &  &  &  &  &  &  &  &  &  & 1 &  &  &  & 1 & 1 \cr
{} &  &  &  &  &  &  &  &  &  &  & 2 & 1 &  & 1 & 2 \cr
{} &  &  &  &  &  &  &  &  &  &  &  &  & 1 & 1 & 1 \cr
{} &  &  &  &  &  &  &  &  &  &  &  &  &  & 2 & 1 \cr
}.
\]
It follows readily that the nontrivial elementary divisors of the Smith normal form of $T$ are all equal to $1$. The torsion subgroup of the group generated by the tails is thus trivial, thereby showing $\B_0(G) = 1$.


\item \label{number:103} 
Let the group $G$ be the representative of this family given by the presentation
\[
\begin{aligned}
\langle g_{1}, \,g_{2}, \,g_{3}, \,g_{4}, \,g_{5}, \,g_{6}, \,g_{7} & \mid & g_{1}^{2} &= 1, \\ 
 & & g_{2}^{2} &= g_{4}, & [g_{2}, g_{1}]  &= g_{4}, \\ 
 & & g_{3}^{2} &= 1, & [g_{3}, g_{1}]  &= g_{5}, \\ 
 & & g_{4}^{2} &= g_{7}, & [g_{4}, g_{1}]  &= g_{7}, \\ 
 & & g_{5}^{2} &= g_{6}g_{7}, & [g_{5}, g_{1}]  &= g_{6}, & [g_{5}, g_{3}]  &= g_{6}, \\ 
 & & g_{6}^{2} &= g_{7}, & [g_{6}, g_{1}]  &= g_{7}, & [g_{6}, g_{3}]  &= g_{7}, \\ 
 & & g_{7}^{2} &= 1\rangle. \\ 
\end{aligned}
\]
We add 14 tails to the presentation as to form a quotient of the universal central extension of the system: 
$g_{1}^{2} =  t_{1}$,
$g_{2}^{2} = g_{4} t_{2}$,
$[g_{2}, g_{1}] = g_{4} t_{3}$,
$g_{3}^{2} =  t_{4}$,
$[g_{3}, g_{1}] = g_{5} t_{5}$,
$g_{4}^{2} = g_{7} t_{6}$,
$[g_{4}, g_{1}] = g_{7} t_{7}$,
$g_{5}^{2} = g_{6}g_{7} t_{8}$,
$[g_{5}, g_{1}] = g_{6} t_{9}$,
$[g_{5}, g_{3}] = g_{6} t_{10}$,
$g_{6}^{2} = g_{7} t_{11}$,
$[g_{6}, g_{1}] = g_{7} t_{12}$,
$[g_{6}, g_{3}] = g_{7} t_{13}$,
$g_{7}^{2} =  t_{14}$.
Carrying out consistency checks gives the following relations between the tails:
\[
\begin{aligned}
g_{5}(g_{3} g_{1}) & = (g_{5} g_{3}) g_{1}  & \Longrightarrow & & t_{12}t_{13}^{-1} & = 1 \\
g_{6}^2 g_{3} & = g_{6} (g_{6} g_{3})& \Longrightarrow & & t_{13}^{2}t_{14} & = 1 \\
g_{5}^2 g_{3} & = g_{5} (g_{5} g_{3})& \Longrightarrow & & t_{10}^{2}t_{11}t_{13}^{-1} & = 1 \\
g_{5}^2 g_{1} & = g_{5} (g_{5} g_{1})& \Longrightarrow & & t_{9}^{2}t_{11}t_{12}^{-1} & = 1 \\
g_{4}^2 g_{1} & = g_{4} (g_{4} g_{1})& \Longrightarrow & & t_{7}^{2}t_{14} & = 1 \\
g_{3}^2 g_{1} & = g_{3} (g_{3} g_{1})& \Longrightarrow & & t_{5}^{2}t_{8}t_{10}t_{11}t_{14} & = 1 \\
g_{2}^2 g_{1} & = g_{2} (g_{2} g_{1})& \Longrightarrow & & t_{3}^{2}t_{6}t_{7}^{-1} & = 1 \\
g_{3} g_{1}^{2} & = (g_{3} g_{1}) g_{1}& \Longrightarrow & & t_{5}^{2}t_{8}t_{9}t_{11}t_{14} & = 1 \\
\end{aligned}
\]
Scanning through the conjugacy class representatives of $G$ and the generators of their centralizers, we obtain the following relations induced on the tails:
\[
\begin{aligned}
{[g_{4} g_{6} g_{7} , \, g_{1} ]}_G & = 1 & \Longrightarrow & & t_{7}t_{12}t_{14} & = 1 \\
\end{aligned}
\]
Collecting the coefficients of these relations into a matrix yields
\[
T = \bordermatrix{
{} & t_{1} & t_{2} & t_{3} & t_{4} & t_{5} & t_{6} & t_{7} & t_{8} & t_{9} & t_{10} & t_{11} & t_{12} & t_{13} & t_{14} \cr
{} &  &  & 2 &  &  & 1 &  &  &  &  &  &  & 1 & 1 \cr
{} &  &  &  &  & 2 &  &  & 1 &  & 1 & 1 &  &  & 1 \cr
{} &  &  &  &  &  &  & 1 &  &  &  &  &  & 1 & 1 \cr
{} &  &  &  &  &  &  &  &  & 1 & 1 & 1 &  & 1 & 1 \cr
{} &  &  &  &  &  &  &  &  &  & 2 & 1 &  & 1 & 1 \cr
{} &  &  &  &  &  &  &  &  &  &  &  & 1 & 1 & 1 \cr
{} &  &  &  &  &  &  &  &  &  &  &  &  & 2 & 1 \cr
}.
\]
It follows readily that the nontrivial elementary divisors of the Smith normal form of $T$ are all equal to $1$. The torsion subgroup of the group generated by the tails is thus trivial, thereby showing $\B_0(G) = 1$.


\item \label{number:104} 
Let the group $G$ be the representative of this family given by the presentation
\[
\begin{aligned}
\langle g_{1}, \,g_{2}, \,g_{3}, \,g_{4}, \,g_{5}, \,g_{6}, \,g_{7} & \mid & g_{1}^{2} &= g_{4}g_{6}, \\ 
 & & g_{2}^{2} &= g_{4}, & [g_{2}, g_{1}]  &= g_{4}, \\ 
 & & g_{3}^{2} &= 1, & [g_{3}, g_{1}]  &= g_{5}, \\ 
 & & g_{4}^{2} &= g_{7}, & [g_{4}, g_{1}]  &= g_{7}, \\ 
 & & g_{5}^{2} &= g_{6}, & [g_{5}, g_{1}]  &= g_{6}, & [g_{5}, g_{3}]  &= g_{6}g_{7}, \\ 
 & & g_{6}^{2} &= g_{7}, & [g_{6}, g_{1}]  &= g_{7}, & [g_{6}, g_{3}]  &= g_{7}, \\ 
 & & g_{7}^{2} &= 1\rangle. \\ 
\end{aligned}
\]
We add 14 tails to the presentation as to form a quotient of the universal central extension of the system: 
$g_{1}^{2} = g_{4}g_{6} t_{1}$,
$g_{2}^{2} = g_{4} t_{2}$,
$[g_{2}, g_{1}] = g_{4} t_{3}$,
$g_{3}^{2} =  t_{4}$,
$[g_{3}, g_{1}] = g_{5} t_{5}$,
$g_{4}^{2} = g_{7} t_{6}$,
$[g_{4}, g_{1}] = g_{7} t_{7}$,
$g_{5}^{2} = g_{6} t_{8}$,
$[g_{5}, g_{1}] = g_{6} t_{9}$,
$[g_{5}, g_{3}] = g_{6}g_{7} t_{10}$,
$g_{6}^{2} = g_{7} t_{11}$,
$[g_{6}, g_{1}] = g_{7} t_{12}$,
$[g_{6}, g_{3}] = g_{7} t_{13}$,
$g_{7}^{2} =  t_{14}$.
Carrying out consistency checks gives the following relations between the tails:
\[
\begin{aligned}
g_{5}(g_{3} g_{1}) & = (g_{5} g_{3}) g_{1}  & \Longrightarrow & & t_{12}t_{13}^{-1} & = 1 \\
g_{6}^2 g_{3} & = g_{6} (g_{6} g_{3})& \Longrightarrow & & t_{13}^{2}t_{14} & = 1 \\
g_{5}^2 g_{3} & = g_{5} (g_{5} g_{3})& \Longrightarrow & & t_{10}^{2}t_{11}t_{13}^{-1}t_{14} & = 1 \\
g_{5}^2 g_{1} & = g_{5} (g_{5} g_{1})& \Longrightarrow & & t_{9}^{2}t_{11}t_{12}^{-1} & = 1 \\
g_{4}^2 g_{1} & = g_{4} (g_{4} g_{1})& \Longrightarrow & & t_{7}^{2}t_{14} & = 1 \\
g_{3}^2 g_{1} & = g_{3} (g_{3} g_{1})& \Longrightarrow & & t_{5}^{2}t_{8}t_{10}t_{11}t_{14} & = 1 \\
g_{2}^2 g_{1} & = g_{2} (g_{2} g_{1})& \Longrightarrow & & t_{3}^{2}t_{6}t_{7}^{-1} & = 1 \\
g_{3} g_{1}^{2} & = (g_{3} g_{1}) g_{1}& \Longrightarrow & & t_{5}^{2}t_{8}t_{9}t_{11}t_{13}t_{14} & = 1 \\
g_{1}^{2} g_{1} & = g_{1} g_{1}^{2}& \Longrightarrow & & t_{7}^{-1}t_{12}^{-1}t_{14}^{-1} & = 1 \\
\end{aligned}
\]
Scanning through the conjugacy class representatives of $G$ and the generators of their centralizers, we see that no new relations are imposed.
Collecting the coefficients of these relations into a matrix yields
\[
T = \bordermatrix{
{} & t_{1} & t_{2} & t_{3} & t_{4} & t_{5} & t_{6} & t_{7} & t_{8} & t_{9} & t_{10} & t_{11} & t_{12} & t_{13} & t_{14} \cr
{} &  &  & 2 &  &  & 1 &  &  &  &  &  &  & 1 & 1 \cr
{} &  &  &  &  & 2 &  &  & 1 &  & 1 & 1 &  &  & 1 \cr
{} &  &  &  &  &  &  & 1 &  &  &  &  &  & 1 & 1 \cr
{} &  &  &  &  &  &  &  &  & 1 & 1 & 1 &  &  & 1 \cr
{} &  &  &  &  &  &  &  &  &  & 2 & 1 &  & 1 & 2 \cr
{} &  &  &  &  &  &  &  &  &  &  &  & 1 & 1 & 1 \cr
{} &  &  &  &  &  &  &  &  &  &  &  &  & 2 & 1 \cr
}.
\]
It follows readily that the nontrivial elementary divisors of the Smith normal form of $T$ are all equal to $1$. The torsion subgroup of the group generated by the tails is thus trivial, thereby showing $\B_0(G) = 1$.


\item \label{number:105} 
Let the group $G$ be the representative of this family given by the presentation
\[
\begin{aligned}
\langle g_{1}, \,g_{2}, \,g_{3}, \,g_{4}, \,g_{5}, \,g_{6}, \,g_{7} & \mid & g_{1}^{2} &= 1, \\ 
 & & g_{2}^{2} &= g_{4}g_{6}, & [g_{2}, g_{1}]  &= g_{4}, \\ 
 & & g_{3}^{2} &= 1, & [g_{3}, g_{1}]  &= g_{5}, \\ 
 & & g_{4}^{2} &= g_{6}g_{7}, & [g_{4}, g_{1}]  &= g_{6}, \\ 
 & & g_{5}^{2} &= g_{7}, & [g_{5}, g_{1}]  &= g_{7}, & [g_{5}, g_{3}]  &= g_{7}, \\ 
 & & g_{6}^{2} &= g_{7}, & [g_{6}, g_{1}]  &= g_{7}, \\ 
 & & g_{7}^{2} &= 1\rangle. \\ 
\end{aligned}
\]
We add 13 tails to the presentation as to form a quotient of the universal central extension of the system: 
$g_{1}^{2} =  t_{1}$,
$g_{2}^{2} = g_{4}g_{6} t_{2}$,
$[g_{2}, g_{1}] = g_{4} t_{3}$,
$g_{3}^{2} =  t_{4}$,
$[g_{3}, g_{1}] = g_{5} t_{5}$,
$g_{4}^{2} = g_{6}g_{7} t_{6}$,
$[g_{4}, g_{1}] = g_{6} t_{7}$,
$g_{5}^{2} = g_{7} t_{8}$,
$[g_{5}, g_{1}] = g_{7} t_{9}$,
$[g_{5}, g_{3}] = g_{7} t_{10}$,
$g_{6}^{2} = g_{7} t_{11}$,
$[g_{6}, g_{1}] = g_{7} t_{12}$,
$g_{7}^{2} =  t_{13}$.
Carrying out consistency checks gives the following relations between the tails:
\[
\begin{aligned}
g_{6}^2 g_{1} & = g_{6} (g_{6} g_{1})& \Longrightarrow & & t_{12}^{2}t_{13} & = 1 \\
g_{5}^2 g_{3} & = g_{5} (g_{5} g_{3})& \Longrightarrow & & t_{10}^{2}t_{13} & = 1 \\
g_{5}^2 g_{1} & = g_{5} (g_{5} g_{1})& \Longrightarrow & & t_{9}^{2}t_{13} & = 1 \\
g_{4}^2 g_{1} & = g_{4} (g_{4} g_{1})& \Longrightarrow & & t_{7}^{2}t_{11}t_{12}^{-1} & = 1 \\
g_{3}^2 g_{1} & = g_{3} (g_{3} g_{1})& \Longrightarrow & & t_{5}^{2}t_{8}t_{10}t_{13} & = 1 \\
g_{2}^2 g_{1} & = g_{2} (g_{2} g_{1})& \Longrightarrow & & t_{3}^{2}t_{6}t_{7}^{-1}t_{12}^{-1} & = 1 \\
g_{3} g_{1}^{2} & = (g_{3} g_{1}) g_{1}& \Longrightarrow & & t_{5}^{2}t_{8}t_{9}t_{13} & = 1 \\
\end{aligned}
\]
Scanning through the conjugacy class representatives of $G$ and the generators of their centralizers, we obtain the following relations induced on the tails:
\[
\begin{aligned}
{[g_{5} g_{6} g_{7} , \, g_{1} ]}_G & = 1 & \Longrightarrow & & t_{9}t_{12}t_{13} & = 1 \\
\end{aligned}
\]
Collecting the coefficients of these relations into a matrix yields
\[
T = \bordermatrix{
{} & t_{1} & t_{2} & t_{3} & t_{4} & t_{5} & t_{6} & t_{7} & t_{8} & t_{9} & t_{10} & t_{11} & t_{12} & t_{13} \cr
{} &  &  & 2 &  &  & 1 & 1 &  &  &  & 1 &  & 1 \cr
{} &  &  &  &  & 2 &  &  & 1 &  &  &  & 1 & 1 \cr
{} &  &  &  &  &  &  & 2 &  &  &  & 1 & 1 & 1 \cr
{} &  &  &  &  &  &  &  &  & 1 &  &  & 1 & 1 \cr
{} &  &  &  &  &  &  &  &  &  & 1 &  & 1 & 1 \cr
{} &  &  &  &  &  &  &  &  &  &  &  & 2 & 1 \cr
}.
\]
It follows readily that the nontrivial elementary divisors of the Smith normal form of $T$ are all equal to $1$. The torsion subgroup of the group generated by the tails is thus trivial, thereby showing $\B_0(G) = 1$.


\item \label{number:106} 
Let the group $G$ be the representative of this family given by the presentation
\[
\begin{aligned}
\langle g_{1}, \,g_{2}, \,g_{3}, \,g_{4}, \,g_{5}, \,g_{6}, \,g_{7} & \mid & g_{1}^{2} &= g_{4}, \\ 
 & & g_{2}^{2} &= g_{6}, & [g_{2}, g_{1}]  &= g_{3}, \\ 
 & & g_{3}^{2} &= g_{6}g_{7}, & [g_{3}, g_{1}]  &= g_{5}, & [g_{3}, g_{2}]  &= g_{6}, \\ 
 & & g_{4}^{2} &= 1, & [g_{4}, g_{2}]  &= g_{5}g_{6}, & [g_{4}, g_{3}]  &= g_{6}g_{7}, \\ 
 & & g_{5}^{2} &= g_{7}, & [g_{5}, g_{1}]  &= g_{6}, & [g_{5}, g_{2}]  &= g_{7}, & [g_{5}, g_{4}]  &= g_{7}, \\ 
 & & g_{6}^{2} &= 1, & [g_{6}, g_{1}]  &= g_{7}, \\ 
 & & g_{7}^{2} &= 1\rangle. \\ 
\end{aligned}
\]
We add 16 tails to the presentation as to form a quotient of the universal central extension of the system: 
$g_{1}^{2} = g_{4} t_{1}$,
$g_{2}^{2} = g_{6} t_{2}$,
$[g_{2}, g_{1}] = g_{3} t_{3}$,
$g_{3}^{2} = g_{6}g_{7} t_{4}$,
$[g_{3}, g_{1}] = g_{5} t_{5}$,
$[g_{3}, g_{2}] = g_{6} t_{6}$,
$g_{4}^{2} =  t_{7}$,
$[g_{4}, g_{2}] = g_{5}g_{6} t_{8}$,
$[g_{4}, g_{3}] = g_{6}g_{7} t_{9}$,
$g_{5}^{2} = g_{7} t_{10}$,
$[g_{5}, g_{1}] = g_{6} t_{11}$,
$[g_{5}, g_{2}] = g_{7} t_{12}$,
$[g_{5}, g_{4}] = g_{7} t_{13}$,
$g_{6}^{2} =  t_{14}$,
$[g_{6}, g_{1}] = g_{7} t_{15}$,
$g_{7}^{2} =  t_{16}$.
Carrying out consistency checks gives the following relations between the tails:
\[
\begin{aligned}
g_{4}(g_{3} g_{1}) & = (g_{4} g_{3}) g_{1}  & \Longrightarrow & & t_{13}t_{15}t_{16} & = 1 \\
g_{4}(g_{2} g_{1}) & = (g_{4} g_{2}) g_{1}  & \Longrightarrow & & t_{9}^{-1}t_{11}t_{15} & = 1 \\
g_{3}(g_{2} g_{1}) & = (g_{3} g_{2}) g_{1}  & \Longrightarrow & & t_{12}^{-1}t_{15} & = 1 \\
g_{6}^2 g_{1} & = g_{6} (g_{6} g_{1})& \Longrightarrow & & t_{15}^{2}t_{16} & = 1 \\
g_{5}^2 g_{1} & = g_{5} (g_{5} g_{1})& \Longrightarrow & & t_{11}^{2}t_{14} & = 1 \\
g_{4}^2 g_{2} & = g_{4} (g_{4} g_{2})& \Longrightarrow & & t_{8}^{2}t_{10}t_{13}t_{14}t_{16} & = 1 \\
g_{3}^2 g_{2} & = g_{3} (g_{3} g_{2})& \Longrightarrow & & t_{6}^{2}t_{14} & = 1 \\
g_{3}^2 g_{1} & = g_{3} (g_{3} g_{1})& \Longrightarrow & & t_{5}^{2}t_{10}t_{15}^{-1} & = 1 \\
g_{2}^2 g_{1} & = g_{2} (g_{2} g_{1})& \Longrightarrow & & t_{3}^{2}t_{4}t_{6}t_{14}t_{15}^{-1} & = 1 \\
g_{2} g_{1}^{2} & = (g_{2} g_{1}) g_{1}& \Longrightarrow & & t_{3}^{2}t_{4}t_{5}t_{8}t_{9}^{2}t_{10}t_{14}^{2}t_{16}^{2} & = 1 \\
\end{aligned}
\]
Scanning through the conjugacy class representatives of $G$ and the generators of their centralizers, we see that no new relations are imposed.
Collecting the coefficients of these relations into a matrix yields
\[
T = \bordermatrix{
{} & t_{1} & t_{2} & t_{3} & t_{4} & t_{5} & t_{6} & t_{7} & t_{8} & t_{9} & t_{10} & t_{11} & t_{12} & t_{13} & t_{14} & t_{15} & t_{16} \cr
{} &  &  & 2 & 1 &  & 1 &  &  &  &  &  &  &  & 1 & 1 & 1 \cr
{} &  &  &  &  & 1 & 1 &  & 1 &  & 1 &  &  &  & 1 & 1 & 1 \cr
{} &  &  &  &  &  & 2 &  &  &  &  &  &  &  & 1 &  &  \cr
{} &  &  &  &  &  &  &  & 2 &  & 1 &  &  &  & 1 & 1 & 1 \cr
{} &  &  &  &  &  &  &  &  & 1 &  & 1 &  &  & 1 & 1 & 1 \cr
{} &  &  &  &  &  &  &  &  &  &  & 2 &  &  & 1 &  &  \cr
{} &  &  &  &  &  &  &  &  &  &  &  & 1 &  &  & 1 & 1 \cr
{} &  &  &  &  &  &  &  &  &  &  &  &  & 1 &  & 1 & 1 \cr
{} &  &  &  &  &  &  &  &  &  &  &  &  &  &  & 2 & 1 \cr
}.
\]
A change of basis according to the transition matrix (specifying expansions of $t_i^{*}$ by $t_j$)
\[
\bordermatrix{
{} & t_{1}^{*} & t_{2}^{*} & t_{3}^{*} & t_{4}^{*} & t_{5}^{*} & t_{6}^{*} & t_{7}^{*} & t_{8}^{*} & t_{9}^{*} & t_{10}^{*} & t_{11}^{*} & t_{12}^{*} & t_{13}^{*} & t_{14}^{*} & t_{15}^{*} & t_{16}^{*} \cr
t_{1} &  &  &  &  &  &  &  &  &  & -1 & -9 & 1 & 1 & 1 &  &  \cr
t_{2} &  &  &  &  &  &  &  &  &  &  & -2 & 1 & 1 &  &  & -1 \cr
t_{3} & 4 & 2 & 4 & 8 & 4 & 8 &  &  & 6 & 2 & 9 & -1 & -1 & -1 &  &  \cr
t_{4} & 2 & 1 & 2 & 4 & 2 & 4 &  &  & 3 &  & 2 &  &  &  &  &  \cr
t_{5} & -3 & -2 & -4 & -8 & -4 & -8 &  &  & -6 & -1 &  &  &  &  &  &  \cr
t_{6} & 3 & 3 & 4 & 8 & 4 & 8 &  &  & 6 & 1 & 2 & -1 & -1 &  &  & 1 \cr
t_{7} &  &  &  &  &  &  &  &  &  &  &  & -1 &  &  &  & 1 \cr
t_{8} & -9 & -4 & -8 & -18 & -12 & -20 &  &  & -14 &  &  &  &  &  &  & -1 \cr
t_{9} & 3 & 2 & 4 & 8 & 5 & 8 &  &  & 6 & 1 &  &  &  &  &  &  \cr
t_{10} & -6 & -3 & -6 & -13 & -8 & -14 &  &  & -10 & -1 & -3 &  &  &  &  &  \cr
t_{11} & 1 &  &  & 2 & 1 & 2 &  &  & 1 &  & 1 &  &  &  &  &  \cr
t_{12} & 9 & 4 & 8 & 18 & 12 & 20 & 1 &  & 14 &  &  & 1 &  &  &  &  \cr
t_{13} & 2 &  & 2 & 3 & 3 & 5 & -1 &  & 3 &  &  &  & 1 &  &  &  \cr
t_{14} &  & 1 & 1 & 2 &  & 1 &  &  & 1 &  &  &  &  & 1 &  &  \cr
t_{15} & 8 & 4 & 6 & 16 & 10 & 17 &  & 2 & 12 &  &  &  &  &  & 1 &  \cr
t_{16} & 9 & 4 & 8 & 18 & 12 & 20 &  & 1 & 14 &  &  &  &  &  &  & 1 \cr
}
\]
shows that the nontrivial elementary divisors of the Smith normal form of $T$ are $1$, $1$, $1$, $1$, $1$, $1$, $1$, $1$, $2$.  The element corresponding to the divisor that is greater than $1$ is $t_{9}^{*}$. This already gives
\[
\B_0(G) \cong \langle t_{9}^{*}  \mid {t_{9}^{*}}^{2} \rangle.
\]

We now deal with explicitly identifying the nonuniversal commutator relation generating $\B_0(G)$.
First, factor out by the tails $t_{i}^{*}$ whose corresponding elementary divisors are either trivial or $1$. Transforming the situation back to the original tails $t_i$, this amounts to the nontrivial expansions given by
\[
\bordermatrix{
{} & t_{2} & t_{4} & t_{5} & t_{6} \cr
t_{9}^{*} & 1 & 1 & 1 & 1 \cr
}
\]
and all the other tails $t_i$ are trivial. We thus obtain a commutativity preserving central extension of the group $G$, given by the presentation
\[
\begin{aligned}
\langle g_{1}, \,g_{2}, \,g_{3}, \,g_{4}, \,g_{5}, \,g_{6}, \,g_{7}, \,t_{9}^{*} & \mid & g_{1}^{2} &= g_{4}, \\ 
 & & g_{2}^{2} &= g_{6}t_{9}^{*} , & [g_{2}, g_{1}]  &= g_{3}, \\ 
 & & g_{3}^{2} &= g_{6}g_{7}t_{9}^{*} , & [g_{3}, g_{1}]  &= g_{5}t_{9}^{*} , & [g_{3}, g_{2}]  &= g_{6}t_{9}^{*} , \\ 
 & & g_{4}^{2} &= 1, & [g_{4}, g_{2}]  &= g_{5}g_{6}, & [g_{4}, g_{3}]  &= g_{6}g_{7}, \\ 
 & & g_{5}^{2} &= g_{7}, & [g_{5}, g_{1}]  &= g_{6}, & [g_{5}, g_{2}]  &= g_{7}, & [g_{5}, g_{4}]  &= g_{7}, \\ 
 & & g_{6}^{2} &= 1, & [g_{6}, g_{1}]  &= g_{7}, \\ 
 & & g_{7}^{2} &= 1, \\ 
 & & {t_{9}^{*}}^{2} &= 1  \rangle,
\end{aligned}
\]
whence the nonuniversal commutator relation is identified as
\[
t_{9}^{*}  = [g_{3}, g_{2}] [g_{5}, g_{1}]^{-1}.  \quad 
\]


\item \label{number:107} 
Let the group $G$ be the representative of this family given by the presentation
\[
\begin{aligned}
\langle g_{1}, \,g_{2}, \,g_{3}, \,g_{4}, \,g_{5}, \,g_{6}, \,g_{7} & \mid & g_{1}^{2} &= g_{4}, \\ 
 & & g_{2}^{2} &= 1, & [g_{2}, g_{1}]  &= g_{3}, \\ 
 & & g_{3}^{2} &= g_{6}g_{7}, & [g_{3}, g_{1}]  &= g_{5}, & [g_{3}, g_{2}]  &= g_{6}g_{7}, \\ 
 & & g_{4}^{2} &= 1, & [g_{4}, g_{2}]  &= g_{5}g_{6}, & [g_{4}, g_{3}]  &= g_{6}g_{7}, \\ 
 & & g_{5}^{2} &= g_{7}, & [g_{5}, g_{1}]  &= g_{6}, & [g_{5}, g_{2}]  &= g_{7}, & [g_{5}, g_{4}]  &= g_{7}, \\ 
 & & g_{6}^{2} &= 1, & [g_{6}, g_{1}]  &= g_{7}, \\ 
 & & g_{7}^{2} &= 1\rangle. \\ 
\end{aligned}
\]
We add 16 tails to the presentation as to form a quotient of the universal central extension of the system: 
$g_{1}^{2} = g_{4} t_{1}$,
$g_{2}^{2} =  t_{2}$,
$[g_{2}, g_{1}] = g_{3} t_{3}$,
$g_{3}^{2} = g_{6}g_{7} t_{4}$,
$[g_{3}, g_{1}] = g_{5} t_{5}$,
$[g_{3}, g_{2}] = g_{6}g_{7} t_{6}$,
$g_{4}^{2} =  t_{7}$,
$[g_{4}, g_{2}] = g_{5}g_{6} t_{8}$,
$[g_{4}, g_{3}] = g_{6}g_{7} t_{9}$,
$g_{5}^{2} = g_{7} t_{10}$,
$[g_{5}, g_{1}] = g_{6} t_{11}$,
$[g_{5}, g_{2}] = g_{7} t_{12}$,
$[g_{5}, g_{4}] = g_{7} t_{13}$,
$g_{6}^{2} =  t_{14}$,
$[g_{6}, g_{1}] = g_{7} t_{15}$,
$g_{7}^{2} =  t_{16}$.
Carrying out consistency checks gives the following relations between the tails:
\[
\begin{aligned}
g_{4}(g_{3} g_{1}) & = (g_{4} g_{3}) g_{1}  & \Longrightarrow & & t_{13}t_{15}t_{16} & = 1 \\
g_{4}(g_{2} g_{1}) & = (g_{4} g_{2}) g_{1}  & \Longrightarrow & & t_{9}^{-1}t_{11}t_{15} & = 1 \\
g_{3}(g_{2} g_{1}) & = (g_{3} g_{2}) g_{1}  & \Longrightarrow & & t_{12}^{-1}t_{15} & = 1 \\
g_{6}^2 g_{1} & = g_{6} (g_{6} g_{1})& \Longrightarrow & & t_{15}^{2}t_{16} & = 1 \\
g_{5}^2 g_{1} & = g_{5} (g_{5} g_{1})& \Longrightarrow & & t_{11}^{2}t_{14} & = 1 \\
g_{4}^2 g_{2} & = g_{4} (g_{4} g_{2})& \Longrightarrow & & t_{8}^{2}t_{10}t_{13}t_{14}t_{16} & = 1 \\
g_{3}^2 g_{2} & = g_{3} (g_{3} g_{2})& \Longrightarrow & & t_{6}^{2}t_{14}t_{16} & = 1 \\
g_{3}^2 g_{1} & = g_{3} (g_{3} g_{1})& \Longrightarrow & & t_{5}^{2}t_{10}t_{15}^{-1} & = 1 \\
g_{2}^2 g_{1} & = g_{2} (g_{2} g_{1})& \Longrightarrow & & t_{3}^{2}t_{4}t_{6}t_{14}t_{16} & = 1 \\
g_{2} g_{1}^{2} & = (g_{2} g_{1}) g_{1}& \Longrightarrow & & t_{3}^{2}t_{4}t_{5}t_{8}t_{9}^{2}t_{10}t_{14}^{2}t_{16}^{2} & = 1 \\
\end{aligned}
\]
Scanning through the conjugacy class representatives of $G$ and the generators of their centralizers, we obtain the following relations induced on the tails:
\[
\begin{aligned}
{[g_{3} g_{5} g_{6} , \, g_{2} g_{4} g_{5} g_{7} ]}_G & = 1 & \Longrightarrow & & t_{6}t_{9}^{-1}t_{12}t_{13}t_{16} & = 1 \\
\end{aligned}
\]
Collecting the coefficients of these relations into a matrix yields
\[
T = \bordermatrix{
{} & t_{1} & t_{2} & t_{3} & t_{4} & t_{5} & t_{6} & t_{7} & t_{8} & t_{9} & t_{10} & t_{11} & t_{12} & t_{13} & t_{14} & t_{15} & t_{16} \cr
{} &  &  & 2 & 1 &  &  &  &  &  &  & 1 &  &  & 1 & 1 & 1 \cr
{} &  &  &  &  & 1 &  &  & 1 &  & 1 & 1 &  &  & 1 & 1 & 1 \cr
{} &  &  &  &  &  & 1 &  &  &  &  & 1 &  &  & 1 & 1 & 1 \cr
{} &  &  &  &  &  &  &  & 2 &  & 1 &  &  &  & 1 & 1 & 1 \cr
{} &  &  &  &  &  &  &  &  & 1 &  & 1 &  &  & 1 & 1 & 1 \cr
{} &  &  &  &  &  &  &  &  &  &  & 2 &  &  & 1 &  &  \cr
{} &  &  &  &  &  &  &  &  &  &  &  & 1 &  &  & 1 & 1 \cr
{} &  &  &  &  &  &  &  &  &  &  &  &  & 1 &  & 1 & 1 \cr
{} &  &  &  &  &  &  &  &  &  &  &  &  &  &  & 2 & 1 \cr
}.
\]
It follows readily that the nontrivial elementary divisors of the Smith normal form of $T$ are all equal to $1$. The torsion subgroup of the group generated by the tails is thus trivial, thereby showing $\B_0(G) = 1$.


\item \label{number:108} 
Let the group $G$ be the representative of this family given by the presentation
\[
\begin{aligned}
\langle g_{1}, \,g_{2}, \,g_{3}, \,g_{4}, \,g_{5}, \,g_{6}, \,g_{7} & \mid & g_{1}^{2} &= g_{4}, \\ 
 & & g_{2}^{2} &= 1, & [g_{2}, g_{1}]  &= g_{3}, \\ 
 & & g_{3}^{2} &= g_{5}g_{6}g_{7}, & [g_{3}, g_{1}]  &= g_{5}, & [g_{3}, g_{2}]  &= g_{5}g_{7}, \\ 
 & & g_{4}^{2} &= 1, & [g_{4}, g_{2}]  &= g_{7}, \\ 
 & & g_{5}^{2} &= g_{6}g_{7}, & [g_{5}, g_{1}]  &= g_{6}, & [g_{5}, g_{2}]  &= g_{6}, \\ 
 & & g_{6}^{2} &= g_{7}, & [g_{6}, g_{1}]  &= g_{7}, & [g_{6}, g_{2}]  &= g_{7}, \\ 
 & & g_{7}^{2} &= 1\rangle. \\ 
\end{aligned}
\]
We add 15 tails to the presentation as to form a quotient of the universal central extension of the system: 
$g_{1}^{2} = g_{4} t_{1}$,
$g_{2}^{2} =  t_{2}$,
$[g_{2}, g_{1}] = g_{3} t_{3}$,
$g_{3}^{2} = g_{5}g_{6}g_{7} t_{4}$,
$[g_{3}, g_{1}] = g_{5} t_{5}$,
$[g_{3}, g_{2}] = g_{5}g_{7} t_{6}$,
$g_{4}^{2} =  t_{7}$,
$[g_{4}, g_{2}] = g_{7} t_{8}$,
$g_{5}^{2} = g_{6}g_{7} t_{9}$,
$[g_{5}, g_{1}] = g_{6} t_{10}$,
$[g_{5}, g_{2}] = g_{6} t_{11}$,
$g_{6}^{2} = g_{7} t_{12}$,
$[g_{6}, g_{1}] = g_{7} t_{13}$,
$[g_{6}, g_{2}] = g_{7} t_{14}$,
$g_{7}^{2} =  t_{15}$.
Carrying out consistency checks gives the following relations between the tails:
\[
\begin{aligned}
g_{5}(g_{2} g_{1}) & = (g_{5} g_{2}) g_{1}  & \Longrightarrow & & t_{13}t_{14}^{-1} & = 1 \\
g_{3}(g_{2} g_{1}) & = (g_{3} g_{2}) g_{1}  & \Longrightarrow & & t_{10}t_{11}^{-1} & = 1 \\
g_{6}^2 g_{2} & = g_{6} (g_{6} g_{2})& \Longrightarrow & & t_{14}^{2}t_{15} & = 1 \\
g_{5}^2 g_{2} & = g_{5} (g_{5} g_{2})& \Longrightarrow & & t_{11}^{2}t_{12}t_{14}^{-1} & = 1 \\
g_{4}^2 g_{2} & = g_{4} (g_{4} g_{2})& \Longrightarrow & & t_{8}^{2}t_{15} & = 1 \\
g_{3}^2 g_{2} & = g_{3} (g_{3} g_{2})& \Longrightarrow & & t_{6}^{2}t_{9}t_{11}^{-1}t_{14}^{-1}t_{15} & = 1 \\
g_{3}^2 g_{1} & = g_{3} (g_{3} g_{1})& \Longrightarrow & & t_{5}^{2}t_{9}t_{10}^{-1}t_{13}^{-1} & = 1 \\
g_{2}^2 g_{1} & = g_{2} (g_{2} g_{1})& \Longrightarrow & & t_{3}^{2}t_{4}t_{6}t_{9}t_{12}t_{15}^{2} & = 1 \\
g_{2} g_{1}^{2} & = (g_{2} g_{1}) g_{1}& \Longrightarrow & & t_{3}^{2}t_{4}t_{5}t_{8}t_{9}t_{12}t_{15}^{2} & = 1 \\
\end{aligned}
\]
Scanning through the conjugacy class representatives of $G$ and the generators of their centralizers, we obtain the following relations induced on the tails:
\[
\begin{aligned}
{[g_{4} g_{6} g_{7} , \, g_{2} g_{4} ]}_G & = 1 & \Longrightarrow & & t_{8}t_{14}t_{15} & = 1 \\
\end{aligned}
\]
Collecting the coefficients of these relations into a matrix yields
\[
T = \bordermatrix{
{} & t_{1} & t_{2} & t_{3} & t_{4} & t_{5} & t_{6} & t_{7} & t_{8} & t_{9} & t_{10} & t_{11} & t_{12} & t_{13} & t_{14} & t_{15} \cr
{} &  &  & 2 & 1 &  & 1 &  &  & 1 &  &  & 1 &  &  & 2 \cr
{} &  &  &  &  & 1 & 1 &  &  & 1 &  & 1 & 1 &  & 1 & 2 \cr
{} &  &  &  &  &  & 2 &  &  & 1 &  & 1 & 1 &  &  & 2 \cr
{} &  &  &  &  &  &  &  & 1 &  &  &  &  &  & 1 & 1 \cr
{} &  &  &  &  &  &  &  &  &  & 1 & 1 & 1 &  & 1 & 1 \cr
{} &  &  &  &  &  &  &  &  &  &  & 2 & 1 &  & 1 & 1 \cr
{} &  &  &  &  &  &  &  &  &  &  &  &  & 1 & 1 & 1 \cr
{} &  &  &  &  &  &  &  &  &  &  &  &  &  & 2 & 1 \cr
}.
\]
It follows readily that the nontrivial elementary divisors of the Smith normal form of $T$ are all equal to $1$. The torsion subgroup of the group generated by the tails is thus trivial, thereby showing $\B_0(G) = 1$.


\item \label{number:109} 
Let the group $G$ be the representative of this family given by the presentation
\[
\begin{aligned}
\langle g_{1}, \,g_{2}, \,g_{3}, \,g_{4}, \,g_{5}, \,g_{6}, \,g_{7} & \mid & g_{1}^{2} &= 1, \\ 
 & & g_{2}^{2} &= 1, & [g_{2}, g_{1}]  &= g_{4}, \\ 
 & & g_{3}^{2} &= 1, & [g_{3}, g_{1}]  &= g_{7}, \\ 
 & & g_{4}^{2} &= g_{5}g_{6}, & [g_{4}, g_{1}]  &= g_{5}, & [g_{4}, g_{2}]  &= g_{5}, \\ 
 & & g_{5}^{2} &= g_{6}g_{7}, & [g_{5}, g_{1}]  &= g_{6}, & [g_{5}, g_{2}]  &= g_{6}, \\ 
 & & g_{6}^{2} &= g_{7}, & [g_{6}, g_{1}]  &= g_{7}, & [g_{6}, g_{2}]  &= g_{7}, \\ 
 & & g_{7}^{2} &= 1\rangle. \\ 
\end{aligned}
\]
We add 15 tails to the presentation as to form a quotient of the universal central extension of the system: 
$g_{1}^{2} =  t_{1}$,
$g_{2}^{2} =  t_{2}$,
$[g_{2}, g_{1}] = g_{4} t_{3}$,
$g_{3}^{2} =  t_{4}$,
$[g_{3}, g_{1}] = g_{7} t_{5}$,
$g_{4}^{2} = g_{5}g_{6} t_{6}$,
$[g_{4}, g_{1}] = g_{5} t_{7}$,
$[g_{4}, g_{2}] = g_{5} t_{8}$,
$g_{5}^{2} = g_{6}g_{7} t_{9}$,
$[g_{5}, g_{1}] = g_{6} t_{10}$,
$[g_{5}, g_{2}] = g_{6} t_{11}$,
$g_{6}^{2} = g_{7} t_{12}$,
$[g_{6}, g_{1}] = g_{7} t_{13}$,
$[g_{6}, g_{2}] = g_{7} t_{14}$,
$g_{7}^{2} =  t_{15}$.
Carrying out consistency checks gives the following relations between the tails:
\[
\begin{aligned}
g_{5}(g_{2} g_{1}) & = (g_{5} g_{2}) g_{1}  & \Longrightarrow & & t_{13}t_{14}^{-1} & = 1 \\
g_{4}(g_{2} g_{1}) & = (g_{4} g_{2}) g_{1}  & \Longrightarrow & & t_{10}t_{11}^{-1} & = 1 \\
g_{6}^2 g_{2} & = g_{6} (g_{6} g_{2})& \Longrightarrow & & t_{14}^{2}t_{15} & = 1 \\
g_{5}^2 g_{2} & = g_{5} (g_{5} g_{2})& \Longrightarrow & & t_{11}^{2}t_{12}t_{14}^{-1} & = 1 \\
g_{4}^2 g_{2} & = g_{4} (g_{4} g_{2})& \Longrightarrow & & t_{8}^{2}t_{9}t_{11}^{-1}t_{14}^{-1} & = 1 \\
g_{4}^2 g_{1} & = g_{4} (g_{4} g_{1})& \Longrightarrow & & t_{7}^{2}t_{9}t_{10}^{-1}t_{13}^{-1} & = 1 \\
g_{3}^2 g_{1} & = g_{3} (g_{3} g_{1})& \Longrightarrow & & t_{5}^{2}t_{15} & = 1 \\
g_{2}^2 g_{1} & = g_{2} (g_{2} g_{1})& \Longrightarrow & & t_{3}^{2}t_{6}t_{8}t_{9}t_{12}t_{15} & = 1 \\
g_{2} g_{1}^{2} & = (g_{2} g_{1}) g_{1}& \Longrightarrow & & t_{3}^{2}t_{6}t_{7}t_{9}t_{12}t_{15} & = 1 \\
\end{aligned}
\]
Scanning through the conjugacy class representatives of $G$ and the generators of their centralizers, we obtain the following relations induced on the tails:
\[
\begin{aligned}
{[g_{3} g_{6} g_{7} , \, g_{1} ]}_G & = 1 & \Longrightarrow & & t_{5}t_{13}t_{15} & = 1 \\
\end{aligned}
\]
Collecting the coefficients of these relations into a matrix yields
\[
T = \bordermatrix{
{} & t_{1} & t_{2} & t_{3} & t_{4} & t_{5} & t_{6} & t_{7} & t_{8} & t_{9} & t_{10} & t_{11} & t_{12} & t_{13} & t_{14} & t_{15} \cr
{} &  &  & 2 &  &  & 1 &  & 1 & 1 &  &  & 1 &  &  & 1 \cr
{} &  &  &  &  & 1 &  &  &  &  &  &  &  &  & 1 & 1 \cr
{} &  &  &  &  &  &  & 1 & 1 & 1 &  & 1 & 1 &  &  & 1 \cr
{} &  &  &  &  &  &  &  & 2 & 1 &  & 1 & 1 &  &  & 1 \cr
{} &  &  &  &  &  &  &  &  &  & 1 & 1 & 1 &  & 1 & 1 \cr
{} &  &  &  &  &  &  &  &  &  &  & 2 & 1 &  & 1 & 1 \cr
{} &  &  &  &  &  &  &  &  &  &  &  &  & 1 & 1 & 1 \cr
{} &  &  &  &  &  &  &  &  &  &  &  &  &  & 2 & 1 \cr
}.
\]
It follows readily that the nontrivial elementary divisors of the Smith normal form of $T$ are all equal to $1$. The torsion subgroup of the group generated by the tails is thus trivial, thereby showing $\B_0(G) = 1$.


\item \label{number:110} 
Let the group $G$ be the representative of this family given by the presentation
\[
\begin{aligned}
\langle g_{1}, \,g_{2}, \,g_{3}, \,g_{4}, \,g_{5}, \,g_{6}, \,g_{7} & \mid & g_{1}^{2} &= g_{4}, \\ 
 & & g_{2}^{2} &= g_{3}g_{5}g_{6}, & [g_{2}, g_{1}]  &= g_{3}, \\ 
 & & g_{3}^{2} &= g_{5}g_{6}g_{7}, & [g_{3}, g_{1}]  &= g_{5}, \\ 
 & & g_{4}^{2} &= 1, & [g_{4}, g_{2}]  &= g_{7}, \\ 
 & & g_{5}^{2} &= g_{6}g_{7}, & [g_{5}, g_{1}]  &= g_{6}, \\ 
 & & g_{6}^{2} &= g_{7}, & [g_{6}, g_{1}]  &= g_{7}, \\ 
 & & g_{7}^{2} &= 1\rangle. \\ 
\end{aligned}
\]
We add 12 tails to the presentation as to form a quotient of the universal central extension of the system: 
$g_{1}^{2} = g_{4} t_{1}$,
$g_{2}^{2} = g_{3}g_{5}g_{6} t_{2}$,
$[g_{2}, g_{1}] = g_{3} t_{3}$,
$g_{3}^{2} = g_{5}g_{6}g_{7} t_{4}$,
$[g_{3}, g_{1}] = g_{5} t_{5}$,
$g_{4}^{2} =  t_{6}$,
$[g_{4}, g_{2}] = g_{7} t_{7}$,
$g_{5}^{2} = g_{6}g_{7} t_{8}$,
$[g_{5}, g_{1}] = g_{6} t_{9}$,
$g_{6}^{2} = g_{7} t_{10}$,
$[g_{6}, g_{1}] = g_{7} t_{11}$,
$g_{7}^{2} =  t_{12}$.
Carrying out consistency checks gives the following relations between the tails:
\[
\begin{aligned}
g_{6}^2 g_{1} & = g_{6} (g_{6} g_{1})& \Longrightarrow & & t_{11}^{2}t_{12} & = 1 \\
g_{5}^2 g_{1} & = g_{5} (g_{5} g_{1})& \Longrightarrow & & t_{9}^{2}t_{10}t_{11}^{-1} & = 1 \\
g_{4}^2 g_{2} & = g_{4} (g_{4} g_{2})& \Longrightarrow & & t_{7}^{2}t_{12} & = 1 \\
g_{3}^2 g_{1} & = g_{3} (g_{3} g_{1})& \Longrightarrow & & t_{5}^{2}t_{8}t_{9}^{-1}t_{11}^{-1} & = 1 \\
g_{2}^2 g_{1} & = g_{2} (g_{2} g_{1})& \Longrightarrow & & t_{3}^{2}t_{4}t_{5}^{-1}t_{9}^{-1}t_{11}^{-1} & = 1 \\
g_{2} g_{1}^{2} & = (g_{2} g_{1}) g_{1}& \Longrightarrow & & t_{3}^{2}t_{4}t_{5}t_{7}t_{8}t_{10}t_{12}^{2} & = 1 \\
\end{aligned}
\]
Scanning through the conjugacy class representatives of $G$ and the generators of their centralizers, we see that no new relations are imposed.
Collecting the coefficients of these relations into a matrix yields
\[
T = \bordermatrix{
{} & t_{1} & t_{2} & t_{3} & t_{4} & t_{5} & t_{6} & t_{7} & t_{8} & t_{9} & t_{10} & t_{11} & t_{12} \cr
{} &  &  & 2 & 1 & 1 &  &  & 1 &  & 1 & 1 & 2 \cr
{} &  &  &  &  & 2 &  &  & 1 & 1 & 1 &  & 1 \cr
{} &  &  &  &  &  &  & 1 &  &  &  & 1 & 1 \cr
{} &  &  &  &  &  &  &  &  & 2 & 1 & 1 & 1 \cr
{} &  &  &  &  &  &  &  &  &  &  & 2 & 1 \cr
}.
\]
It follows readily that the nontrivial elementary divisors of the Smith normal form of $T$ are all equal to $1$. The torsion subgroup of the group generated by the tails is thus trivial, thereby showing $\B_0(G) = 1$.


\item \label{number:111} 
Let the group $G$ be the representative of this family given by the presentation
\[
\begin{aligned}
\langle g_{1}, \,g_{2}, \,g_{3}, \,g_{4}, \,g_{5}, \,g_{6}, \,g_{7} & \mid & g_{1}^{2} &= g_{4}, \\ 
 & & g_{2}^{2} &= 1, & [g_{2}, g_{1}]  &= g_{3}, \\ 
 & & g_{3}^{2} &= 1, & [g_{3}, g_{1}]  &= g_{5}, \\ 
 & & g_{4}^{2} &= 1, & [g_{4}, g_{2}]  &= g_{5}g_{7}, & [g_{4}, g_{3}]  &= g_{6}g_{7}, \\ 
 & & g_{5}^{2} &= g_{7}, & [g_{5}, g_{1}]  &= g_{6}, & [g_{5}, g_{2}]  &= g_{7}, & [g_{5}, g_{3}]  &= g_{7}, & [g_{5}, g_{4}]  &= g_{7}, \\ 
 & & g_{6}^{2} &= 1, & [g_{6}, g_{1}]  &= g_{7}, & [g_{6}, g_{2}]  &= g_{7}, \\ 
 & & g_{7}^{2} &= 1\rangle. \\ 
\end{aligned}
\]
We add 17 tails to the presentation as to form a quotient of the universal central extension of the system: 
$g_{1}^{2} = g_{4} t_{1}$,
$g_{2}^{2} =  t_{2}$,
$[g_{2}, g_{1}] = g_{3} t_{3}$,
$g_{3}^{2} =  t_{4}$,
$[g_{3}, g_{1}] = g_{5} t_{5}$,
$g_{4}^{2} =  t_{6}$,
$[g_{4}, g_{2}] = g_{5}g_{7} t_{7}$,
$[g_{4}, g_{3}] = g_{6}g_{7} t_{8}$,
$g_{5}^{2} = g_{7} t_{9}$,
$[g_{5}, g_{1}] = g_{6} t_{10}$,
$[g_{5}, g_{2}] = g_{7} t_{11}$,
$[g_{5}, g_{3}] = g_{7} t_{12}$,
$[g_{5}, g_{4}] = g_{7} t_{13}$,
$g_{6}^{2} =  t_{14}$,
$[g_{6}, g_{1}] = g_{7} t_{15}$,
$[g_{6}, g_{2}] = g_{7} t_{16}$,
$g_{7}^{2} =  t_{17}$.
Carrying out consistency checks gives the following relations between the tails:
\[
\begin{aligned}
g_{5}(g_{2} g_{1}) & = (g_{5} g_{2}) g_{1}  & \Longrightarrow & & t_{12}^{-1}t_{16}^{-1}t_{17}^{-1} & = 1 \\
g_{4}(g_{3} g_{2}) & = (g_{4} g_{3}) g_{2}  & \Longrightarrow & & t_{12}^{-1}t_{16} & = 1 \\
g_{4}(g_{3} g_{1}) & = (g_{4} g_{3}) g_{1}  & \Longrightarrow & & t_{13}t_{15}t_{17} & = 1 \\
g_{4}(g_{2} g_{1}) & = (g_{4} g_{2}) g_{1}  & \Longrightarrow & & t_{8}^{-1}t_{10}t_{12}^{-1}t_{17}^{-1} & = 1 \\
g_{3}(g_{2} g_{1}) & = (g_{3} g_{2}) g_{1}  & \Longrightarrow & & t_{11}^{-1}t_{12}^{-1}t_{17}^{-1} & = 1 \\
g_{6}^2 g_{1} & = g_{6} (g_{6} g_{1})& \Longrightarrow & & t_{15}^{2}t_{17} & = 1 \\
g_{5}^2 g_{1} & = g_{5} (g_{5} g_{1})& \Longrightarrow & & t_{10}^{2}t_{14} & = 1 \\
g_{4}^2 g_{2} & = g_{4} (g_{4} g_{2})& \Longrightarrow & & t_{7}^{2}t_{9}t_{13}t_{17}^{2} & = 1 \\
g_{3}^2 g_{1} & = g_{3} (g_{3} g_{1})& \Longrightarrow & & t_{5}^{2}t_{9}t_{12}t_{17} & = 1 \\
g_{2}^2 g_{1} & = g_{2} (g_{2} g_{1})& \Longrightarrow & & t_{3}^{2}t_{4} & = 1 \\
g_{4} g_{2}^{2} & = (g_{4} g_{2}) g_{2}& \Longrightarrow & & t_{7}^{2}t_{9}t_{11}t_{17}^{2} & = 1 \\
g_{2} g_{1}^{2} & = (g_{2} g_{1}) g_{1}& \Longrightarrow & & t_{3}^{2}t_{4}t_{5}t_{7}t_{8}^{2}t_{9}t_{12}^{2}t_{14}t_{17}^{3} & = 1 \\
\end{aligned}
\]
Scanning through the conjugacy class representatives of $G$ and the generators of their centralizers, we see that no new relations are imposed.
Collecting the coefficients of these relations into a matrix yields
\[
T = \bordermatrix{
{} & t_{1} & t_{2} & t_{3} & t_{4} & t_{5} & t_{6} & t_{7} & t_{8} & t_{9} & t_{10} & t_{11} & t_{12} & t_{13} & t_{14} & t_{15} & t_{16} & t_{17} \cr
{} &  &  & 2 & 1 &  &  &  &  &  &  &  &  &  &  &  &  &  \cr
{} &  &  &  &  & 1 &  & 1 &  & 1 &  &  &  &  &  &  &  & 1 \cr
{} &  &  &  &  &  &  & 2 &  & 1 &  &  &  &  &  &  & 1 & 2 \cr
{} &  &  &  &  &  &  &  & 1 &  & 1 &  &  &  & 1 &  & 1 & 1 \cr
{} &  &  &  &  &  &  &  &  &  & 2 &  &  &  & 1 &  &  &  \cr
{} &  &  &  &  &  &  &  &  &  &  & 1 &  &  &  &  & 1 & 1 \cr
{} &  &  &  &  &  &  &  &  &  &  &  & 1 &  &  &  & 1 & 1 \cr
{} &  &  &  &  &  &  &  &  &  &  &  &  & 1 &  &  & 1 & 1 \cr
{} &  &  &  &  &  &  &  &  &  &  &  &  &  &  & 1 & 1 & 1 \cr
{} &  &  &  &  &  &  &  &  &  &  &  &  &  &  &  & 2 & 1 \cr
}.
\]
It follows readily that the nontrivial elementary divisors of the Smith normal form of $T$ are all equal to $1$. The torsion subgroup of the group generated by the tails is thus trivial, thereby showing $\B_0(G) = 1$.


\item \label{number:112} 
Let the group $G$ be the representative of this family given by the presentation
\[
\begin{aligned}
\langle g_{1}, \,g_{2}, \,g_{3}, \,g_{4}, \,g_{5}, \,g_{6}, \,g_{7} & \mid & g_{1}^{2} &= g_{4}, \\ 
 & & g_{2}^{2} &= 1, & [g_{2}, g_{1}]  &= g_{3}, \\ 
 & & g_{3}^{2} &= g_{7}, & [g_{3}, g_{1}]  &= g_{5}, & [g_{3}, g_{2}]  &= g_{7}, \\ 
 & & g_{4}^{2} &= 1, & [g_{4}, g_{2}]  &= g_{5}, & [g_{4}, g_{3}]  &= g_{6}g_{7}, \\ 
 & & g_{5}^{2} &= g_{7}, & [g_{5}, g_{1}]  &= g_{6}, & [g_{5}, g_{2}]  &= g_{7}, & [g_{5}, g_{3}]  &= g_{7}, & [g_{5}, g_{4}]  &= g_{7}, \\ 
 & & g_{6}^{2} &= 1, & [g_{6}, g_{1}]  &= g_{7}, & [g_{6}, g_{2}]  &= g_{7}, \\ 
 & & g_{7}^{2} &= 1\rangle. \\ 
\end{aligned}
\]
We add 18 tails to the presentation as to form a quotient of the universal central extension of the system: 
$g_{1}^{2} = g_{4} t_{1}$,
$g_{2}^{2} =  t_{2}$,
$[g_{2}, g_{1}] = g_{3} t_{3}$,
$g_{3}^{2} = g_{7} t_{4}$,
$[g_{3}, g_{1}] = g_{5} t_{5}$,
$[g_{3}, g_{2}] = g_{7} t_{6}$,
$g_{4}^{2} =  t_{7}$,
$[g_{4}, g_{2}] = g_{5} t_{8}$,
$[g_{4}, g_{3}] = g_{6}g_{7} t_{9}$,
$g_{5}^{2} = g_{7} t_{10}$,
$[g_{5}, g_{1}] = g_{6} t_{11}$,
$[g_{5}, g_{2}] = g_{7} t_{12}$,
$[g_{5}, g_{3}] = g_{7} t_{13}$,
$[g_{5}, g_{4}] = g_{7} t_{14}$,
$g_{6}^{2} =  t_{15}$,
$[g_{6}, g_{1}] = g_{7} t_{16}$,
$[g_{6}, g_{2}] = g_{7} t_{17}$,
$g_{7}^{2} =  t_{18}$.
Carrying out consistency checks gives the following relations between the tails:
\[
\begin{aligned}
g_{5}(g_{2} g_{1}) & = (g_{5} g_{2}) g_{1}  & \Longrightarrow & & t_{13}^{-1}t_{17}^{-1}t_{18}^{-1} & = 1 \\
g_{4}(g_{3} g_{2}) & = (g_{4} g_{3}) g_{2}  & \Longrightarrow & & t_{13}^{-1}t_{17} & = 1 \\
g_{4}(g_{3} g_{1}) & = (g_{4} g_{3}) g_{1}  & \Longrightarrow & & t_{14}t_{16}t_{18} & = 1 \\
g_{4}(g_{2} g_{1}) & = (g_{4} g_{2}) g_{1}  & \Longrightarrow & & t_{9}^{-1}t_{11}t_{13}^{-1}t_{18}^{-1} & = 1 \\
g_{3}(g_{2} g_{1}) & = (g_{3} g_{2}) g_{1}  & \Longrightarrow & & t_{12}^{-1}t_{13}^{-1}t_{18}^{-1} & = 1 \\
g_{6}^2 g_{1} & = g_{6} (g_{6} g_{1})& \Longrightarrow & & t_{16}^{2}t_{18} & = 1 \\
g_{5}^2 g_{1} & = g_{5} (g_{5} g_{1})& \Longrightarrow & & t_{11}^{2}t_{15} & = 1 \\
g_{4}^2 g_{2} & = g_{4} (g_{4} g_{2})& \Longrightarrow & & t_{8}^{2}t_{10}t_{14}t_{18} & = 1 \\
g_{3}^2 g_{2} & = g_{3} (g_{3} g_{2})& \Longrightarrow & & t_{6}^{2}t_{18} & = 1 \\
g_{3}^2 g_{1} & = g_{3} (g_{3} g_{1})& \Longrightarrow & & t_{5}^{2}t_{10}t_{13}t_{18} & = 1 \\
g_{2}^2 g_{1} & = g_{2} (g_{2} g_{1})& \Longrightarrow & & t_{3}^{2}t_{4}t_{6}t_{18} & = 1 \\
g_{4} g_{2}^{2} & = (g_{4} g_{2}) g_{2}& \Longrightarrow & & t_{8}^{2}t_{10}t_{12}t_{18} & = 1 \\
g_{2} g_{1}^{2} & = (g_{2} g_{1}) g_{1}& \Longrightarrow & & t_{3}^{2}t_{4}t_{5}t_{8}t_{9}^{2}t_{10}t_{13}^{2}t_{15}t_{18}^{3} & = 1 \\
\end{aligned}
\]
Scanning through the conjugacy class representatives of $G$ and the generators of their centralizers, we obtain the following relations induced on the tails:
\[
\begin{aligned}
{[g_{3} g_{5} g_{6} g_{7} , \, g_{2} g_{5} g_{7} ]}_G & = 1 & \Longrightarrow & & t_{6}t_{12}t_{13}^{-1}t_{17}t_{18} & = 1 \\
\end{aligned}
\]
Collecting the coefficients of these relations into a matrix yields
\[
T = \bordermatrix{
{} & t_{1} & t_{2} & t_{3} & t_{4} & t_{5} & t_{6} & t_{7} & t_{8} & t_{9} & t_{10} & t_{11} & t_{12} & t_{13} & t_{14} & t_{15} & t_{16} & t_{17} & t_{18} \cr
{} &  &  & 2 & 1 &  &  &  &  &  &  &  &  &  &  &  &  & 1 & 1 \cr
{} &  &  &  &  & 1 &  &  & 1 &  & 1 &  &  &  &  &  &  & 1 & 1 \cr
{} &  &  &  &  &  & 1 &  &  &  &  &  &  &  &  &  &  & 1 & 1 \cr
{} &  &  &  &  &  &  &  & 2 &  & 1 &  &  &  &  &  &  & 1 & 1 \cr
{} &  &  &  &  &  &  &  &  & 1 &  & 1 &  &  &  & 1 &  & 1 & 1 \cr
{} &  &  &  &  &  &  &  &  &  &  & 2 &  &  &  & 1 &  &  &  \cr
{} &  &  &  &  &  &  &  &  &  &  &  & 1 &  &  &  &  & 1 & 1 \cr
{} &  &  &  &  &  &  &  &  &  &  &  &  & 1 &  &  &  & 1 & 1 \cr
{} &  &  &  &  &  &  &  &  &  &  &  &  &  & 1 &  &  & 1 & 1 \cr
{} &  &  &  &  &  &  &  &  &  &  &  &  &  &  &  & 1 & 1 & 1 \cr
{} &  &  &  &  &  &  &  &  &  &  &  &  &  &  &  &  & 2 & 1 \cr
}.
\]
It follows readily that the nontrivial elementary divisors of the Smith normal form of $T$ are all equal to $1$. The torsion subgroup of the group generated by the tails is thus trivial, thereby showing $\B_0(G) = 1$.


\item \label{number:113} 
Let the group $G$ be the representative of this family given by the presentation
\[
\begin{aligned}
\langle g_{1}, \,g_{2}, \,g_{3}, \,g_{4}, \,g_{5}, \,g_{6}, \,g_{7} & \mid & g_{1}^{2} &= 1, \\ 
 & & g_{2}^{2} &= 1, & [g_{2}, g_{1}]  &= g_{3}, \\ 
 & & g_{3}^{2} &= g_{4}g_{5}, & [g_{3}, g_{1}]  &= g_{4}, & [g_{3}, g_{2}]  &= g_{4}, \\ 
 & & g_{4}^{2} &= g_{5}g_{6}, & [g_{4}, g_{1}]  &= g_{5}, & [g_{4}, g_{2}]  &= g_{5}, \\ 
 & & g_{5}^{2} &= g_{6}g_{7}, & [g_{5}, g_{1}]  &= g_{6}, & [g_{5}, g_{2}]  &= g_{6}, \\ 
 & & g_{6}^{2} &= g_{7}, & [g_{6}, g_{1}]  &= g_{7}, & [g_{6}, g_{2}]  &= g_{7}, \\ 
 & & g_{7}^{2} &= 1\rangle. \\ 
\end{aligned}
\]
We add 16 tails to the presentation as to form a quotient of the universal central extension of the system: 
$g_{1}^{2} =  t_{1}$,
$g_{2}^{2} =  t_{2}$,
$[g_{2}, g_{1}] = g_{3} t_{3}$,
$g_{3}^{2} = g_{4}g_{5} t_{4}$,
$[g_{3}, g_{1}] = g_{4} t_{5}$,
$[g_{3}, g_{2}] = g_{4} t_{6}$,
$g_{4}^{2} = g_{5}g_{6} t_{7}$,
$[g_{4}, g_{1}] = g_{5} t_{8}$,
$[g_{4}, g_{2}] = g_{5} t_{9}$,
$g_{5}^{2} = g_{6}g_{7} t_{10}$,
$[g_{5}, g_{1}] = g_{6} t_{11}$,
$[g_{5}, g_{2}] = g_{6} t_{12}$,
$g_{6}^{2} = g_{7} t_{13}$,
$[g_{6}, g_{1}] = g_{7} t_{14}$,
$[g_{6}, g_{2}] = g_{7} t_{15}$,
$g_{7}^{2} =  t_{16}$.
Carrying out consistency checks gives the following relations between the tails:
\[
\begin{aligned}
g_{5}(g_{2} g_{1}) & = (g_{5} g_{2}) g_{1}  & \Longrightarrow & & t_{14}t_{15}^{-1} & = 1 \\
g_{4}(g_{2} g_{1}) & = (g_{4} g_{2}) g_{1}  & \Longrightarrow & & t_{11}t_{12}^{-1} & = 1 \\
g_{3}(g_{2} g_{1}) & = (g_{3} g_{2}) g_{1}  & \Longrightarrow & & t_{8}t_{9}^{-1} & = 1 \\
g_{6}^2 g_{2} & = g_{6} (g_{6} g_{2})& \Longrightarrow & & t_{15}^{2}t_{16} & = 1 \\
g_{5}^2 g_{2} & = g_{5} (g_{5} g_{2})& \Longrightarrow & & t_{12}^{2}t_{13}t_{15}^{-1} & = 1 \\
g_{4}^2 g_{2} & = g_{4} (g_{4} g_{2})& \Longrightarrow & & t_{9}^{2}t_{10}t_{12}^{-1}t_{15}^{-1} & = 1 \\
g_{3}^2 g_{2} & = g_{3} (g_{3} g_{2})& \Longrightarrow & & t_{6}^{2}t_{7}t_{9}^{-1}t_{12}^{-1} & = 1 \\
g_{3}^2 g_{1} & = g_{3} (g_{3} g_{1})& \Longrightarrow & & t_{5}^{2}t_{7}t_{8}^{-1}t_{11}^{-1} & = 1 \\
g_{2}^2 g_{1} & = g_{2} (g_{2} g_{1})& \Longrightarrow & & t_{3}^{2}t_{4}t_{6}t_{7}t_{10}t_{13}t_{16} & = 1 \\
g_{2} g_{1}^{2} & = (g_{2} g_{1}) g_{1}& \Longrightarrow & & t_{3}^{2}t_{4}t_{5}t_{7}t_{10}t_{13}t_{16} & = 1 \\
\end{aligned}
\]
Scanning through the conjugacy class representatives of $G$ and the generators of their centralizers, we see that no new relations are imposed.
Collecting the coefficients of these relations into a matrix yields
\[
T = \bordermatrix{
{} & t_{1} & t_{2} & t_{3} & t_{4} & t_{5} & t_{6} & t_{7} & t_{8} & t_{9} & t_{10} & t_{11} & t_{12} & t_{13} & t_{14} & t_{15} & t_{16} \cr
{} &  &  & 2 & 1 &  & 1 & 1 &  &  & 1 &  &  & 1 &  &  & 1 \cr
{} &  &  &  &  & 1 & 1 & 1 &  & 1 & 1 &  &  & 1 &  &  & 1 \cr
{} &  &  &  &  &  & 2 & 1 &  & 1 & 1 &  &  & 1 &  &  & 1 \cr
{} &  &  &  &  &  &  &  & 1 & 1 & 1 &  & 1 & 1 &  &  & 1 \cr
{} &  &  &  &  &  &  &  &  & 2 & 1 &  & 1 & 1 &  &  & 1 \cr
{} &  &  &  &  &  &  &  &  &  &  & 1 & 1 & 1 &  & 1 & 1 \cr
{} &  &  &  &  &  &  &  &  &  &  &  & 2 & 1 &  & 1 & 1 \cr
{} &  &  &  &  &  &  &  &  &  &  &  &  &  & 1 & 1 & 1 \cr
{} &  &  &  &  &  &  &  &  &  &  &  &  &  &  & 2 & 1 \cr
}.
\]
It follows readily that the nontrivial elementary divisors of the Smith normal form of $T$ are all equal to $1$. The torsion subgroup of the group generated by the tails is thus trivial, thereby showing $\B_0(G) = 1$.


\item \label{number:114} 
Let the group $G$ be the representative of this family given by the presentation
\[
\begin{aligned}
\langle g_{1}, \,g_{2}, \,g_{3}, \,g_{4}, \,g_{5}, \,g_{6}, \,g_{7} & \mid & g_{1}^{2} &= g_{4}, \\ 
 & & g_{2}^{2} &= 1, & [g_{2}, g_{1}]  &= g_{3}, \\ 
 & & g_{3}^{2} &= g_{6}, & [g_{3}, g_{1}]  &= g_{5}, & [g_{3}, g_{2}]  &= g_{6}, \\ 
 & & g_{4}^{2} &= 1, & [g_{4}, g_{2}]  &= g_{5}g_{6}g_{7}, & [g_{4}, g_{3}]  &= g_{6}g_{7}, \\ 
 & & g_{5}^{2} &= g_{7}, & [g_{5}, g_{1}]  &= g_{6}, & [g_{5}, g_{2}]  &= g_{7}, & [g_{5}, g_{4}]  &= g_{7}, \\ 
 & & g_{6}^{2} &= 1, & [g_{6}, g_{1}]  &= g_{7}, \\ 
 & & g_{7}^{2} &= 1\rangle. \\ 
\end{aligned}
\]
We add 16 tails to the presentation as to form a quotient of the universal central extension of the system: 
$g_{1}^{2} = g_{4} t_{1}$,
$g_{2}^{2} =  t_{2}$,
$[g_{2}, g_{1}] = g_{3} t_{3}$,
$g_{3}^{2} = g_{6} t_{4}$,
$[g_{3}, g_{1}] = g_{5} t_{5}$,
$[g_{3}, g_{2}] = g_{6} t_{6}$,
$g_{4}^{2} =  t_{7}$,
$[g_{4}, g_{2}] = g_{5}g_{6}g_{7} t_{8}$,
$[g_{4}, g_{3}] = g_{6}g_{7} t_{9}$,
$g_{5}^{2} = g_{7} t_{10}$,
$[g_{5}, g_{1}] = g_{6} t_{11}$,
$[g_{5}, g_{2}] = g_{7} t_{12}$,
$[g_{5}, g_{4}] = g_{7} t_{13}$,
$g_{6}^{2} =  t_{14}$,
$[g_{6}, g_{1}] = g_{7} t_{15}$,
$g_{7}^{2} =  t_{16}$.
Carrying out consistency checks gives the following relations between the tails:
\[
\begin{aligned}
g_{4}(g_{3} g_{1}) & = (g_{4} g_{3}) g_{1}  & \Longrightarrow & & t_{13}t_{15}t_{16} & = 1 \\
g_{4}(g_{2} g_{1}) & = (g_{4} g_{2}) g_{1}  & \Longrightarrow & & t_{9}^{-1}t_{11}t_{15} & = 1 \\
g_{3}(g_{2} g_{1}) & = (g_{3} g_{2}) g_{1}  & \Longrightarrow & & t_{12}^{-1}t_{15} & = 1 \\
g_{6}^2 g_{1} & = g_{6} (g_{6} g_{1})& \Longrightarrow & & t_{15}^{2}t_{16} & = 1 \\
g_{5}^2 g_{1} & = g_{5} (g_{5} g_{1})& \Longrightarrow & & t_{11}^{2}t_{14} & = 1 \\
g_{4}^2 g_{2} & = g_{4} (g_{4} g_{2})& \Longrightarrow & & t_{8}^{2}t_{10}t_{13}t_{14}t_{16}^{2} & = 1 \\
g_{3}^2 g_{2} & = g_{3} (g_{3} g_{2})& \Longrightarrow & & t_{6}^{2}t_{14} & = 1 \\
g_{3}^2 g_{1} & = g_{3} (g_{3} g_{1})& \Longrightarrow & & t_{5}^{2}t_{10}t_{15}^{-1} & = 1 \\
g_{2}^2 g_{1} & = g_{2} (g_{2} g_{1})& \Longrightarrow & & t_{3}^{2}t_{4}t_{6}t_{14} & = 1 \\
g_{2} g_{1}^{2} & = (g_{2} g_{1}) g_{1}& \Longrightarrow & & t_{3}^{2}t_{4}t_{5}t_{8}t_{9}^{2}t_{10}t_{14}^{2}t_{16}^{2} & = 1 \\
\end{aligned}
\]
Scanning through the conjugacy class representatives of $G$ and the generators of their centralizers, we see that no new relations are imposed.
Collecting the coefficients of these relations into a matrix yields
\[
T = \bordermatrix{
{} & t_{1} & t_{2} & t_{3} & t_{4} & t_{5} & t_{6} & t_{7} & t_{8} & t_{9} & t_{10} & t_{11} & t_{12} & t_{13} & t_{14} & t_{15} & t_{16} \cr
{} &  &  & 2 & 1 &  & 1 &  &  &  &  &  &  &  & 1 &  &  \cr
{} &  &  &  &  & 1 & 1 &  & 1 &  & 1 &  &  &  & 1 &  & 1 \cr
{} &  &  &  &  &  & 2 &  &  &  &  &  &  &  & 1 &  &  \cr
{} &  &  &  &  &  &  &  & 2 &  & 1 &  &  &  & 1 & 1 & 2 \cr
{} &  &  &  &  &  &  &  &  & 1 &  & 1 &  &  & 1 & 1 & 1 \cr
{} &  &  &  &  &  &  &  &  &  &  & 2 &  &  & 1 &  &  \cr
{} &  &  &  &  &  &  &  &  &  &  &  & 1 &  &  & 1 & 1 \cr
{} &  &  &  &  &  &  &  &  &  &  &  &  & 1 &  & 1 & 1 \cr
{} &  &  &  &  &  &  &  &  &  &  &  &  &  &  & 2 & 1 \cr
}.
\]
A change of basis according to the transition matrix (specifying expansions of $t_i^{*}$ by $t_j$)
\[
\bordermatrix{
{} & t_{1}^{*} & t_{2}^{*} & t_{3}^{*} & t_{4}^{*} & t_{5}^{*} & t_{6}^{*} & t_{7}^{*} & t_{8}^{*} & t_{9}^{*} & t_{10}^{*} & t_{11}^{*} & t_{12}^{*} & t_{13}^{*} & t_{14}^{*} & t_{15}^{*} & t_{16}^{*} \cr
t_{1} &  &  &  &  &  &  &  &  &  & -1 & -9 & 1 & 1 & 1 &  &  \cr
t_{2} &  &  &  &  &  &  &  &  &  &  & -2 & 1 & 1 &  &  & -1 \cr
t_{3} & 4 & 2 & 4 & 8 & 4 & 8 &  &  & 6 & 2 & 9 & -1 & -1 & -1 &  &  \cr
t_{4} & 2 & 1 & 2 & 4 & 2 & 4 &  &  & 3 &  & 2 &  &  &  &  &  \cr
t_{5} & -3 & -2 & -4 & -8 & -4 & -8 &  &  & -6 & -1 &  &  &  &  &  &  \cr
t_{6} & 3 & 3 & 4 & 8 & 4 & 8 &  &  & 6 & 1 & 2 & -1 & -1 &  &  & 1 \cr
t_{7} &  &  &  &  &  &  &  &  &  &  &  & -1 &  &  &  & 1 \cr
t_{8} & -9 & -4 & -8 & -18 & -12 & -20 &  &  & -14 &  &  &  &  &  &  & -1 \cr
t_{9} & 3 & 2 & 4 & 8 & 5 & 8 &  &  & 6 & 1 &  &  &  &  &  &  \cr
t_{10} & -6 & -3 & -6 & -13 & -8 & -14 &  &  & -10 & -1 & -3 &  &  &  &  &  \cr
t_{11} & 1 &  &  & 2 & 1 & 2 &  &  & 1 &  & 1 &  &  &  &  &  \cr
t_{12} & 9 & 4 & 8 & 18 & 12 & 20 & 1 &  & 14 &  &  & 1 &  &  &  &  \cr
t_{13} & 2 &  & 2 & 3 & 3 & 5 & -1 &  & 3 &  &  &  & 1 &  &  &  \cr
t_{14} &  & 1 & 1 & 2 &  & 1 &  &  & 1 &  &  &  &  & 1 &  &  \cr
t_{15} & 19 & 9 & 16 & 38 & 24 & 41 &  & 2 & 29 &  &  &  &  &  & 1 &  \cr
t_{16} & 9 & 4 & 8 & 18 & 12 & 20 &  & 1 & 14 &  &  &  &  &  &  & 1 \cr
}
\]
shows that the nontrivial elementary divisors of the Smith normal form of $T$ are $1$, $1$, $1$, $1$, $1$, $1$, $1$, $1$, $2$.  The element corresponding to the divisor that is greater than $1$ is $t_{9}^{*}$. This already gives
\[
\B_0(G) \cong \langle t_{9}^{*}  \mid {t_{9}^{*}}^{2} \rangle.
\]

We now deal with explicitly identifying the nonuniversal commutator relation generating $\B_0(G)$.
First, factor out by the tails $t_{i}^{*}$ whose corresponding elementary divisors are either trivial or $1$. Transforming the situation back to the original tails $t_i$, this amounts to the nontrivial expansions given by
\[
\bordermatrix{
{} & t_{2} & t_{4} & t_{5} & t_{6} \cr
t_{9}^{*} & 1 & 1 & 1 & 1 \cr
}
\]
and all the other tails $t_i$ are trivial. We thus obtain a commutativity preserving central extension of the group $G$, given by the presentation
\[
\begin{aligned}
\langle g_{1}, \,g_{2}, \,g_{3}, \,g_{4}, \,g_{5}, \,g_{6}, \,g_{7}, \,t_{9}^{*} & \mid & g_{1}^{2} &= g_{4}, \\ 
 & & g_{2}^{2} &= t_{9}^{*} , & [g_{2}, g_{1}]  &= g_{3}, \\ 
 & & g_{3}^{2} &= g_{6}t_{9}^{*} , & [g_{3}, g_{1}]  &= g_{5}t_{9}^{*} , & [g_{3}, g_{2}]  &= g_{6}t_{9}^{*} , \\ 
 & & g_{4}^{2} &= 1, & [g_{4}, g_{2}]  &= g_{5}g_{6}g_{7}, & [g_{4}, g_{3}]  &= g_{6}g_{7}, \\ 
 & & g_{5}^{2} &= g_{7}, & [g_{5}, g_{1}]  &= g_{6}, & [g_{5}, g_{2}]  &= g_{7}, & [g_{5}, g_{4}]  &= g_{7}, \\ 
 & & g_{6}^{2} &= 1, & [g_{6}, g_{1}]  &= g_{7}, \\ 
 & & g_{7}^{2} &= 1, \\ 
 & & {t_{9}^{*}}^{2} &= 1  \rangle,
\end{aligned}
\]
whence the nonuniversal commutator relation is identified as
\[
t_{9}^{*}  = [g_{3}, g_{2}] [g_{5}, g_{1}]^{-1}.  \quad 
\]


\item \label{number:115} 
Let the group $G$ be the representative of this family given by the presentation
\[
\begin{aligned}
\langle g_{1}, \,g_{2}, \,g_{3}, \,g_{4}, \,g_{5}, \,g_{6}, \,g_{7} & \mid & g_{1}^{2} &= g_{4}, \\ 
 & & g_{2}^{2} &= g_{6}, & [g_{2}, g_{1}]  &= g_{3}, \\ 
 & & g_{3}^{2} &= g_{6}, & [g_{3}, g_{1}]  &= g_{5}, & [g_{3}, g_{2}]  &= g_{6}g_{7}, \\ 
 & & g_{4}^{2} &= 1, & [g_{4}, g_{2}]  &= g_{5}g_{6}g_{7}, & [g_{4}, g_{3}]  &= g_{6}g_{7}, \\ 
 & & g_{5}^{2} &= g_{7}, & [g_{5}, g_{1}]  &= g_{6}, & [g_{5}, g_{2}]  &= g_{7}, & [g_{5}, g_{4}]  &= g_{7}, \\ 
 & & g_{6}^{2} &= 1, & [g_{6}, g_{1}]  &= g_{7}, \\ 
 & & g_{7}^{2} &= 1\rangle. \\ 
\end{aligned}
\]
We add 16 tails to the presentation as to form a quotient of the universal central extension of the system: 
$g_{1}^{2} = g_{4} t_{1}$,
$g_{2}^{2} = g_{6} t_{2}$,
$[g_{2}, g_{1}] = g_{3} t_{3}$,
$g_{3}^{2} = g_{6} t_{4}$,
$[g_{3}, g_{1}] = g_{5} t_{5}$,
$[g_{3}, g_{2}] = g_{6}g_{7} t_{6}$,
$g_{4}^{2} =  t_{7}$,
$[g_{4}, g_{2}] = g_{5}g_{6}g_{7} t_{8}$,
$[g_{4}, g_{3}] = g_{6}g_{7} t_{9}$,
$g_{5}^{2} = g_{7} t_{10}$,
$[g_{5}, g_{1}] = g_{6} t_{11}$,
$[g_{5}, g_{2}] = g_{7} t_{12}$,
$[g_{5}, g_{4}] = g_{7} t_{13}$,
$g_{6}^{2} =  t_{14}$,
$[g_{6}, g_{1}] = g_{7} t_{15}$,
$g_{7}^{2} =  t_{16}$.
Carrying out consistency checks gives the following relations between the tails:
\[
\begin{aligned}
g_{4}(g_{3} g_{1}) & = (g_{4} g_{3}) g_{1}  & \Longrightarrow & & t_{13}t_{15}t_{16} & = 1 \\
g_{4}(g_{2} g_{1}) & = (g_{4} g_{2}) g_{1}  & \Longrightarrow & & t_{9}^{-1}t_{11}t_{15} & = 1 \\
g_{3}(g_{2} g_{1}) & = (g_{3} g_{2}) g_{1}  & \Longrightarrow & & t_{12}^{-1}t_{15} & = 1 \\
g_{6}^2 g_{1} & = g_{6} (g_{6} g_{1})& \Longrightarrow & & t_{15}^{2}t_{16} & = 1 \\
g_{5}^2 g_{1} & = g_{5} (g_{5} g_{1})& \Longrightarrow & & t_{11}^{2}t_{14} & = 1 \\
g_{4}^2 g_{2} & = g_{4} (g_{4} g_{2})& \Longrightarrow & & t_{8}^{2}t_{10}t_{13}t_{14}t_{16}^{2} & = 1 \\
g_{3}^2 g_{2} & = g_{3} (g_{3} g_{2})& \Longrightarrow & & t_{6}^{2}t_{14}t_{16} & = 1 \\
g_{3}^2 g_{1} & = g_{3} (g_{3} g_{1})& \Longrightarrow & & t_{5}^{2}t_{10}t_{15}^{-1} & = 1 \\
g_{2}^2 g_{1} & = g_{2} (g_{2} g_{1})& \Longrightarrow & & t_{3}^{2}t_{4}t_{6}t_{14}t_{15}^{-1} & = 1 \\
g_{2} g_{1}^{2} & = (g_{2} g_{1}) g_{1}& \Longrightarrow & & t_{3}^{2}t_{4}t_{5}t_{8}t_{9}^{2}t_{10}t_{14}^{2}t_{16}^{2} & = 1 \\
\end{aligned}
\]
Scanning through the conjugacy class representatives of $G$ and the generators of their centralizers, we obtain the following relations induced on the tails:
\[
\begin{aligned}
{[g_{3} g_{5} g_{6} , \, g_{2} g_{4} g_{5} g_{7} ]}_G & = 1 & \Longrightarrow & & t_{6}t_{9}^{-1}t_{12}t_{13}t_{16} & = 1 \\
\end{aligned}
\]
Collecting the coefficients of these relations into a matrix yields
\[
T = \bordermatrix{
{} & t_{1} & t_{2} & t_{3} & t_{4} & t_{5} & t_{6} & t_{7} & t_{8} & t_{9} & t_{10} & t_{11} & t_{12} & t_{13} & t_{14} & t_{15} & t_{16} \cr
{} &  &  & 2 & 1 &  &  &  &  &  &  & 1 &  &  & 1 &  &  \cr
{} &  &  &  &  & 1 &  &  & 1 &  & 1 & 1 &  &  & 1 &  & 1 \cr
{} &  &  &  &  &  & 1 &  &  &  &  & 1 &  &  & 1 & 1 & 1 \cr
{} &  &  &  &  &  &  &  & 2 &  & 1 &  &  &  & 1 & 1 & 2 \cr
{} &  &  &  &  &  &  &  &  & 1 &  & 1 &  &  & 1 & 1 & 1 \cr
{} &  &  &  &  &  &  &  &  &  &  & 2 &  &  & 1 &  &  \cr
{} &  &  &  &  &  &  &  &  &  &  &  & 1 &  &  & 1 & 1 \cr
{} &  &  &  &  &  &  &  &  &  &  &  &  & 1 &  & 1 & 1 \cr
{} &  &  &  &  &  &  &  &  &  &  &  &  &  &  & 2 & 1 \cr
}.
\]
It follows readily that the nontrivial elementary divisors of the Smith normal form of $T$ are all equal to $1$. The torsion subgroup of the group generated by the tails is thus trivial, thereby showing $\B_0(G) = 1$.

\end{enumerate}

\end{document}